\documentclass[11pt]{article} 
\usepackage[utf8]{inputenc}
\usepackage{tikz} %%% parallel arrows 
\usepackage[all]{xy}
\usepackage{amsmath,amsthm, amssymb,amsfonts}
\numberwithin{equation}{subsection} %%% ajoute une numérotation des équations avec les sections %%%
\usepackage{authblk}
\xyoption{arc}
\usepackage{eucal}%%% mathcal change
\usepackage{mathrsfs}
\usepackage{verbatim}
\usepackage{graphicx} %
\usepackage{hyperref}
\usepackage{enumitem} % % % % 
\usepackage{extpfeil} %%%% sert à nommer les fleches avec deux tetes dans un sens : fibrations...
\usepackage{dsfont}
\usepackage[T1]{fontenc}
\usepackage[hmargin=1in,vmargin=1in]{geometry}
\newtheorem{thm}{Theorem}[section]
\newtheorem{prop}[thm]{Proposition}
\newtheorem{pdef}[thm]{Proposition-Definition}%%% proposition definition%%%
\newtheorem{lem}[thm]{Lemma}
\newtheorem*{thmsans}{Theorem}%%% thm sans 

\newtheorem{cor}[thm]{Corollary}

\newtheoremstyle{bidule}% name
{3pt}% Space above
{3pt}% Space below
{}% Body font
{}% Indent amount
{\scshape}% Theorem head font
{.}% Punctuation after theorem head
{.5em}% Space after theorem head
{}% Theorem head spec (can be left empty, meaning `normal')
\newtheorem{df}[thm]{Definition}
\theoremstyle{definition}
\newtheorem{rmk}{Remark}[section]
\newtheorem{ex}[thm]{Example}

\newtheorem*{term}{Terminology}
\newtheorem{hypo}{Hypothesis}[subsubsection] % % % hypothese 
\newtheorem*{note}{Note}
\newtheorem*{thank}{Acknowledgments}
\newtheorem*{warn}{Warning}
\newtheorem*{claim}{Claim}
\newtheorem{nota}{Notations}[subsubsection]
\newtheorem*{conv}{Conventions}
\newcommand{\E}{\mathscr{E}}
\newcommand{\Ev}{\tx{Ev}}
\newcommand{\colim}{\text{colim}}
\newcommand*{\dt}{{\mbox{\large\bfseries .}}} % % % bold point
\newcommand{\C}{\mathcal{C}}
\newcommand{\Ub}{\mathcal{U}}%%% Ub comme Oublie%%% forgetful functor%%%%%
\renewcommand{\Pr}{\mathcal{P}}%%% projection functor %%%% 
\newcommand{\Ca}{\mathcal{C}}

\newcommand{\Ma}{\mathcal{M}}
\renewcommand{\O}{\mathcal{O}}
\newcommand{\K}{\mathcal{K}}%%%ù K %% genericl locally presentable category %%%%
\newcommand{\Ar}{\text{Arr}}
\renewcommand{\Im}{\text{Im}}
\newcommand{\D}{\mathcal{D}}
\newcommand{\Ba}{\mathcal{B}}
\newcommand{\Ga}{\mathcal{G}}
\newcommand{\Xa}{\mathcal{X}}
\newcommand{\Ra}{\mathcal{R}}
\newcommand{\A}{\mathscr{A}}
\newcommand{\Aa}{\mathcal{A}}
\newcommand{\B}{\mathscr{B}}
\newcommand{\M}{\mathscr{M}}
\newcommand{\V}{\mathbb{V}}
\newcommand{\Ja}{\mathbf{J}} %%acyclic cofibrations%%%
\newcommand{\J}{\mathcal{J}} %% index category for limit and colimits%%%
\newcommand{\X}{\mathscr{X}}

\newcommand{\T}{\mathcal{T}}
\newcommand{\Ta}{\mathbf{T}} 
\newcommand{\Na}{\mathcal{Na}}
\newcommand{\Ea}{\mathcal{E}} %%%
\newcommand{\Qa}{\mathcal{Q}} 
\newcommand{\W}{\mathscr{W}}
\newcommand{\Fa}{\mathcal{F}}
\newcommand{\Nv}{\mathscr{N}}
\newcommand{\I}{\mathbf{I}}
\newcommand{\Un}{\mathbb{I}} %%%% unity in mset
\newcommand{\Po}{\mathbf{P}}
\renewcommand{\le}{\mathscr{L}}
\newcommand{\In}{\mathscr{I}}

\newcommand{\Za}{\mathcal{Z}}

\newcommand{\Rc}{\mathscr{R}}

\renewcommand{\P}{\mathscr{P}}
\renewcommand{\S}{\mathbb{S}}
\newcommand{\Y}{\mathscr{Y}}
\renewcommand{\L}{\text{L}} %%%% normal L 
\newcommand{\Fb}{\mathbf{F}}%%%%%% Free functor 
\newcommand{\Ya}{\mathcal{Y}}%%%%%% Yoneda functor of point
\newcommand{\Sim}{\mathscr{S}}
\newcommand*\circled[1]{%
  \tikz[baseline=(C.base)]\node[draw,circle,inner sep=0.3pt](C) {#1};\!
}
\newcommand{\co}{\epsilon} %%% cone %% on chosit epsilon comme dans le livre de Carlos
\newcommand{\ho}{\otimes_h} %%%  the horizontal compositon in a double category 
\newcommand{\vo}{\otimes_v} %%%  the vertical compositon in a double category 

\newcommand{\Rcal}{\mathcal{R}}
\renewcommand{\to}{\longrightarrow}
\newcommand{\ol}{\overline}
\newcommand{\ul}{\underline}
\renewcommand{\o}{\textbf{O}}
\renewcommand{\bf}{\mathbf}
\newcommand{\U}{\mathbb{U}}
\newcommand{\N}{\mathbb{N}}

\newcommand{\Ob}{\text{Ob}}% set of objects  
\newcommand{\n}{\textbf{n}} %object of Delta
\newcommand{\m}{\textbf{m}} % object of Delta
\newcommand{\p}{\textbf{p}} % object of Delta
\renewcommand{\d}{\textbf{d}}% object of Delta
\newcommand{\0}{\textbf{0}} % 0 of Delta
\renewcommand{\1}{\textbf{1}} %  1 of Delta
\renewcommand{\2}{\textbf{2}} %  2 of Delta
\newcommand{\uc}{\mathds{1}} % the unit category

\newcommand{\tx}{\text}
\newcommand{\tld}{\widetilde}
\newtheorem{obs}{Observations}
\renewcommand{\to}{\longrightarrow}
\DeclareMathOperator\Id{Id}
\DeclareMathOperator\Hom{Hom}

\DeclareMathOperator\ob{Ob}
\DeclareMathOperator\Set{\textbf{Set}} % Category of Set
\DeclareMathOperator\Lax{Lax}%% Lax functors 
\DeclareMathOperator\Lan{\textbf{Lan}}%% Left Kan Extension %%%%%%
\DeclareMathOperator\Cat{\mathbf{Cat}}%% Cat
\DeclareMathOperator\Opera{\text{Oper}}%% Operads 
\DeclareMathOperator\pistar{\pi^{\star}} %%% pi étoile
\DeclareMathOperator\Graph{\textbf{Graph}}%% Graph
\DeclareMathOperator\Laxalg{\Lax_{\O-alg}}%% algebres Operads 
\DeclareMathOperator\Laxalgp{\Lax_{\O'-alg}}%% algebres Operad O' (read O prime) henc 'p' in the end of  \Laxalgp
\DeclareMathOperator\msx{\M_{\S}(X)}%% Pre-Coseg-Cat
\DeclareMathOperator\msy{\M_{\S}(Y)}%% Pre-Coseg-Cat Y= set of object
\DeclareMathOperator\mset{\M_{\S}(\Set)}%% Pre-Coseg-Cat
\DeclareMathOperator\msetprojp{\M_{\S}(\Set)^{+}_{\tx{proj}}}%% Pre-Coseg-Cat proj
\DeclareMathOperator\kset{\K_{\Set}}%% K-Set
\DeclareMathOperator\kx{\K_X}%% K-X
\DeclareMathOperator\ky{\K_Y}%% K-Y
\DeclareMathOperator\kxproj{\K_{X\tx{-proj}}}%% K-X
\DeclareMathOperator\kyproj{\K_{Y\tx{-proj}}}
\DeclareMathOperator\kxinj{\K_{X\tx{-inj}}}%% K-X

\newcommand{\fstar}{f^{\star}}%% f étoile en haut
\newcommand{\flstar}{f_{\star}}%% f étoile en bas
\newcommand{\gstar}{g^{\star}}%% g étoile
\newcommand{\fex}{{f_!}}%% f exclamtion
\DeclareMathOperator\Le{\mathcal{L}}%% 
\DeclareMathOperator\mcatx{\M\tx{-}\Cat(X)}%% 
\DeclareMathOperator\mgraphx{\M\tx{-}\Graph(X)}%% 
\DeclareMathOperator\os{\otimes_{\S}}%% 
\DeclareMathOperator\pstar{p^{\star}}%% p étoile
\DeclareMathOperator\e{\mathbf{E}} %%% total space of fibration 
\DeclareMathOperator\bz{\mathbf{B}} %%% Base of fibration 
\DeclareMathOperator\cof{\mathbf{cof}}%% cofibration 
\DeclareMathOperator\kbinj{\mathbf{K}_{\tx{inj}}}%% cofibration 
\DeclareMathOperator\kb{\mathbf{K}} % % % 
\DeclareMathOperator\arproj{\M^2_{\tx{proj}}}
\DeclareMathOperator\jmsxproj{\mathbf{\Ja}_{\msx\tx{-proj}}} % % 
\DeclareMathOperator\kbproj{\mathbf{K}_{\tx{proj}}}%% cofibration 
\DeclareMathOperator\kbxproj{\mathbf{K}(X)_{\tx{proj}}}%% cofibration 
\DeclareMathOperator\kbyproj{\mathbf{K}(Y)_{\tx{proj}}}%% cofibration 
\DeclareMathOperator\Mb{\mathbf{M}}%% M
\DeclareMathOperator\fib{\mathbf{fib}} %%% fibration
\DeclareMathOperator\Lb{\tx{L}} % % % localization
\DeclareMathOperator\we{\mathbf{we}} %%% weak equivalence
\DeclareMathOperator\pr{\textbf{pr}}%% projections
\DeclareMathOperator\Pf{P} %%% projection fibration
\DeclareMathOperator\cn{\mathbf{[n]}} %%% crochet n 
\DeclareMathOperator\msn{\M_{\S}(\cn)}%% Pre-Coseg-Cat n
\DeclareMathOperator\kn{\K_{\cn}}%% 
\DeclareMathOperator\opt{\tx{-op}}%% op avec tiret
\DeclareMathOperator\op{\tx{op}}%% op tout court
\DeclareMathOperator\msxinj{\M_{\S}(X)_{\tx{inj}}} % % % msx with the injective model structure.
\DeclareMathOperator\msxproj{\M_{\S}(X)_{\tx{proj}}} % % % msx with the projective model structure.
\DeclareMathOperator\msxprojc{\M_{\S}(X)^{\mathbf{c}}_{\tx{proj}}} % % % msx with the projective model structure.
\DeclareMathOperator\msxprojp{\M_{\S}(X)^{+}_{\tx{proj}}} % % % msx with the projective model structure.
\DeclareMathOperator\msyprojp{\M_{\S}(Y)^{+}_{\tx{proj}}} % % % msx with the projective model structure.
\DeclareMathOperator\msyproj{\M_{\S}(Y)_{\tx{proj}}} % % % msx with the projective model structure.
\DeclareMathOperator\msxinjp{\M_{\S}(X)^{+}_{\tx{inj}}} 
\DeclareMathOperator\fainf{\Fa^{\infty}}
\DeclareMathOperator\homsx{ho(\msx)} % % % % Homotopy category
\DeclareMathOperator\morsx{Mor(\S_{\ol{X}})} % % % %  all  morphisms 
\DeclareMathOperator\sx{\S_{\ol{X}}} % % % %  sx
\DeclareMathOperator\Ho{\mathbf{ho}} % % % %  homotopy category
\DeclareMathOperator\lkm{\mathbf{L}_{\kb} \Mb} % % % locaization 
\DeclareMathOperator\sxop{(\S_{\ol{X}})^{2\tx{-op}}} % % % %  all  morphisms  % \cite{DHKS}

\DeclareMathOperator\Map{\tx{Map}}  % % % %  function complex
\DeclareMathOperator\degb{\textbf{deg}} % % % %  degree
\DeclareMathOperator\inj{\mathbf{inj}} % % % % inj
\DeclareMathOperator\msetproj{\mset_{\tx{proj}}} % % % % mset proj
\DeclareMathOperator\ksetproj{\kset{\tx{-proj}}} % % % % mset proj 
\DeclareMathOperator\kc{\K_{\Ca\dt}} 
\DeclareMathOperator\iro{\mathbf{ir}-\O} % 
\DeclareMathOperator\lr{\mathbf{lr}} % 
\DeclareMathOperator\Cr{\tx{P}} % % % cr= cofibrant replacement functor. 
\DeclareMathOperator\hco{\O-\mathbf{hc}} % % homotopy compatible pair
\DeclareMathOperator\siminj{\Sim_{\tx{inj}}} % % % cosegalification injective
\DeclareMathOperator\simproj{\Sim_{\tx{proj}}} % % % cosegalification projective % \Reedycatx
\DeclareMathOperator\Reedycatx{\Cat^{\times}_{\tx{Reedy}}} 
\DeclareMathOperator\Reedycat{\Cat_{\tx{Reedy}}} 
\DeclareMathOperator\cb{\mathbf{c}}  
\DeclareMathOperator\Ua{\tx{U}} 
\DeclareMathOperator\laxlatch{\mathbf{Latch}_{lax}}  %% laxlatch  
\DeclareMathOperator\latch{\mathbf{Latch}}  %% laxlatch  
\DeclareMathOperator\Depi{\Delta_{\tx{epi}}} %%
\DeclareMathOperator\g{\mathbf{g}} %%
\DeclareMathOperator\sk{\mathbf{sk}} %%
\DeclareMathOperator\bicatd{\mathbf{Bicat_2}} %%

\title{Lax Diagrams and Enrichment}
\author{Hugo V. Bacard \thanks{This research is supported in part by the Agence Nationale de la Recherche
grant ANR-09-BLAN-0151-02 (HODAG) }}
 \affil{Laboratoire J.-A. Dieudonné, UNS}
\date{}
\begin{document}
\maketitle

\begin{abstract}
We introduce a new type of weakly enriched categories over a given symmetric monoidal model category $\M$; these are called \emph{Co-Segal categories}. Their definition derives from the philosophy of classical (enriched) Segal categories. We study their homotopy theory by giving a model structure on them. One of the motivations of introducing these structure was to have an alternative definition of higher linear categories following  Segal-like methods. 
\end{abstract}
\setcounter{tocdepth}{1}
\tableofcontents

\section{Introduction}

In this paper we pursue the idea initiated in \cite{SEC1}, of having a theory of weakly enriched categories over a symmetric monoidal model category $\M=(\ul{M}, \otimes,I)$.  We introduce the notion of  Co-Segal $\M$-category.   Before going to what motivated the consideration of these `categories' we present very briefly hereafter the main philosophy. \ \\
\ \\
In a classical  $\M$-category $\Ca$, for every triple $(A,B,C)$ of objects of $\Ca$, the composition is given by a map of $\M$: $$\Ca(A,B) \otimes \Ca(B,C) \to \Ca(A,C).$$  In  a \textbf{Co-Segal category} such composition map is not explicit; we have instead the following diagram:
\[
\xy
(-15,0)*+{\Ca(A,B) \otimes \Ca(B,C)}="X";
(20,0)*+{\Ca(A,B,C)}="Y";
(20,15)*+{\Ca(A,C)}="E";
{\ar@{->}^{}"X";"Y"};
{\ar@{->}^{}_{\wr}"E";"Y"};
\endxy
\]
\\
where the vertical map $\Ca(A,C) \to \Ca(A,B,C)$ is required to be a \emph{weak equivalence}. As one can see when this weak equivalence is an isomorphism or an identity (the strict case) then we will have a classical composition and everything is as usual.
In the non-strict case, one gets a weak composition given by any choice of a weak inverse of that vertical map.  
\ \\ 

The previous diagram is obtained by `reversing the morphisms' in the Segal situation, hence the terminology ‘Co-Segal’. The diagrams below outline this idea:

\begin{minipage}{14cm}
\begin{minipage}{7cm}
\[
\xy
(-15,0)*+{\Ca(A,B) \otimes \Ca(B,C)}="X";
(20,0)*+{\Ca(A,B,C)}="Y";
(20,15)*+{\Ca(A,C)}="E";
{\ar@{->}^{}_-{\sim}"Y";"X"};
{\ar@{->}^{}"Y";"E"};
\endxy
\]
\begin{center}\textbf{In a Segal category} \end{center}
\end{minipage}
\begin{minipage}{7cm}
\[
\xy
(-15,0)*+{\Ca(A,B) \otimes \Ca(B,C)}="X";
(20,0)*+{\Ca(A,B,C)}="Y";
(20,15)*+{\Ca(A,C)}="E";
{\ar@{->}^{}"X";"Y"};
{\ar@{->}^{}_-{\wr}"E";"Y"};
\endxy
\]
\begin{center} \textbf{In a Co-Segal category}\end{center}
\end{minipage}
\end{minipage}

\ \\
\ \\
If the tensor product $\otimes$ of the category $\M=(\ul{M}, \otimes,I)$ is different from the cartesian product $\times$ e.g $\M$ is a Tannakian category, the so called \emph{Segal map}  $ \Ca(A,B,C) \to \Ca(A,B) \otimes \Ca(B,C)$  appearing in the Segal situation is not `natural'; it's a map going   into a product where there is no \emph{a priori} a way to have a projection on each factor. The  Co-Segal formalism was introduced precisely to bypass this problem.\ \\ %and offer an alternative definition of weakly enriched categories for arbitrary monoidal category $\M$. \ \\

In \cite{SEC1}, following an idea introduced by Leinster \cite{Lei3}, we define a Segal enriched category $\Ca$ having a set of objects $X$, as a \ul{colax} morphism of $2$-categories 
$$ \Ca: \P_{\ol{X}}: \to \M$$
satisfying the usual Segal conditions. As we shall see a Co-Segal category is defined as a \ul{lax} morphism of $2$-categories 
$$\Ca: (\S_{\ol{X}})^{2\tx{-op}} \to \M  $$
satisfying the Co-Segal conditions (Definition \ref{cosegal-diagram} ). Here $\P_{\ol{X}}$ is a $2$-category over $\Delta$ while $\sx \subset \P_{\ol{X}}$ is over $\Delta_{\tx{epi}}$. These $2$-categories are probably example of what should be a \emph{ locally Reedy $2$-category}, that is a $2$-category such that each category of $1$-morphisms is a Reedy category and the composition is coherent with the Reedy structures. \ \\

To develop a homotopy theory of these Co-Segal categories we follow the same philosophy as for Segal categories, that is we consider the more general objects consisting of lax morphisms $\Ca: (\S_{\ol{X}})^{2\tx{-op}} \to \M  $ without demanding the Co-Segal conditions yet; these are called \emph{Pre-Co-Segal categories}. 

As $X$ runs through $\Set$ we have a category $\mset$ of all \emph{Pre-Co-Segal categories} with morphisms between them. We have a natural Grothendieck bifibration $\Ob: \mset \to \Set$. \ \\
\ \\
The main result of this paper is the following
\begin{thmsans}
Let $\M$ be a symmetric monoidal model category which is cofibrantly generated and such that all the objects are cofibrant. Then the following holds.
\begin{enumerate}
\item the category $\mset$, of Pre-Co-Segal categories admits a  model structure which is cofibrantly generated,
\item  fibrant objects are Co-Segal categories,
\item If $\M$ is combinatorial  then so is $\mset$.
\end{enumerate}
\end{thmsans}
\begin{comment}
The proof of this theorem uses the same method as in the book of Simpson \cite{Simpson_HTHC} who in turn follows ideas which go back to Jardine \cite{Jardine_simpresh}, Joyal \cite{Joyal_simpsh}. The previous model structure on $\mset$  is obtained as a left Bousfield localization of an existing model structure with respect to a \emph{Co-Segalification functor}. 
\end{comment}
\subsection*{Plan of the paper}\ 
We begin the paper  by the definition of a lax diagram  in a $2$-category $\M$, which are simply lax functors of $2$-category in the sense of Bénabou \cite{Ben2}. We point out that $\M$-categories are special cases of lax diagrams as earlier observed by Street \cite{Str}.\ \\

Then in section \ref{operad-lax-morphisms} we recall some basic definitions about multisorted operads or colored operads. The idea is to use the powerful language of operads to treat $2$-categories and lax morphisms in terms of $\O$-algebras and lax morphisms of $\O$-algebras for some suitable operad. The operads we're working with are the ones enriched in $\Cat$. 

Given two $\O$-algebras $\Ca\dt$ and $\Ma\dt$ there is a category $\Laxalg(\Ca\dt, \Ma\dt)$ of lax morphisms and morphism of lax morphisms. After setting up some definitions we prove that:
\begin{itemize}%[label=$-$, align=left, leftmargin=*, noitemsep]
\item for a locally presentable $\O$-algebra $\Ma\dt$ the category $\Laxalg(\Ca\dt, \Ma\dt)$ is also locally presentable (Theorem \ref{laxalg-local-pres});
\item If $\Ma\dt$ is a special Quillen $\O$-algebra (Definition \ref{quillen-alg}) and under some hypothesis, the category $\Laxalg(\Ca\dt, \Ma\dt)$ carries a model structure (Theorem \ref{model-laxalg}).
\end{itemize}

In section \ref{section-cosegal-cat} we introduce the language of Co-Segal categories  starting with an overview of the one-object case. We've tried as much as possible to make this section independent from the previous ones. We only use the language of lax functor between $2$-categories rather than lax morphisms of $\O$-algebras. 
We introduce first the notion of an $\S$-diagram in $\M$ which correspond to \emph{Pre-Co-Segal categories} (Definition \ref{s-diagram}). Then we define a Co-Segal category to be an $\S$-diagram satisfying the \textbf{Co-Segal conditions} (Definition \ref{cosegal-cat}).
After giving some definitions we show that 
\begin{itemize}%[label=$-$, align=left, leftmargin=*, noitemsep]
\item A strict Co-Segal $\M$-category is the same thing as a strict (semi) $\M$-category (Proposition \ref{equiv-semicat});
\item The Co-Segal conditions are stable under weak equivalences (Proposition \ref{stab-coseg}).
\end{itemize} 

In section \ref{properties-msc} we show that the category $\msx$ of Pre-Co-Segal categories with a fixed set of objects $X$ is:
\begin{itemize}%[label=$-$, align=left, leftmargin=*, noitemsep]
\item is cocomplete if $\M$ is so (Theorem \ref{MX-cocomplete}) ; and 
\item locally presentable if $\M$ is so (Theorem \ref{MX-local-pres}).
\end{itemize} 
 For both of these two theorems, we've presented a `direct proof' i.e which doesn't make use of the language of operads; the idea is to make the content accessible for a reader who is not familiar with operads. \\
 
In section  \ref{lr-category}  we consider the notion of locally Reedy $2$-category. The main idea is to provide a \emph{direct} model structure on the category $\msx$ (Corollary \ref{direct-proj-model-msx}). 

In section  \ref{section-model-msx} we give two type of model structures on $\msx$, using a different method. These model structures play an important role in the later sections. 
We show precisely that if $\M$ is a symmetric monoidal model category, which is cofibrantly generated and such that all the objects are cofibrant, then we have:
\begin{itemize}
\item a \emph{projective} model structure on $\msx$ denoted $\msxproj$ (Theorem \ref{main-proj-msx});
\item an \emph{injective} model structure on $\msx$ denoted $\msxinj$ (Theorem \ref{main-inj-msx});
\item the identity functor $\Id: \msxproj  \rightleftarrows \msxinj: \Id$ is a Quillen equivalence (Corollary \ref{quillen-equiv-inj-proj});
\end{itemize}
These model structures are both cofibrantly generated (and combinatorial if $\M$ is so). The projective model structure is the same as the one given by Corollary \ref{direct-proj-model-msx}.\ \\

The section \ref{variation-objects} is dedicated to study of the category $\mset$ of all Pre-Co-Segal categories. We show that:
\begin{itemize}
\item $\mset$ inherits the cocompleteness and local presentability of $\M$ (Theorem \ref{mset-loc-pres}); and
\item that $\mset$ carries a \emph{fibered projective} model strucuture which is cofibrantly generated. And if $\M$ is combinatorial then so is $\mset$ (Theorem \ref{fibered-model-mset} and Corollary \ref{mset-combinatorial}).
\end{itemize} 
\ \\

In section \ref{secion-cosegalification},  we begin by constructing for each set $X$, an endofunctor $\Sim: \msx \to \msx$, called  \emph{`Co-Segalifcation'} which takes any Pre-Co-Segal category to a Co-Segal category (Proposition \ref{global-cosegalification}). Assuming that $\msx$ is left proper if $\M$ is so (Hypothesis \ref{left-proper-hypo}) we prove that:
\begin{itemize}
\item There exists a  \emph{new injective} model structure on $\msx$ denoted $\msxinjp$ which is combinatorial and such that the fibrant objects are Co-Segal categories. $\msxinjp$ is the left Bousfield localization of $\msxinj$ with respect to some set of maps $\kbinj$ (Theorem \ref{model-msxinjp}).
\item  There is also a  \emph{new projective} model structure on $\msx$ denoted $\msxprojp$ which is combinatorial and such that the fibrant objects are Co-Segal categories. The model category $\msxprojp$ is the left Bousfield localization of $\msxproj$ with respect to some set of maps $\kbproj$ (Theorem \ref{msxprojp}).
\item We have also a \emph{new fibered projective} model structure on $\mset$ denoted $\msetprojp$ which is combinatorial and such that the fibrant objects are Co-Segal categories (Theorem \ref{fibered-model-mset-plus}). 
\end{itemize} 

In section \ref{mode-mcat-deux-cat} we reviewed the basics about $\M$-categories for a $2$-category $\M$. For a fixed set of objects $X$, we show that if $\M$ is locally a model category  (Definition \ref{model-2-cat})  and all the objects are cofibrant, then the category $\mcatx$ has a model structure  which is cofibrantly generated and combinatorial if $\M$ is so.  We leave the reader who might be interested to give a fibered model structure on $\M$-$\Cat$ and even the `canonical model structure' in the sense of Berger-Moedijk \cite{Berger_Moer_htpy_cat}. \ \\

It seems clear that all the previous results on Co-Segal categories should hold if we replace the monoidal model category $\M$ by a $2$-category which is locally a model category.

%%%%%%%%%%%%%%%%%%%%%%%%%%%%%%%%%%%%%%%%%%%%%%%%%%%%%%%
%%%%%%%%%%%%%%%THANKS %%%%%%%%%%%%%%%%%%%%%%%%%%%%%%%%%%%%%%%%

\begin{thank}
I would like to warmly thank my supervisor Carlos Simpson for his support and encouragement. This work owes him a lot.
This paper is an extension of a previous work originally motivated by a question of Julie Bergner and Tom Leinster; their ideas have undoubtedly influenced this work.  I would like to thank Jacob Lurie for helpful conversations. I would also like to thank Bertrand Toën for helpful comments.
\end{thank}
%%%%%%%%%%%%%%%%%%%%%%%%%%%%%%%%%%%%%%%%%%%%%%%%%%%%%%%
%%%%%%%%%%%%%%%%%%%%%%%%%%%%%%%%%%%%%%%%%%%%%%%%%%%%%%%

%
\section{Lax Diagrams}\label{lax-diagram}

In the following we fix $\M$ a  bicategory (or $2$-category). For a sufficiently large universe $\V$ we will assume that all the $2$-categories we will consider (including $\M$) have  a $\V$-small set of $2$-morphisms.  
\begin{df}
A \textbf{lax diagram} in $\M$ is a lax morphism $F : \D \to \M$, where  $\D$ is a strict $2$-category. 
\end{df}

For each $\D$ we will consider  $\Lax(\D,\M)$  the $1$-category of lax morphisms from $\D$ to $\M$ and \textbf{icons}  in the sense of Lack \cite{Lack_icons}.
\begin{enumerate}[label=$-$, align=left, leftmargin=*, noitemsep]
\item The objects of $\Lax(\D,\M)$  are lax morphisms,
\item the morphisms are \emph{icons} (see \cite{Lack_icons}) .
\end{enumerate}
Icons are what we call later \textbf{simple transformations} (Definition \ref{simple-trans}).
The reader can find for example in \cite{Lack_icons,Lei1}  these definitions.

\begin{warn}
Note that in general there is only a bicategory and (not a $2$-category) of lax morphisms. This bicategory is described as follows
\begin{enumerate}[label=$-$, align=left, leftmargin=*, noitemsep]
\item The objects of are lax morphisms,
\item the $1$-morphisms are transformations of lax morphisms,
\item the $2$-morphisms are modifications of transformations. 
\end{enumerate}
\end{warn}

By definition of a lax morphism $F: \D \to \M$ we have a function between the corresponding set of objects 
$$\Ob(F): \Ob(\D) \to \Ob(\M).$$  
This defines a function $\Ob : \Ob[\Lax(\D,\M)] \to \Hom[\Ob(\D), \Ob(\M)]$ which sends $F$ to $\Ob(F)$.\ \\

Given a function  $\phi$ from $\Ob(\D)$ to  $\Ob(\M)$ we will say that $F \in \Lax(\D,\M)$ is \emph{over $\phi$} if $\Ob(F)= \phi$. 
We will denote by $\Lax_{/ \phi}(\D,\M)$ be the full subcategory of $\Lax(\D,\M)$ consisting of objects over $\phi$ and transformations of lax morphisms.

\begin{comment}
Recall that by definition, a transformation of a lax morphism $\sigma : F \to G$  consists, roughly speaking, of:
\begin{itemize}
\item  a $1$-morphism of $\M$,  $\sigma_A : FA \to  GA$ for  every $A \in \Ob(\D)$
\item a natural transformation $\sigma_{AB}: \sigma_B \otimes F(-) \to G(-) \otimes \sigma_A$ for each pair $(A,B) \in \Ob(\D)^2$ submitted to some coherences (see \cite{Lei1}).
\end{itemize} 

For $F,G$ in $\Lax_{/ \phi}(\D,\M)$ since $FA=GA$ we will consider the following subcategory. 

\begin{nota}
Let $\phi$ a function from $\Ob(\D)$ to $\Ob(\M)$.\\
We will denote by $\Lax_{/\phi}^{r}(\D,\M)$ the subcategory of $\Lax_{/\phi}(\D,\M)$ of lax morphisms over $\phi$ and transformations $\sigma$ such that for every $A$ we have $\sigma_A= \Id_{\phi(A)}$.   
\end{nota}   
\end{comment}

\paragraph*{$\M$-categories are lax morphisms}  Given an $\M$-category $\Ca$ having a set of objects $X$, then we can define a lax morphism denoted again $\Ca: \ol{X} \to \M$. Here $\ol{X}$ is the universal nontrivial groupoid associated to $X$; we refer it as the \emph{undiscrete} or \emph{coarse} category associated to $X$. In this context one interprets the lax morphism as the nerve of the enriched category. \ \\

This identification of $\M$-categories as lax morphisms goes back to Bénabou \cite{Ben2} as pointed out by Street \cite{Str}. Bénabou defined them as \emph{polyads} as the plural form of monad.\ \\

We pursue the spirit of this identification which is somehow the `universal lax situation'. 

%%%%%%%%%%%%%%%%%%%%%%%%%%%%%%%%%%%%%%%%%%%%%%%%%%%%%%%%%%%%%%%%%
%%%%%%%%%%%%%%%%%%%%%%%%%%%%%%%%%%%%%%%%%%%%%%%%%%%%%%%%%%%%%%%%%

\section{Operads and Lax morphisms}\label{operad-lax-morphisms}
%% Dire que le but de cette section est de voir un lax morphism comme algèbre ou morphisme d'algèbre sur une opérade colorée en categorie%%%
In the following we use the language of \emph{multisorted operads}  also called `colored operads' to treat the theory of $2$-categories and lax functors as $\O$-algebras and morphism of $\O$-algebras of a certain multisorted operad $\O$. When there is no confusion we will simply say operads to mean multisorted operads.  \ \\

Although the results of this section will be stated for a general operad $\O$, one should keep in mind the special case where $\O$ is the operad we will see in the Example \ref{operad-principal} below.\ \\

We recall briefly hereafter the definition of the type of operad we will consider. For a detailed definition of these one can see, for example,  \cite{Ber_Moer_multisorted} or \cite{Leinster_HOHC}. In the later reference multisorted operads are called \emph{multicategories}. \ \\

Let $C$ be a nonempty set (thought as a set of coulors or sorts).\ \\

A $C$-multisorted operad $\O$ in $\Cat$, or a $\Cat$-operad, consists of the following data. 
\begin{enumerate}
\item For all $n\geq 0$ and each $(n+1)$-tuple  $(i_1,...,i_n; j)$ of elements of $C$ there is a category $\O(i_1,...,i_n; j)$. 
\item For each $i \in C$, we have an identity operation expressed as a functor $\1 \xrightarrow{1_i} \O(i;i)$, where $\1$ is the unit category.
\item There is a composition operation:
$$ \O(i_1,...,i_n; j) \times \O(h_{1,1},...,h_{1,k_1}; i_1) \times \cdots \times  \O(h_{n,1},...,h_{n,k_n}; i_n) \to  \O(h_{1,1},...,h_{n,k_n}; j) $$ 
$$ (\theta, \theta_1, ..., \theta_n) \mapsto \theta \circ (\theta_1, ..., \theta_n). $$  
\item The composition satisfies associativity and unity conditions. The reader can find all the details in \cite[Chap.2]{Leinster_HOHC} or in \cite[Part I]{Kriz_May_OAMM}. 
\end{enumerate}
When the set $C$ has only one element (one colour) we recover the definition of an operad.
\begin{rmk}
In the condition $(1)$ above, when $n=0$ we have no colour in `input', so we will denote by $\O(0,i)$ this category. Here the `0' means  zero input.

This category $\O(0,i)$ allows us to have an `identity' or `unity object' when we want to have the notion of \emph{unital} $\O$-algebra. For this reason we will always set $\O(0,i)=\1$. 
\end{rmk}

For a fixed set of colors $C$, we have a category of $C$-multisorted operads in $\Cat$ with the obvious notion of morphism. The reader can find a definition in \cite{Ber_Moer_multisorted}. 
We follow the same notation as in  \cite{Ber_Moer_multisorted} and will denote by  $\Opera_C(\Cat)$the category of $C$-multisorted operads in $\Cat$. Similarily  if $\E$ is a monoidal category, we will denote by $\Opera_C(\E)$ the category of $C$-multisorted operads in $\E$. \ \\
\\
Below we give an example of a multi-sorted operad which will play an important role in the upcoming sections. This is the multi-sorted operad whose algebras are $2$-categories i.e enriched categories over $\Cat$. The construction we present here is equivalent to the one given  in 
\cite[Section 1.5.4]{Ber_Moer_multisorted}.

\begin{ex}\label{operad-principal}
Let $X$ be a nonempty set and $\ol{X}$ be the associated indiscrete or coarse category. Recall that $\ol{X}$ is the category with $X$ as the set objects and such that there is exactly one morphism between any pair of elements.\  \\

Let $C=X \times X$ be the set of pairs of elements of $X$. There is  a one-to-one corespondance between $C$  and the set of morphisms of $\ol{X}$.
We will denote by $\Nv(\ol{X})$ the nerve of $\ol{X}$ and by $\Nv(\ol{X})_n$ its set of $n$-simplices.\ \\
 
We define a $C$-multisorted operad $\O_X$ as follows.
%%% Dire que Clemens et Ieke ont considéré ça dans leur papier.
\begin{itemize}
\item for $n> 0$ we take  
 \begin{equation*}
 \O_X(i_1,...,i_n; j) = \begin{cases}
 \1=\tx{the unit category }  &  \tx{if $(i_1,...,i_n) \in \Nv(\ol{X})_n$ and  $j= i_n \circ \cdots \circ i_1$} \\
 \varnothing=\tx{the empty category} & \tx{if not}   \\
  \end{cases}
\end{equation*}

 \item For $n=0$  we set 
 \begin{equation*}
\O_X(0,i)=\begin{cases}
 \1&  \tx{if $i=\Id_A$ in $\ol{X}$ i.e $i=(A,A)$ for some $A\in X$} \\
 \varnothing & \tx{if not}   \\
  \end{cases}
\end{equation*}
 \item The `identity-operation' functor $\1 \to \O_X(i,i)$ is the identity $\Id_{\1}$.
 \item The composition:
 $$ \O_X(i_1,...,i_n; j) \times \O_X(h_{1,1},...,h_{1,k_1}; i_1) \times \cdots \times  \O_X(h_{n,1},...,h_{n,k_n}; i_n) \to  \O_X(h_{1,1},...,h_{n,k_n}; j) $$
 
 is either one of the (unique) functors:
 \begin{equation*}
\begin{cases}
 \1 \times \cdots \times \1 \xrightarrow{\cong} \1 &   \\
 \varnothing \xrightarrow{\Id} \varnothing &    \\
 \varnothing \to \1 &\\
  \end{cases}
\end{equation*}
\item The associativity and unity axioms are straightforward.
\end{itemize}
\end{ex}

We will see in a moment that an $\O_X$-algebra in $\Cat$ is equivalent to a $2$-category having $X$ as its set of objects. 
\begin{claim}
Given a nonempty category $\Ba$, If we replace everywhere $\ol{X}$  by $\Ba$ in the construction above, one gets a multisorted operad  $\O_{\Ba}$ in $\Cat$ where the set of colours $C$ is the set $\Ar(\Ba)$ of all morphisms of $\Ba$. An $\O_{\Ba}$-algebra is the same thing as a lax morphism from $\Ba$ to $(\Cat, \times, \1)$.\ \\
And more generally given a symmetric monoidal category $\M=(\ul{M}, \otimes, I)$ having an initial object $\0$, we can construct a multisorted $\M$-operad $\O_{\Ba}$, replacing $\1$ and $\varnothing$ respectively by $I$ and $\0$. As in the previous case an $\O_{\Ba}$-algebra in $\M$ will be the same thing as a lax morphism from $\Ba$ to $\M$. 
\end{claim}

\begin{df}
Let $\O$ be a $C$-multisorted operad in $\Cat$. \ \\

An \textbf{$\O$-algebra} $\Ma\dt$ is given by the following data.
\begin{itemize}
\item For each $i \in C$ we have a category $\Ma_i$. 
\item $\Ma_0= \1$. 
\item For each $(n+1)$-tuple  $(i_1,...,i_n; j)$ of elements of $C$ there is a functor:
$$ \O(i_1,...,i_n; j) \times \Ma_{i_1} \times \cdots \times \Ma_{i_n} \xrightarrow{\theta_{_{i\dt|j}}}  \Ma_j$$
\item We have also a functor  $\O(0,i) \times \Ma_0 \to \Ma_i $ which gives  a functor  
$$ \1 \xrightarrow{e_i} \Ma_i.$$
\item These functors are compatible with the associativity and unity of the composition of $\O$. 
\end{itemize}
\end{df}

\begin{nota}
Given $ (x, m_1, ..., m_n) \in \O(i_1,...,i_n; j) \times \Ma_{i_1} \times \cdots \times \Ma_{i_n} $  we will use the suggestive notation $\otimes_x(m_1, ..., m_n)=\theta_{_{i\dt|j}}(x, m_1, ..., m_n)$. The idea is to think each functor $\theta_{_{i\dt|j}}(x,-)$ as a general  tensor product. \ \\
\end{nota}

The following proposition shows us how the theory of lax functors and operads are related within the theory of enriched categories.
% dire que la proposition suivante relie l'enrichissement, les o-alg et les lax morphism. 
\begin{prop}
Let $X$ be a nonempty set. We have an equivalence between the following data.
\renewcommand{\labelenumi}{\roman{enumi})}
\begin{enumerate}
\item An $\O_X$-algebra in $\Cat$,
\item A $2$-category with $X$ as the set of objects.  %%% en fait c'est plutôt des 
\item A lax morphism $F: \ol{X} \to (\Cat, \times, \1)$ 
\end{enumerate}
\end{prop}

\begin{rmk}
We can also include a fourth equivalence between the strict homomorphism from $\P_{\ol{X}}$ to $(\Cat, \times, \1)$, where $\P_{\ol{X}}$ is the $2$-path category associated to $\ol{X}$ (see \cite{SEC1}). And as claimed above, one can replace everywhere $\ol{X}$ by an arbitrary category $\Ba$. The fourth equivalence will be a homomorphism from $\P_{\Ba}$ to  $(\Cat, \times, \1)$.
\end{rmk}

\begin{proof}[\textbf{Sketch of proof}]

The equivalence between ii) and iii) is well known and is left to the reader. \ \\
We simply show how we get a $2$-category from an $\O_X$-algebra. The implication $\tx{ii)} \Rightarrow\tx{i)}$ will follow immediately by `reading backwards' the argumentation we present hereafter. \ \\

Let  $\Ma\dt$  be an $\O_X$-algebra in $\Cat$. We construct a $2$-category $\Ma$ as follows.

\begin{enumerate}
\item $\Ob(\Ma)= X$\\
\item Given a pair $(A,B) \in X^2=C$, we have a category $\Ma_{(A,B)}$ and we set $ \Ma(A,B)= \Ma_{(A,B)}$.\\
\item Given $A,B,C$ in $X$, if we set $i_1=(A,B), i_2=(B,C)$ and $j=(A,C)$ we have $\O(i_1,i_2; j) = \1$ and the functor 
$ \O_X(i_1,i_2; j) \times \Ma_{i_1} \times  \Ma_{i_2} \to  \Ma_j$ gives the composition: 

$$\Ma(A,B) \times  \Ma(B,C) \xrightarrow[\tx{canonical}]{\cong}  \1 \times \Ma(A,B) \times  \Ma(B,C) \to  \Ma(A,C). $$ 

\item Each $\O_X(i,i)$ acts trivially on $\Ma_i$ i.e the map  $\O_X(i,i) \times \Ma_i \to  \Ma_i$ is the canonical isomorphism $\1 \times \Ma_i \xrightarrow{\cong} \Ma_i$.\\
\item One gets the associativity  of the composition in $\Ma$ using the fact the following functors are invertible and have the same codomain: 

\begin{itemize}
 \item $\O_X(i_1\circ i_2, i_3; j)  \times \O_X(i_1, i_2; i_1\circ i_2 ) \times \O_X(i_3;i_3) \xrightarrow{\cong} \O_X(i_1,i_2, i_3; j)$    \\
\item $\O_X(i_1, i_2 \circ i_3; j)  \times \O_X(i_1, i_1) \times \O_X(i_2, i_3 ; i_2\circ i_3) \xrightarrow{\cong}\O_X(i_1,i_2, i_3; j)$  \\
\end{itemize}
with $i_1=(A,B), i_2=(B,C), i_3=(C,D)$ and $i_1 \circ i_2=(A,C)$, $i_2 \circ i_3=(B,D)$. This provides a natural isomorphism between the domains of the two functors. Putting these together with the fact that the action of $\O_X$ on $\Ma\dt$ is compatible with the composition of $\O_X$, we get the desired natural isomorphism expressing the associativity of the composition in $\Ma$. \\
\item For each $i$ of the form $(A,A)$ we have $\O_X(0,i)=\1$ and the unity condition of the algebra provides a morphism $\1 \to \Ma(A,A)$ which satisfies the desired conditions of an identity morphism in a $2$-category.
\end{enumerate}
\end{proof}
%%%%%%%%%%%%%%%%%%%%%%%%%%%%%%%%%%%%%%%%%%%%%%%%%%%%%%%%%%%%%%%%%
%%%%%%%%%%%%%%%%%%%%%%%%%%%%%%%%%%%%%%%%%%%%%%%%%%%%%%%%%%%%%%%%%
The functor $\Ob: \Cat \to \Set$ which sends a category to its set of objects, commutes with the cartesian product, so that it's actually a (strict) monoidal functor. As a consequence  we get a functor $$\Ob: \Opera_C(\Cat)  \to \Opera_C(\Set).$$
For $\Rcal \in \Opera_C(\Set)$ and $\O \in \Opera_C(\Cat)$ we will say that \emph{$\O$ is over $\Rcal$} if $\Ob(\O)= \Rcal$.

\begin{rmk}
It's not hard to see that since the functor $\Ob$ is monoidal, for any $\O$-algebra $\Ma\dt$ then $\Ob(\Ma\dt)$ is automatically an $\Ob(\O)$-algebra. 
\end{rmk}
\subsection{Lax morphism of $\O$-algebra}
We now consider the type of morphism of $\O$-algebras we are going to work with. Our definition is different than the standard definition of morphism of algebras. The idea is to recover the definition of lax functor between $2$-categories when $\O$ is of the form $\O_X$. \ \\

\begin{df}\label{o-morphism}
Let $\O$ be an object of  $\Opera_C(\Cat)$ and $\Ca\dt$, $\Ma\dt$ be two $\O$-algebras.\ \\
\\
 A \textbf{lax} morphism $\Fa\dt: \Ca\dt \to \Ma\dt $ of $\O$-algebras, or simply  a lax $\O$-morphism,is given by the following data and axioms.\ \\\\
\emph{\textbf{Data}:} 
\begin{itemize}
 \item A family of functors $\{ \Fa_i: \Ca_i \to \Ma_i\}_{i \in C}$. 
 \item For each $(n+1)$-tuple  $(i_1,...,i_n; j)$, a family of natural transformation $\{ \varphi=\varphi(i_{\dt};j) \}$ :
\[
\xy
(-20,20)*+{ \O(i_1,...,i_n; j) \times \Ca_{i_1} \times \cdots \times \Ca_{i_n}}="X";
(30,20)*+{\Ca_j}="Y";
(-20,0)*+{\O(i_1,...,i_n; j) \times \Ma_{i_1} \times \cdots \times \Ma_{i_n}}="A";
(30,0)*+{\Ma_j}="B";
{\ar@{->}^{\ \ \ \ \ \ \ \ \ \ \ \ \ \ \ \ \ \ \  \theta}"X";"Y"};
{\ar@{->}_{ \ \\ \ \ \ \ \ \ \ \ \ \ \ \ \ \ \ \ \ \ \rho}"A";"B"};
{\ar@{->}^{\Fa_j}"Y";"B"};
{\ar@{->}_{\Id \times \Fa_{i_1} \times \cdots \times \Fa_{i_n}}"X";"A"};
%% insertion de phi%%%
{\ar@{=>}^{\varphi}(21,2);(28,9)};
\endxy
 \]
 \ \\
 $$\otimes_x(\Fa_{i_1}  c_1,..., \Fa_{i_n} c_n)  \xrightarrow{\varphi(x,c_1,...c_n)} \Fa_j [ \otimes_x(c_1,..., c_n)]$$
\end{itemize} 
\ \\
\emph{\textbf{Axioms}:}  The natural transformations $\varphi_{_{i_{\dt}|j}}(-)$  satisfy the following coherence conditions, which are the ` $2$-dimensional' analogue of those satisfied by  $\theta_{_{i_{\dt}|j}}(-)$ and $\rho_{_{i_{\dt}|j}}(-)$:
 $$\varphi_{_{i_{\dt}|j}} \otimes_{\theta} \{ \Id_{\Id_{\O(i\dt|j)}} \times [ (\prod_{i}  \varphi_{h_{i,\dt|i}}) \otimes  \Id_{\tx{shuffle}}] \} = \varphi_{h_{\dt,\dt|j}} \otimes \{ \Id_{ \gamma_{_{h_{\dt,\dt}|i\dt|j}} } \times \Id_{\Id_{\prod \Ma_{h_{\dt,\dt}}}} \}. $$  
 \ \\
\ \\ 
 More explicitly, given 
 \begin{itemize}
 \item $(x,x_1,...,x_n) \in \O(i\dt|j) \times \O(h_{_{1,\dt}}|i_1) \times \cdots \times \O(h_{_{n,\dt}}|i_n)$
 \item $(d_{_{1,1}},...,d_{_{1,k_1}}, ... , d_{_{n,1}}, ... ,d_{_{n,k_n}}) \in \Ma_{h_{_{1,1}}} \times \cdots \times \Ma_{h_{_{1,k_1}}} \times \cdots \times \Ma_{h_{_{n,k_n}}}$  
\item $\otimes_{\gamma(x, x_1,...,x_n)}(d_{_{1,1}},...,d_{_{n,k_n}})= c$
\item $\otimes_{x_i}(d_{_{i,1}},...,d_{_{i,k_i}})= c_i$,  $i\in \{1,...,n\}$,
\item $\otimes_x(c_1,..., c_n)=c$
\item $\varphi_i=\varphi(x_i, d_{_{i,1}},...,d_{_{i,k_i}}): \otimes_{x_i}(\Fa d_{_{i,1}}  ,..., \Fa d_{_{i,k_i}}) \to \Fa [ \otimes_{x_i}(d_{_{i,1}},...,d_{_{i,k_i}})]= \Fa c_i $
 \end{itemize}
  
we require the equality :

$$ \varphi(\gamma(x, x_1,...,x_n), d_{_{1,1}},...,d_{_{n,k_n}})=\varphi(x,c_1,...c_n) \circ [ \otimes_x(\varphi_1, ..., \varphi_n)] .$$

%$ \varphi^{\infty}(x,c_1, ..., c_n) \circ [\otimes_x (\varphi^{\infty}(x_1,d_{1,1}, ..., d_{1,k_1}), \cdots, \varphi^{\infty}(x_n,d_{n,1}, ..., d_{n,k_n}))].$ 

\end{df}

In the next paragraph we make some comments about the coherence conditions in the previous definition.
\paragraph{\textbf{Coherences}} \label{coherences-lax}\ \\
The previous coherence can be easily understood when we think that the family of functors $\{ \Fa_i: \Ca_i \to \Ma_i\}_{i \in C}$ equipped with the family of natural transformations $\{ \varphi_{_{i_{\dt}|j}}(x) \}_{x \in \Ob(\O(i_1,...,i_n; j))}$, is an $\O$-algebra of some arrow-category we are about to describe. \ \\

Let's consider $\Ar(\Cat)_+$ the double category  given by the following data. 
\begin{itemize}
\item  The objects are the arrows of $\Cat$ i.e functors $F$ 
\item A morphism from  $F$ to $G$ consists of a triple $(\alpha,\beta,\varphi)$ where $\alpha$  and $\beta$ are functors and $\varphi$ a natural transformation as shown in the following diagram:
\[
\xy
(0,20)*+{\cdot}="X";
(30,20)*+{\cdot}="Y";
(0,0)*+{\cdot}="A";
(30,0)*+{\cdot}="B";
{\ar@{->}^{\alpha}"X";"Y"};
{\ar@{->}_{\beta}"A";"B"};
{\ar@{->}^{G}"Y";"B"};
{\ar@{->}_{F}"X";"A"};
%% insertion de phi%%%
{\ar@{=>}^{\varphi}(21,2);(28,9)};
\endxy
 \]

We will represent such morphism as a column or a row: 
$ 
\left( \begin{array}{c}
\alpha \\ 
\beta \\ 
\varphi
\end{array} \right)
$,
$(\alpha; \beta; \varphi)$.
\\
\item  The horizontal composition $\ho$ and vertical composition $\vo$ in $\Ar(\Cat)_+$ are given as follows
$$
\left( \begin{array}{c}
\alpha' \\ 
\beta' \\ 
\varphi'
\end{array} \right)_{_{G \to H}} \ho
\left( \begin{array}{c}
\alpha \\ 
\beta \\ 
\varphi
\end{array} \right)_{_{F \to G}} = 
\left( \begin{array}{c}
\alpha' \circ \alpha \\ 
\beta ' \circ \beta \\ 
\varphi'_{\alpha x} \circ \beta'(\varphi_x) 
\end{array} \right)_{_{F \to H}} 
$$
%%%%%%%%%%%%% comp verticale: %%%%%%%%%%
$$
\left( \begin{array}{c}
\alpha \\ 
\beta \\ 
\varphi'
\end{array} \right)_{_{K \to L}} \vo
\left( \begin{array}{c}
\beta \\ 
\gamma \\ 
\varphi
\end{array} \right)_{_{F \to G}} = 
\left( \begin{array}{c}
 \alpha \\ 
\gamma \\ 
L(\varphi_x) \circ \varphi'_{Fx}
\end{array} \right)_{_{K F \to LG}} 
$$
\end{itemize}
\ \\
It's not hard to see that $\Ar(\Cat)_+$ carries a monoidal structure with the cartesian product of functors where the unity is the identity functor $\Id_{\1}$. The product of two morphisms $(\alpha,\beta,\varphi)$ and $(\alpha',\beta',\varphi')$ is given by:
$$
\left( \begin{array}{c}
\alpha \\ 
\beta \\ 
\varphi
\end{array} \right) \times
\left( \begin{array}{c}
\alpha' \\ 
\beta' \\ 
\varphi'
\end{array} \right) = 
\left( \begin{array}{c}
\alpha \times \alpha' \\ 
\beta  \times \beta' \\ 
\varphi  \times \varphi' 
\end{array} \right) 
$$
\ \\
\begin{rmk}
We have a functor  $\Cat \hookrightarrow \Ar(\Cat)_+$ sending a natural transformation $\sigma: F \to G$ to $(\Id_{s}; \Id_t; \sigma)$ where $s$ and $t$ are the source and target of both $F$ and $G$.  
\end{rmk}
Given an object $F$ of $\Ar(\Cat)_+$, we will use the notation  $\O \odot F := \Id_{\O} \times F$.  With the monoidal category  $(\Ar(\Cat)_+, \times, \Id_{\1})$ we can say that the coherence conditions on  $\varphi_{_{i_{\dt}|j}}$ are equivalent to say that the family $\{\Fa_i \}_{i \in C}$ is an $\O$-algebra in   $(\Ar(\Cat)_+, \times, \Id_{\1})$ where the maps
$$\O(i_1,...,i_n; j) \odot [\Fa_{i_1} \times \cdots \times \Fa_{i_n}]  \to \Fa_j$$
are given by the family of triple $(\theta_{_{i_{\dt}|j}}, \rho_{_{i_{\dt}|j}}, \varphi_{_{i_{\dt}|j}})$. 

%%%%%%%%%%%% Composition de morphism$ 
\subsection*{Composition of lax $\O$-morphisms} 
Let $\Ca\dt, \Ma\dt, \Na\dt$ be three $\O$-algebras and $(\Fa\dt,\varphi_{_{i_{\dt}|j}}):\Ca\dt \to \Ma\dt$,\\ $(\Ga\dt,\psi_{_{i_{\dt}|j}}):\Ma\dt \to \Na\dt$ be two lax  $\O$-morphisms. \ \\

We define the composite $\Ga\dt \circ \Fa\dt$ to be the lax $\O$-morphism given by the following data.
\begin{itemize}
\item The family of functors $\{ \Ga_i \circ \Fa_i: \Ca_i \to \Na_i\}_{i \in C}$. 
\item For each $(n+1)$-tuple  $(i_1,...,i_n; j)$, the family of natural transformation $\{ \chi_{_{i_{\dt}|j}}(x) \}_{x \in \Ob(\O(i_1,...,i_n; j))}$ where:
$$\chi_{_{i_{\dt}|j}}(x)= G_j[\varphi_{_{i_{\dt}|j}}(x)] \circ \psi_{_{i_{\dt}|j}}(x)_{_{\prod \Fa_i(-)}}.$$\\
\item More precisely the component of $\chi_{_{i_{\dt}|j}}(x)$ at $(c_1,...,c_n) \in \Ca_{i_1} \times \cdots \times \Ca_{i_n} $ is the morphism:
$$\chi_{_{i_{\dt}|j}}(x)_{(c_1,...,c_n)}= G_j[\varphi_{_{i_{\dt}|j}}(x)_{(c_1,...,c_n)}] \circ \psi_{_{i_{\dt}|j}}(x)_{(\Fa_{i_1}c_1,...,\Fa_{i_n}c_n)}.$$
\end{itemize}
We leave the reader to check that these data satisfy the coherence conditions  of the Definition \ref{o-morphism}.

\begin{rmk}\ \\
The identity $\O$-morphims of an algebra $(\Ma\dt,\theta_{_{i_{\dt}|j}}) $ is given by the family of functors $\{\Id_{\Ma_i}\}_{i \in C}$ and  natural transformations $\{ \Id_{\theta_{_{i_{\dt}|j}}(x)} \}_{x \in \Ob(\O(i_1,...,i_n; j))}$.
\end{rmk}

\subsection{Morphisms of lax $\O$-morphisms}
$\O$-algebras and lax $\O$-morphisms form naturally a category. But there is an obvious notion of $2$-morphism we now describe. A $2$-morphism is the analogue of the transformations of lax functors. 
\begin{df}
Let $(\Fa\dt,\varphi_{_{i_{\dt}|j}})$ and  $(\Fa\dt', \varphi_{_{i_{\dt}|j}}')$ be two lax $\O$-morphisms from $\Ca\dt$  to $ \Ma\dt$.\ \\
\ \\
A $2$-morphism $\sigma\dt: \Fa\dt \to \Fa\dt'$ is given by the following data and axioms.
\ \\
\ \\
\emph{\textbf{Data}:} A family of natural transformations $ \{\sigma_i: \Fa_i \to \Fa_i' \}_{i \in C}$.

\ \\
\emph{\textbf{Axioms}:} 
For any $x \in \O(i_1,...,i_n; j)$, and any $(c_1,...,c_n) \in \Ca_{i_1} \times \cdots \times \Ca_{i_n} $,  the following commutes :
\[
\xy
(-40,25)*+{\otimes_x(\Fa_{i_1}c_1,...,\Fa_{i_n}c_n)}="X";
(30,25)*+{\Fa_j[\otimes_x(c_1,...,c_n)]}="Y";
(-40,0)*+{\otimes_x(\Fa'_{i_1}c_1,...,\Fa'_{i_n}c_n)}="A";
(30,0)*+{\Fa'_j[\otimes_x(c_1,...,c_n)]}="B";
{\ar@{->}^{\varphi(x,c_1,...,c_n)}"X";"Y"};
{\ar@{->}_{\varphi'(x,c_1,...,c_n)}"A";"B"};
{\ar@{->}^{\sigma_{j,\otimes_x(c_{\dt})}}"Y";"B"};
{\ar@{->}_{\otimes_x(\sigma_{i_1,c_1},...,\sigma_{i_n,c_n})}"X";"A"};
\endxy
 \]

\end{df}
The composition of $2$-morphisms is the obvious one i.e component-wise. We will denote by $\Laxalg(\Ca\dt, \Ma\dt)$ the category of lax $\O$-morphisms between two $\O$-algebras $\Ca\dt$ and $\Ma\dt$. \\
\subsection{Locally presentable $\O$-algebras}
Below we extend the notion of  locally presentable category $\M$ to $\O$-algebras for an operad $\O \in  \Opera_C(\Cat)$.

\begin{df}
Let $(\Ma\dt,\theta_{_{i\dt|j}})$ be an $\O$-algebra. We say that $\Ma\dt$ is a  locally presentable $\O$-algebra if the following conditions holds. 
\begin{itemize}
\item For every $i \in C$ the category $\Ma_i$ is a locally presentable category in the usual sense.
\item For every $(i_1,...,i_n; j)$ the functor $\theta_{_{i\dt|j}}$ 
%$\theta_{_{i\dt|j}}: \O(i_1,...,i_n; j) \times \Ma_{i_1} \times \cdots \times \Ma_{i_n} \to  \Ma_j$
preserves the colimits on each factor `$i_k$' $(1 \leq k \leq n)$ that is for every $(m_l)_{l\neq k} \in \prod_{l,l \neq k} \Ma_{i_l}$ and every $x  \in \O(i_1,...,i_n; j)$ the functor  
$$\theta_{_{i\dt|j}}{(x;(m_l))}:=\theta_{_{i\dt|j}}{(x; ..., m_l, ...m_{k-1}, -, m_{k+1},...)}: \Ma_{i_k} \to \Ma_{j}$$ 
preserves all colimits.
\end{itemize}
\end{df}

\begin{ex}\ \
\begin{enumerate}
\item If  $\O$ is the operad of enriched categories, then any  symmetric closed monoidal category $\M$ which is locally presentable is automatically a locally presentable $\O$-algebra. The second condition of the definition follows from the fact that being closed monoidal  imply that the tensor product of $\M$ (which is a left adjoint) preserves colimits on each factor.  
\item More generally any biclosed monoidal category $\M$ (see \cite[1.5]{Ke}), not necessarily symmetric, which is locally presentable is a locally presentable $\O$-algebra.
\item Any $2$-category (or bicategory) such that the composition preserves the colimits on each factor and all the category of morphisms are locally presentable,  is a  locally presentable $\O_X$-algebra for the operad $\O_X$ of the Example \ref{operad-principal}.
\end{enumerate}
\end{ex}

\begin{rmk}
In the same fashion way we will say that $\Ma\dt$ is a cocomplete $\O$-algebra if all the $\Ma_i$  are cocomplete and if the second condition of the previous definition holds. 
\end{rmk}
The main result in this section is the following.
\begin{thm}\label{laxalg-local-pres}
Let $\Ma\dt$ be a locally presentable $\O$-algebra. For any $\O$-algebra $\Ca\dt$ the category of lax $\O$-morphisms $\Laxalg(\Ca\dt, \Ma\dt)$ is locally presentable.  
\end{thm}
 
\begin{proof}
 See Appendix \ref{Proof-local-present-O-alg}
\end{proof}
 
\subsection{Special Quillen $\O$-algebra} 

In the following we consider an \emph{ad-hoc} notion of Quillen $\O$-algebra. 
\begin{df}\label{quillen-alg}
Let $(\Ma\dt,\theta_{_{i\dt|j}})$ be an $\O$-algebra. We say that $\Ma\dt$ is a \textbf{special Quillen $\O$-algebra} if the following conditions holds. 
\begin{enumerate}
\item $\Ma\dt$ is complete and cocomplete,
\item For every $i \in C$ the category $\Ma_i$ is a Quillen closed model category in the usual sense.
\item For every $x \in \O(i_1,...,i_n; j)$, the functor $\otimes_x$ preserves (trivial) cofibrations with cofibrant domain. This means that  for every  $n$-tuple of morphisms $(g_k)_{k}$ in $\Ma_{i_1} \times \cdots \times \Ma_{i_n}$, such that each $g_k$ has a cofibrant domain, then 
$\otimes_x(g_1, ..., g_n)$ is a (trivial) cofibration in $\Ma_j$ if all $g_1,...,g_n$ are (trivial) cofibrations.  
\end{enumerate}
Say that $\Ma\dt$ is cofibrantly generated if all the $\Ma_i$ are cofibrantly generated. Similarly if each $\Ma_i$ is combinatorial we will say that $\Ma\dt$ is combinatorial. 
\end{df}

\begin{ex}\ \
\begin{itemize}
\item Any model category is obviously a special Quillen $\O$-algebra with the tautological operad (no operations except the $1$-ary identity  operation). 
\item  Another example of special Quillen algebra is a symmetric monoidal model category. In fact using the pushout-product axiom one has that (trivial) cofibrations with cofibrant domain are closed by tensor product.% (applying \cite[lemma 4.2.4]{Hov-model}).
\end{itemize}
\end{ex}

\begin{rmk}
Note that in our definition we did not include a generalized pushout product axiom; it doesn't seem relevant, for our purposes,  to impose this axioms in general. But if one is interested of having such axiom, a first approximation will be of course to mimic the monoidal situation. Below we give a sketchy one. \ \\
\ \\
\ul{\textbf{Axiom}:} Say that $\Ma\dt$ is pushout-product compatible  if:\ \\
\begin{itemize}
\item for every $x \in \O(i_1,...,i_n; j)$ 
\item for every cofibrations  $f: a_k \to b_k  \in \Ma_{i_k}$, $g: a_l \to b_l \in \Ma_{i_l}$,
\item for every $(n-2)$-tuple of cofibrant objects $(c_r)_{r \neq l, r \neq l} $  
\end{itemize} 
then the map $$\delta: \otimes_x(-, a_k ,-, b_l, -) \cup_{\otimes_x(-, a_k , -, a_l, -)} \otimes_x(-, b_k , -, a_l, -) \to
 \otimes_x(-, b_k , -, b_l, -)$$ 
is a cofibration which is moreover a trivial cofibration if either $f$ of $g$ is. 
 
 \[
\xy
(-40,25)*+{\otimes_x(c_1,..., a_k , -, a_l,..., c_n)}="X";
(30,25)*+{\otimes_x(c_1,..., a_k , -, b_l,..., c_n)}="Y";
(-40,0)*+{\otimes_x(c_1,..., b_k , -, a_l,..., c_n)}="A";
(30,0)*+{\otimes_x(c_1,..., b_k , -, b_l,..., c_n)}="B";
{\ar@{->}^{\otimes_x(\Id,..., \Id , -, g,..., \Id)}"X";"Y"};
{\ar@{->}_{\otimes_x(\Id,..., \Id , -, g,..., \Id)}"A";"B"};
{\ar@{->}^{\otimes_x(\Id,..., f , -, \Id,..., \Id)}"Y";"B"};
{\ar@{->}_{\otimes_x(\Id,..., f , -, \Id,..., \Id)}"X";"A"};
(0,10)*+{.}="E";
{\ar@{.>}^{}"Y";"E"};
{\ar@{.>}_{}"A";"E"};
{\ar@{->}^-{\delta}"E";"B"};
\endxy
 \]
\end{rmk}
\ \\
The main result in this section is to say that under some hypothesis on the triple $(\O, \Ca\dt, \Ma\dt)$ then there is a model structure on $\Laxalg(\Ca\dt,\Ma\dt)$.  We don't know for the moment if we have the same result without any restriction. 
We will denote by $\kc=\prod_i \Hom(\Ca_i, \Ma_i)$. 

\begin{df}
Let $(\Ca\dt,\rho)$ and $(\Ma\dt, \theta)$ be two $\O$-algebras.
\begin{enumerate}
\item Say that $\Ca\dt$ is  \textbf{$\O$-well-presented}, or \textbf{$\O$-identity-reflecting} (henceforth $\iro$-algebra) if
for every $n+1$-tuple $(i_1,...,i_n; j)$ the following functor reflects identities
$$\rho: \O(i_1,...,i_n; j) \times \Ca_{i_1} \times \cdots \times \Ca_{i_n} \to \Ca_j.$$
This means that the image of $(u, f_1, ...,f_n) \in \O(i_1,...,i_n; j) \times \Ca_{i_1} \times \cdots \times \Ca_{i_n} $ is an identity morphism in $\Ca_j$ (if and) only if all $u, f_1,...,f_n$ are simultaneously identities.
\item Say that $(\Ca\dt, \Ma\dt)$ is an \textbf{$\O$-homotopy-compatible pair} if  $\Fb: \kc \to \kc$ preserves level-wise trivial cofibrations, where $\kc$ is endowed with the injective model structure. Here $\Fb$ is the left adjoint of the functor $\Ub$ which forgets the laxity maps (see Appendix \ref{ub-has-left-adjoint}).
\end{enumerate}
\end{df}

The motivation of these definitions is explained in the Appendix \ref{pushout-laxalg}.\ \\
With the previous material we have
\begin{thm}\label{model-laxalg}
For an $\iro$-algebra $\Ca\dt$, and a special Quillen $\O$-algebra $\Ma\dt$ assume that 
\begin{itemize}
\item $(\Ca\dt, \Ma\dt)$ is an $\O$-homotopy compatible pair, 
\item all objects of $\Ma\dt$ are cofibrant, 
\item $\Ma$ is cofibrantly generated with $\I\dt$ (resp. $\Ja\dt$) the generating set of (trivial) cofibrations
\end{itemize} 
then there is a model structure on $\Laxalg(\Ca\dt, \Ma\dt)$ which is cofibrantly generated. A map $\sigma: \Fa \to \Ga$ is 
\begin{itemize}
\item a weak equivalence if $\Ub \sigma$ is a weak equivalence in $\kc$, 
\item a fibration if $\Ub \sigma$ is a fibration in $\kc$, 
\item a  cofibration if it has the LLP with respect to all maps which are both fibrations and weak equivalences,
\item a  trivial cofibration if it has the LLP with respect to all fibrations.
\item the set $\Fb(\I\dt)$ and $\Fb(\Ja\dt)$ constitute respectively the set of generating cofibrations and trivial cofibrations in $\Laxalg(\Ca\dt, \Ma\dt)$.
\end{itemize} 
The pair $$\Ub : \Laxalg(\Ca\dt, \Ma\dt) \leftrightarrows \prod_i \Hom(\Ca_i, \Ma_i) : \Fb$$
 is a Quillen pair, where $\Fb$ is left Quillen and $\Ub$ right Quillen. 
\end{thm}

\begin{proof}
The idea is to transfer the (product) model structure on $\K_{\Ca\dt}=\prod_i \Hom(\Ca_i, \Ma_i)$ through the monadic adjunction $\Fb \dashv \Ub$ using a lemma of Schwede-Shipley \cite{Sch-Sh-Algebra-module}. In fact $\Laxalg(\Ca\dt,\Ma\dt)$ is equivalent to $\T$-alg for the monad $\T= \Ub \Fb$. The method is exactly the same as in the proof of theorem \ref{main-proj-msx}.\ \\
All we have to check is that the pushout of  $\Fb \sigma$ is a weak equivalence for every  generating trivial cofibration $\sigma$ in $\K_{\Ca\dt}$. This is exposed in the Appendix \ref{pushout-laxalg}. 
\end{proof}
 
\section*{An alternative description of  $\Laxalg(\Ca\dt, \Ma\dt)$ }

In the following we fix a multi-sorted operad $\O$ and an $\O$-algebra $\Ca\dt$. Our goal is to describe the category $\Laxalg(\Ca\dt, \Ma\dt)$ as subcategory of $\Laxalgp(\1\dt, \Ma\dt)$ for some operad $\O'= \O_{\Ca\dt} = \int_{\Ca\dt}$; here $\1\dt$ is the terminal algebra. This will simplify many construction such as pushouts and colimit in general. 
%%%%%%%%%%%%%%%%%%%%%%%%%%%%%%%%%%%%%%%%%%%%%%%%%%%%%%%%%%%%%%
\subsection*{Definition of $\O_{\Ca\dt}$}
By definition of $\Ca\dt$, for each $(n+1)$-tuple we have an action of $\O$  given by a functor 
$$\theta_{_{i\dt|j}}: \O(i_1,...,i_n; j) \times \Ca_{i_1} \times \cdots \times \Ca_{i_n} \to  \Ca_j.$$
When there is no confusion we will omit the subscript and will write simply $\theta$.\ \\
\ \\
The set $D$ of colours or sorts of $\O_{\Ca\dt}$ is the set of object of $\Ca\dt$, that is $ D= \coprod_{i \in C} \Ob(\Ca_i)$.\ \\
Given an $(n+1)$-tuple $(c_1, ..., c_n, c_j) \in \Ca_{i_1} \times \cdots \times \Ca_{i_n} \times \Ca_j$, we define the category of operations 
$\O_{\Ca\dt}(c_1, ..., c_n, c_j)$ as follows:
\begin{itemize}
\item the objects are pairs $(x, h)$, with $x \in  \O(i_1,...,i_n; j)$ and $h : \theta(x,c_1,..., c_n) \to c_j $ a morphism in $\Ca_j$  
\item  a morphism from $(x, h)$  to $(y, k)$  is a morphism $u : x \to y$ in  $\O(i_1,...,i_n; j)$ such that $h= k \circ \theta(u,c_1,..., c_n)$; or  equivalently $\theta (u, c_{\dt})$ is a morphism from $h$ to $k$ in the slice category $\Ca_j /c_j$.
\end{itemize} 

If $\gamma$ is the mutliplication or substitution of $\O$, then we define the associated multiplication $\gamma_{_{\Ca\dt}}$ in the natural way to be a mixture of  $\gamma$ and $\theta$. \ \\

Given  $[(x_i, h_i)]_{1 \leq i \leq n}$  with $(x_i, h_i) \in \O_{\Ca\dt}(d_{i,1}, ..., d_{1,k_i}, c_i)$ then we set
$$\gamma_{_{\Ca\dt}} [(x_1,h_1), ..., (x_n,h_n)] := [\gamma(x_1, ..., x_n), \theta(d_{1,1}, ..., d_{n,k_n})].$$

\begin{prop}
The data $\O_{\Ca\dt}(c_1, ..., c_n, c_j)$  with $\gamma_{_{\Ca\dt}}$ constitute a $D$-multisorted $\Cat$-operad.
\end{prop}

\begin{proof}
The associativity of $\gamma_{_{\Ca\dt}}$ follows from the associativty of $\gamma$ and $\theta$.
\end{proof}
\begin{rmk}
Note that there is a function $p : D \to C$ between the set of colours which is just the subscript-reading operation: for  $c_i \in \Ob(\Ca_i)$  $p(c_i)=i$. Pulling back $\O$ along $p$ we get a $D$ multisorted operad $p^{\star}\O$. \ \\
We have then that for each $(c_1, ..., c_n, c_j) \in D^{n+1}$, $\pstar \O(c_1, ..., c_n, c_j) =\O(i_1,...,i_n; j)$.\ \\ 
The projection on the first factor is a functor $\pi: \O_{\Ca\dt}(c_1, ..., c_n, c_j) \to  \O(i_1,...,i_n; j)$ ( $\pi (x,h)=x$) and it's not hard to see that these functors $\pi$ fit coherently to form a morphism of $D$-multisorted operads denoted again $\pi: \O_{\Ca\dt} \to \pstar \O$.\ \\

For an $\O$-algebra $\Ma$, by $p$ we have an $\pstar \O$-algebra and by $\pi$ we have an $\O_{\Ca\dt}$-algebra $\pistar[\pstar \Ma\dt]$. When there is no confusion we will simply write $\pistar \Ma\dt$.  
\end{rmk}

\begin{df}
Let $\Ca$ and $\D$ be  two small $1$-categories. A \textbf{prefunctor}  $F: \Ca \to \Ma$ is an object given by the same data and axioms of a functor except the preservations of identities, that is we do not require to have  $F(\Id_A)= \Id_{FA}$ for $A \in \Ca$. 
\end{df}
In other terms a prefunctor is the same thing as a morphism between the underlying graphs which is compatible with the composition on both sides. \ \\ 
The compatibility of the composition forces each $F(\Id_A)$ to be an idempotent in $\Ma$. Obviously any functor is a prefunctor.\ \\ 

In the same fashion way given two $\O$-algebras $\Ca\dt$ and $\Ma\dt$ a \textbf{prelax $\O$-morphism}  
$\Fa\dt: \Ca\dt \to \Ma\dt$  is the same thing as a lax $\O$-morphism except that each $\Fa_i: \Ca_i \to \Ma_i$ is a prefunctor.

\begin{prop}
Let $\Ca\dt$ and $\Ma\dt$ be two $\O$-algebras. We have an equivalence between the following data:
\begin{enumerate}
\item a \textbf{prelax  $\O$-morphism} from $\Ca\dt$ to $\Ma\dt$,
\item a \textbf{lax $\O_{\Ca\dt}$-morphism} from $\1\dt$ to $ \pistar \Ma\dt$.
\end{enumerate}
\end{prop}

\begin{proof}
Simply write the definition of each object. 
\end{proof}

%%%%%%%%%% mettre l'équivalence entre lax oalg et lax (\star, \m) pour O_Ca

\section{Co-Segal Categories}\label{section-cosegal-cat}
\subsection{The one-object case}
\begin{conv}\ \
\begin{itemize}
\item By \textbf{semi-monoidal category} we mean the same structure as a monoidal category  except that no unit object is required. Obviously any monoidal category has an underlying semi-monoidal category.\\
\item A \textbf{lax functor} between semi-monoidal categories is the same thing as a lax functor between monoidal categories without the data involving the units. A strict lax functor will be call as well `monoidal functor'. \\
\item More generally we will say \textbf{semi-bicategory} (resp. semi-$2$-category) to be the same thing as bicategory (resp. $2$-category) except that we don't require the identities $1$-morphisms.\\
\item We have also the notion of lax morphism, transformation of lax morphisms, between semi-bicategories in the natural way.\\ 
\item For a semi-bicategory $\A$ and a bicategory $\B$, a lax morphism from $\A$ to $\B$ will be a morphism from $\A$ to the underlying semi-bicategory of $\B$ which will be denoted again $\B$. 
\end{itemize}
\end{conv}
In the following we fix $\M=(\ul{\tx{M}},\otimes, I)$ a monoidal category.
\subsection{Overview}\ \\
\indent As we identify  $\M$-categories with one object and monoids of $\M$, we shall expect that a Co-Segal category with one object will be a kind of \emph{homotopical semi-monoid}\footnote{In the standard terminology we would have said `up-to-homotopy' monoid but this terminology is already used for another notion of weak monoid (see \cite{SEC1}).}  of $\M$. We will call them \emph{Co-Segal semi-monoids}.\ \\   

To define a Co-Segal category $\C$ with one object $A$, we need  a sequence of objects of $\M$% 

 \begin{equation*}
  \begin{cases}
  \C(A,A)   \rightsquigarrow \C(\1) &  \tx{: the `hom-space' of $A$} \\
  \C(A,A,A) \rightsquigarrow \C(\2)   \\
   ~~~~~   \cdots \\
  \C(\n\ast A)=\C(\underbrace{A,...,A}_{(n+1)\tx{-A}}) \rightsquigarrow \C(\n)   & n\geq 1  \\
  \end{cases}
\end{equation*}

 together with the following data. 

\begin{enumerate}
\item A diagram  expressing a `quasi-multiplication'
\[
\xy
(-15,0)*+{\C(A,A) \otimes \C(A,A)}="X";
(30,0)*+{\C(A,A,A)}="Y";
(30,18)*+{\C(A,A)}="E";
{\ar@{->}^{\mu_{1,1}}"X";"Y"};
{\ar@{->}^{\tx{weak.equiv}}_{\gamma_0}"E";"Y"};
{\ar@{.>}^{\tx{generic lifting} \ \ \ \ \ \ \ }"X";"E"};
\endxy
\]

\item Some other semi-multiplications: $\C(\n\ast A) \otimes \C(\m \ast A) \xrightarrow{\mu_{n,m}} \C((\n+\m) \ast A)$\\

\item In the `semi-cubical' diagram below each face is commutative and the weak equivalences  $\gamma_i$ must satisfy : $ \gamma_1 \circ \gamma_0 = \gamma_2 \circ \gamma_0$ to have an associativity up-to homotopy. 

\[
\xy
%Le digram de dessus%%%
(-40,20)*+{\C(A,A) \otimes \C(A,A) \otimes \C(A,A)}="A";
(30,20)*+{\C(A,A) \otimes \C(A,A,A)}="B";
(-90,0)*+{\C(A,A,A) \otimes \C(A,A)}="C";
(-20,0)*+{\C(A,A,A,A)}="D";
%%% fleches au dessus
{\ar@{->}^{\Id \otimes \mu_{1,1}}"A";"B"};
{\ar@{->}_{\mu_{1,1} \otimes \Id}"A";"C"};
{\ar@{->}^{\mu_{1,2}}"B";"D"};
{\ar@{->}_{\mu_{2,1}}"C";"D"};
% le diagram d'en dessous%%%
(30,-20)*+{\C(A,A) \otimes \C(A,A)}="Y";
(-90,-40)*+{\C(A,A) \otimes \C(A,A)}="Z";
(-20,-40)*+{\C(A,A,A)}="W";
{\ar@{->}^{\mu_{1,1}}"Y";"W"};
{\ar@{->}^{\gamma_0 \otimes \Id}"Z";"C"};
{\ar@{->}^{\Id \otimes \gamma_0}"Y";"B"};
{\ar@{->}_{\mu_{1,1}}"Z";"W"};
%%%%%%%%% les flechess courbées%%%%%%%%%%%
{\ar@/_2.3pc/^{\gamma_2}"W"; "D"};
{\ar@/^2.3pc/^{\gamma_1}"W"; "D"};
(-20,-60)*+{\C(A,A)}="X";
{\ar@{->}^{\gamma_0}"X";"W"};
\endxy
\]
\\
\item We have other commutative diagrams of the same type as above which give the coherences of this  weak associativity of the quasi-multiplication etc.\\
\end{enumerate}
\ \\
\indent As one can see when all of the maps `$\gamma_i$' are isomorphism we will have the data of semi-category with one object i.e a semi-monoid of $\M$. In this case we know from Mac Lane \cite{Mac} that a semi-monoid in $\M$ is given by a monoidal functor:
$$ \Nv(\C) : (\Delta_{epi},+,\0) \to \M $$
which we interpret as the nerve of the semi-enriched category $\C$ with one object. \ \\

\begin{rmk}
The object $\0$ doesn't play any role here since there is no morphism from any another object to it.  So we can restrict this functor to the underlyinng semi-monoidal categories (see Definition \ref{def-semi-mon} below). 
\end{rmk}

\subsection{Definitions}\
%%%Notations
\begin{nota}\
\begin{itemize}
\item We will denote by $\n$ the set $\{0, \cdots , n-1 \}$ with the natural order on it. 
\item The objects of $\Delta$ will be identified with those $\n$ and the morphisms will be the nondecreasing functions. The object $\0$ corresponds to the empty set.
\end{itemize} 
\end{nota}
%%%%%%% fin notations
\begin{df}
Let $\Upsilon = (\Delta_{\tx{epi}},+) $ be the semi-monoidal subcategory of $(\Delta,+,\0)$ described as follows. 
\begin{itemize}
\item $\ob(\Upsilon)= \ob(\Delta)-\{\0\}$.
\item The morphisms are: 
 \begin{equation*}
\Upsilon(\m,\n) =
  \begin{cases}
  \{ f \in \Delta(\m,\n), ~~~ \text{$f$ is surjective} \}  & \text{if $\m \geq \n > \0  $}\\
  \varnothing & \text{otherwise.} \\
  \end{cases}
\end{equation*}

\end{itemize}
\end{df}

\begin{rmk}\
\begin{enumerate}
\item For $f \in \Upsilon(\m,\n)$, by definition $f$ is surjective and nondecreasing then it follows that $f$ preserves the  `endpoints' i.e $f(0)=0$ and $f(m-1)=n-1$.\\
\item For $\n \geq \1$ we denote by $\sigma^{n}_i$ the unique map  of $\Upsilon$ from $\n + \1$ to $\n$ such that $\sigma_i(i)=\sigma_i(i+1)$ for $i \in \n=\{0, \cdots , n-1 \}$. The maps $\sigma^{n}_i$ generate all the maps in $\Upsilon$ (see \cite{Mac}) and satisfies the simplicial identities:
$$\sigma^{n}_j \circ \sigma^{n+1}_i = \sigma^{n}_i \circ \sigma^{n+1}_{j+1}, ~~~~~~~~~~~~~~~\\\ \\\ i \leq j .$$

%\item A morphism $f : \m \to \n$ can be interpreted as `doing some multiplication' when $\n \neq \0 $, 
\item Mac Lane \cite{Mac} pointed out that just like $(\Delta,+,0)$, $\Upsilon$ contains the universal semi-monoid which still corresponds to the object $\1$ together the (unique) map $\sigma^{1}_0: \2 \to \1$. 
\end{enumerate}
\end{rmk}
\ \\
Now we can take as definition. 
\begin{df}\label{def-semi-mon}
Let $\M=(\ul{\tx{M}},\otimes, I)$ be a monoidal category. A semi-monoid of $\M$ is a \textbf{monoidal functor} 
$$F : \Upsilon \to \M .$$
\end{df}

We now assume that $\M$ is equipped with a class of map called homotopy or weak equivalences. We refer the reader to \cite{SEC1} for the definition of \emph{base of enrichment}. 

\begin{df}\label{Sim-semimon}
Let $(\M,\W)$ be base of enrichment.

A \textbf{Co-Segal semi-monoid} of $(\M,\W)$ is a \textbf{lax monoidal} functor
$$ F : \Upsilon^{op} \to \M $$
satisfying the Co-Segal conditions: \\
for every $f \in \Upsilon(\m,\n)$ the morphism $F(f) : F(\n) \to F(\m)$ is a weak equivalence i.e  $F(f) \in \W$. 
\end{df}

\begin{rmk}\
\begin{enumerate}
\item It's important to notice that in the first definition we use $\Upsilon = (\Depi,+)$ while in the second we use $\Upsilon^{op} = (\Depi^{op},+)$. 
\item Here as usual, the underlying semi-monoid is the object $F(\1)$.
\item Finally it's important to notice that since the morphism of $\Upsilon$ are generated by the maps $\sigma^{n}_i$ and because $\W$ is stable by composition, it suffices to require the Co-Segal conditions only for the maps $F(\sigma^{n}_i)$.
\end{enumerate}
\end{rmk}

 To understand the definition one needs to see the data that $F$ carries. 
 
\begin{obs}\label{obs_semimon} \
\begin{enumerate}
\item By definition of a lax morphism for every $\bf{\n,\m}$ we have a `laxity map' 
$$ F_{n,m}: F(\n) \otimes F(\m) \to F(\n+\m).$$
In particular for $\m=\n=\1$ we have a map $F_{1,1}: F(\1) \otimes F(\1) \to F(\2)$.\\

\item The Co-Segal condition for $f=\sigma^{1}_0: \2 \to \1$ says that the map $F(\sigma^{1}_0): F(\1) \to F(\2)$ is a weak equivalence. If we combine this map with the previous laxity map we will have a quasi-multiplication as we described earlier:
\[
\xy
(-15,0)*+{F(\1) \otimes F(\1)}="X";
(30,0)*+{F(\2)}="Y";
(30,18)*+{F(\1)}="E";
{\ar@{->}^{F_{1,1}}"X";"Y"};
{\ar@{->}^{\tx{weak.equiv}}_{F(\sigma^{1}_0)}"E";"Y"};
\endxy
\] 

\item For every  $f : \n \to \n'$  and $g : \m \to \m'$  the following diagram commutes 

\[
\xy
(-30,20)*+{F(\n') \otimes F(\m')}="A";
(30,20)*+{F(\n'+\m')}="B";
(-30,0)*+{F(\n) \otimes F(\m)}="C";
(30,0)*+{F(\n+\m)}="D";
{\ar@{->}^{F_{n',m'}}"A";"B"};
{\ar@{->}_{F(f) \otimes F(g)}"A";"C"};
{\ar@{->}^{F(f+g)}"B";"D"};
{\ar@{->}^{F_{n,m}}"C";"D"};
\endxy
\]
\\
\item For every triple of objects $(\m,\n,\p)$ using the laxity maps and the maps `$F(f)$' we have some semi-cubical commutative diagrams as before, which will give the associativity up-to-homotopy and the suitable coherences. 
\end{enumerate}
\end{obs}
\
\begin{term}\
\begin{itemize}
\item When all the maps $F(f)$ are isomorphisms then we will say $F$ is a  strict Co-Segal semi-monoid or a Co-Segal semigroup.
\item Without the Co-Segal conditions in the Definition \ref{Sim-semimon} we will say that $F$ is a pre-semi-monoid.
\end{itemize}
\end{term}

\begin{prop}\label{equiv-semimon}
We have an equivalence between the following data:
\begin{itemize}
\item a classical semi-monoid or semigroup of $\M$
\item a strict Co-Segal semi-monoid of $\M$.
\end{itemize}
\end{prop}

When we will define the morphisms between Co-Segal semi-monoids, this equivalence will automatically be an equivalence of categories.

\begin{proof}[\scshape{Sketch of proof}]\
\renewcommand\labelenumi{\alph{enumi})}
\begin{enumerate}
\item Let $F: \Upsilon \to \M$ be a semi-monoid. We define the corresponding Co-Segal semi-monoid $\tld{F}$ as follows.\\
\begin{itemize}
\item We set $\tld{F}(\n)= \tld{F}(\1):= F(\1)$ for every $\n$, and 
for every $f:\m \to \n$ we set $\tld{F}(f):= \Id_{F(\1)}$.\\ 
 
\item Finally the laxity maps correspond to the multiplication of the semi-monoid $F(\1)$ i.e  $\tld{F}_{n,m}$ is the composite:

$$F(\1)\otimes F(\1) \xrightarrow{\Id} F(\2) \xrightarrow{F(\sigma_0^{1})} F(\1)$$
%%% mettre en évidence qui est F tild (n) etc ... 
for $\m \geq \1 , \n \geq \1 $.\\ 
\end{itemize}

\item Conversely let $ G : \Upsilon^{op} \to \M $ be a \ul{strict} Co-Segal semi-monoid. We get a semi-monoid $[G]$ in the following manner.\\
\begin{itemize}
\item $[G](\1)= G(\1)$,\\
\item $[G](\n)= G(\1)^{\otimes n} = \underbrace{G(\1) \otimes \cdots \otimes G(\1)}_{\tx{n-times}}$,\\
\item We have a multiplication $\mu : [G](\1) \otimes [G](\1) \to [G](\1)$ which is the map $G(\sigma^{1}_0)^{-1} \circ G_{1,1}$ obtained from the diagram below:
\[
\xy
(-15,0)*+{G(\1) \otimes G(\1)}="X";
(30,0)*+{G(\2)}="Y";
(30,18)*+{G(\1)}="E";
{\ar@{->}^{G_{1,1}}"X";"Y"};
{\ar@{->}^{\cong}_{G(\sigma^{1}_0)}"E";"Y"};
{\ar@{-->}^{\mu}"X";"E"};
\endxy
\] 
\\
\item On morphism, we define $[G]$ on the generators by 
$[G](\sigma^{n}_i):= \Id_{G(\1)^{\otimes i}}\otimes \mu \otimes \Id_{G(\1)^{\otimes n-i-1}} $ \\ 
\item Finally one gets the associativity from the semi-cubical diagram mentioned before. 
\end{itemize} 
\end{enumerate}
\end{proof}
%%%%%%%%%%%%Nouvelle section: Cas plusieurs objets 
\subsection{The General case: Co-Segal Categories}
%%%%%%%%%%%Subsection %%%%%%%%
\subsection{$\S$-Diagrams} 

In addition to the notations of the previous section, we will also use the following ones. 
\begin{nota}\ \\
$\tx{Cat}_{\leq 1}=$ the $1$-category of small categories with functors.\\
$\bicatd$= the $2$-category of bicategories, lax morphisms and icons (\cite[Thm 3.2]{Lack_icons}). \footnote{Note that $\bicatd$ is not the standard one which includes all transformation. The standard one is \textbf{not} a $2$-category.}\\
$\frac{1}{2}\bicatd=$ the category of semi-bicategories, lax morphisms and icons.\\
$\P_{\C}=$ the $2$-path-category associated to a small category $\C$ (see \cite{SEC1}) .\\
$\Delta^{+}=$ the category $\Delta$ without the object $\0$ i.e the category of finite nonempty ordinals.\\
$\Upsilon^{+}=$ the category $\Upsilon$ without the object $\0$.\\
$\uc=\{\o,\o \xrightarrow{Id_{\o}} \o \}=$ the unit category. \\
$\ol{X}=$ the coarse category associated to a nonempty set $X$ (see \cite{SEC1}).\\
$(\B)^{2\tx{-op}}=$ the $2$-opposite (semi) bicategory of $\B$. We keep the same $1$-cells but reverse the $2$-cells i.e
$$(\B)^{2\tx{-op}}(A,B):= \B(A,B)^{\tx{op}}.$$
$(\M,\W)=$ a base of enrichment, with $\M$ is a general bicategory. \\
$2$-$\tx{Iso}(\M)=$ the class of invertible $2$-morphisms of $\M$. Recall that $(\M,2\tx{-Iso})$ is the smallest base of enrichment. 
\end{nota}
\
\begin{note}
We will freely identify bicategories and $2$-categories. And as usual  monoidal categories will be identified with bicategories with one object.
\end{note}

Recall that the $2$-path category $\P_{\C}$ is a generalized version of the monoidal category $(\Delta,+,\0)$ in the sense that when  $\C \cong \uc$ then $\P_{\C} \cong (\Delta,+,\0)$. It has been shown in \cite{SEC1} that a classical enriched category with a small set of objects was the same thing as a homomorphism in the sense of Bénabou from $\P_{\ol{X}}$ to $\M$, for some set $X$.\ \\

In what follows we introduce a generalized version of the semi-monoidal category $\Upsilon=(\Delta_{\tx{epi}},+)$ just like we did for $(\Delta,+,\0)$. For $\C$ a small category, we consider $\S_{\C}$, a semi-$2$-category  contained in $\P_{\C}$, such that $\S_{\uc}$ `is' $\Upsilon$.

\begin{pdef}\label{path-bicat}
Let $\C$ be a small category.\\
There exists a strict semi-$2$-category $\S_{\C}$ having the following properties.

\begin{itemize}[label=$-$]
\item the objects of $\S_{\C}$ are the objects of $\C$,
\item for every pair $(A,B)$ of objects, $\S_{\C}(A,B)$ is a category over $\Upsilon$ i.e we have a functor called \textbf{length} or \textbf{degree} 
$$ \le_{AB} : \S_{\C}(A,B) \to \Upsilon$$  
\item $\le_{AA}$ becomes naturally a monoidal functor when we consider the composition on $\S_{\C}(A,A)$,
\item if $\C \cong \uc$, say $ob(\C)=\{\o\}$ and $\C(\o,\o)=\{\Id_{\o}\}$, we have an isomorphism of semi-monoidal categories:
$$\S_{\C}(\o,\o) \xrightarrow{\sim} \Upsilon $$
\item the operation $\C \mapsto \S_{\C}$ is functorial in $\C$:
\[
\xy
(0,8)*+{\S_{[-]}: \Cat_{\leq 1}}="X";
(30,8)*+{\frac{1}{2}\bicatd}="Y";
(3,0)*+{\C \xrightarrow{F} \D}="E";
(35,0)*++{\S_{\C} \xrightarrow{\S_{F}} \S_{\D}}="W";
{\ar@{->}"X";"Y"};
{\ar@{|->}"E";"W"};
\endxy
\]
\end{itemize}
\end{pdef}
\
\begin{proof}
$\S_{\C}$  is the object obtained from the genuine fibred product of  semi-$2$-categories:

\[
\xy
(0,20)*+{\S_{\C}}="A";
(30,20)*+{\P_{\C}}="B";
(0,0)*+{(\Delta_{\tx{epi}},+)}="C";
(30,0)*+{(\Delta,+,\0)}="D";
{\ar@{^{(}->}^-{i}"A";"B"};
{\ar@{.>}_-{\le}"A";"C"};
{\ar@{^{(}->}_-{i}"C";"D"};
{\ar@{->}^{\le}"B";"D"};
\endxy
\]
\end{proof}

\begin{note}
We will be interested in particular to the cases where $\C$ is of the form $\ol{X}$, the \emph{ indiscrete} or \emph{ coarse category} associated to a nonempty set $X$. In that case an object of $\S_{\ol{X}}(A,B)$ can be identified with an $(n+1)$-tuple $(E_0, \cdots, E_n)$ of elements of $X$ for some $n$,  with $E_0=A$ and $E_n=B$.  For simplicity we will use small letters: $r, s, t,$..., to represent such chains $(E_0, \cdots, E_n)$.

A morphism $u : t \to s$ of  $\S_{\ol{X}}(A,B)$  can be viewed as an operation which deletes some letters of $t$ to get $s$,  keeping $A$ and $B$ fixed.
\end{note}
In the upcoming definitions we consider  a $2$-category $\M$ which is also a special Quillen $\O$-algebra for the operad `$\O_X$' of $2$-categories. This situation covered also the special case of a $2$-category which is locally a model category (Definition \ref{model-2-cat}).

\begin{df}\label{s-diagram}
Let $\M$ be a $2$-category which is a special Quillen algebra.\ \\ 
An $\S$-diagram of $\M$ is a lax morphism $F : (\S_{\C})^{2\tx{-op}} \to \M$ for some  $\C$.
We will say for short that $F$ is an $\S_{\C}$-diagram of $\M$.
\end{df}
One can observe that this definition is the generalization of Definition \ref{Sim-semimon} without the Co-Segal conditions. 
\begin{df}\label{cosegal-diagram}
Let $\M$ be a $2$-category which is a special Quillen algebra.\ \\
A \textbf{Co-Segal $\S$-diagram} is an $\S$-diagram 
$$F : (\S_{\C})^{2\tx{-op}} \to \M $$
satisfying the \textbf{Co-Segal conditions}: for every pair $(A,B)$ of object of $\C$,   the component 
$$F_{AB}: \S_{\C}(A,B)^{op} \to \M(FA,FB)$$
takes its values in the subcategory of weak equivalences.  This means that  for every $u : s \to s'$ in  $\S_{\C}(A,B)$,  the $2$-morphism\\
 $$F_{AB}(u): F_{AB}(s') \to F_{AB}(s)$$ is a weak equivalence in the model category $\M(FA,FB)$. 
\end{df}

\begin{term}
When all the maps  $F_{AB}(u)$ are $2$-isomorphisms , then we will say that $F$ is as strict Co-Segal $\S_{\C}$-diagram of $\M$.
\end{term}

\begin{obs}\label{Co-Segal-final-cond}
By construction of $\S_{\C}$ for every pair of objects $(A,B)$  and  for every $t  \in \S_{\C}(A,B)$ we have a unique element $f \in \C(A,B)$ and a unique morphism $u_t: t \to [1,f]$. \ \\
Concretely $t$ is a chain of composable morphisms such that the composite is $f$, or equivalently $t$ is a `presentation' (or factorization) of $f$ with respect to the composition.  It follows that for any morphism $v:t \to s$  we have that $u_t= u_s \circ v$.\ \\

Since in each $\M(FA,FB)$ the weak equivalence have the $3$-out-of-$2$ property and are closed by composition, it's easy to see that $F$ satisfies the Co-Segal conditions if and only if $F(u_t)$ is a weak equivalence for all $t$ and all pair $(A,B)$.  
\end{obs}

\begin{df}\label{cosegal-cat}
A \textbf{Co-Segal $\M$-category} is a Co-Segal $\S_{\ol{X}}$-diagram for some set $X$.
\end{df}

\subsubsection{Weak unity in Co-Segal $\M$-categories}
The definition of a Co-Segal $\M$-category  gives rise to a weakly enriched semi-category, which means that there is no identity morphism. But there is a natural notion of weak unity we are going to explain very briefly. This is the same situation as for $A_{\infty}$-categories which arised with weak identity morphisms (see \cite{Fukaya_1,Fukaya_2}). If $\C$ is a Co-Segal $\M$-category denote by $[\C]$ the Co-Segal $\Ho(\M)$-category we get by the change of enrichment (=base change)  $\Lb: \M \to \Ho(\M)$. $[\C]$ is a strict Co-Segal category which means that it's a semi-enriched $\Ho(\M)$-category. Then we can define

\begin{df}
Say that a Co-Segal $\M$-category $\C$ has  weak  identity morphisms if $[\C]$ is a classical enriched category over $\Ho(\M)$ (with identity morphisms) .
\end{df}
There is a natural question which is to find out whether or not it's relevant to consider a ``direct'' identity morphism i.e without using the base change $\Lb: \M \to \Ho(\M)$. At this level we don't know for the moment if such consideration is `natural'. Below we give alternative definition.

\begin{df}
 A Co-Segal $\M$-category $\C$ has  weak  identity morphisms if  for any object $A$ of $\C$ there is a map $I_A: I \to \C(A,A)$ such that for any object $B$ the following commutes  up-to-homotopy:
\[
\xy
(-15,18)*+{I \otimes \Ca(A,B)}="P";
(-15,0)*+{\Ca(A,A) \otimes \Ca(A,B)}="X";
(30,0)*+{\Ca(A,A,B)}="Y";
(30,18)*+{\Ca(A,B)}="E";
{\ar@{->}^-{\varphi_{AAB}}"X";"Y"};
{\ar@{->}^-{\tx{weak.equiv}}_{\wr}"E";"Y"};
{\ar@{->}^-{\cong}"P";"E"};
{\ar@{->}^-{I_A \otimes \Id}"P";"X"};
\endxy
\]
\end{df}

The above diagram  will give the left invariance of $I_A$; the same type of diagram will give the right invariance. Note that we've limited the invariance to the `$1$-simplices' $\C(A,B)$  of $\C$ i.e we do not require such a diagram with $\C(A_0, ...A_n)$ with $n>1$.
 There are two reasons that suggest this limitation. The first one comes from the fact that for unital $A_{\infty}$-categories, the unity condition is only required for the binary multiplication `$m_2$' (see for example Kontsevich-Soibelman \cite[Sec. 4.2]{Kon_Soi_NCG1},  Lyubashenko \cite[Def. 7.3]{Lyub_A_inf}). 
 
The other reason is that $\C(A,B)$ and $\C(A,..., A_i, ..., B)$ have the same homotopy type (the Co-Segal conditions); thus if $\C(A,B)$ is weakly invariant under $I_A$ we should have the same thing for $\C(A,..., A_i, ..., B)$.  Finally we should mention that in the grand scheme of algebra, imposing further conditions reduces the class of objects. 

The question of weak unities will be treated separately in another work.

\subsubsection{The classical examples}
In the following discussion we will use the following conventions.
\begin{itemize}[label=$-$]
\item By \textbf{semi-enriched category} we mean a structure given by the same data and axioms of an enriched category without the identities . We will  say as well \textbf{$\M$-semi-category} to mention  the base $\M$ which contains the `Hom'.  This is the generalized version of semi-monoids.
\item As for $\M$-categories, we have the morphism between $\M$-semi-categories by simply ignoring the data involving the identities.
\item Our $\M$-categories and $\M$-semi-categories will always have a small set of objects. 
\end{itemize}

The following proposition is the generalized version of Proposition \ref{equiv-semimon}. 

\begin{prop}\label{equiv-semicat}
We have an equivalence between the following data:
\begin{enumerate}
\item an  $\M$-semi-category 
\item a strict Co-Segal $\S_{\ol{X}}$-diagram of $\M$.
\end{enumerate}
\end{prop}

The proof is very similar and is straightforward. We give hereafter an outline for the case where $\M$ is a monoidal category.

\begin{proof}[Sketch of proof]

Let $\A$ be an $\M$-semi-category  with $X=Ob(\A)$.  We define the corresponding strict Co-Segal $\S_{\ol{X}}$-diagram $F=(F,\varphi)$ as follows:

\begin{enumerate}[label=$\ast$, align=left, leftmargin=*, noitemsep]
\item each component $F_{AB}: \S_{\ol{X}}(A,B) \to \M$ is a constant functor :
 \begin{equation*}
  \begin{cases}
  F_{AB}([\n,s])= F_{AB}([\1,(A,B)]):= \A(A,B) &  \tx{for all $[\n,s]$} \\
   F_{AB}(f):= \Id_{\A(A,B)} &    \tx{for all $f:[\n,s] \to [\n',s']$ in  $\S_{\ol{X}}(A,B)$} \\
  \end{cases}
\end{equation*}
\item the laxity maps are given by the composition:
$$\varphi_{s,t}:= c_{ABC} : \A(B,C) \otimes \A(A,B)  \to \A(A,C)$$
\end{enumerate}
\ \\
Conversely let $F : (\S_{\ol{X}})^{2\tx{-op}} \to \M $ be a \ul{strict} Co-Segal  $\S_{\ol{X}}$-diagram. We simply show how we get the composition of the $\M$-semi-category which is denoted  by $\M_{F}^{X}$. \\ %The associativity is simple and left to the reader.
\begin{enumerate}[label=$\ast$, align=left, leftmargin=*, noitemsep]
\item First we have Ob$(\M_{F}^{X})=X$.\\
\item We take $\M_{F}^{X}(A,B) := F_{AB}([\1,(A,B)]$, for every  $A,B \in X$. \\
\item The laxity map $\varphi_{s,t}$ for $s=[\1,(A,B)]$, $t=[\1,(B,C)]$ is a map of $\M$
$$\varphi_{s,t}: \M_{F}^{X}(B,C) \otimes \M_{F}^{X}(A,B)  \to \M_{F}^{X}(A,B,C)$$
where $\M_{F}^{X}(A,B,C):=F_{AC}([\2,(A,B,C)])$. \\
\item Now in $\S_{\ol{X}}(A,C)$  we have a unique map $[\2,(A,B,C)] \xrightarrow{\sigma^{1}_0} [\1,(A,C)] $ parametrized by the map $\sigma^{1}_0 : \2 \to \1$ of $\Upsilon$.  The image of this map by $F_{AC}$ is a map   
$$F(\sigma^{1}_0) :  \M_{F}^{X}(A,C)  \to \M_{F}^{X}(A,B,C)$$ 
which is invertible by hypothesis.\\

\item And we take the composition $c_{ABC}= F(\sigma^{1}_0)^{-1} \circ  \varphi_{s,t}$ as illustrated in the  the diagram below:
\[
\xy
(-15,0)*+{\M_{F}^{X}(B,C) \otimes \M_{F}^{X}(A,B)}="X";
(30,0)*+{\M_{F}^{X}(A,B,C)}="Y";
(30,18)*+{\M_{F}^{X}(A,C)}="E";
{\ar@{->}^-{\varphi_{s,t}}"X";"Y"};
{\ar@{->}^-{\cong}_{F(\sigma^{1}_0)}"E";"Y"};
{\ar@{-->}^-{c_{ABC}}"X";"E"};
\endxy
\] 
\end{enumerate}

\end{proof}

\begin{rmk}
The previous equivalence will turn to be an equivalence of categories when we will have the morphisms of $\S$-diagrams.
\end{rmk}

%%%%%%%%Nouvelle Sous-section %%%%%%
\subsection{Morphism of $\S$-Diagrams}
As our $\S$-diagrams are lax  morphisms of semi-bicategories, one can guess that a morphism of $\S$-diagrams will be a \emph{transformations of lax morphisms} in the sense of Bénabou. This  is the same approach as in \cite{SEC1} where the morphism of \emph{path-objects} were defined as transformations of colax morphisms.

But just like in \cite{SEC1}  not every transformation will give  a morphism of semi-enriched categories.  In \cite{SEC1}, a general transformation is called `$\M$-premorphism' and an $\M$-morphism was defined as special $\M$-premorphism. \\

\begin{warn}
In the following, we will only consider the transformations which will give the classical notion of morphism between semi-enriched categories. We decide not to mention `$\M$-premorphisms' between $\S$-diagrams. 
\end{warn}
 
 We recall hereafter the definition of the transformations of morphisms of semi-bicategories we are going to work with. The following definition is slightly different from the standard one, even though in the monoidal case, it is the standard one. 
 
\begin{df}\label{simple-trans}
\indent Let $\B$ and $\M$ be two semi-bicategories and $F=(F,\varphi)$, $G=(G,\psi)$ be two lax morphisms from $\B$ to $\M$ such that $FA=GA$ for every object $A$ of $\B$. \\

 A \textbf{simple transformation} $\sigma: F \to G$ 
\[
\xy
(0,0)*+{\B}="A";
(30,0)*+{\M}="C";
{\ar@/_1.3pc/_{G}"A";"C"};
{\ar@/^1.3pc/^{F}"A";"C"};
{\ar@{=>}_{\sigma}(15,4);(15,-4)};
\endxy.
\]
is given by the following data and axioms. \\

\textbf{Data:}  A  natural transformation for each pair of objects $(A,B)$ of  $\B$:
\[
\xy
(-10,0)*+{\B(A,B)}="A";
(30,0)*+{\M(FA,FB)}="C";
{\ar@/_1.3pc/_{G_{AB}}"A";"C"};
{\ar@/^1.3pc/^{F_{AB}}"A";"C"};
{\ar@{->}_{\sigma}(15,4);(15,-4)};
\endxy.
\]

hence a $2$-morphism of $\M$,  $\sigma_t : Ft \to Gt $, for each $t$ in  $\B(A,B)$, natural in $t$. \\

\textbf{Axioms:} The following commutes :
\[
\xy
(-10,20)*+{Fs \otimes Ft}="A";
(30,20)*+{F(s \otimes t)}="B";
(-10,0)*+{Gs \otimes Gt}="C";
(30,0)*+{G(s \otimes t)}="D";
{\ar@{->}^{\varphi_{s,t}}"A";"B"};
{\ar@{->}_{\sigma_s \otimes \sigma_t}"A";"C"};
{\ar@{->}^{\sigma_{s \otimes t}}"B";"D"};
{\ar@{->}^{\psi_{s,t}}"C";"D"};
\endxy
\]
\end{df}

With this definition we can now give the definition of morphism of $\S$-diagrams.

\begin{df}\label{mor-s-diag}
Let $F$ and $G$ be respectively an $\S_{\C}$-diagram and an $\S_{\D}$-diagram of $\M$.  A morphism of $\S$-diagrams from $F$ to $G$ is a pair 
$ (\Sigma,\sigma)$ where:\\
\begin{enumerate}
\item $\Sigma : \C \to \D$ is a functor such that for every $A \in \tx{Ob}(\C)$ we have $FA= G (\Sigma A)$,\\
\item $\sigma : F \to G \circ \S_{\Sigma}$ is a simple tranformation of lax morphisms:
\[
\xy
(-18,0)*+{(\S_{\C})^{2\tx{-op}}}="X";
(30,0)*+{(\S_{\D})^{2\tx{-op}}}="Y";
(9,-18)*+{\M}="E";
{\ar@{->}^{(\S_{\Sigma})^{2\tx{-op}}}"X";"Y"};
{\ar@{->}_{F}"X";"E"};
{\ar@{->}^{G}"Y";"E"};
{\ar@{=>}^{\sigma}(1,-10);(15,-10)};
\endxy
\] 

\end{enumerate} 
 When all the components `$\sigma_t$' of $\sigma$ are weak equivalences  we will say that $ (\Sigma,\sigma)$ is a level-wise weak equivalences.
\end{df}
\
\begin{nota}\
\begin{enumerate}
\item For a small category $\C$, we will denote by $\tx{Lax}^{\ast}[(\S_{\C})^{2\tx{-op}},\M ]$ the category of $\S_{\C}$-diagrams  with morphism of $\S_{\C}$-diagrams. 
\item We will denote by $\M_{\S}(\C)$ the subcategory of $\tx{Lax}^{\ast}[(\S_{\C})^{2\tx{-op}},\M]$  with morphisms  of the form $(\Id_{\C}, \sigma)$. It follows that the morphisms in  $\M_{\S}(\C)$ are simply determined by the simple transformations `$\sigma$'. 
\item For $\C=\ol{X}$, we will write $\M_{\S}(X)$ to mean $\M_{\S}(\ol{X})$.
\end{enumerate}
\end{nota}
 
\begin{prop}\label{stab-coseg}
Let $\M$ be a $2$-category which is a base of enrichment or a special Quillen algebra, and $F : (\S_{\C})^{2\tx{-op}}  \to \M$, $G: (\S_{\C})^{2\tx{-op}} \to \M$ be two $\S$-diagrams in $\M$. 
For  a level-wise weak equivalence  $ (\Sigma,\sigma): F \to G$ we have:
\begin{enumerate}
\item If $G$ is a Co-Segal $\S$-diagram then so is $F$,
\item If $F$ is a Co-Segal $\S$-diagram and if $\Sigma$  is surjective on objects and full then $G$ is also a Co-Segal $\S$-diagram. 
\end{enumerate}
\end{prop}

\begin{rmk}
In the category $\M_{\S}(\C)$  the condition required in $(2)$ is automatically fulfilled because the morphism in $\M_{\S}(\C)$  are of the form $(\Id_{\C},\sigma)$.
\end{rmk}

\begin{proof}[Sketch of proof]
The key of the proof is to use  the `3-out-of-2' property of weak equivalences in $\M$. This says that whener we have a composable pair of morphisms  $(f,g)$, then if $2$ members of the set  $\{f,g, g\circ f\}$ are weak equivalences  then so is the third.\ \\

For the assertion $(1)$, we need to show that for every $u : s \to s'$ in  $\S_{\C}(A,B)$,  we have $F_{AB}(u): F_{AB}(s') \to F_{AB}(s)$ is  a weak equivalence in $\M$. 
To simplify the notations we will not mention the subscript `AB' on the components of $F$ and $G$. 

By definition of $(\Sigma,\sigma)$  for every $u : s \to s'$ in $\S_{\C}(A,B)$, the following diagram commutes:

\[
\xy
(-10,20)*+{F(s')}="A";
(30,20)*+{G[\S_{\Sigma}(s')]}="B";
(-10,0)*+{F(s)}="C";
(30,0)*+{G[\S_{\Sigma}(s)]}="D";
{\ar@{->}^-{\sigma_{s'}}_-{\sim}"A";"B"};
{\ar@{->}_-{F(u)}"A";"C"};
{\ar@{->}^-{G[\S_{\Sigma}(u)]}_-{\wr}"B";"D"};
{\ar@{->}^-{\sigma_s}_-{\sim}"C";"D"};
\endxy
\]
\\
Since all the three maps are weak equivalences by hypothesis, we deduce by 
$3$-out-of-$2$ that $F(u)$ is also a weak equivalence, which gives $(1)$. 
\ \\

For the assertion $(2)$ we proceed as follows. The assumptions on $\Sigma$ implie that for any  morphism $v: t \to t'$ in $\S_{\D}(U,V)$  there exists a pair of objects  $(A,B)$ of $\C$  and $s',s$ in  $\S_{\C}(A,B)$ together with a maps $u: s \to s'$ such that:\\
 $\Sigma A= U$, $\Sigma B=V$,\\
 $\S_{\Sigma}(s)=t$, $\S_{\Sigma}(s')=t'$, \\
 $ \S_{\Sigma}(u)=v$. \\
 
 And we have the same type of commutative diagram: 
 
 \[
\xy
(-10,20)*+{F(s')}="A";
(30,20)*+{G(t')}="B";
(-10,0)*+{F(s)}="C";
(30,0)*+{G(t)}="D";
{\ar@{->}^{\sigma_{s'}}_-{\sim}"A";"B"};
{\ar@{->}_{F(u)}^-{\wr}"A";"C"};
{\ar@{->}^{G(v)}"B";"D"};
{\ar@{->}^{\sigma_s}_-{\sim}"C";"D"};
\endxy
\]
\\
Just like in the previous case we have by  $3$ for $2$ that  $G(v)$ is also a weak equivalence.

\end{proof}

 %%%%%%mettre la définition %%%%

%%%%%%%%%%% Nouvelle section : stability par homotopy %%%%%%%%%%%
\section{Properties of $\M_{\S}(\Ca)$}\label{properties-msc}
This section is devoted to the study  of the properties that $\M_{\S}(X)$ inherits from $\M$ e.g (co)-completeness, accessibility, etc. For simplicity we consdier here only the cases  $\C=\ol{X}$, for some nonempty set $X$. The methods are the same for an arbitrary $\C$.
\paragraph*{\textbf{ Environment}:}
We assume that $\M=(\ul{M},\otimes,I)$ is a \textbf{symmetric closed} monoidal category (see \cite{Ke} for a definition).  One of the consequences of this hypothesis is the fact that for every object $A$ of $\M$ the two functors 
$$-\otimes A :  \ul{M} \to \ul{M}  \ \ \ \ \ \ \ \  \tx{and (hence)} \ \ \ \ \  A \otimes - :  \ul{M} \to \ul{M}$$
preserve the colimits. Note that these conditions turn $\M$ into a special Quillen algebra. 

\begin{warn}\ \
\begin{enumerate}
 \item  When we say that `$\M$ is (co)-complete' we mean of course that the underlying category $\ul{M}$ is (co)-complete.
 And by colimits and limits in $\M$ we mean colimits and limits in $\ul{M}$.
\item We will say as well that $\M$ is locally presentable, accessible when $\ul{M}$ is so.
\item Some results in this section are presented without proof since they are easy and are sometime considered as `folklore' in category theory.
\end{enumerate}
\end{warn}

\subsection{$\M_{\S}(X)$ is locally presentable if $\M$ is so}
Our goal here is to prove the following
\begin{thm}\label{MX-local-pres}
Let $\M$ be a symmetric closed monoidal category which is locally presentable. Then for every nonempty set $X$ the category $\M_{\S}(X)$ is locally presentable. 
\end{thm}

To prove this, we proceed in the same way as in the paper of Kelly and Lack \cite{Kelly-Lack-loc-pres-vcat} where they established that $\M$-$\Cat$  is locally presentable if $\M$ is so.\ \\

The idea is to use the fact that given a locally presentable category $\K$  and a monad $\T$ on $\K$, then if $\T$ preserve the directed  colimits then the category of algebra of $\T$ (called  the Eilenberg-Moore category of $\T$) is also locally presentable (see \cite[Remark 2.78]{Adamek-Rosicky-loc-pres}).\ \\

In our case we will have:
\begin{itemize}
\item $\K$ is the category $ \prod_{(A,B) \in X^{2}} \Hom[\S_{\ol{X}}(A,B)^{op}, \M].$ We will write $\K_X$ to emphasize that it depends of the set $X$\\
\item There is a forgetful functor $\Ub : \M_{\S}(X) \to \K_X$ which is faithful and injective on object, therefore we can consider $\M_{\S}(X)$ is a subcategory of $\K_X$.\\ %%% à mon avis c'est une "reflective" subcategory"
\item There is left adjoint $\Gamma$ of $\Ub$ inducing the monad $\T= U\Gamma$.\\
\item The category of algebra of $\T$ is precisely  $\M_{\S}(X)$.
\end{itemize}  

\begin{rmk}
The theory of locally presentable categories tell us that any (small) diagram category of a locally presentable category, is  locally presentable (see \cite[Corollary 1.54]{Adamek-Rosicky-loc-pres}). It follows that  each $\Hom[\S_{\ol{X}}(A,B)^{op}, \M]$ is locally presentable  if $\M$ is so. Finally $\K_X$ is locally presentable since it's a (small) product of locally presentable categories. \ \\
\end{rmk}
But before proving the Theorem \ref{MX-local-pres} we must first show that $\M_{\S}(X)$ is co-complete to be able to consider (filtered) colimits. This is given by the following
\begin{thm}\label{MX-cocomplete}
Given a co-complete symmetric monoidal category $\M$, for any nonempty set $X$ the category  $\M_{\S}(X)$ is co-complete.
\end{thm}

\begin{proof}[Proof of Theorem \ref{MX-cocomplete}]
See Appendix \ref{proof-MX-cocomplete}
\end{proof}
%%%%%%%%%%%%%%%%%
\begin{prop}\label{monad-finitary}
The monad $\T= \Ub \Gamma: \K_X \to \K_X$ is finitary, that is, it preserves filtered colimits.  
\end{prop}

\begin{proof}[Proof of the proposition]
Filtered and directed colimits are essentially the same and it's known that a functor preserves filtered colimits if and only if it preserves directed colimits (see  \cite[Chap. 1, Thm 1.5 and Cor ]{Adamek-Rosicky-loc-pres}).
This allows us to reduce the proof to directed colimits. \ \\

%For each pair $(A,B) \in X^2$ consider the projection functor $\Pcal_{AB} : \K_X \to  \Hom[\S_{\ol{X}}(A,B)^{op}, \M]$.\ \\

Recall that  colimits in $\K_X=\prod_{(A,B) \in X^{2}} \Hom[\S_{\ol{X}}(A,B)^{op}, \M]$  are computed factor-wise.\ \\  %so it will suffices to prove that each functor $\T_{AB}= \Pcal_{AB} \circ \T: \K_X \to \Hom[\S_{\ol{X}}(A,B)^{op}, \M]$ 
%preserve directed colimits.\ \\

For $\Fa=(\Fa_{AB})$ recall  that $\T \Fa= ([\Gamma \Fa]_{AB})$ where $\Gamma$ is the left adjoint of $\Ub$ (see Appendix \ref{adjoint-ub}).\footnote{Note that actually $\T \Fa= ([\Ub\Gamma \Fa]_{AB})$ but since $\Ub$ consists to forget the laxity maps it's not necessary to mention it.} To simplify the notation we will not mention the subscript `AB'. For each pair $(A,B)$  the $AB$-component of  $\Gamma \Fa$ is $\Gamma \Fa: \S_{\ol{X}}(A,B)^{op} \to  \M$ the functor given by:
\begin{itemize}
\item for $t \in \S_{\ol{X}}(A,B)$ we have 
$$\Gamma \Fa(t)=  \coprod_{(t_0, \cdots ,t_l) \in \tx{Dec}(t)} \Fa(t_0) \otimes \cdots \otimes \Fa(t_l).$$
\item for $u: t \to t'$ , we have 
$$\Gamma \Fa(u)=  \coprod_{(u_0, \cdots ,u_l) \in \tx{Dec}(u)} \Fa(u_0) \otimes \cdots \otimes \Fa(u_l)$$

where $u_i: t_i \to t'_i$. 
\end{itemize}

 Let  $\lambda < \kappa$ be  an ordinal and ${(\Fa^{k})}_{k \in \lambda}$ be a $\lambda$-directed diagram in $\K_X$ whose colimit is denoted by $\Fa^{\infty}$.\ \\ 
 
For any $l$  the diagonal functor $d: \lambda \to \prod_{i=0...l} \lambda$ is cofinal therefore the following colimits are the same%

\begin{equation*}
\begin{cases}
\colim_{(k_0, ...,  k_l) \in \lambda^{l+1}}  \{ \Fa^{k_0}(t_0) \otimes \cdots \otimes \Fa^{k_l}(t_l) \} \\
\\
\colim_{k \in \lambda}  \{  \Fa^{k}(t_0) \otimes \cdots \otimes \Fa^{k}(t_l) \} \\
\end{cases}
\end{equation*}

The first colimit is easy to compute as $\M$ is symmetric closed  and we have 
$$\colim_{(k_0, ...,  k_l) \in \lambda^{l+1}}  \{  \Fa^{k_0}(t_0) \otimes \cdots \otimes \Fa^{k_n}_{t_l} \} = \Fa^{\infty}(t_0) \otimes \cdots \otimes \Fa^{\infty}_{t_l} .$$

Consequently $\colim_{k \in \lambda}  \{  \Fa^{k}(t_0) \otimes \cdots \otimes \Fa^{k}_{t_l} \}= \Fa^{\infty}(t_0) \otimes \cdots \otimes \Fa^{\infty}(t_l). $\ \\

From this we deduce successively that: 
\begin{equation*}
\begin{split}
\colim_{k \in \lambda} \Gamma \Fa^{k}(t)
&=\colim_{k \in \lambda} \{ \coprod_{(t_0, \cdots ,t_l) \in \tx{Dec}(t)}  \Fa^{k}(t_0) \otimes \cdots \otimes \Fa^{k}(t_l) \}\\
 &= \coprod_{(t_0, \cdots ,t_l) \in \tx{Dec}(t)} \colim_{k \in \lambda} \{ \Fa^{k}(t_0) \otimes \cdots \otimes \Fa^{k}(t_l) \}\\
 &= \coprod_{(t_0, \cdots ,t_l) \in \tx{Dec}(t)} \Fa^{\infty}(t_0) \otimes \cdots \otimes \Fa^{\infty}(t_l) \\
&=\Gamma \Fa^{\infty}(t) \\
\end{split}
\end{equation*}

which shows that  $\T = \Ub \Gamma$ preserves directed colimits as desired.  
%%% on procède comme Kelly-Lack %%%
\end{proof}

Now we can give the proof of Theorem \ref{MX-local-pres}  as follows. 

\begin{proof}[Proof of Theorem \ref{MX-local-pres}]
Thanks to Theorem \ref{MX-monadic-KX} we know that $\Ub: \M_{\S}(X) \to \K_X$ is monadic therefore $\M_{\S}(X)$ is equivalent to the category $\T$-alg of $\T$-algebra. Now since $\T$ is a finitary monad on the locally presentable category $\K_X$, we  know from a classical result that $\T$-alg (hence $\M_{\S}(X)$) is also locally presentable  (see \cite[Remark 2.78]{Adamek-Rosicky-loc-pres}).
\end{proof}

%%%%%%%%%%%%%%%%%%%% Model structure %%%%%%%%%%%%%%%%%%%%%%%%%%%%%

%%%%%%%%%%%%%%%%%%%%%%%%%%%%%%%%%%%%%%%%%%%%%%%%%%%%%%%%%%%%%%%%%%%%%
%%%%%%%%%%%%%%%%%%%%%%%%%%%%%%%%%%%%%%%%%%%%%%%%%%%%%%%%%%%%%%%%%%%%%
%%%%%%%%%%%%%%%%%%%%%%%%%%%%%%%%%%%%%%%%%%%%%%%%%%%%%%%%%%%%%%%%%%%%%
%%%%%%%%%%%%%%%%%%%%%%%%%%%%%%%%%%%%%%%%%%%%%%%%%%%%%%%%%%%%%%%%%%%%%
%%%%%%%%%%%%%%%%%%%%MODEL STRUCTURE %%%%%%%%%%%%%%%%%%%%%%%% %%%%%%%%%%%%%%%%%%%%%%%%%%%%%%%%%%%%%%%%%%%%%%%%% 
%%%%%%%%%%%%%%%%%%%%%%%%%%%%%%%%%%%%%%%%%%%%%%%%%%%%%%%%%%%%%%%%%%%%%
%%%%%%%%%%%%%%%%%%%%%%%%%%%%%%%%%%%%%%%%%%%%%%%%%%%%%%%%%%%%%%%%%%%%%
%%%%%%%%%%%%%%%%%%%%%%%%%%%%%%%%%%%%%%%%%%%%%%%%%%%%%%%%%%%%%%%%%%%%%
%%%%%%%%%%%%%%%%%%%%%%%%%%%%%%%%%%%%%%%%%%%%%%%%%%%%%%%%%%%%%%%%%%%%%
%%%%%%%%%%%%%%%%%%%%%%%%%%%%%%%%%%%%%%%%%%%%%%%%%%%%%%%%%%%%%%%%%%%%%
%%%%%%%%%%%%%%%%%%%%%%%%%%%%%%%%%%%%%%%%%%%%%%%%%%%%%%%%%%%%%%%%%%%%%
\section{Locally Reedy $2$-categories} \label{lr-category}
In the following we give an \emph{ad hoc} definition of a locally Reedy $2$-category.  One can generalize this notion to $\O$-algebra but we will not go through that here. The horizontal composition in $2$-categories will be denoted by $\otimes$.   
\begin{comment}
\ \\
Recall that if $P$ is a property of categories e.g complete, locally presentable, etc; then we say that a $2$-category $\Ba$ is locally $P$ if all the hom-categories $\Ba(x,y)$ have the property $P$. Similarly for morphisms of $2$-categories 
\end{comment}
\begin{df}\label{Reedy-2-cat}
A small $2$-category $\Ca$ is called a \textbf{locally Reedy $2$-category} if the following holds.
\begin{enumerate}
\item For each pair $(A,B)$ of objects, the category $\Ca(A,B)$ is a classical Reedy $1$-category;
\item The composition $\otimes : \Ca(A,B) \times \Ca(B,C) \to \Ca(A,C)$ is a functor of Reedy categories i.e takes direct (resp. inverse) morphisms to direct (resp. inverse) morphism.
\end{enumerate}
\end{df}
\begin{ex}\label{ex-lr-category}
\begin{enumerate}
\item The examples  that motivated the above definition are of course the $2$-categories, $\P_{\D}$ and $\S_{\D}$, and their respective $1$-opposite and $2$-opposite $2$-categories: $(\P_{\D})^{\tx{op}}$, $(\S_{\D})^{2\tx{-op}}$, etc.  In particular $(\Delta, +, 0)$ is a monoidal Reedy category (a locally Reedy $2$-category with one object).
\item Any classical Reedy $1$-category $\D$ can be considered as a $2$-category with two objects $0$ and $1$  with  $\Hom(0,1)= \D$, $\Hom(1,0)= \varnothing$; $\Hom (0, 0)= \Hom(1,1)= \uc$ (the unit category); the composition is the obvious one (left and right isomorphism of the cartesian product). It follow that any classical category is also a locally Reedy $2$-category. 
\item As any set is a (discrete) Reedy category, it follows that any $1$-category viewed as a $2$-category with only identity $2$-morphisms is a locally Reedy $2$-category. In that case the linear extension are constant functor.
\end{enumerate}
\end{ex}

\begin{warn}
It's important to notice that we've chosen to say `locally Reedy $2$-category' rather than `Reedy $2$-category'. The reason is that  the later terminology may refer to the notion of `Reedy $\Ma$-category' (=enriched Reedy category)  introduced by Angeltveit \cite{Angel_Reedy} when  $\Ma=(\Cat,\times, \uc)$.  %Our terminology `$2$-Reedy' means that the Reedy structure is on dimension $2$ i.e on $2$-morphisms.  
\end{warn}
In our definition we've implicitly used the fact that if $\Aa$ and $\Ba$ are two classical Reedy categories, then there is a natural Reedy structure on the cartesian product $\Aa \times \Ba$ (see \cite[Prop. 15.1.6]{Hirsch-model-loc}). This way we form a monoidal category  of Reedy categories and morphisms of Reedy categories with the cartesian product;  the unit is the same i.e $\uc$. We will denote by $\Reedycatx$ this monoidal category.\ \\ 

Our  definition is equivalent to say that

\begin{df}
A locally Reedy $2$-category is a category enriched over $\Reedycatx$.
\end{df}

\begin{rmk}
\begin{enumerate}
\item This definition can be generalized to locally Reedy $n$-categories, but we won't consider it, since the spirit of this paper is to use lower dimensional objects to define higher dimensional ones.
\item One can replace $\Reedycatx$ by a suitable monoidal category $\M$-$\Reedycat^{\otimes}$ of Reedy $\M$-categories in the sense of  Angeltveit \cite{Angel_Reedy}; but we don't know how relevant this would be. 
\item We can also enrich over the category of generalized Reedy categories in the sense of Berger-Moerdijk \cite{Ber-Moer_Reedy_cat} and in the sense of Cisinski \cite[Chap. 8]{Cisinski-prefaisceau-model}. 
\item We can also push the definition far  by considering not only Reedy $\O$-algebras, but defining first a Reedy multisorted operad as being multicategory enriched over $\Reedycatx$,$\M$-$\Reedycat^{\otimes}$, etc. 
\end{enumerate}
\end{rmk}

\subsection{(Co)lax-latching and (Co)lax-matching objects}
Let $\Ca$ be a locally Reedy $2$-category (henceforth $\lr$-category). Given a lax morphism or colax morphism $\Fa: \Ca \to \M$ we would like to define the corresponding latching and matching objects of $\Fa$ at a $1$-morphism $z$ of $\Ca$. We will concentrate our discussion to the lax-latching object; leaving the other cases to the reader. Our definitions are restricted to the case where $\Ca$ is equipped with a \textbf{global linear extension} which respects the composition. This means that we have an ordinal $\lambda$  such that the linear extension $\degb : \Ca(A,B) \to \lambda$ satisfies $\degb(g \otimes f)= \degb (g)  + \degb(f)$.

\begin{note}
We don't know many examples other than the $2$-categories that motivated this consideration, but we choose to have a common language for both $\P_{\ol{X}}$, $\sx$ and the others $2$-categories we can construct out of them. However it's clear to see that for a classical Reedy $1$-category $\D$, if we view $\D$ as an $\lr$-category with two objects (see Example \ref{ex-lr-category}) and if we declare both $\degb(\Id_0)=\degb(\Id_1)=0$ then $\D$ has this property. \\

We will consider lax morphism $\Fa : \Ca \to \M$ which are unitary in the sense of Bénabou i.e such that $\Fa(\Id) = \Id$ and the laxity maps $ \Fa \Id \otimes \Fa f  \to \Fa(\Id \otimes f)$ are the natural left and right isomorphisms.
\end{note}

Let $\lambda$ be an infinite ordinal containing $\omega$. We can make $\lambda$ into a monoidal category with the addition; and we can consider it a usual as a  $2$-category with one object having as hom-category  $\lambda$. 
\begin{df}\label{lr-simple}
A locally Reedy $2$-category $\Ca$ is \textbf{simple} if there exist an infinite ordinal $\lambda$ such that the linear extension form a strict $2$-functor $\degb : \Ca \to \lambda$.  Here on object $\degb: \Ob(\Ca) \to \{ \ast \}$ 
\end{df}

From this definition we have the following consequences:
\begin{itemize}[label=$-$]
\item First if $\Ca$ is a simple $\lr$-category then the composition reflects the identies; thus $\Ca$ is an $\iro$-algebra for the operad of $2$-categories.
\item For any object $A \in \Ca$, we have $\degb(\Id_A)=\0$ since $\degb(\Id_A)= \degb(\Id_A \otimes \Id_A)= \degb(\Id_A) + \degb(\Id_A)$.
\item Another important consequence is that a $1$-morphism $z$, cannot appear in the set 
$$\otimes^{-1}(z) =\coprod_{l >1} \{(s_1, ..., s_l); \otimes(s_1, ..., s_l)=z  \}$$

if $\degb(s_i)> \0$ for all $i$. 
\end{itemize}

\paragraph{Using the Grothendieck construction} For each pair $(A,B)$  we have a composition diagram which is organized into a functor $\cb: \Delta_{\tx{epi}} \to \Cat$ and represented as:
\[
\xy
(-75,0)+(-10,0)*++{\Ca(A,B)}="X";
(-40,0)+(-8,0)*++{\coprod \Ca(A,A_{1})\times \Ca(A_{1},B)}="Y";
(16,0)+(-5,0)*++{\coprod \Ca(A,A_{1}) \times \Ca(A_{1},A_{2})\times \Ca(A_{2},B)}="Z";
(44,0)*++{\cdots};
{\ar@{->}"Y";"X"};
{\ar@<-0.5ex>"Z";"Y"};
{\ar@<0.5ex>"Z";"Y"};
\endxy
\]
\\
More precisely one defines:
\begin{itemize}[label=$-$]
\item $\cb(\1)= \Ca(A,B)$,
\item  $\cb(\n)= \coprod_{(A,...,B)} \Ca(A,A_{1}) \times \Ca(A,A_1)\times \cdots\times \Ca(A_{n-1},B)$.
\end{itemize}
\ \\
Note that the morphisms in  $\Depi$ are generated by the maps  $\sigma^{n}_i: \n + \1 \to \n$ which are characterized by  $\sigma^{n}_i(i)=\sigma^{n}_i(i+1)$ for $i \in \n=\{0, \cdots , n-1 \}$ (see \cite[p.177]{Mac}). Then the functor $\cb(\sigma^{n}_i): \cb(\n+1)\to  \cb(\n)$ is the functor which consists to compose at the vertex $A_{i+1}$. \ \\

Let $\int\cb$ be the category we obtained by the Grothendieck construction:
\begin{itemize}[label=$-$]
\item the objects are pairs $(\n, a)$ with $a \in \Ob(\cb(\n))$. Such object $(\n, a)$ can be identified with an $n$-tuple of $1$-morphisms $(s_1,..., s_n)$ with $s_i \in \Ca(A_{i-1}, A_i)$.
\item  a morphism $\gamma: (\n,a) \to (\m,b)$ is a pair $\gamma=(f,u)$ where $f: \n \to \m$ is a morphism of $\Depi$ and 
$u: \cb(f)a \to b$ is a morphism in $\cb(\m)$. Here $\cb(f): \cb(\n) \to \cb(\m)$ is a functor (image of $f$ by $\cb$).
\item the composite of $\gamma=(f,u)$ and $\gamma'=(g,v)$ is $\gamma' \circ \gamma:= (g \circ f, v \circ \cb(g)u)$.
\end{itemize}
One can easily check that these data define a category.  Note that for each $\n \in \Depi$ the subcategory of objects over $\n$ is isomorphic to $\cb(\n)$.
\begin{claim}
For a simple $\lr$-category $\Ca$ and for each pair $(A,B)$ there is a natural Reedy structure on $\int \cb$. 
\end{claim}
In fact one has a linear extension by setting $\degb(\n, a)= \degb(a)= \degb(s_1)+ \cdots +\degb(s_n)$ for $a=(s_1,..., s_n)$. A morphism $\gamma=(f,u): (\n,a) \to (\m,b)$ is said to be a direct (resp. inverse) if $u$ is a direct (resp. inverse) morphism in $\cb(\m)$. The factorization axiom follows from the fact that $\cb(\m)$ is a Reedy category.

\begin{rmk}\ \
\begin{enumerate}
\item There is a general statement for the category $\int F$ associated to any functor \\ $F: \D \to \Reedycatx$. 
\item For any morphism of $ h: x \to y$ of $\D$ and any morphism $u: a \to b$ of $F(x)$,  the following commutes in $\int F$:
\[
\xy
(0,20)*+{(x,a)}="A";
(30,20)*+{(y, F(h)a)}="B";
(0,0)*+{(x,b)}="C";
(30,0)*+{(y,F(h)b)}="D";
{\ar@{->}^-{(h, \Id_{F(h)a})}"A";"B"};
{\ar@{->}_-{(\Id_x, u)}"A";"C"};
{\ar@{->}^-{(\Id_y,F(h)u)}"B";"D"};
{\ar@{->}_-{(h, \Id_{F(h)b})}"C";"D"};
\endxy
\]  

%This diagram can be read as: `$(h,u)=(\Id, u) \circ (h, \Id)= (h, \Id) \circ (\Id, u)$'.  
\end{enumerate}
\end{rmk}
 
 Let $\int \overrightarrow{\cb} \subset \int \cb$ be the direct category. We will denote by $ \int \overrightarrow{\cb} \downarrow (\n,a)$ the slice category at $(\n,a)$. 
\begin{df}\label{gen-latching-cat}
Let $z \in  \Ca(A,B)$ be a $1$-morphism. Define the \textbf{generalized latching category at $z$}, $\partial^{\bullet}_{\Ca/z}$, to be the subcategory of $ \int \overrightarrow{\cb} \downarrow (\1,z)$  described as follows. 

\begin{itemize}[label=$-$]
\item the objects are \ul{direct} morphisms $\gamma: (\n,a) \to (\1,z)$ such that $z$ doesn't appear in $a$; or equivalently we have $a=(s_1,..., s_n)$  with $\degb(s_i) < \degb(z)$ for all $i$.
\item the morphisms are the usual morphisms of the comma category $ \int \overrightarrow{\cb} \downarrow (\1,z)$: 
\[
\xy
(0,20)*+{(\n,a)}="A";
(30,20)*+{(\1,z)}="B";
(0,0)*+{(\m,b)}="C";
{\ar@{->}^{\gamma}"A";"B"};
{\ar@{->}_{\delta}"A";"C"};
{\ar@{->}_{\gamma'}^{}"C";"B"};
\endxy
\]  
\item the composition is the one in $ \int \overrightarrow{\cb} \downarrow (\1,z)$.
\end{itemize}
\end{df} 

\paragraph{Lax functor and diagram on $\int \cb$} Given a lax functor $\Fa: \Ca \to \M$ we can define a natural functor denoted again $\Fa$ on $\int \cb$ as follows.

\begin{enumerate}
\item $\Fa(\n,a)=\otimes(\Fa s_1, ... \Fa s_n)$ for $a=(s_1,..., s_n)$;
\item To define $\Fa$ on morphisms it suffices to define the image of morphisms $\gamma=(\sigma^{n}_i,u)$ since they generated all the other morphisms. Such morphism $\gamma: (\n+\1, a) \to (\n, b)$ for $a=(s_1,..., s_{n+1})$ and $b=(t_1,...,t_n)$ corresponds to a $n$ direct morphisms $\{\alpha_l: s_l \to t_l \}_{l \neq i,l \neq i+1} \cup \{ \alpha_i: s_i \otimes s_{i+1} \to t_i \} $. 
With these notations one defines $\Fa(\gamma): \Fa(\n+1,a) \to \Fa(\n,b)$ to be the composite:
$$ \otimes(\Fa s_1, ..., \Fa s_{n+1}) \xrightarrow{\Id \otimes ... \otimes \varphi(s_i,s_{i+1}) \otimes ... \Id} \otimes(\Fa s_1, ..., \Fa(s_i \otimes s_{i+1}),... \Fa s_n) \xrightarrow{\otimes \Fa( \alpha_l)}  \otimes(\Fa t_1, ..., \Fa t_{n})$$
where $\varphi(s_i,s_{i+1}): \Fa s_i \otimes \Fa s_{i+1} \to \Fa(s_i \otimes s_{i+1})$ is the laxity map.  
\end{enumerate}
These data won't define a functor until we show that  $\Fa(\gamma \circ \gamma')= \Fa(\gamma) \circ \Fa(\gamma')$. But this is given by the coherence axioms for the lax functor $\Fa : \Ca \to \M$ as we are going to explain. First of all we will denote by $\varphi_{\sigma_i}$ the above map $\Id \otimes \cdots \otimes \varphi \otimes \cdots \otimes \Id$ which uses the laxity map of the  $i$th and $(i+1)$th terms. For any map $f: \n \to \m$ of $\Depi$ and any object $a \in \cb(\n)$ there is a canonical map $f_a: (\n,a) \to (\m,\cb(f)a)$ given by $f_a= (f, \Id_{\cb(f)a})$. As pointed out by Mac Lane \cite[p.177]{Mac}, each morphism $f: \n \to \m$ of $\Depi$ has a \emph{unique presentation} $f= \sigma_{j_1} \circ \cdots \circ \sigma_{j_{n-m}}$ where the string of subscripts $j$ satisfy:
$$0 \leq j_1 < \cdots < j_{m-n}< n-1.$$ 

With the previous notations we can define $\varphi_f= \Fa(f_a):= \varphi_{\sigma_{j_1}} \circ \cdots \circ \varphi_{\sigma_{j_{n-m}}}$ to be the \textbf{laxity map governed by $f$}. Here we omit $a$ in $\varphi_f$ for simplicity.
\begin{prop}\label{prop-lax-commute}
Given $f: \n \to \m$ and $g: \m \to \m'$ then $\Fa({g\circ f}_a)= \Fa(g_{\cb(f)a}) \circ \Fa(f_a)$ i.e $\varphi_{g\circ f}= \varphi_g \circ \varphi_f$. 
\end{prop}

The proposition will follow from the
\begin{lem}
The maps $\varphi_{\sigma_i}$ respect the simplicial identities $\sigma_j \circ \sigma_i= \sigma_i \circ \sigma_{j+1}$ ($i \leq j$). This means that we have $\varphi_{\sigma_j} \circ \varphi_{\sigma_i}= \varphi_{\sigma_i} \circ \varphi_{\sigma_{j+1}}.$ 
\end{lem}

\begin{proof}[Sketch of proof]
If $i<j$ then the assertion follows from the bifunctoriality of $\otimes$. In fact  given two morphisms $u,v$ of $\M$ then 
$u \otimes v= (\Id \otimes v) \circ (u \otimes \Id) = ( u \otimes \Id) \circ (\Id \otimes v)$.  So the only point which needs to be clarified is when $i=j$. In that case the equality is given by the coherence  condition which says that the following commutes:
\[
\xy
(0,20)*+{\Fa s_i \otimes \Fa s_{i+1} \otimes \Fa s_{i+2}}="A";
(60,20)*+{\Fa s_i \otimes \Fa (s_{i+1} \otimes s_{i+2}) }="B";
(0,0)*+{\Fa (s_{i} \otimes s_{i+1}) \otimes \Fa s_{i+2} }="C";
(60,0)*+{\Fa (s_{i} \otimes s_{i+1} \otimes s_{i+2})}="D";
{\ar@{->}^-{\sigma_{i+1}}"A";"B"};
{\ar@{->}_-{\sigma_i}"A";"C"};
{\ar@{->}^-{\sigma_i}"B";"D"};
{\ar@{->}_-{\sigma_i}"C";"D"};
\endxy
\]   
\end{proof}

To prove the proposition one needs to see how we build a presentation of $g\circ f$ out of the presentation of $f= \sigma_{j_1} \circ \cdots \circ \sigma_{j_{n-m}}$ and $g=\sigma_{l_1} \circ \cdots \circ \sigma_{l_{m-m'}}$  where $0 \leq l_1 < \cdots < l_{m-m'}< m-1$ and $0 \leq j_1 < \cdots < j_{m-n}< n-1$.  By induction one reduces to the case where $g= \sigma_l$. Then $g\circ f= \sigma_l \circ \sigma_{j_1}  \cdots \sigma_{j_{n-m}}$,  and we proceed as follows.
\begin{enumerate}
\item if $l \geq j_1$ then we use the simplicial identities to replace $\sigma_l \circ \sigma_{j_1}$ by $\sigma_{j_1} \circ \sigma_{l+1}$; if not we're done.
\item then $g \circ f= \sigma_{j_1} \circ ( \sigma_{l+1} \sigma_{j_2} \cdots \sigma_{j_{n-m}})$ and we apply the first step with $g'=\sigma_{l+1}$ and $f'=\sigma_{j_2} \cdots \sigma_{j_{n-m}}$. 
\item after a finite number of steps one has the presentation of $g \circ f=\sigma_{k_1} \circ \cdots \circ \sigma_{k_{n-m +1}}$ with $0 \leq k_1 < \cdots < k_{m-n+1}< n-1$. 
\end{enumerate}

\begin{proof}[Proof of Proposition \ref{prop-lax-commute}]
By definition $\Fa({g\circ f})= \varphi_{\sigma_{k_1}} \circ \cdots \circ \varphi_{\sigma_{k_{n-m'}}}$ and we have 
$$\Fa(g_{\cb(f)a}) \circ \Fa(f_a)=  \varphi_{\sigma_{l_1}} \cdots  \varphi_{\sigma_{l_{m-m'}}} \circ  \varphi_{\sigma_{j_1}} \cdots  \varphi_{\sigma_{j_{n-m}}}.$$
From the last expression we apply the lemma to each step we've followed to get the presentation of $g \circ f$;  after a finite number of steps we end up with the expression $\varphi_{\sigma_{k_1}} \circ \cdots \circ \varphi_{\sigma_{k_{n-m'}}}$ which completes the proof.
\end{proof}
%%%%%%%%%%%%%%%% Comment begins %%%%%%%%%%%%%%%%%%

As we already said, each morphism $f : \n \to \m$ induces a functor $\cb(f): \cb(\n) \to \cb(\m)$. A morphism $u$ in $\cb(\n)$ is a $n$-tuple of morphism $u=(u_1,...,u_n)$. When there is no confusition we will write $\otimes (\Fa(u))$ to  mean $\otimes(\Fa u_1,..., \Fa u_n)$; the image of $u$ by $f$ will be denoted by $fu$ instead of $\cb(f)u$ and we may consider $\otimes(\Fa(fu))$.\ \\

According to these notations we can define shortly $\Fa$ on the morphisms of $\int \cb$ by:
$$\Fa(\gamma)= \otimes (\Fa u) \circ \varphi_f \ \ \ \ \ \ \ \ \  \tx{with} \ \ \ \gamma=(f,u)$$

\begin{lem}
For every morphism $f: \n \to \m$ of $\Depi$ and every morphism $u=(u_1, ..., u_n)$ of $\cb(\n)$  we have an equality:
$$\otimes[\Fa(fu)] \circ \varphi_f= \varphi_f \circ \otimes(\Fa u).$$
This means that the following commutes:
\[
\xy
(0,20)*+{\otimes(\Fa s_i)}="A";
(30,20)*+{\otimes[\Fa f(s_i)]}="B";
(0,0)*+{\otimes(\Fa t_i)}="C";
(30,0)*+{\otimes[\Fa f(t_i)]}="D";
{\ar@{->}^-{\varphi_f}"A";"B"};
{\ar@{->}_-{\otimes(\Fa u)}"A";"C"};
{\ar@{->}^-{\otimes[\Fa(fu)]}"B";"D"};
{\ar@{->}_-{\varphi_f}"C";"D"};
\endxy
\] 
\end{lem}

\begin{proof}[Sketch of proof]
By standard arguments (repeating the process), one reduces the assertion to the case where $f= \sigma_i$. In this case the result follows from the functoriality of the coherence for the laxity maps, that is that all the following diagram commutes:
\[
\xy
(0,20)*+{\Fa s_i \otimes \Fa s_{i+1}}="A";
(30,20)*+{\Fa (s_{i} \otimes s_{i+1}) }="B";
(0,0)*+{\Fa t_{i} \otimes \Fa t_{i+1} }="C";
(30,0)*+{\Fa (t_{i} \otimes t_{i+1} )}="D";
{\ar@{->}^-{\varphi_{\sigma_i}}"A";"B"};
{\ar@{->}_-{\Fa u_i \otimes \Fa u_{i+1}}"A";"C"};
{\ar@{->}^-{\Fa(u_i \otimes u_{i+1})}"B";"D"};
{\ar@{->}_-{\varphi_{\sigma_i}}"C";"D"};
\endxy
\]  
\end{proof}

With the previous lemma at hand we conclude that

\begin{prop}\label{f-int}
Given two composable morphisms $\gamma$ and $\gamma'$ then we have $\Fa(\gamma' \circ \gamma)= \Fa(\gamma') \circ \Fa(\gamma)$ i.e $\Fa$ is a functor on $\int \cb$.
\end{prop} 

\begin{proof}
Let $\gamma=(f,u)$ and $\gamma'=(g,v)$. Then by definition $\gamma'\circ  \gamma= [(g\circ f), v \circ g(u)]$. \ \\

In one hand we have by definition  $\Fa(\gamma' \circ \gamma)= \otimes [\Fa (v \circ g(u))] \circ \varphi_{g\circ f}$. On the other hand we have $\Fa(\gamma)= \otimes (\Fa u) \circ \varphi_f$ and $\Fa(\gamma')= \otimes (\Fa v) \circ \varphi_g $. Using the functoriality of $\Fa$ and $\otimes$ together with the fact that $\varphi_{g\circ f}= \varphi_g \circ \varphi_f$ we establish that

\begin{equation*}
\begin{split}
\Fa(\gamma' \circ \gamma)
&= \otimes [\underbrace{\Fa (v \circ g(u))}_{\Fa v \circ \Fa g(u)}] \circ \underbrace{\varphi_{g\circ f}}_{=\varphi_g \circ \varphi_f} \\
\\
 &=\otimes(\Fa v) \circ  \underbrace{\otimes[\Fa g(u)] \circ \varphi_g}_{= \varphi_g \circ \otimes(\Fa u)} \circ \varphi_f  \\
 \\
 &=\underbrace{\otimes(\Fa v) \circ  \varphi_g}_{=\Fa(\gamma')} \circ \underbrace{\otimes(\Fa u) \circ \varphi_f}_{=\Fa(\gamma)}  \\
 \\
&=\Fa(\gamma') \circ \Fa(\gamma)\\
\end{split}
\end{equation*}
\end{proof}

\begin{obs}
As we pointed out earlier, for each $\n$ the category of objects in $\int \cb$ over $\n$ is isomorphic to $\cb(\n)$ and we have an embedding $\cb(\n) \xhookrightarrow{\iota_n} \int \cb$. By the universal property of the coproduct we get a functor 
$$\iota: \coprod \cb(\n) \to \int \cb.$$
Given a family of functors $ \{ \Fa_{AB}: \Ca(A,B) \to \M \} $ we can define a functor for each $\n$, $\Fa_n: \cb(\n) \to \M$ by the above formula $\Fa(s_1,..., s_n):= \otimes( \Fa s_1, ..., \Fa s_n)$. These functors define in turn a functor
 $$\coprod \Fa_n: \coprod \cb(\n) \to \M.$$ 
Then the left Kan extension of $\coprod \Fa_n$ along $\iota: \coprod \cb(\n) \to \int \cb$
\textbf{creates laxity maps}. This is the same idea we use to define the `free lax-morphism' generated by the family $ \{ \Fa_{AB}: \Ca(A,B) \to \M \} $ (see Appendix \ref{ub-has-left-adjoint}).
\end{obs}

Denote by $\Ua$ the canonical forgetful functor $\Ua: \partial^{\bullet}_{\Ca/z} \to \int \cb $. 

\begin{df}
Let $\Fa: \Ca \to \M$ be a lax functor. 
\begin{enumerate}
\item The \textbf{lax-latching} object of $\Fa$ at $z$ is the colimit
$$\laxlatch(\Fa,z):= \colim_{\partial^{\bullet}_{\Ca/z}} \Ua_{\star} \Fa.$$
\item Define the \textbf{lax-matching} object of $\Fa$ at $z$ to be the usual matching object of the component $\Fa_{AB} : \Ca(A,B) \to \M(\Fa A, \Fa B)$ at $z \in \Ca(A,B)$.
\end{enumerate}
\end{df}
\begin{rmk}\ \
\begin{enumerate}
\item By the universal property of the colimit, there is a unique map $ \varepsilon: \laxlatch(\Fa,z) \to \Fa z$. 
\item There is a canonical map $\eta: \latch(\Fa,z) \to \laxlatch(\Fa,z)$, from the classical latching object to the lax-latching object and we have an equality
$$ \latch(\Fa,z) \to \Fa z= \varepsilon \circ \eta.$$
\end{enumerate}
\end{rmk}

\subsubsection{Locally direct categories}
We focus our study to lax diagrams indexed by locally directed category $\Ca$ which are also simple in the sense of Definition \ref{lr-simple}. This case is precisely what motivated our considerations. Indeed the $2$-category $\sxop$ has the property that each $\sx(A,B)^{\tx{op}}$ is a direct category (like $\Delta_{\tx{epi}}^{\tx{op}}$). \\

Given a lax morphism $\Fa: \Ca \to \M$ and an element $m \in \lambda$, we would like to consider a truncation `$\Fa^{\leq m}$' just like for simplicial sets. The problem is that to define such a morphism we need to change $\sx$ into `$\sx^{\leq m}$'  and consider the corresponding notion of lax functor. \ \\
An obvious attempt is to define `$\sx^{\leq m}$' as follows.
\begin{itemize}[label=$-$]
\item The objects are the same i.e $\Ob(\sx^{\leq m})=X$;
\item the category of morphisms between from $A$ to $B$ is $\sx^{\leq m}(A,B):= \sx(A,B)_{\leq m}$, the full subcategory of objects of degree $\leq m$. 
\end{itemize}
But these data don't define a $2$-category since the horizontal composition $z \otimes z'$ is only defined if $\degb(z)+\degb(z') \leq m$. This is the same situation where $(\Delta,+, \0)$ is a monoidal category but any truncation $(\Delta_{\leq \n},+,0)$ fails to be stable by addition.

By the above observation we need to enlarge a bit our $2$-categories and consider a more geral notion of \textbf{$2$-groupement} \emph{à la} Bonnin \cite{Bonnin}. The notion of \emph{groupement} was introduced by Bonnin \cite{Bonnin} as a generalization of a category. The concept of groupement covers the idea of a category without  a set of objects in the following sense. For a small $1$-category $\D$ denote by $\Ar(\D)$ the `set' of all morphisms on $\D$.  We can embed the set of objects $\Ob(\D)$ in $\Ar(\D)$ using  the identity morphism and the composition gives a partial multiplication on $\Ar(\D)$. This way the category structure is transfered on  $\Ar(\D)$ and we no longer mention a set of objects. 

\begin{warn}
We will not provide an explicit definition of $2$-groupement but will use the terminology to refer a sort of $2$-category where the horizontal composition is partially defined. Our discussion will be limited to $\Ca^{\leq m}$ for locally directed category $\Ca$ which is simple.
\end{warn}
From now $\Ca^{\leq m}$ is the $2$-groupement or the `almost $2$-category' having the same objects as $\Ca$ and all $1$-morphisms of degree $\leq m$; the $2$-morphisms are the same.

\begin{df}
A lax $\g$-morphism $\Ga:  \Ca^{\leq m} \to \M$ consists of:

\begin{enumerate}
\item A family of functors $\Ga_{AB}: \Ca(A,B)^{\leq m} \to \M(\Ga A, \Ga B)$;
\item laxity maps $\varphi: \Ga s \otimes \Ga t \to \Ga (s \otimes t)$ if $\degb(s) + \degb(t) \leq m$;
\item the laxity maps respect the functoriality i.e the following commutes when all the laxity maps exist
\[
\xy
(0,20)*+{\Ga s \otimes \Ga t}="A";
(30,20)*+{\Ga (s \otimes t) }="B";
(0,0)*+{\Ga s' \otimes \Ga t' }="C";
(30,0)*+{\Ga (s' \otimes t' )}="D";
{\ar@{->}^-{\varphi}"A";"B"};
{\ar@{->}_-{\Ga u \otimes \Ga v}"A";"C"};
{\ar@{->}^-{\Ga(u \otimes v)}"B";"D"};
{\ar@{->}_-{\varphi}"C";"D"};
\endxy
\]  
\item a coherence condition which say that the following commutes if all the laxity maps are defined 
\[
\xy
(0,20)*+{\Ga r \otimes \Ga s \otimes \Ga t}="A";
(40,20)*+{\Ga r \otimes \Ga (s \otimes t) }="B";
(0,0)*+{\Ga (r \otimes s) \otimes \Ga t }="C";
(40,0)*+{\Ga (r \otimes s \otimes t)}="D";
{\ar@{->}^-{\varphi}"A";"B"};
{\ar@{->}_-{\varphi}"A";"C"};
{\ar@{->}^-{\varphi}"B";"D"};
{\ar@{->}_-{\varphi}"C";"D"};
\endxy
\]   
\end{enumerate} 
\end{df}

There is an obvious notion of transformation of lax $\g$-morphism given by the same data except that we limit everything to the $1$-morphisms of degree $\leq m$. We will denote by $\Lax_{\g}(\Ca^{\leq m},\M)$ the category of lax $\g$-morphisms and transformations (the $\g$ here stands for groupement).\ \\
We leave the reader to check that any lax functor $\Fa : \Ca \to \M$ induces a lax $\g$-morphisms $\Fa^{\leq m} : \Ca^{\leq m} \to \M$, functorially in $\Fa$. Thus we have a truncation functor 
$$\tau_m: \Lax(\Ca, \M) \to \Lax_{\g}(\Ca^{\leq m},\M).$$ 

It's natural to ask if this functor has a left adjoint. This is the same situation with simplicial sets. In the next paragraph we will see that there is an affirmative answer to that question. 
\paragraph{Lax Left Kan extensions} 
Our problem can be interpreted as an existence of a \emph{lax left Kan extension}. Proceeding by induction on $m$ we reduces our original question to the existence of a left adjoint of the truncation functor
$$\tau_m: \Lax_{\g}(\Ca^{\leq m+1}, \M) \to \Lax_{\g}(\Ca^{\leq m},\M).$$ 

\begin{prop}
For every $m \in \lambda$ there is a left adjoint to $\tau_m$
$$\sk_m: \Lax_{\g}(\Ca^{\leq m},\M) \to  \Lax_{\g}(\Ca^{\leq m+1}, \M) $$  
\end{prop}

\begin{proof}[Sketch of proof]
Given a lax functor $\Fa : \Ca \to \M$ and a $1$-morphism $z$, we defined previously the lax-latching  object of $\Fa$ at $z$ to be 
$$\laxlatch(\Fa,z):= \colim_{\partial^{\bullet}_{\Ca/z}} \otimes (\Fa s_1,...., \Fa s_n) $$
where the colimit is taking over the sub-comma category  of  (direct) morphisms $ \gamma: (n,s_1,..., s_n) \to (1,z)$ such that $\degb(s_j)< \degb(z)$ for all $j$. \ \\
Then given a lax $\g$-morphism $\Ga: \Ca^{\leq m} \to \M$ and a $1$-morphism $z$ of degree $m+1$, all the values $\Ga s_j$ are defined for $\degb(s_j) < m+1$. Using the coherence of the lax $\g$-morphism and the functoriality of its components, one can show like in Proposition \ref{f-int}, that we have a functor $L_z: \partial^{\bullet}_{\Ca/z} \to \M$ by the formula $L_z(\gamma)= \otimes(\Ga s_1,..., \Ga s_n)$. Since $\M$ is (locally) complete, then we can define
$$\laxlatch(\Ga,z)= \colim_{\partial^{\bullet}_{\Ca/z}} \otimes (\Ga s_1,...., \Ga s_n)= \colim_{\partial^{\bullet}_{\Ca/z}} L_z.$$

For each $\gamma$ we have a canonical map $\iota_{\gamma}: \otimes (\Ga s_1,...., \Ga s_n) \to \laxlatch(\Ga,z)$. If $n=1$ then $\gamma$ is just a $2$-morphism $s \to z$ of $\Ca(A,B)$ with $z \in \Ca(A,B)$; and $\iota_{\gamma}$ is a \emph{structure map}  $\Ga s \to \laxlatch(\Ga,z)$. If $n>1$ then  $\iota_{\gamma}$ is a laxity map (eventually composed with a structure map); in particular when $(s_1,s_2) \in \otimes^{-1}(z)$ and $\gamma= \Id$, we have a \emph{pure laxity map}  
$$\iota_{\gamma}: \Ga s_1 \otimes \Ga s_2 \to \laxlatch(\Ga,z).$$

So for every $1$-morphism $z$ of degree $m+1$ the object $\laxlatch(\Ga,z)$ comes equipped with structure maps and laxity map wich are compatible with the old ones. If we assemble these data for all $z$ of degree $m+1$ we can define $\sk_m \Ga: \Ca^{\leq m+1} \to \M$ as follows
\begin{itemize}[label=$-$]
\item $(\sk_m \Ga) z:= \laxlatch(\Ga,z)$;
\item $(\sk_m \Ga)_{|\Ca^{\leq m}}= \Ga$
\item the structure maps  $\Ga s \to \laxlatch(\Ga,z)$ give the components 
$$(\sk_m \Ga)_{AB}: \Ca(A,B)^{\leq m+1} \to \M(\Ga A, \Ga B)$$
\item the laxity maps are the obvious ones.
\item the coherences come with the definition of each $\laxlatch(\Ga,z)$
\end{itemize}
We leave the reader to check that these data define a lax $\g$-morphism  $\sk_m \Ga: \Ca^{\leq m+1} \to \M$ and that $\sk_m$ is indeed a left adjoint to $\tau_m$. 
\end{proof}

\begin{rmk}
\begin{enumerate}
\item According to the description of $\sk_m$, given a transformation $\alpha: \Fa \to \Ga$ in $\Lax_{\g}(\Ca^{\leq m},\M)$, one defines $\sk_m(\alpha)$ as the transformation given the maps $\alpha_s: \Fa s \to \Ga s$  ($\degb(s)< m$) together with the maps  $\laxlatch(\Fa,z) \to \laxlatch(\Ga,z)$ ($\degb(z)=m+1$)  induced by the universal property of the colimits. In particual for each pair $(A,B)$ we have a natural transformation $(\sk_m \alpha)_{AB}: (\sk_m \Fa)_{AB} \to (\sk_m \Ga)_{AB}$ extending $\alpha_{AB}: \Fa_{AB} \to \Ga_{AB}$. 
\item It turns out that a map $\alpha : \Fa \to \Ga$ in $\Lax_{\g}(\Ca^{\leq m+1},\M)$ is determined by its restriction $\alpha^{\leq m}$ together with the following commutative squares for all $z$ of degree $m+1$:

\[
\xy
(0,20)*+{\Fa z}="A";
(30,20)*+{\Ga z}="B";
(0,0)*+{\laxlatch(\Fa,z)}="C";
(30,0)*+{\laxlatch(\Ga,z)}="D";
{\ar@{->}^-{}"A";"B"};
{\ar@{->}_-{}"C";"A"};
{\ar@{->}^-{}"D";"B"};
{\ar@{->}_-{}"C";"D"};
\endxy
\] 
\end{enumerate}
\end{rmk}

\subsubsection{Colimits and Factorization system}
Let $\M$ be a $2$-category which is locally complete and such that each $\M(U,V)$ has a factorization system. For simplicity we will reduce our study to the case where $\M$ is a monoidal category having a factorization system $(L,R)$.  Let $\Ca$ be as above and consider:

\begin{itemize}[label=$-$]
\item $\Rc=$ the class of lax morphisms $\alpha: \Fa \to \Ga$ such that for all $z$, the map 
$$g_z: \Fa z \cup_{\laxlatch(\Fa,z)} \laxlatch(\Ga,z) \to \Ga z$$ is in $R$;
\item $\le=$ the class of lax morphisms  $\alpha: \Fa \to \Ga$ such that for all $z$ the map $\alpha_z: \Fa z \to \Ga_z$ is in $L$.
\end{itemize} 
Similarly for each $m \in \lambda$ there are two classes $\le_m$ and $\Rc_m$ in $\Lax_{\g}(\Ca^{\leq m},\M)$. 

\begin{lem}\label{direct-colimit}
With the above notations the following holds.
\begin{enumerate}
\item The functor $\tau_m: \Lax_{\g}(\Ca^{\leq m+1}, \M) \to \Lax_{\g}(\Ca^{\leq m},\M)$ creates colimits.
\item Let $\alpha: \Fa \to \Ga$  be an object $\Lax_{\g}(\Ca^{\leq m+1},\M)$  such that $\tau_m\alpha$ has a factorization of type $(\le_m,\Rc_m)$:
$$\tau_m\Fa \xrightarrow{i} \K \xrightarrow{p} \tau_m\Ga.$$
 
Then there is a factorization of $\alpha$ of type $(\le_{m+1},\Rc_{m+1})$ in $\Lax_{\g}(\Ca^{\leq m+1}, \M)$.
\item Let $\alpha: \Fa \to \Ga$ be in $\le_{m+1}$ (resp. $\Rc_{m+1}$). If $\tau_m \alpha$ has the LLP (resp. RLP) with respect to all maps in $\Rc_m$ (resp. $\le_m$) then $\alpha$ has the LLP (resp. RLP)  with respect to all maps in $\le_{m+1}$ (resp. $\Rc_{m+1}$).
\end{enumerate}
\end{lem}

%\begin{cor}
%The pair $(\le, \Rc)$ is a factorization system in $\Lax(\Ca, \M)$.
%\end{cor}
We dedicate the next paragraph for the proof of the lemma.

\paragraph{Proof of Lemma \ref{direct-colimit}}
\subparagraph{Proof of (1)}
Let $\Xa: \In \to \Lax_{\g}(\Ca^{\leq m+1}, \M) $ be a diagram such that $\tau_m \Xa$ has a colimit $\Ea$ in  $\Lax_{\g}(\Ca^{\leq m},\M)$. For $i \in \In$ we have a canonical map $e_i: \tau_m \Xa_i \to \Ea$. \ \\
Let $z$ be a $1$-morphism of degree $m+1$ in $\Ca^{\leq m+1}$. By the universal property of the colimit there is canonical map 
$$\pi_i:\laxlatch(\tau_m \Xa_i,z)  \to \laxlatch(\Ea,z).$$
Note that $\laxlatch(\tau_m \Xa_i,z) \cong \laxlatch(\Xa_i,z)$ so we can drop the $\tau_m$ here. Furthermore the maps $\pi_i$ are functorial in $i$, that is we have an obvious functor $\pi: \In  \to (\M \downarrow \laxlatch(\Ea,z))$.\ \\

Let $\Lambda_z$ be the diagram in $\M$ made of the following spans (= pushout data) which are connected in the obvious manner:  

\[
\xy
%%%%%%%%%%%face arriere%%%%%
(-10,20)+(-5,0)*+{\Xa_i z}="A";
(10,20)+(10,0)*+{\Xa_j z}="B";
(-10,0)+(-5,0)*+{\laxlatch(\Xa_i,z)}="C";
(10,0)+(10,0)*+{\laxlatch(\Xa_j,z)}="D";
{\ar@{->}^-{\Xa(i \to j)}"A";"B"};
{\ar@{->}_-{\varepsilon_i}"C";"A"+(0,-3)};
{\ar@{->}_-{\varepsilon_j}"D";"B"};
{\ar@{->}_-{}"C";"D"};
%%%%%%%%%%%%%%%%%%%%face avant%%%%
%(-10,20)+(-25,-15)*+{V}="E";
(-10,0)+(-25,-15)*+{\laxlatch(\Ea,z)}="G";
%{\ar@{->}_-{ai_1}"E"+(0,-3);"G"};
%%%%%%%%%%% fleches sortantes arrieres vers avant %%%%%%%
%{\ar@{->}_-{h}"A";"E"};
{\ar@{->}_-{\pi_i}"C";"G"};
%{\ar@{.>}^-{r}"E";"B"};
{\ar@{->}^-{\pi_j}"D";"G"};
\endxy
\]

Denote by $\tld{\Ea}z$ the colimit of $\Lambda_z$. The are several ways to compute this colimit. One can proceed as follows. 
\begin{itemize}[label=$-$]
\item Introduce $\Xa_{\infty} z= \colim_{\In} \Xa_i z = \colim_{\In} \Ev_z \circ \Xa$; we have a canonical map $\delta_i: \Xa_i z \to \Xa_{\infty} z$.
\item Let $\O_i(z) = \Xa_{\infty} z \cup_{\laxlatch(\Xa_i,z)} \laxlatch(\Ea,z)$ be the object obtained by the pushout
 
\[
\xy
(0,20)*+{\laxlatch(\Xa_i,z)}="A";
(30,20)*+{\Xa_{\infty} z}="B";
(0,0)*+{\laxlatch(\Ea,z)}="C";
(30,0)*+{\O_i(z)}="D";
{\ar@{->}^-{\delta_i \circ  \varepsilon_i}"A";"B"};
{\ar@{->}_-{\pi_i}"A";"C"};
{\ar@{->}^-{}"B";"D"};
{\ar@{->}_-{}"C";"D"};
\endxy
\] 

\item The objects $\O_i(z)$ are functorial in $i$, that is we have a functor $\O(z): \In \to \M$ that takes $i$ to $\O_i(z)$. 
\item Then it's easy to see that  $\tld{\Ea}z \cong \colim_{\In} \O(z)$.
\end{itemize}

So for each $i$ and each $z$ of degree $m+1$ we have a canonical map $\iota_i : \Xa_i z \to \tld{\Ea}z$ and the following commutes 
\[
\xy
(0,20)*+{\Xa_i z}="A";
(35,20)*+{\tld{\Ea}z}="B";
(0,0)*+{\laxlatch(\Xa_i,z)}="C";
(35,0)*+{\laxlatch(\Ea,z)}="D";
{\ar@{->}^-{\iota_i}"A";"B"};
{\ar@{->}_-{\varepsilon_i}"C";"A"};
{\ar@{->}^-{\tx{can}}"D";"B"};
{\ar@{->}_-{\tx{can}}"C";"D"};
\endxy
\] 
The objects $\tld{\Ea}z$ together with the obvious maps defined a unique lax $\g$-morphism $\tld{\Ea}:\Ca^{\leq m+1} \to  \M$ such that $\tau_m(\tld{\Ea}) = \Ea$. We leave the reader to check that $\tld{\Ea}$ equipped with the natural cocone satisfies the universal property of the colimit in $\Lax_{\g}(\Ca^{\leq m+1}, \M)$. This proves the assertion $(1)$.  

\subparagraph{Proof of  (2)} Let $\alpha: \Fa \to \Ga$ be in $\Lax_{\g}(\Ca^{\leq m+1}, \M)$ and $z$ be of degree $m+1$. By hypotesis the following commutes

\[
\xy
(0,20)*+{\Fa_i z}="A";
(70,20)*+{\Ga z}="B";
(0,0)*+{\laxlatch(\Fa,z)}="C";
(70,0)*+{\laxlatch(\Ga,z)}="D";
(35,0)*+{\laxlatch(\K,z)}="E";
{\ar@{->}^-{\alpha}"A";"B"};
{\ar@{->}_-{\varepsilon}"C";"A"};
{\ar@{->}^-{\varepsilon}"D";"B"};
{\ar@{->}_-{i}"C";"E"};
{\ar@{->}_-{p}"E";"D"};
\endxy
\] 

So we have a unique map $\Fa z \cup_{\laxlatch(\Fa,z)} \laxlatch(\K,z) \to \Ga z$. We use the factorization in $\M$ to factorize this map as
$$\Fa z \cup_{\laxlatch(\Fa,z)} \laxlatch(\K,z) \xrightarrow{i'} \K' z \xrightarrow{p'} \Ga z$$
where $i' \in L$ and $p' \in R$. Write $p_z=p'$ and $i_z$ for the composite
$$ \Fa z \to \Fa z \cup_{\laxlatch(\Fa,z)} \laxlatch(\K,z) \xrightarrow{i'} \K' z.$$ 

If we assemble these data for all $z$ of degree $m+1$, we have an object $\K' \in \Lax_{\g}(\Ca^{\leq m+1}, \M)$ such that $\tau_m \K'= \K$ with maps $i: \Fa \to \K' \in \le_{m+1}$ and $p: \K' \to \Ga \in \Rc_{m+1}$ such that $\alpha= p \circ i$. And the assertion $(2)$ follows.     

\paragraph{Proof of (3)} Consider a lifting problem in $\Lax_{\g}(\Ca^{\leq m+1}, \M)$ defined by $\alpha: \Fa \to \Ga$ and $p: \Xa \to \Ya$:

\[
\xy
(0,20)*+{\Fa}="A";
(25,20)*+{\Xa}="B";
(0,0)*+{\Ga}="C";
(25,0)*+{\Ya}="D";
{\ar@{->}^-{}"A";"B"};
{\ar@{->}_-{\alpha}"A";"C"};
{\ar@{->}^-{p}"B";"D"};
{\ar@{->}_-{}"C";"D"};
\endxy
\] 

By hypothesis, in the two cases, there is a solution $h: \tau_m \Ga \to \tau_m \Xa$ for the truncated problem in $\Lax_{\g}(\Ca^{\leq m}, \M)$. The idea is to extend $h$ into a lax $\g$-morphism $h': \Ga  \to \Xa$. As usual we reduce the problem to find $h'_z$ for $z$ of degree $m+1$.  For each $z$ of degree $m+1$, we have by $h$ a canonical map $ \laxlatch(\Ga,z)  \to  \laxlatch(\Xa,z)$; if we compose with  $\varepsilon$ we get a map
$$ \laxlatch(\Ga,z)  \to  \laxlatch(\Xa,z) \xrightarrow{\varepsilon} \Xa z.$$
By the universal property of the pushout we get a unique map $\Fa z \cup_{\laxlatch(\Fa,z)} \laxlatch(\Ga,z) \to  \Xa z$ and  the following commutes:

\[
\xy
(0,20)*+{\Fa z \cup_{\laxlatch(\Fa,z)} \laxlatch(\Ga,z)}="A";
(50,20)*+{\Xa z}="B";
(0,0)*+{\Ga z}="C";
(50,0)*+{\Ya z}="D";
{\ar@{->}^-{}"A";"B"};
{\ar@{->}_-{g_z}"A";"C"};
{\ar@{->}^-{p_z}"B";"D"};
{\ar@{->}_-{}"C";"D"};
\endxy
\] 

So if either $g_z \in L$ or $p_z \in R$ we can find a lift $h'_z: \Ga z \to \Xa z$ making everything commutative. In particular the following commutes:

\[
\xy
(0,20)*+{\Ga z}="A";
(35,20)*+{\Xa z}="B";
(0,0)*+{\laxlatch(\Ga ,z)}="C";
(35,0)*+{\laxlatch(\Xa,z)}="D";
{\ar@{->}^-{h'_z}"A";"B"};
{\ar@{->}_-{\varepsilon}"C";"A"};
{\ar@{->}^-{\varepsilon}"D";"B"};
{\ar@{->}_-{h}"C";"D"};
\endxy
\] 

Thus the collection of $h$ together with the maps $h'_z$ constitutes a lax $\g$-morphism $\Ga \to \Xa$ which is obviously a solution to the original problem.  $\qed$

\subsection{Application: a model structure}
We apply the previous material to establish the following

\begin{thm}\label{direct-proj-model}
Let $\M$ be a $2$-category which is a locally model category and  $\Ca$ be a locally Reedy category which is simple and such that the degree $\degb: \Ca \to \lambda$ has a minimal value $m_0$.\\

Then there is a model structure on the category $\Lax(\Ca,\M)_u$ of all unitary lax morphisms; where a morphism $\alpha: \Fa \to \Ga$ is:

\begin{enumerate}
\item a \textbf{weak equivalence} if for all $1$-morphism $z$, $\alpha_z: \Fa z \to \Ga z$ is a weak equivalence.
\item a \textbf{fibration} if for all $1$-morphism $z$, $\alpha_z: \Fa z \to \Ga z$ is a fibration.
\item a \textbf{cofibration} if for all $z$ the canonical map 
$$g_z: \Fa z \cup_{\laxlatch(\Fa,z)} \laxlatch(\Ga,z) \to \Ga z$$
is a cofibration
\end{enumerate}
\end{thm}

\begin{cor}\label{direct-proj-model-msx}
For a monoidal model category $\M$, the category $\msx$ has a model structure, called the \textbf{projective model structure}, with the above three classes of weak equivalences, fibration, cofibration.
\end{cor}

\begin{proof}[Proof of Theorem \ref{direct-proj-model}]
The proof is very similar to the one given by Hovey \cite[Thm 5.1.3]{Hov-model} for classical directed diagrams.\\
An easy exercise shows that the above three classes of maps are closed under retracts. Following Hovey, we will say that $\alpha$ is a \emph{ good trivial cofibration} if for all $z$ 
$$g_z: \Fa z \cup_{\laxlatch(\Fa,z)} \laxlatch(\Ga,z) \to \Ga z$$ is a trivial cofibration.\ \\

Let $\degb^{-1}(m_0)$ be the set of all $1$-morphisms of degree $m_0$. For $\Fa \in \Lax(\Ca,\M)_u$, since $m_0$ is minimal then for any $1$-morphism $z \in \degb^{-1}(m_0)$, there is no laxity maps $\Fa s_1 \otimes  \Fa s_2 \to \Fa z$ other than the isomorphism:
$$\Fa \Id \otimes  \Fa z \xrightarrow{\cong} \Fa z  \ \ \ \ \ \ \  (\Fa \Id = \Id).$$

Consequently it's not hard to see that the category $\Lax_{\g}(\Ca^{\leq m_0}, \M)_u$ is just a product of copies of $\M$, i.e we have an isomorphism:

$$\Lax_{\g}(\Ca^{\leq m_0}, \M)_u \xrightarrow{\cong} \prod_{\degb^{-1}(m_0)} \M.$$
The above functor is the product of the evaluation functor at each $z \in \degb^{-1}(m_0)$.  From the previous isomorphism we deduce that $\Lax_{\g}(\Ca^{\leq m_0}, \M)_u$ is cocomplete, and by Lemma \ref{direct-colimit} we establish (by induction) that:
\begin{itemize}[label=$-$]
\item $\Lax(\Ca, \M)_u$ is cocomplete; it's also complete since lax morphisms behave nicely with limits.
\item Any map $\alpha$ can be factorized as a cofibration followed by a trivial fibration.
\item Any map $\alpha$ can be factorized as a good trivial cofibration followed by a  fibration. 
\item Good trivial cofibrations have the LLP with respect to all fibrations; and  trivial fibrations have the RLP with respect to all cofibrations. 
\end{itemize}
Finally following the same method as Hovey one shows using a retract argument that every map which is both a weak equivalence and a cofibration is a good trivial cofibration.
\end{proof}

\begin{rmk}
For a classical Reedy $1$-category $\D$, if we view it as an $\lr$-category which is simple, then the previous theorem gives the same model structure for diagrams in $\M$ indexed by $\D$ (see \cite[Thm 5.1.3]{Hov-model}).
\end{rmk}
%%%%%%%%%%% Vérifier ce qui suit %%%%%

\section{A model structure on $\M_{\S}(X)$}  \label{section-model-msx}
In this section we want to show, with a different method that for a fixed set $X$, the category $\S_X$-diagrams whose objects are called \emph{pre-cosegal categories} has a model structure when $\M$ is monoidal model category. In the first case we will assume that $\M$ is cofibrantly generated model category  with a set $\I$ (resp. $\Ja$) of  generating cofibrations (resp. acyclic cofibrations). \ \\

The model structure will be obtained by transfer of the model structure on the category \\ 
$ \K_X= \prod_{(A,B) \in X^{2}} \Hom[\S_{\ol{X}}(A,B)^{op}, \M]$ along the monadic adjunction $\M_{\S}(X) \leftrightarrows \K_X $.\ \\

On $\K_X$ we will consider for our purposes the \emph{projective} and \emph{injective} model structure. Each of these model stuctures is the product of the one  on each factor $\K_{X,AB}=\Hom[\S_{\ol{X}}(A,B)^{op}, \M]$. Since each $\S_{\ol{X}}(A,B)$ is an inverse category (like $\Delta_{epi}$) the projective and Reedy model structure on the presheaf category $\K_{X,AB}$ are the same. In fact the identity  is an isomorphism of model categories  between $(\K_{X,AB})_{\tx{proj}}$ and $(\K_{X,AB})_{\tx{Reedy}}$ see \cite[3.17]{Barwick_localization}, \cite[Ch. 5]{Hov-model}. In the last reference one views $\K_{X,AB}$ as a functor category where the source is the directed category  $\S_{\ol{X}}(A,B)^{op}$.\ \\
The reader  can find in \cite[Prop 3.3]{Barwick_localization}, \cite[Ch. 11.6; Ch.15 ]{Hirsch-model-loc},\cite[Ch. 5]{Hov-model}, \cite[A.3.3]{Lurie_HTT} \cite[Ch. 7.6.2]{Simpson_HTHC}, a description of these model structures on diagram categories.\ \\ 

Denote by $\kxproj$ (resp. $\kxinj$) the projective (resp. injective) model structure on $\K_X$. These are  cofibrantly generated model categories as  (small) product of such model categories. The generating cofibrations and acyclic cofibration are respectively  $\I_{\bullet}= \prod_{(A,B) \in X^2} \I_{AB}$  and  $\Ja_{\bullet}= \prod_{(A,B) \in X^2} \Ja_{AB}$, where $\I_{AB}$ (resp. $\Ja_{AB}$) is the corresponding set of cofibration (resp. acyclic cofibrations) in $\Hom[\S_{\ol{X}}(A,B)^{op}, \M]$. For $\kxproj$ one can actually tell more about the sets $\I_{AB}$ and $\Ja_{AB}$; the reader can find a nice description in the  above references.

In contrast to the projective model structure, there is not an explicit characterization in $\kxinj$ for the generating set of (trivial) cofibrations. The generating cofibrations are known so far to be (trivial) cofibrations between presentable objects see  \cite{Barwick_localization}, \cite{Lurie_HTT}, \cite{Simpson_HTHC} and references therein. \ \\
 
\paragraph{\textbf{ Extra hypothesis on $\M$}} For the moment we will assume that all objects of $\M$ are cofibrant \ \\
\ \\
The following lemma due  to Schwede-Shipley \cite{Sch-Sh-Algebra-module} is the key step for the transfer of the model structure on $\kx$ to  $\msx$ through the monadic adjunction 
$$\Ub: \msx \rightleftarrows \kx: \Gamma$$

\begin{lem}\label{transfer-model-str}
Let $\T$ be a monad on a cofibrantly generated model category $\K$, whose underlying functor commutes with directed colimits. Let $\I$ be the set of generating cofibrations and $\Ja$ be the set of generating  acyclic cofibrations for $\K$. Let $\I_{\T}$ and $\Ja_{\T}$ be the images of these sets under the free $\T$-algebra functor. Assume that the domains of  $\I_{\T}$ and $\Ja_{\T}$ are small relative to  $\I_{\T} \tx{-cell}$ and $\Ja_{\T} \tx{-cell}$ respectively. Suppose that
\begin{enumerate}
\item every regular $\Ja_{\T}$-cofibration is a weak equivalence, or 
\item every object of $\K$ is fibrant and every $\T$-algebra has a path object. 
\end{enumerate}

Then the category of $\T$-algebras is a cofibranty generated model category with $\I_{\T}$ a generating set of cofibrations and $\Ja_{\T}$ a generating  set of acyclic cofibrations.
\end{lem} 

In our case we will need only to show that the condition $(1)$ holds. We do this in the next paragraph.

\begin{note}
In the formulation of Schwede-Shipley \cite{Sch-Sh-Algebra-module},  $\I_{\T} \tx{-cell}$ and $\Ja_{\T} \tx{-cell}$ are respectively denoted by $\I_{\T} \tx{-cof}_{\tx{reg}}$ and $\Ja_{\T} \tx{-cof}_{\tx{reg}}$.
\end{note} 
%%% Je suspecte que c'est la structure injective et non projective qui va marcher !!!! 
\subsection{Pushouts in $\M_{\S}(X)$}
Our goal here is to understand the pushout  in $\M_{\S}(X)$ of $\Gamma \alpha$ where $\alpha: \Aa \to \Ba $ is a (trivial) cofibration in $\kxinj$ or $\kxproj$.  \ \\

By construction  $\Gamma$ preserves level-wise cofibrations and weak equivalences in $\K_X$ so $\Gamma \alpha$ is clearly a level-wise (trivial) cofibration if $\alpha$ is a (trivial) cofibration. \ \\
For an object $\Fa$  of $\M_{\S}(X)$ we want to analyse the pushout of 
$\Gamma\Ba  \xhookleftarrow{\Gamma \alpha} \Gamma\Aa \to  \Fa $. Before going to this task we start below with  a constant case;  we consider three objects with lax morphisms which are coherent. The goal is to outline how one builds laxity maps when we move each of the three objects. 

\paragraph{Analyze of the constant case}

Let  $m_1, m_2, m_3, m_{12}, m_{23}$ and $m$ be objects of $\M$ with maps:
\begin{itemize}[label=$-$]
\item  $\varphi: m_1 \otimes m_2 \otimes m_3 \to m$,
\item  $\varphi_{1,2}: m_1 \otimes m_2 \to m_{12}$,
\item  $\varphi_{2,3}: m_2 \otimes m_3 \to m_{23}$,
\item  $\varphi_{1,23}: m_1 \otimes m_{23} \to m$,
\item  $\varphi_{12,3}: m_{12} \otimes m_3 \to m$,
\end{itemize}
Assume moreover that the following `associativity condition' holds: $$\varphi_{12,3} \circ (\varphi_{1,2} \otimes \Id_{m_3})= \varphi_{1,23} \circ (\Id_{m_1} \otimes \varphi_{2,3})= \varphi.$$
 
These equalities are piece of the coherence conditions required for the laxity maps: think  $F(s)=m_1,F(t)=m_2, F(u)=m_3, F(s \otimes t)=m_{12}$, $\varphi_{t,s}= \varphi_{1,2},$ etc.  We've considered only three generic objects because the coherences for lax morphisms involves three terms. 
\begin{term}
We will say that the objects $m_i, m_{ij}$ together with the maps $\varphi$ satisfying the previous equality form a \textbf{3-ary coherent system}. There is also a notion of $n$-ary-coherent system when we consider $n$ objects $m_1, ..., m_n$ with compatible laxity maps. These are the `constant data' of lax morphism between $\O$-algebras.
\end{term}
\ \\
Given three maps $\alpha_i: m_i \to m_i'$ ( $i \in \{ 1,2,3\}$), we consider successively:
\begin{itemize}[label=$-$]
\item  $\alpha_{12}: m_{12} \to \Ra_{12}$ the pushout of $\alpha_1 \otimes \alpha_{2}$ along $\varphi_{1,2}$. $\Ra_{12}= m_1' \otimes m_2' \cup_{m_1 \otimes m_2} m_{12}$. 
\item  $\alpha_{23}: m_{23} \to \Ra_{23}$ the pushout of $\alpha_2 \otimes \alpha_3$ along $\varphi_{2,3}$.
\end{itemize}
These pushouts come with canonical maps:
$\tld{\varphi}_{i,i+1}: m_i' \otimes m_{i+1}' \to \Ra_{i,i+1}$ ( $i \in \{ 1,2\}$).

\begin{df}
Define the \textbf{coherent object} $m'$ to be the colimit of the diagram below:
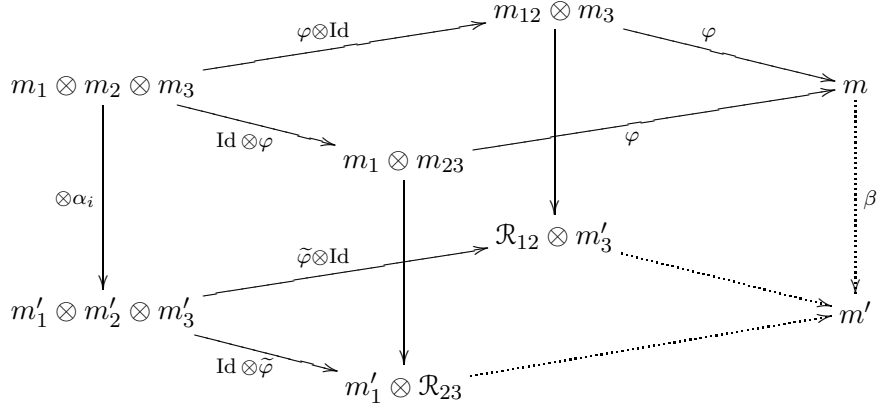
\[
\xy
%Le digram de dessus%%%
(-60,20)*+{m_1 \otimes m_2 \otimes m_3}="A";
(20,30)+(-20,0)*+{m_{12} \otimes m_3}="B";
(-20,-10)+(0,20)*+{m_1 \otimes m_{23}}="C";
(60,0)+(0,20)+(-20,0)*+{m}="D";
{\ar@{->}^{\varphi \otimes \Id}"A";"B"};
{\ar@{->}_{\Id \otimes \varphi}"A";"C"};
{\ar@{->}^{\varphi}"B";"D"};
{\ar@{->}_{\varphi}"C";"D"};
% le diagram d'en dessous%%%
(-60,-10)*+{m_1' \otimes m_2' \otimes m_3'}="X";
(20,0)+(-20,0)*+{\Ra_{12} \otimes m_3'}="Y";
(-20,-40)+(0,20)*+{m_1' \otimes \Ra_{23}}="Z";
(60,-30)+(0,20)+(-20,0)*+{m'}="W";
{\ar@{->}^{\tld{\varphi} \otimes \Id }"X";"Y"};
{\ar@{->}_{\otimes \alpha_i}"A";"X"};
{\ar@{..>}^{}"Y";"W"};
{\ar@{->}^{}"C";"Z"};
{\ar@{->}_{}"B";"Y"};
{\ar@{..>}^{\beta}"D";"W"};
{\ar@{->}_{\Id \otimes \tld{\varphi}}"X";"Z"};
{\ar@{.>}_{}"Z";"W"};
\endxy
\]
\end{df}

\begin{prop}\label{prop-cube-cof}
With the above notations, assume that all  objects of the ambient category $\M$ are cofibrant. Then
if each $\alpha_i: m_i \to m_i'$  is a (trivial) cofibration, then the canonical map $\beta: m \to m'$ is a (trivial) cofibration as well.
\end{prop}

\begin{rmk}\label{tensor-cofib-obj}
The reason we demand the objects to be cofibrant is the fact that tensoring with a cofibrant object preserves (trivial) cofibrations. This is a consequence of the pushout-product axiom. 
\end{rmk}
The strategy to prove the proposition is to `divide then conquer';  we will use the following lemma which treats the case where one of the faces in the original semi-cube is a pushout square. This is a classical Reedy-style lemma (see for example Lemma 7.2.15 in\cite{Hirsch-model-loc}). 

\begin{lem}\label{semi-cub-cof}
Let $Q$ be a semi-cube in a model category $\M$ whose colimit is an object $m'$:
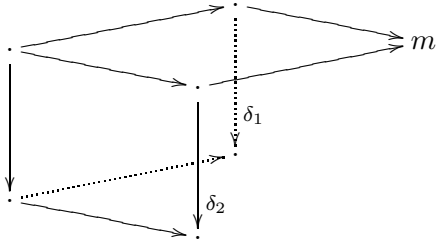
\[
\xy
%%%%%%%%%%%face arriere%%%%%
(-10,20)*+{.}="A";
(20,20)+(0,10)+(0,-4)*+{.}="B";
(-10,0)*+{.}="C";
(20,0)+(0,10)+(0,-4)*+{.}="D";
{\ar@{->}^{}"A";"B"};
{\ar@{->}_{}"A"+(0,-2);"C"};
{\ar@{.>}^<<<<<<<<<<<<<{\delta_1}"B"+(0,-2);"D"};
{\ar@{.>}^{}"C";"D"};
%%%%%%%%%%%%%%%%%%%%face avant%%%%
(-10,20)+(-15,-10)+(40,0)+(0,5)*+{.}="E";
(20,20)+(-15,-10)+(40,0)+(0,10)+(0,5)+(0,-4)*+{m}="F";
(-10,0)+(-15,-10)+(40,0)+(0,5)*+{.}="G";
%(20,0)+(-15,-10)+(40,0)+(0,10)+(0,5)+(0,-4)*+{m'}="H";
{\ar@{->}^{}"E";"F"};
{\ar@{->}^>>>>{\delta_2}"E"+(0,-2);"G"};
%{\ar@{.>}_{\beta}"F";"H"};
%{\ar@{.>}^{}"G";"H"};
%%%%%%%%%%% fleches sortantes arrieres vers avant %%%%%%%
{\ar@{->}_{}"A";"E"};
{\ar@{->}_{}"B";"F"};
{\ar@{->}^{}"C";"G"};
%{\ar@{.>}^{}"D";"H"};
\endxy
\]

Assume that the face containing $\delta_1$  is a pushout square. Then if $\delta_2$ is a (trivial) cofibration, then the canonical map  $\beta: m \to m'$ is also a (trivial) cofibration.
\end{lem}

In practice we will use the lemma when all the vertical map are (trivial) cofibrations. 

\begin{proof}[Proof of the lemma ]
We simply treat the case where $\delta_2$ is a trivial cofibration; the  method is the same when $\delta_2$ is just a cofibration.
 $\beta$ will be a trivial cofibration if we show that it has the LLP with respect to all fibrations.\ \\
Consider a lifting problem defined by $\beta$ and a fibration $p: x \twoheadrightarrow y$:
\[
\xy
(0,20)*+{m}="A";
(20,20)*+{x}="B";
(0,0)*+{m'}="C";
(20,0)*+{y}="D";
{\ar@{->}^{i}"A";"B"};
{\ar@{->}_{\beta}"A";"C"};
{\ar@{->>}^{p}"B";"D"};
{\ar@{->}^{j}"C";"D"};
\endxy
\]   

A solution to this problem is map out of $m'$, $h:m' \to x$, satisfying the obvious equalities. Since $m'$ is a colimit-object, we simply have to show that we can complete in a suitable way the semi-cube $Q$ into a commutative cube ending at $x$; the map $h$ will be then induced by universal property of the colimit. \ \\
 
If we join the lifting problem to the universal cube we get a commutative diagram displayed below  
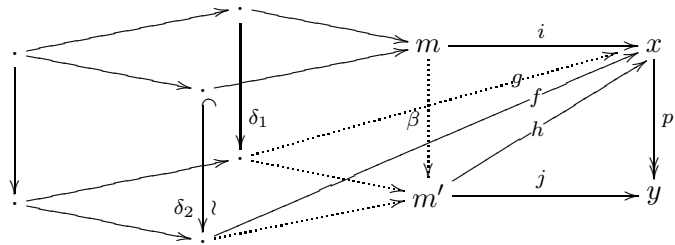
\[
\xy
%%%%%%%%%%%face arriere%%%%%
(-10,20)*+{.}="A";
(20,20)+(0,10)+(0,-4)*+{.}="B";
(-10,0)*+{.}="C";
(20,0)+(0,10)+(0,-4)*+{.}="D";
{\ar@{->}^{}"A";"B"};
{\ar@{->}_{}"A"+(0,-2);"C"};
{\ar@{->}^<<<<<<<<<<<<<{\delta_1}"B"+(0,-2);"D"};
{\ar@{->}^{}"C";"D"};
%%%%%%%%%%%%%%%%%%%%face avant%%%%
(-10,20)+(-15,-10)+(40,0)+(0,5)*+{.}="E";
(20,20)+(-15,-10)+(40,0)+(0,10)+(0,5)+(0,-4)*+{m}="F";
(-10,0)+(-15,-10)+(40,0)+(0,5)*+{.}="G";
(20,0)+(-15,-10)+(40,0)+(0,10)+(0,5)+(0,-4)*+{m'}="H";
{\ar@{->}^{}"E";"F"};
{\ar@{^{(}->}_>>>>{\delta_2}^>>>>{\wr}"E"+(0,-2);"G"};
{\ar@{.>}_{\beta}"F";"H"};
{\ar@{.>}^{}"G";"H"};
%%%%%%%%%%% fleches sortantes arrieres vers avant %%%%%%%
{\ar@{->}_{}"A";"E"};
{\ar@{->}_{}"B";"F"};
{\ar@{->}^{}"C";"G"};
{\ar@{.>}^{}"D";"H"};
%%% lifting problem%%%%
(20,20)+(-15,-10)+(40,0)+(0,10)+(0,5)+(0,-4)+(30,0)*+{x}="O";
(20,0)+(-15,-10)+(40,0)+(0,10)+(0,5)+(0,-4)+(30,0)*+{y}="P";
{\ar@{->}^{i}"F";"O"};
{\ar@{->}^{j}"H";"P"};
{\ar@{->>}^{p}"O";"P"};
{\ar@{->}|>>>>>>>>>>>>>>>{f}"G";"O"};
{\ar@{.>}|>>>>>>>>>>>>>>{g}"D";"O"+(-5,-1)};
{\ar@{->}|{h}"H";"O"+(-1,-2)};
\endxy
\]  
 
Since $\delta_2$ is a trivial cofibration there is a solution $f$ to the lifting problem defined by $\delta_2$ and $p$.  With the map $f$ we get a commutative square starting from the horn defining the pushout square in the back (the one containing $\delta_1$) and ending at $x$; so by universal property of the pushout, there is a unique map $g$ making the obvious diagram commutative. \ \\

With the maps $f$ and $g$ we have a commutative cube ending at $x$, so by universal property of the colimit we have a unique map $h: m' \to x$ making everything commutative. In particular $h$ satisfies the equality $i= h \circ \beta$.\\  
 
By construction the two cubes ending at $y$ obtained with the maps $j$ and $p\circ h$ are the same, so by unicity of the map out of  the colimit we have $j= p \circ h$. 
Consequently $h$ is a solution to the original lifting problem and $\beta$ is a  trivial cofibration as desired.  
\end{proof}

\begin{rmk}\label{rmk-cub-cofib-general}
The statement of the lemma remains valid if we consider a more general situation where the pushout square containing $\delta_1$ is  replaced by another commutative square in which the morphism $\varepsilon$ out of the pushout is a (trivial) cofibration:
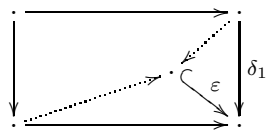
\[
\xy
(-15,15)*+{.}="A";
(15,15)*+{.}="B";
(-15,0)*+{.}="C";
(15,0)*+{.}="R";
(6,7)*+{.}="E";
{\ar@{->}^{}"A";"B"};
{\ar@{->}_{}"A";"C"};
{\ar@{->}^-{\delta_1}"B";"R"};
{\ar@{.>}^{}"B";"E"};
{\ar@{.>}^-{}"C";"E"};
{\ar@{->}^{}"C";"R"};
{\ar@{^{(}->}^{\varepsilon}"E";"R"};
\endxy
\]

\end{rmk}
\subsubsection{Proof of  Proposition \ref{prop-cube-cof}}

To prove the proposition, we will present the cube defining $m'$ as a concatenation of other universals cube where each of them satisfies the condition of the previous lemma. The proof is organized as follows.
\begin{itemize}
\item First we treat the case where only $m_1$ moves that is $\alpha_2=\Id_{m_2}$ and $\alpha_3=\Id_{m_3}$. We will denote by $z_1$ the coherent object defined with these data and denote by $Q_1$ the induced universal cube. Denote by $\beta_1:  m \to z_1$ the canonical map.
\item The lower face of the cube $Q_1$ is a coherent system ending at $z_1$. We construct $z_2$ to be the coherent object with respect to that associative system,  where only $m_2$ moves i.e $\alpha_1=\Id_{m_1'}$ and $\alpha_3=\Id_{m_3}$. We will denote by $Q_2$ the new universal cube. There is a canonical map  $\beta_2: z_1 \to z_2$. 
\item Finally with the lower face of $Q_2$, we  treat the case where only $m_3$ moves, which is similar to the first case. We have a coherent object $z_3$ with a new cube $Q_3$; there is also a canonical map $\beta_3: z_2 \to z_3$. 
\item By universal property we have  $z_3 \cong m'$, thus we can take $\beta= \beta_3 \circ  \beta_2 \circ \beta_1$. 
\item Each cube $Q_i$ is constructed from a semi-cube satisfying the conditions of the previous lemma, thus each $\beta_i$ will be a  (trivial) cofibration and the result will follow.
\end{itemize}

%%%%%%%%%%%%% Notations %%%%%%%%%%%%%%%%%%
We need some piece of notations for the rest of the proof. 
\begin{nota}\ \
\begin{enumerate}
\item Let $O_{12}$ and $P_{12}$ be the objects obtained from the pushout squares:

\begin{tabular}{cc}
$S_1=$ 
\begin{tabular}{c}
\xy
(-15,15)*+{m_1 \otimes m_2}="A";
(15,15)*+{m_{12}}="B";
(-15,0)*+{m_1' \otimes m_2}="C";
(15,0)*+{O_{12}}="R";
{\ar@{->}^-{ \varphi}"A";"B"};
{\ar@{->}_-{\alpha_1 \otimes \Id}"A";"C"};
{\ar@{->}^-{h_{12}}"B";"R"};
{\ar@{->}^-{\tld{\varphi}}"C";"R"};
\endxy
\end{tabular}
\ \ \ \ \ \ \ \ \ 
$S_2=$
\begin{tabular}{c}
\xy
(-15,15)*+{m_1' \otimes m_2}="A";
(15,15)*+{O_{12}}="B";
(-15,0)*+{m_1' \otimes m_2'}="C";
(15,0)*+{P_{12}}="R";
{\ar@{->}^-{\tld{\varphi}}"A";"B"};
{\ar@{->}_-{\Id \otimes \alpha_2}"A";"C"};
{\ar@{->}^-{k_{12}}"B";"R"};
{\ar@{->}^-{\tld{\varphi}'}"C";"R"};
\endxy
\end{tabular}
\end{tabular}
\ \\
\ \\
\ \\
From lemma \ref{collapse_po} we know that the `vertical' concatenation `$\frac{S_1}{S_2}$' of these pushout squares is `the' pushout square defining $\Ra_{12}$;  it follows that $P_{12} \cong \Ra_{12}$.\ \\

Now since colimits distribute over the tensor product, tensoring $S_1$ and $S_2$ by $m_3$ gives two pushout squares $S_1 \otimes m_3$ and $S_2 \otimes m_3$. The concatenation  of  the later squares is the pushout square
\begin{center}
$D=$ 
\begin{tabular}{c}
\xy
(-15,15)*+{m_1 \otimes m_2 \otimes m_3}="A";
(15,15)*+{m_{12} \otimes m_3}="B";
(-15,0)*+{m_1' \otimes m_2' \otimes m_3}="C";
(15,0)*+{\Ra_{12} \otimes m_3}="R";
{\ar@{->}^-{\varphi \otimes \Id}"A";"B"};
{\ar@{->}_-{\alpha_1 \otimes \alpha_2 \otimes \Id }"A";"C"};
{\ar@{->}^-{p_{12} \otimes \Id}"B";"R"};
{\ar@{->}^-{ \tld{\varphi} \otimes \Id}"C";"R"};
\endxy
\end{tabular}
\end{center}

\item Let $K_{23}$ and $L_{23}$ be the objects obtained from the the pushout squares:

\begin{tabular}{cc}
$T_1=$ 
\begin{tabular}{c}
\xy
(-15,15)*+{m_2 \otimes m_3}="A";
(15,15)*+{m_{23}}="B";
(-15,0)*+{m_2' \otimes m_3}="C";
(15,0)*+{K_{23}}="R";
{\ar@{->}^-{ \varphi}"A";"B"};
{\ar@{->}_-{\alpha_2 \otimes \Id}"A";"C"};
{\ar@{->}^-{}"B";"R"};
{\ar@{->}^-{\tld{\varphi}}"C";"R"};
\endxy
\end{tabular}
\ \ \ \ \ \ \ \ \ 
$T_2=$
\begin{tabular}{c}
\xy
(-15,15)*+{m_2' \otimes m_3}="A";
(15,15)*+{K_{23}}="B";
(-15,0)*+{m_2' \otimes m_3'}="C";
(15,0)*+{L_{23}}="R";
{\ar@{->}^-{\tld{\varphi}}"A";"B"};
{\ar@{->}_-{\Id \otimes \alpha_3}"A";"C"};
{\ar@{->}^-{}"B";"R"};
{\ar@{->}^-{\tld{\varphi}'}"C";"R"};
\endxy
\end{tabular}
\end{tabular}
\ \\
\ \\
\ \\
As usual the concatenation of $T_1$ and $T_2$ is the pushout square defining $\Ra_{23}$ so we can take $L_{23}= \Ra_{23}$. And if we tensor everywhere by $m_1'$ we still have pushout square $ m_1' \otimes T_1$ and $m_1' \otimes T_2$ and their concatenation is the pushout square:

\begin{center}
$E=$ 
\begin{tabular}{c}
\xy
(-15,15)*+{m_1' \otimes m_2 \otimes m_3}="A";
(15,15)*+{m_1' \otimes m_{23}}="B";
(-15,0)*+{m_1' \otimes m_2' \otimes m_3'}="C";
(15,0)*+{ m_1' \otimes \Ra_{23} }="R";
{\ar@{->}^-{\Id \otimes\varphi }"A";"B"};
{\ar@{->}_-{\Id \otimes \alpha_2 \otimes \alpha_3 }"A";"C"};
{\ar@{->}^-{}"B";"R"};
{\ar@{->}^-{ \Id \otimes \tld{\varphi}}"C";"R"};
\endxy
\end{tabular}
\end{center}

\end{enumerate}

\end{nota}

\paragraph*{Step 1: Moving $m_1$}\ \\
In this case we consider the  following semi-cube  whose colimit is $z_1$:

\[
\xy
%%%%%%%%%%%face arriere%%%%%
(-10,20)*+{m_1 \otimes m_2 \otimes m_3}="A";
(20,20)+(0,10)+(0,-4)+(5,0)*+{m_{12} \otimes m_3}="B";
(-10,0)*+{m_1' \otimes m_2 \otimes m_3}="C";
(20,0)+(0,10)+(0,-4)+(5,0)*+{O_{12}\otimes m_3}="D";
{\ar@{->}^{}"A";"B"};
{\ar@{->}_{}"A"+(0,-2);"C"};
{\ar@{.>}^<<<<<<<<<<<<<{}"B"+(0,-2);"D"};
{\ar@{->}^{}"C";"D"};
%%%%%%%%%%%%%%%%%%%%face avant%%%%
(-10,20)+(-15,-10)+(40,0)+(0,5)*+{m_1 \otimes m_{23}}="E";
(20,20)+(-15,-10)+(40,0)+(0,10)+(0,5)+(0,-4)+(5,0)*+{m}="F";
(-10,0)+(-15,-10)+(40,0)+(0,5)*+{m_1' \otimes m_{23}}="G";
%(20,0)+(-15,-10)+(40,0)+(0,10)+(0,5)+(0,-4)*+{z_1}="H";
{\ar@{->}^{}"E";"F"};
{\ar@{->}^>>>>{}"E"+(0,-2);"G"};
%{\ar@{.>}_{\beta_1}"F";"H"};
%{\ar@{.>}^{}"G";"H"};
%%%%%%%%%%% fleches sortantes arrieres vers avant %%%%%%%
{\ar@{->}_{}"A";"E"};
{\ar@{->}_{}"B";"F"};
{\ar@{->}^{}"C";"G"};
%{\ar@{.>}^{}"D";"H"};
\endxy
\]  
 
The face in the back is precisely the pushout square $S_1\otimes m_3$ and the map $\delta_2= \alpha_1 \otimes \Id_{m_{23}}$ is a (trivial) cofibration since  $\alpha_1$ is so (Remark \ref{tensor-cofib-obj}); then by lemma \ref{semi-cub-cof}  we know that the canonical map $\beta_1: m \to z_1$ is also a (trivial) cofibration. 

\paragraph*{Step 2: Moving $m_2$}\ \\
Introduce the  following semi-cube whose colimit is $z_2$:

\[
\xy
%%%%%%%%%%%face arriere%%%%%
(-10,20)*+{m_1' \otimes m_2 \otimes m_3}="A";
(20,20)+(0,10)+(0,-4)+(5,0)*+{O_{12} \otimes m_3}="B";
(-10,0)*+{m_1' \otimes m_2' \otimes m_3}="C";
(20,0)+(0,10)+(0,-4)+(5,0)*+{\Ra_{12}\otimes m_3}="D";
{\ar@{->}^{}"A";"B"};
{\ar@{->}_{}"A"+(0,-2);"C"};
{\ar@{.>}^<<<<<<<<<<<<<{}"B"+(0,-2);"D"};
{\ar@{->}^{}"C";"D"};
%%%%%%%%%%%%%%%%%%%%face avant%%%%
(-10,20)+(-15,-10)+(40,0)+(0,5)*+{m_1' \otimes m_{23}}="E";
(20,20)+(-15,-10)+(40,0)+(0,10)+(0,5)+(0,-4)+(5,0)*+{z_1}="F";
(-10,0)+(-15,-10)+(40,0)+(0,5)*+{m_1' \otimes N_{23}}="G";
%(20,0)+(-15,-10)+(40,0)+(0,10)+(0,5)+(0,-4)*+{z_1}="H";
{\ar@{->}^{}"E";"F"};
{\ar@{->}^>>>>{}"E"+(0,-2);"G"};
%{\ar@{.>}_{\beta_1}"F";"H"};
%{\ar@{.>}^{}"G";"H"};
%%%%%%%%%%% fleches sortantes arrieres vers avant %%%%%%%
{\ar@{->}_{}"A";"E"};
{\ar@{->}_{}"B";"F"};
{\ar@{->}^{}"C";"G"};
%{\ar@{.>}^{}"D";"H"};
\endxy
\]  

The two faces not containing $z_1$ are pushout squares; the one in the back is $S_2 \otimes m_3$ and the other one is $m_1' \otimes T_1$. All the vertical maps appearing there are (trivial) cofibrations since $\alpha_2$ is so, therefore by lemma \ref{semi-cub-cof}
 the canonical map $\beta_2: z_1 \to z_2$ is also a (trivial) cofibration.

\paragraph*{Step 3: Moving $m_3$}\ \\   
This time we consider the semi-cube below whose colimit is denoted by $z_3$:
\[
\xy
%%%%%%%%%%%face arriere%%%%%
(-10,20)*+{m_1' \otimes m_2' \otimes m_3}="A";
(20,20)+(0,10)+(0,-4)+(5,0)*+{\Ra_{12} \otimes m_3}="B";
(-10,0)*+{m_1' \otimes m_2' \otimes m_3'}="C";
(20,0)+(0,10)+(0,-4)+(5,0)*+{\Ra_{12}\otimes m_3'}="D";
{\ar@{->}^{}"A";"B"};
{\ar@{->}_{}"A"+(0,-2);"C"};
{\ar@{.>}^<<<<<<<<<<<<<{}"B"+(0,-2);"D"};
{\ar@{->}^{}"C";"D"};
%%%%%%%%%%%%%%%%%%%%face avant%%%%
(-10,20)+(-15,-10)+(40,0)+(0,5)*+{m_1' \otimes N_{23}}="E";
(20,20)+(-15,-10)+(40,0)+(0,10)+(0,5)+(0,-4)+(5,0)*+{z_2}="F";
(-10,0)+(-15,-10)+(40,0)+(0,5)*+{m_1' \otimes \Ra_{23}}="G";
%(20,0)+(-15,-10)+(40,0)+(0,10)+(0,5)+(0,-4)*+{z_1}="H";
{\ar@{->}^{}"E";"F"};
{\ar@{->}^>>>>{}"E"+(0,-2);"G"};
%{\ar@{.>}_{\beta_1}"F";"H"};
%{\ar@{.>}^{}"G";"H"};
%%%%%%%%%%% fleches sortantes arrieres vers avant %%%%%%%
{\ar@{->}_{}"A";"E"};
{\ar@{->}_{}"B";"F"};
{\ar@{->}^{}"C";"G"};
%{\ar@{.>}^{}"D";"H"};
\endxy
\] 

The face on the left is a pushout square and corresponds to $m_1' \otimes T_2$. The map $\delta_2= \Id_{\Ra_{12}} \otimes \alpha_3$ in the face on the back is a (trivial) cofibration since $\alpha_3$ is so; applying lemma \ref{semi-cub-cof} again we deduce that the canonical map $\beta_3: z_2 \to z_3$ is also a (trivial) cofibration.\ \\

One can easily see that the (vertical) concatenation of the previous universal cubes constitutes a universal cube for the original semi-cube defining $m'$. By unicity of the colimit we can take $m'=z_3$ and $\beta = \beta_3 \circ \beta_2 \circ \beta_1$. Since each $\beta_i$ is a (trivial) cofibration, by composition  $\beta$ is a (trivial) cofibration as well, which is just we wanted to prove.$\qed$

\begin{rmk}
The proposition remains valid if we allow the objects $m_{12}$ and $m_{23}$ to move by (trivial) cofibrations. This time we will have to use the more general version of Lemma \ref{semi-cub-cof} pointed out in Remark \ref{rmk-cub-cofib-general}.
\end{rmk}

\subsubsection{The main lemma} 
In the following our goal is to establish that
\begin{lem}\label{pushout-MX}
Given a diagram $\Gamma\Ba \xhookleftarrow{\Gamma \alpha } \Gamma\Aa \to  \Fa$ consider the pushout in $\msx$ 
\[
\xy
(-10,20)*+{\Gamma\Aa}="A";
(20,20)*+{\Fa}="B";
(-10,0)*+{\Gamma\Ba}="C";
(20,0)*+{\Ga}="D";
{\ar@{->}^{\sigma}"A";"B"};
{\ar@{_{(}->}_{\Gamma \alpha}"A"+(0,-3);"C"};
{\ar@{.>}^{H_{\alpha}}"B";"D"};
{\ar@{.>}^{}"C";"D"};
\endxy
\]   

Then if $\alpha$ is  a level-wise trivial cofibration in $\kx$ then $\Ub H_{\alpha}$ is a level-wise trivial cofibration; in particular $H_{\alpha}$ is a weak equivalence in $\msx$. 
\end{lem}

\begin{proof}
This is a special case of Lemma \ref{lem-pushout-laxalg} in the Appendix. In fact $(\sx)^{\tx{$2$-op}}$ is an $\O$-algebra where $\O$ is the multisorted operad for (nonunital) $2$-categories; $\M$ is a special Quillen $\O$-algebra with all the objects cofibrant. Furthermore:
\begin{enumerate}
\item  $(\sx)^{\tx{$2$-op}}$ is an $\iro$-algebra  in the sense of Definition \ref{iro-hco}. This follows from the fact that  the composition in $\sx$ is a concatenation of chains and the $2$-morphisms  are parametrized by the morphisms in $\Delta_{\tx{epi}}$. In fact the composition of $2$-morphisms is simply a generalization of the ordinal addition of morphisms in $(\Delta_{\tx{epi}},+,\0)$ ; consequently the concatenation of $2$-morphisms cannot be the identiy unless all of them are identities. 
\item The pair $((\sx)^{\tx{$2$-op}},\M)$ is an $\hco$-pair  in the sense of Definition \ref{iro-hco}, since the left adjoint $\Gamma$ preserves the level-wise trivial cofibrations (see Remark \ref{gamma-preserve-we}). 
\end{enumerate}
We have $\msx= \Laxalg((\sx)^{\tx{$2$-op}},\M)$.
\end{proof}

\begin{rmk}
It's important to notice that in the lemma we've considered $\alpha$ a level-wise cofibration in $\K_X$; these are precisely the injective cofibrations therein. But this situation covers also the projective case,  since projective cofibrations are also injective ones. \ \\
So in either $\kxproj$ or $\kxproj$, the pushout of   $\Gamma \alpha$ is a lewel-wise weak equivalence and the condtion $(1)$ of lemma \ref{transfer-model-str} will hold. 
\end{rmk}

%%%%%%%%%%%%%%%%%%%%%%%%%%%%%%%%%%%%%%%%%%%%%%%%%%%%%%%%%%
%%%%%%%%%%%%%%%%%%%%%%%%%%%%%%%%%%%%%%%%%%%%%%%%%%%%%%%%%
\subsection{The projective model structure}\ \\
According to a well know result on diagram categories in cofibrantly generated model category, see \cite[Theorem 11.6.1]{Hirsch-model-loc},  each diagram category $\Hom[\S_{\ol{X}}(A,B)^{op}, \M]$ has a cofibrantly generated model structure which is know to be the projective model structure. \ \\
\\
In these settings a morphism $\sigma: \Fa \to \Ga$ is: 
\begin{itemize}
\item A weak equivalence in $\Hom[\S_{\ol{X}}(A,B)^{op}, \M]$ if it is a level-wise equivalence: for every $w$  the component $\sigma_w : \Fa w \to \Ga w$ is a weak equivalence in $\M$,
\item A  fibration  in $\Hom[\S_{\ol{X}}(A,B)^{op}, \M]$ if it is a level-wise fibration:  $\sigma_w : \Fa w \to \Ga w$ is a (acyclic) fibration in $\M$.
\item A trivial fibration is a map which is both a fibration and a weak equivalence. 
\end{itemize}

\paragraph*{Left adjoint of evaluations}

For any object $w  \in \S_{\ol{X}}(A,B)^{op}$ the evalutation functor at $w$ : $\Ev_{w}: \Hom[\S_{\ol{X}}(A,B)^{op}, \M] \to \M$ has a left adjoint  $$\Fb_{-}^{w}: \M \to \Hom[\S_{\ol{X}}(A,B)^{op}, \M]$$
One defines $\Fb_{-}^{w}$ by `the body' of the Yoneda functor $\Ya_w$ (see \cite[Section 11.5.21 ]{Hirsch-model-loc}): 
$$  \Fb_{m}^{w}= m \otimes \Ya_w =  \coprod_{\Hom(w,-)} m ,\ \ \ \ \ \ \ \tx{for $m \in \M.$ }$$  
\\
This means that for $v  \in \S_{\ol{X}}(A,B)^{op}$, $\Fb_{m}^{w}(v)$ is the coproduct of copies of $m$ indexed by the set $\Hom_{\S_{\ol{X}}(A,B)^{op}}(w,v)$. The fact that $ \Fb_{-}^{w}$ has the desired properties follows from the Yoneda lemma.\ \\

With the functor $\Fb$ we have that the set of generating cofibrations is: 
$$\I_{AB}= \coprod_{w \in \S_{\ol{X}}(A,B)^{op}} \Fb_{\I}^{w}= \coprod_{w \in \S_{\ol{X}}(A,B)^{op}} \{\Fb_{m}^{w} \xrightarrow{ \Fb^w_{\alpha}}  \Fb_{m'}^{w} \}_{(m \xrightarrow{\alpha} m') \in \I}$$ 
\\
Similarly the set of generating acyclic cofibrations is: 
$$\Ja_{AB}= \coprod_{w \in \S_{\ol{X}}(A,B)^{op}} \Fb_{\Ja}^{w}$$
\ \\ 

Consider the product model structure on $ \kxproj= \prod_{(A,B) \in X^{2}} \Hom[\S_{\ol{X}}(A,B)^{op}, \M]_{\tx{proj}}$ where the three class of maps, cofibrations, fibrations, weak equivalences, are the natural ones i.e factor-wise cofibrations, fibrations, and weak equivalences (see \cite[Example 1.16]{Hov-model}). \ \\

\subsubsection{The main theorem}
\begin{thm}\label{main-proj-msx}
The category $\M_{\S}(X)$ has a combinatorial model structure where:
\begin{itemize}[label=$-$, align=left, leftmargin=*, noitemsep]
\item a weak equivalence is a map $\sigma$ such that $\Ub(\sigma)$ is a weak equivalence in $\kxproj$,
\item a fibration is a map $\sigma$ such that $\Ub(\sigma)$ is a fibration  in $\kxproj$,
\item a cofibration is a map having the left lifting property (LLP) with respect to all trivial fibrations,
\item the set of generating cofibrations is $\Gamma(\I)$,
\item the set of generating acyclic cofibrations is $\Gamma(\Ja)$.
\end{itemize}
We will refer  this model structure as the `projective' model structure on $\msx$ and denoted by $\msxproj$.
\end{thm}

\begin{proof}
Thanks to our lemma \ref{pushout-MX}, the condition $(1)$ of  lemma \ref{transfer-model-str}  holds. It follows from the lemma that $\msx$ is a cofibrantly generated model category with the corresponding set of generating (trivial) cofibrations.  And from Theorem \ref{MX-local-pres}) we know that $\msx$ is locally presentable. 
\end{proof}

\subsection{Lifting the injective model structure on $\kx$}
By the same argument as in the projective case we establish the following.
\begin{thm}\label{main-inj-msx}
The category $\M_{\S}(X)$ has a combinatorial model structure where:
\begin{itemize}[label=$-$, align=left, leftmargin=*, noitemsep]
\item a weak equivalence is a map $\sigma$ such that $\Ub(\sigma)$ is a weak equivalence in $\kxinj$,
\item a fibration is a map $\sigma$ such that $\Ub(\sigma)$ is a fibration  in $\kxinj$,
\item a cofibration is a map having the left lifting property (LLP) with respect to all trivial fibrations,
\item the set of generating cofibrations is $\Gamma(\I)$,
\item the set of generating acyclic cofibrations is $\Gamma(\Ja)$.
\end{itemize}
We will refer  this model structure as the `injective' model structure on $\msx$ and denoted it by $\msxinj$.
\end{thm}

\begin{proof}
The same as for the previous theorem.
\end{proof}

\begin{cor}\label{quillen-equiv-inj-proj}
The identity functor $\Id: \msxproj  \rightleftarrows \msxinj: \Id$ is a Quillen equivalence.
\end{cor}

\begin{proof}
The weak equivalences are the same and a projective (trivial) cofibration is also an injective (trivial) cofibration.
\end{proof}
\begin{rmk}
Note that in both  $\msxinj$ and $\msxproj$ the fibrations and weak equivalences are the underlying ones in $\kxinj$ and $\kxproj$ respectively. Since limits in $\msx$ are computed level-wise, it's easy to see that both $\msxproj$ and $\msxinj$ are \emph{right proper} if $\M$. In fact one establish first that  $\kxproj$ and $\kxinj$ are also right proper. For left properness the situation is a bit complicated, we will discuss it later.  
\end{rmk}
%%%%%%%%%%%%%%%%%%%%%%%%%%%%%%%%%%%%%%%%%%%%%%%%%%%%%%%%%%%%%%%%%%%%%
%%%%%%%%%%%%%%%%%%%%%%%%%%%%%%%%%%%%%%%%%%%%%%%%%%%%%%%%%%%%%%%%%%%%%
%%%%%%%%%%%%%%%%%%%%FIN MODEL STRUCTURE%%%%%%%%%%%%%%%%%%%%%%%%%%% %%%%%%%%%%%%%%%%%%%%%%%%%%%%%%%%%%%%%%%%%%%%%%%%%
%%%%%%%%%%%%%%%%%%%%%%%%%%%%%%%%%%%%%%%%%%%%%%%%%%%%%%%%%%%%%%%%%%%%%
%%%%%%%%%%%%%%%%%%%%%%%%%%%%%%%%%%%%%%%%%%%%%%%%%%%%%%%%%%%%%%%%%%%%%
%%%%%%%%%%%%%%%%%%%%%%%%%%%%%%%%%%%%%%%%%%%%%%%%%%%%%%%%%%%%%%%%%%%%%
%%%%%%%%%%%%%%%%%%%%%%%%%%%%%%%%%%%%%%%%%%%%%%%%%%%%%%%%%%%%%%%%%%%%%
%%%%%%%%%%%%%%%%%%%%%%%%%%%%%%%%%%%%%%%%%%%%%%%%%%%%%%%%%%%%%%%%%%%%%
%%%%%%%%%%%%%%%%%%%%%%%%%%%%%%%%%%%%%%%%%%%%%%%%%%%%%%%%%%%%%%%%%%%%%

\section{Variation of the set of objects} \label{variation-objects}

Let $\Set$ be the category of sets of some universe $\U \subsetneq \U' $. So far we've considered the category $\M_{\S}(X)$ for a fixed set $X \in \U$. In this section we are going to move $X$. \ \\

Since the construction of $\S_{\ol{X}}$ is funtorial in $X$, any funtion $f: X \to Y$ induces a homomorphism (=strict $2$-functor) $\S_f: \S_{\ol{X}} \to \S_{\ol{Y}}$. We then have a functor $\fstar: \msx \to \msy$. Below we will see that there is a left adjoint $\fex$ of $\fstar$. When no confusion arises we will simply write again $f$ to mean $\S_f$. \ \\

Let $\M_{\S}(\Set)$ be the category described as follows:
\begin{description}
\item[objects] are pairs $(X, \Fa)$ with $X \in \Set$ and $\Fa \in \M_{\S}(X)$,
\item[morphisms] from $(X,\Fa$) to $(Y, \Ga)$ are pairs $(f,\sigma)$ with $f \in \Set(X,Y)$ and $\sigma \in \M_{\S}(X)( \Fa, \fstar \Ga)$. 
\end{description}  

In the same way we have a category $\K_{\Set}$ and a forgetful functor $\Ub: \M_{\S}(\Set) \to \K_{\Set}$. 

\begin{lem}
If $\M$ is a symmetric closed monoidal which is cocomplete then:

\begin{enumerate}
\item $\Ub$ is monadic 
\item The monad induced by $\Ub$ preserves directed colimits.
\end{enumerate}
\end{lem}
\begin{proof}
The assertion $(1)$ is easy and is treated in the same way as in the fixed set case.  For the  assertion $(2)$ we simply need to see how one computes colimit in $\K_{\Set}$.  Each function $f: X \to Y$ induces an adjoint pair:
$\fex : \K_X \rightleftarrows \K_Y : \fstar$.\ \\

Every diagram $\J : \D \to  \K_{\Set}$,  induces a diagram $\pr_1(\J) : \D \to \Set$ and we can take the colimit $X_{\infty}= \colim \pr_1(\J)$. For each $d \in \D$ the canonical map $i_d: X_d \to X_{\infty}$ induces an object $i_{d !} \J d$ in $\K_{X_{\infty}}$. It's not hard to see that $\J$ induces a diagram $i_{!} \J : \D \to \K_{X_{\infty}}$ where the morphisms connecting the different $i_{d !} \J d$ are induced by the universal property of the adjoint. \ \\

The colimit of $\J$ is the colimit of the pushforward diagram $i_{!} \J$. Given a directed diagram $\J$, one has to show that the pushforward of the colimit of $\J$ is the colimit of the pushforward diagram. One proceeds \textbf{exactly} in the same manner as Kelly and Lack \cite[Lemma 3.2, Thm 3.3]{Kelly-Lack-loc-pres-vcat} who treated the case for $\M$-categories. 
\end{proof}

\begin{thm}\label{mset-loc-pres}
Let  $\M$ be a symmetric monoidal closed category.
\begin{enumerate}
\item If $\M$ is cocomplete then so is $\M_{\S}(\Set)$ ,
\item If $\M$ is locally presentable then so is $\M_{\S}(\Set)$. 
\end{enumerate}
\end{thm}

\begin{proof}
All is proved in the same way as for $\M_{\S}(X)$. 
\end{proof}
\subsection{The projective model structure on $\mset$}
\subsubsection{$\fstar$ has a left adjoint}
Let $f: X \to Y$ be a function. 
As pointed above we have an adjunction $ \fex : \kx \rightleftarrows \ky : \fstar$ which is  just the product adjunction for each pair $(A,B)$ 
$$\fex_{AB}: \Hom[\S_{\ol{X}}(A,B)^{op}, \M] \rightleftarrows \Hom[\S_{\ol{Y}}(A,B)^{op}, \M] : \fstar_{AB}.$$
The last adjunction is a Quillen adjunction between the projective model structure: this is Proposition 3.6 in \cite{Barwick_localization}. It follows that $ \fex : \kx \rightleftarrows \ky : \fstar$ is also a Quillen adjunction betwen the projective model structure. \ \\ 

In what follows we will show that we have also a Quillen adjunction between the projective model structures on $\msx$ and $\msy$.  Let's denote again the functor $\fstar: \msy \to \msx$ the pullback functor. By definition $\fstar$ preserves everything which is level-wise: (trivial) fibrations , weak equivalences, limits in $\msy$ (limits are computed level-wise).  To show that  we have a Quillen adjunction it suffices to show that $\fstar$ has a left adjoint since it already preserves  fibrations and trivial fibrations (see \cite[lemma 1.3.4]{Hov-model} ). We will use the adjoint functor theorem for locally presentable category since $\msx$ and $\msy$ are such categories.

\begin{thm}
A functor between locally presentable categories is a right adjoint if and only if it preserves limits and $\lambda$-directed colimits for some regular cardinal $\lambda$. 
\end{thm}

\begin{proof}
See \cite[1.66]{Adamek-Rosicky-loc-pres}
\end{proof}
\begin{prop}
For a symmetric monoidal closed  model category $\M$ and a function $f : X \to Y$ the following holds.
\begin{enumerate}
\item The functor $\fstar$ has a left adjoint $\fex: \msx \to \msy$,
\item The adjunction $ \fex : \msx \rightleftarrows \msy : \fstar$ is a Quillen adjunction
\item We have a square of Quillen adjunctions 

\[
\xy
(0,15)*+{\msx}="A";
(30,15)*+{\msy}="B";
(0,0)*+{\kx}="C";
(30,0)*+{\ky}="S";
{\ar@{->}_{\fex}"A"+(7,-1);"B"+(-7,-1)};
{\ar@{<-}^{\fstar}"A"+(7,1);"B"+(-7,1)};
{\ar@{->}^{\Ub}"A"+(1,-3);"C"+(1,3)};
{\ar@{<-}_{\Gamma}"A"+(-1,-3);"C"+(-1,3)};
{\ar@{->}^{\Ub}"B"+(1,-3);"S"+(1,3)};
{\ar@{<-}_{\Gamma}"B"+(-1,-3);"S"+(-1,3)};
{\ar@{->}_{\fex}"C"+(7,-1);"S"+(-7,-1)};
{\ar@{<-}^{\fstar}"C"+(7,1);"S"+(-7,1)};
\endxy
\] 
in which only two squares are commutative:
\begin{itemize}[label=$-$]
\item $\Ub \circ \fstar= \fstar \circ \Ub$ and 
\item $\Gamma \circ \fex = \fex \circ \Gamma$ .
\end{itemize}
\end{enumerate}
\end{prop}
\begin{proof}
Since $\fstar$ preserves limits and thanks to  the adjoint functor theorem, it suffices to show that it also preserves directed colimits. But as the functor $\Ub : \msy \to \ky$ preserves filtered colimits (Proposition \ref{monad-finitary}) , it follows that filtered colimits in $\msy$ are computed level-wise and since $\fstar: \msy \to \msx$ preserves every level-wise property it certainly preserves them and the assertion $(1)$ follows.\ \\ 

The assertion $(2)$ is a consequence of \cite[lemma 1.3.4]{Hov-model}: from $(1)$ we know that $\fstar$ is a right adjoint functor but as it preserves (trivial) fibrations, the adjunction $\fex \dashv \fstar$ is automatically a  Quillen adjunction. The assertion $(3)$ is clear.
\end{proof}

Recall that for $\Fa \in \msx$, $\Ga \in \msy$  a morphism $\sigma  \in \Hom_{\mset}(\Fa,\Ga)$ is a pair $\sigma=(f,\sigma)$ where $f \in \Set(X,Y)$ and $\sigma \in \Hom_{\msx}(\Fa, \fstar \Ga)$. An easy excercise shows that 
\begin{prop}
The canonical functor $\Pf: \mset \to \Set$ is a Grothendieck fibration (or fibered category). The fiber category over  $X \in \Set$ being $\msx$ and the inverse image functor is  $\fstar$ for $f \in \Set(X,Y)$.
\end{prop}
\begin{rmk}
Note that for $f \in \Set(X,Y)$ and $\Ga \in \msy$,  $\fstar \Ga$ is the composite 
$$(\S_{\ol{X}})^{2\tx{-op}} \xrightarrow{(\S_{f})^{2\tx{-op}} }   (\S_{\ol{Y}})^{2\tx{-op}} \xrightarrow{\Ga}   \M.$$ The identities $\Id_{\Ga(f(s))}$ gives, in a tautological way, a  canonical cartesian lifting of $f$, therefore  $P$ has a cleavage (or is cloven).
\end{rmk}
As we saw previously the inverse image functor has a left adjoint $\fex$ so we deduce that 
\begin{prop}
The canonical functor $\Pf: \mset \to \Set$ is a  bifibration, that is $\Pf^{op}: \mset^{op} \to \Set^{op}$ is also a Grothendieck fibration (or $\Pf$ is cofibered).
\end{prop}
\begin{proof}
Apply lemma 9.1.2 in \cite{Jacobs_cat_type}.
\end{proof}
\begin{rmk}

From the adjunction $\fex \dashv \fstar$, it's not hard to see that  $\Pf^{op}$ has a cleavage; thus $\Pf$ is a cloven bifibration.
\end{rmk}
\subsubsection{A fibered model structure on $\mset$}
In what follows we give a first model structure on $\mset$ using the previous bifibration $\Pf: \mset \to \Set$. The key ingredient is to use Roig's work  \cite{Roig_bifib} on Quillen model structure on the `total space' of a Grothendieck bifibration. As pointed out by Stanculescu \cite{Stanculescu_bifib}, there is a gap in Roig's theorem. A reformulation was given by  Stanculescu in \emph{loc. cit}    and is recalled hereafter.

\begin{thm}[Stanculescu]\label{Stanculescu_bifib}
Let $\Pf: \e \to \bz $ be cloven Grothendieck bifibration. 
Assume that 
\renewcommand{\theenumi}{\roman{enumi}}
\begin{enumerate}
\item $\e$ is complete and cocomplete,
\item the base category $\bz$ as a model structure $(\cof, \we, \fib)$
\item for each object $X \in \bz$ the fibre category $\e_X$ admits a model structure $(\cof_X, \we_X, \fib_X)$,
\item for every morphism $f: X \to Y$ of $\bz$, the adjoint pair is $(\fex,\fstar)$ is a Quillen pair,
\item for $f= \Pf(\sigma)$ a weak equivalence in $\bz$, the functor $\fstar$ preserves and reflects weak equivalences,
\item for $f= \Pf(\sigma)$ a a trivial cofibration in $\bz$, the unit of the adjoint pair $(\fex,\fstar)$ is a  weak equivalence.
\end{enumerate}

Then there is a model structure on $\e$ where a map $\sigma:\Fa \to \Ga $ in  $\e$ is
\begin{itemize}
\item  a weak equivalence if $f = \Pf(\sigma) \in \we$ and  $\sigma^f : \Fa \to \fstar \Ga  \in \we_X$ 
\item  a cofibration  if $f = \Pf(\sigma) \in \cof$ and  $\sigma_f :  \fex \Fa \to \Ga  \in \cof_Y$ 
\item a fibration  if $f = \Pf(\sigma) \in \fib$ and  $\sigma^f :   \Fa \to \fstar \Ga  \in \fib_X$ 
\end{itemize} 
\end{thm}

Let $\Set_{min}$ be the category of $\Set$ with the minimal model structure: weak equivalences are isomorphisms, cofibration and fibrations are all morphisms. In particular trivial cofibrations and fibrations are simply isomorphisms. Recall that if $f: X \to Y$ an isomorphism then  $(\S_{f})^{2\tx{-op}}$ is also an isomorphism, and we can take $\fex=(f^{-1})^{\star}$ ; one clearly see that the conditions $(v)$ and $(vi)$ of the theorem hold one the nose.\ \\ 

Let's fix the projective model structure on each $\msx$ as $X$ runs through $\Set$. By virtue of the previous theorem we deduce that
\begin{thm}\label{fibered-model-mset}
For a symmetric closed monoidal model category $\M$, the category $\mset$ has a Quillen model structure where a map $\sigma=(f,\sigma):\Fa \to \Ga $ is 
\begin{enumerate}
\item  a weak equivalence if $f: X \to Y$ is an isomorphism of sets and  $ \sigma: \Fa \to \fstar \Ga $ is a weak equivalence in $\msx$,
\item a cofibration if the adjoint map $ \tld{\sigma}: \fex \Fa \to \Ga $ is a cofibration in $\msy$,
\item a fibration if $ \sigma: \Fa \to \fstar \Ga $ is a fibration in $\msx$.
\end{enumerate} 
We will denote $\mset$ endowed with this model structure by $\mset$-proj. 
\end{thm}

\begin{proof}
$\mset$ is complete and cocomplete as any locally presentable category. The other conditions of Theorem \ref{Stanculescu_bifib} are clearly fulfilled.
\end{proof}

\begin{rmk}
If we replace everywhere $\mset$ by $\kset$ in the previous theorem we will get as well a fibered model structure on $\kset$. The adjunction $ \Ub : \mset \rightleftarrows \kset : \Gamma$ is a Quillen adjunction. 
\end{rmk}
\subsubsection{Cofibrantly generated} 
In the following we simply show that the fibered model structures on $\mset$ and $\kset$ are cofibrantly generated. 
\paragraph{Some natural $\S$-diagrams} 
The discussion we present here follows closely Simpson's considerations in \cite[13.2]{Simpson_HTHC}.\ \\

We will denote by $\cn$ the category  described as follows:\\
 $\Ob(\cn)=\{0,...,,n \}$ is the set of first $n+1$ natural numbers and
\begin{equation*}
\Hom_{\cn}(i,j) =
  \begin{cases}
     \{(i,j) \} & \text{if $i<j$} \\
     \{\Id_i=(i,i) \}  & \text{if $i=j$ }\\
     \varnothing & \text{if $i>j$ }
  \end{cases}
\end{equation*}
The composition is the obvious one. In the $2$-category $\S_{\cn}$ there is a special object in the category of morphism $\S_{\cn}(0,n)$ which is represented as: $0 \to 1 \to \cdots \to n$. It is the reduced string of length $n$ (i.e with no repetition of object)  in $\S_{\cn}(0,n)$; or equivalently the maximal nondegenerated simplex in the nerve.  We will denote this $1$-morphism by $s_n$. Let $\Fb^{s_n}_{-}: \M \to \Hom[\S_{\cn}(0,n)^{op}, \M] $ be the left adjoint of the evaluation at $s_n$. \ \\

We have as usual the categories $\msn$ and $\kn$ with the monadic adjunction $ \Ub: \msn \leftrightarrows \kn: \Gamma$. This adjunction is moreover a Quillen adjunction. For the record $\msn$ is the category of lax morphisms from $\S_{\cn}^{2\opt}$ to $\M$ and $\kn= \prod_{(i,j) \in \Ob(\cn)^2} \Hom[\S_{\cn}(i,j)^{\op}, \M] $. \ \\

For $B \in \Ob(\M)$ we will denote by $\delta(s_n, B)$ to be the object of $\kn$ given by:
 
 \begin{equation*}
\delta(s_n, B)_{ij} =
  \begin{cases}
   \Fb^{s_n}_{B}  & \text{if $i=0, j=n$} \\
    (\varnothing, \Id_{\varnothing})  & \text{the constant functor other wise}.
  \end{cases}
\end{equation*}

For $B \in \Ob(\M)$ define  $h(\cn; B) \in \msn $ to be $\Gamma \delta(s_n, B)$. 

\begin{lem}\label{gen-set}
For any $B \in \Ob(\M)$ and $\Fa \in \msy$ the following are equivalent.

\begin{enumerate}
\item a morphism $\sigma : h(\cn; B) \to \Fa$ in $\mset$
\item a sequence of elements  $(y_0, ..., y_n)$ of $Y$ together with a morphism $B \to \Fa(y_0,....,y_n)$ in $\M$ .
\end{enumerate}
\end{lem}

\begin{proof}[Sketch of proof]

A morphism $\sigma=(f,\sigma) : h(\cn; B) \to \Fa$ is by definition a function $f : \{0,...,n\} \to Y$ together with a morphism $\sigma: h(\cn; B) \to \fstar \Fa$ in $\msn$ . Setting $y_i= f(i)$ we get $f s_n= (y_0,...,y_n)$ and  by adjunction we have:

\begin{equation*}
\begin{split}
\Hom_{\mset}[(h(\cn; B), \Fa]
&=\Hom_{\msn}[(h(\cn; B),\fstar \Fa ]\\
&= \Hom_{\msn}[\Gamma \delta(s_n, B),\fstar \Fa ]  \\
&\cong \Hom_{\kn}[\delta(s_n, B), \Ub (\fstar \Fa) ]   \\
&\cong \Hom[\Fb^{s_n}_{B},\fstar \Fa_{y_0 y_n} ] \\
&\cong \Hom[B, \Fa_{y_0 y_n} (f s_n) ]\\
&= \Hom[B, \Fa(y_0,....,y_n)]\\
\end{split}
\end{equation*}
\end{proof}

\begin{rmk}
$h$ and $\delta$ are, in an obvious way, left adjoint to the evaluation on generic $1$-morphisms of lenght $n$: $\Ev_n: \mset \to \M$ and $\Ev_n: \kset \to  \M$ respectively. It's not hard to see that they are moreover left Quillen functors. 
\end{rmk}

Recall that $\I$ and $\Ja$ are the respective generating set of cofibration and trivial cofibrations in $\M$. Since projective (trivial) fibrations are the object-wise (trivial) fibrations, it follows that a map $\sigma=(f,\sigma): \Fa \to \Ga$ is a (trivial) fibration if for all $s$ the map $\Fa s \xrightarrow{\sigma_s}   \Ga f(s)$ has the right lifting property with respect to ($\I$) $\Ja$.\ \\

If we combine this observation and  lemma \ref{gen-set} we deduce that 
\begin{prop}\label{prop-gen-set}
The following sets constitute respectively a set of generating cofibrations and trivial cofibrations in $\msetproj$. 
$$\I_{\mset}= \coprod_{\cn, n\geq 1} \{ h(\cn; q): h(\cn; A) \to h(\cn; B) \}_{q :A \to B \in \I}$$
$$\Ja_{\mset}= \coprod_{\cn, n\geq 1} \{ h(\cn; q): h(\cn; A) \to h(\cn; B) \}_{q :A \to B \in \Ja} $$

\ \\
Similarly the following sets constitute respectively a set cofibrations and trivial cofibrations in $\ksetproj$:
$$\I_{\kset}= \coprod_{\cn, n\geq 1} \{ \delta(s_n, q): \delta(s_n, A)\to \delta(s_n, B) \}_{q :A \to B \in \I}$$
$$\Ja_{\kset}= \coprod_{\cn, n\geq 1} \{ \delta(s_n, q): \delta(s_n, A)\to \delta(s_n, B) \}_{q :A \to B \in \Ja} $$ 
\end{prop}

\begin{proof}
 All is proved by adjointness. \ \\

Let $\Fa \in \msx$, $\Ga \in \msy$, $f: X \to Y$ and $\sigma: \Fa \to \fstar \Ga$ a morphism in $\msx$.  For every $1$-morphism $s=(y_0,...,y_n)$ in $\S_{\ol{X}}$ using  lemma \ref{gen-set}, it's easy to see that for all $q :A \to B \in \Ar(\M)$ we have isomorphism between Hom in arrow-categories:

$$ \Hom_{\Ar[\mset]}[h(\cn; q), \sigma] \cong \Hom_{\Ar[\kset]}[\delta(s_n; q), \Ub \sigma]  \cong  \Hom_{\Ar(\M)}[q, \sigma_s]$$

These are bijection of set of commutative squares.  It follows that $\sigma_s$ has the RLP with respect to $\I$ (resp. $\Ja$) if and only if $\Ub \sigma$ has the RLP with respect to $\I_{\kset}$ (resp. $\Ja_{\kset}$); finally $\Ub \sigma$ has the RLP with respect to $\I_{\kset}$ (resp.  $\Ja_{\kset}$) if and only if $\sigma$ has the RLP with respect to $\Gamma \I_{\kset}= \I_{\mset}$ (resp. $\Ja_{\mset}$).
\end{proof}

\begin{rmk}\label{cofibrant-dom}
Since the functor $h$ and $\delta$ are left Quillen funtors they preserve (trivial) cofibrations. In particular if $A$ is cofibrant then so are $h(\cn; A)$ and $\delta(s_n, A)$. It follows that if the domain of maps in $\I$ are cofibrant in $\M$ then so are the domain of maps in $\I_{\mset}$ and $\I_{\kset}$. 
\end{rmk}

\begin{cor}\label{mset-combinatorial}
The fibered model structure on $\mset$ (resp. $\kset$) is a combinatorial model category.
\end{cor}

\begin{proof}
Combine Proposition \ref{prop-gen-set}, Remark \ref{cofibrant-dom} and Theorem \ref{mset-loc-pres}. 
\end{proof}

\section{Co-Segalification\index{Segal!Co-Segalification} of $\S$-diagrams} \label{secion-cosegalification}

\paragraph*{\textbf{ Environment}:} In this section $(\M,\W)$ is a \textbf{symmetric monoidal model category} where $\W$ represents the class of weak equivalences. We refer the reader to \cite{Hov-model} for the definition of (symmetric) monoidal model categories.

For simplicity we consider in this section only $\S_{\ol{X}}$-diagrams of $(\M,\W)$. For a general category $\C$ the methods we will use will be the same. \\
\begin{nota}\ \\
A cofibration of $\M$ will be represented by an arrow of the form: $ \hookrightarrow$.\\
A fibration will be represented by:  $\twoheadrightarrow$\\
A weak equivalence will be represented by an arrow: $ \xrightarrow{\sim}$.\\
An isomorphism will be represented by: $ \xrightarrow{\cong}$.\\
$\aleph_0=$ the first countable cardinal.   $\aleph_0$ is identified with the ordinal $\omega=(\N,<)$. \\
$\kappa=$ a regular uncountable cardinal. \\
$\tx{End}[\msx]=$ the category of endofunctors of $\msx$.\\
$\bf{I}=$ the class of cofibrations of $\M$.\\
$\bf{I}\tx{-inj}=$ the class of $\bf{I}$-injective maps. \\
$\M^{[\1]}=\Hom([\1],\M)=$ the category of arrows of $\M$ (here $[\1]$ is the interval category).\\  
$\emptyset=$ the initial object of $\M$. 
\end{nota} 
\ \\
All along our discussion $X$ is a fixed nonempty $\kappa$-small set.\ \\
\ \\
The purpose of this section is to build a process which associates to any $\S_{\ol{X}}$-diagram $F$ a  Co-Segal  $\S_{\ol{X}}$-diagram $\Sim(F)$. This process will be needed in the upcoming sections when we will localize the previous model structures on the category $\msx$.\\

We are going to construct a functor $\Sim: \msx \to \msx$ equiped a natural transformation 
 $$\eta_{\Sim}: \Id_{\msx} \hookrightarrow \Sim$$ whose component at each $F$, $\eta_{\Sim,F}: F \hookrightarrow \Sim(F)$, will be a \emph{cofibration} in $\msx$.\ \\
 
The natural transformation $\eta_{\Sim}$ will arise automatically from the construction of the functor $\Sim$. \\

The functor $\Sim$ will be obtained as a colimit of a $\kappa$-sequence of cofibrations in $\msx$:

$$\Id_{\msx}=\Sim_0 \hookrightarrow  \Sim^1 \hookrightarrow \Sim^2 \cdots \hookrightarrow \Sim^{n-1} \hookrightarrow \Sim^{n} \hookrightarrow \cdots $$ 

\begin{comment}
\begin{df}\label{level-wise-cof}
Let $F$ and  $G$ be two objects of $\Hom[\S_{\ol{X}}(A,B)^{op},\M]$ and $\sigma : F \to G$ a natural transformation. \\
We say that $\sigma$ is a level-wise cofibration if for all $t \in \Ob[\S_{\ol{X}}(A,B)]$, $\sigma_t$  is a cofibration in $\M$ i.e $\sigma_t \in \I$.

\ \\
More generally if  $F=\{ F_{A B} \}_{(A,B) \in X^2}$ and $G=\{ G_{A B} \}_{(A,B) \in X^2}$ are in  $\prod_{(A,B) \in X^{2}} \Hom[\S_{\ol{X}}(A,B)^{op}, \M]$ and  $\sigma=\{\sigma_{AB}\}_{(A,B)\in X^2} \in \Hom(F,G)$, we say that  $\sigma$ is a level-wise cofibration if each $\sigma_{AB}$ is a level-wise cofibration in the previous sense.
\end{df}

\begin{df}\label{def-level-wise-cof-MX}
Let $F$ and  $G$ be two objects of $\msx$ and $\sigma \in \Hom_{\msx}(F,G)$.  
We say that $\sigma$ is a level wise cofibration if for all $t \in \Ob(\S_{\ol{X}})$, $\sigma_t$  is a cofibration in $\M$ i.e $\sigma_t \in \I$.
\end{df}
\end{comment}

\subsection{Co-Segalification by Generators and Relations} 
Recall that an $\S_{\ol{X}}$-diagram  $F$ is given by a family of functors $\{ F_{A B} \}_{(A,B) \in X^2}$ together with some laxity maps $\{ \varphi_{s,t} \}$ and  the suitable coherences. \\ 
Here each $F_{AB}$ is a classical functor $F_{AB} :  \S_{\ol{X}}(A,B)^{op} \to \M$, with  $\S_{\ol{X}}(A,B)$  a category over $\Delta_{epi}$.\ \\

Such $F$ is said to be a \emph{Co-Segal $\S_{\ol{X}}$-diagram} if for every pair $(A,B)$ and any morphism $u: t \to s$ of $\S_{\ol{X}}(A,B)$, the morphism  $F(u): F(s) \to F(t)$  is a \emph{weak equivalence} in $\M$. Following the Observation \ref{Co-Segal-final-cond} we know that it suffices to have these conditions for $u=u_t$ for all $t$, where $u_t :t \to (A,B)$ is the unique map from $t$ to $(A,B)$. \ \\

The functor $\Sim$ we are about to construct will have the property that $\Sim(F)(u_t)$ will be a \emph{trivial fibration} in $\M$ for all $t$. But since $\M$ is a model category
 \footnote{Here we adopt the modern language and simply say `model category' to mean what Quillen \cite[Ch.5]{Quillen_HA} called `closed model category'.} 
$\Sim(F)(u_t)$ is a trivial fibration if and only if  $\Sim(F)(u_t) \in \bf{I}\tx{-inj}$ i.e it has the \emph{right lifting property} (RLP) with respect to the
 class $\bf{I}$ of all cofibrations  (see \cite[Lemma 1.1.10]{Hov-model}, \cite[Ch.5]{Quillen_HA}). 
This lifting property amounts to say that whenever we have a commutative diagram in $\M$%

% So our strategy to modify $F$ into  $\Sim(F)$ is to make sure that whenever we have a commutative diagram in $\M$:
%mettre S(F) et h%
\[
\xy
(-10,20)*+{U}="A";
(20,20)*+{\Sim[F](A,B)}="B";
(-10,0)*+{V}="C";
(20,0)*+{\Sim[F](t)}="D";
{\ar@{->}^{f}"A";"B"};
{\ar@{_{(}->}_{h}"A"+(0,-3);"C"};
{\ar@{->}^{\Sim[F](u_t)}"B";"D"};
{\ar@{->}^{g}"C";"D"};
{\ar@{.>}^{k}"C";"B"};
\endxy
\]  

with $h \in \bf{I}$ then we can find a lifting i.e  there exists  $k: V \to \Sim[F](A,B)$ such that $k \circ h= f$ and $\Sim[F](u_t) \circ k= g$.  \ \\

If we consider separately in $\M$ the  map $F(u_t)$ the classical trick to produce $\Sim[F](u_t)$ is to use  the  \emph{small object argument} which gives, up-to some hypothesis on $\bf{I}$, a functorial factorization $F(u_t)= \beta_t(F) \circ  \alpha_t(F)$ with:

\begin{equation*}
  \begin{cases}
  \alpha_t(F): F(A,B) \to D  &  \tx{an $\bf{I}$-cell complex }\\
   \beta_t(F) : D \to F(t) & \tx{an element of $\bf{I}$-inj} \\
  \tx{ for some $D \in \Ob(\M)$}. \\
  \end{cases}
\end{equation*}
\\
The map $\alpha_t(F)$ is obtained as a transfinite composition of pushouts of a coproduct of the maps in $\I$. The smallness  or \emph{compacity} of $D$ is used to show that $\beta_t(F)$ has the RLP with respect to $\I$. The reader can find an exposition of the small object argument for example in \cite[Section 7.12]{Dwyer_Spalinski}, \cite[Theorem 2.1.14]{Hov-model}.\ \\

In this situation we can set $\Sim[F](A,B) = D$, $\Sim[F](t) = F(t)$, $\Sim[F](u_t) = \beta_t(F)$ and the natural transformation $\eta_{\Sim, F}$ will be given by $\alpha_t(F): F(A,B) \to \Sim[F](A,B)$ and $\Id_{F(t)}: F(t) \to \Sim[F](t)$. \ \\ 

In our case  we want to use the same trick i.e using a transfinite composition of pushout of maps of some class $\I_{\msx} \subset \Ar(\M)$, but we want these pushouts as well as the other operations to take place in $\msx$. \ \\ % that the maps   %we ignore for one moment that 
  %%% il faut continuer à expliquer la stratégie ///%The construction of the functor $\Sim_{1}$  is splited in to many pieces of functors $\{ \Sim_{1}(f,u) \}_{ f \in \bf{I} }$ for each $u: t \to (A,B)$. 
  
  %%%%%%%%%%%%%%%%%%

\subsubsection{\textbf{An important adjunction}}
Let $t$ be a $1$-morphism of $\S_{\ol{X}}$ of length $> 1$ i.e $t \in \Ob(\S_{\ol{X}}(A,B)^{\op}) $ for some  pair of elements $(A,B)$ of $X$. Recall that $t$ corresponds to a sequences $(A_0, A_1, ..., A_n)$ with $A_0= A$ and $A_n=B$. \ \\ 
\ \\
Let $\Pr_t: \M_{\S}(X) \to \M^{[\1]}$ be the evaluation functor at $u_t: (A,B) \to t$: 

\begin{itemize}
\item For $\Fa \in \M_{\S}(X)$ we have $\Pr_t(\Fa)=\Fa(u_t)$,
\item  For $\sigma \in \Hom_{\M_{\S}(X)}[\Fa,\Ga]$, we have $\Pr_t(\sigma)=(\sigma_{(A,B)},\sigma_t)$ which corresponds to the commutative square:
 \[
\xy
(-10,20)*+{\Fa(A,B)}="A";
(20,20)*+{\Ga(A,B)}="B";
(-10,0)*+{\Fa(t)}="C";
(20,0)*+{\Ga(t)}="D";
{\ar@{->}^{\sigma_{(A,B)}}"A";"B"};
{\ar@{->}_{\Fa(u_t)}"A";"C"};
{\ar@{->}^{\Ga(u_t)}"B";"D"};
{\ar@{->}^{\sigma_t}"C";"D"};
\endxy
\]
\end{itemize}

\begin{prop}
For every object $t$ of length $>1$ the  following holds.

\begin{enumerate}
\item The functor $\Pr_t$ has a left adjoint, that is there exists a functor 
$$\Pr_{t!}: \M^{[\1]} \to \M_{\S}(X)$$
 such that for every $\Fa \in \M_{S}(X)$ and every  $h \in \M^{[\1]}$  we have an isomorphism of sets:

$$\Hom_{\M_{\S}(X)}[\Pr_{t!}h, \Fa] \cong \Hom_{\M^{[\1]}}[h,\Fa(u_t)]$$
which is natural in both $h$ and $F$.
\item $\Pr_{t!}$ is a left Quillen functor. 
\end{enumerate}
\end{prop}

\begin{proof}[Sketch of proof]
For the assertion $(1)$, we write $\Pr_t$ as the composite of the following functors: 
$$\M_{\S}(X) \xrightarrow{\Ub} \prod_{(A',B') \in X^{2}} \Hom[\S_{\ol{X}}(A',B')^{op}, \M] \xrightarrow{\pr_{AB}} \Hom[\S_{\ol{X}}(A,B)^{op}, \M] \xrightarrow{\Ev_{u_t}} \M^{[\1]} $$
\\
where:
\begin{itemize}%[label=$-$, align=left, leftmargin=*, noitemsep]
\item $\Ub$ is the functor which forgets the laxity maps, 
\item $\pr_{AB}$ is the functor which gives the component at $(A,B)$,
\item $\Ev_{u_t}$ is the evaluation at $u_t$.
\end{itemize}

Thanks to lemma \ref{adjoint-ub} in the Appendix, $\Ub$ has left adjoint $\Gamma$.  $\Ev_{u_t}$ has a left adjoint  $\Fb^{u_t}$ (see Appendix \ref{lemme_adjunction_2}). Finally $\pr_{AB}$ has clearly a left adjoint $\delta_{AB}$ as explained below.  
The composite of these left adjoints gives a left adjoint of $\Pr_t$.  \ \\

The functor $\delta_{AB}$ is simply the `Dirac extension'.  For $\Fa \in \Hom[\S_{\ol{X}}(A,B)^{op}, \M]$ we define $\delta(\Fa) \in \K_X$ by 
 
 \begin{equation*}
\delta(\Fa)_{A'B'} =
  \begin{cases}
\Fa  & \text{if $(A',B')=(A,B)$} \\
    (\varnothing, \Id_{\varnothing})  & \text{the constant functor other wise}.
  \end{cases}
\end{equation*}

One can easily see that $\delta$ is a functor and  that we have indeed an isomorphism of sets:
$$\Hom[\Fa, \Ga_{AB} ] \cong \Hom[\delta(\Fa),\Ga]$$
which is natural in both $\Fa$ and $\Ga$; this completes the proof of the assertion $(1)$. \ \\

The assertion $(2)$ follows from the fact that all the three functors $\Gamma, \delta$ and $\Fb^{u_t}$ are left Quillen functors. In fact $\Gamma$ is a left Quillen functor by construction of the model structure on $\msx$ (injective or projective). $\delta$ is clearly a left Quillen functor, for $\Fb^{u_t}$ see Corollary \ref{adjoint-preserve-proj-cof} and Corollary \ref{lem-pres-level-wise-cof}.  It follows that $\Pr_{t!}$ is a composite of left Quillen functors therefore it's a left Quillen functor.  
\end{proof}
\begin{comment}
\begin{lem}
Let $\M$ be a symmetric  monoidal model category. \\

Then we have:
\begin{enumerate}
\item The functor  $\Ub$ (resp. $\pr_{AB}$, resp. $\Ev_t$) has a left adjoint $\Gamma$ (resp. $\delta$, resp. $\Psi_t$).
\item All the three functors  $\Gamma$,  $\delta$, and $\Psi_t$ are left Quillen where  $\msx$, $\kx$ and $\M^{[\1]}$ are endowed with the respective projective model structure.
\end{enumerate}
\end{lem}

\begin{rmk}
The three functors preserve the level-wise cofibrations. This will be relevant if we consider the injective model structure on the three categories.
\end{rmk}

\begin{proof}
For the construction of $\Gamma$  see . For the construction of $\Psi_t$ see Appendix \ref{lemme_adjunction_2}.\ \\

It's not hard to see that the functor $\delta$ preserves the projective (trivial) cofibrations. $\Gamma$ is a left Quillen functor by definition of the model structure on $\msx$ and thanks to the Lemma \ref{lem-pres-level-wise-cof} of the Appendix  \ref{preservation-cof-gamma-psi} we know that  $\Psi_t$ is also a left Quillen functor which complete the proof of the assertion $(2)$ .
\end{proof}
\end{comment}
\ \\
For any map $h:U \to V$ of $\M$, we have a tautological commutative square: 
\[
\xy
(0,20)*+{U}="A";
(20,20)*+{V}="B";
(0,0)*+{V}="C";
(20,0)*+{V}="D";
{\ar@{->}^{h}"A";"B"};
{\ar@{->}^{h}"A";"C"};
{\ar@{->}^{\Id_V}"B";"D"};
{\ar@{->}^{\Id_V}"C";"D"};
\endxy
\]  
which says that $(h,\Id_V)$ is, in a natural way, a morphism in $\M^{[\1]}$ from $h$ to $\Id_V$. We will denote by $h_{/V}$ this morphism.\footnote{The notation `$h_{/V}$' is inspired from the fact that the commutative square above is the (unique) canonical map from $h$ to $\Id_V$ in the slice category $\M_{/V}$. We recall that $\Id_V$ is final in $\M_{/V}$.}\ \\
%%%%%%%%%%%%%%%%%%%%%%%%%%%%%%%%%%%%%%%%%%%%%%%%%%%%%%%
%%%%%%%%%%%%%%%%%%%%%%%%%%%%%%%%%%%%%%%%%%%%%%%%%%%%%%%
\begin{comment}
If we consider the pushout of $h$ along itself (=the cokernel pair of $h$), then we have an induced map $\varepsilon: V \cup_U V \to V$. We will denote by $Z_h$ `the' object obtained (functorially) from the factorization cofibration-trivial fibration of $\varepsilon$:
$ \varepsilon: V \cup_U V \xhookrightarrow{i_h} Z_h \twoheadrightarrow V$. \ \\
We have an induced commutative square:
\[
\xy
(0,20)*+{U}="A";
(25,20)*+{V}="B";
(0,0)*+{V}="C";
(25,0)*+{Z_h}="D";
(14,10)*+{V \cup_U V}="E";
{\ar@{->}^{h}"A";"B"};
{\ar@{->}_{h}"A";"C"};
{\ar@{->}^{a_1}"B";"D"};
{\ar@{->}_{a_2}"C";"D"};
{\ar@{.>}_{}"C";"E"};
{\ar@{.>}_{}"B";"E"};
{\ar@{.>}^-{i_h}"E";"D"};
\endxy
\]  
This square represents a morphism $\xi_h: h \to a_1$ in $\M^{[\1]}$  which is a projective (trivial) cofibration if $h$ is so.  
\end{comment}
%%%%%%%%%%%%%%%%%%%%%%%%%%%%%%%%%%%%%%%%%%%%%%%%%%%%%%%%%%%%%%%%%%%%%%%%%%%%%%%%%%%%%%%%%%%%%%%%%%%%%%%%%%%%%%%%%%%%%%%%%%%%%%%%%%%%%%%%%%%%%%%%%%%%%%%%%%%%%%%%%%%%%
%%%%%%%%%%%%%%%%%%%%%%%%%%%%%%%%%%%%%%%%%%%%%%%%%%%%%%%
\begin{lem}\label{cof-rectification}
For  a symmetric monoidal model category $\M$ which is also tractable, for any pushout square in either $\msxinj$ or $\msxproj$: 
\[
\xy
(-10,20)*+{ \Pr_{t!}(h)}="A";
(20,20)*+{\Fa}="B";
(-10,0)*+{\Pr_{t!}\Id_V}="C";
(20,0)*+{\Ga}="D";
{\ar@{->}^{\sigma}"A";"B"};
{\ar@{_{(}->}_{\Pr_{t!}(h_{/V})}"A"+(0,-3);"C"};
{\ar@{-->}^{H}"B";"D"};
{\ar@{-->}^{\ol{\sigma}}"C";"D"};
\endxy
\]   

the following holds.
\begin{enumerate}%[ref=(\arabic{*})]
\item If $h:U \to V$ is a  cofibration in $\M$ then $H$ is a cofibration in  $\msxinj$. \label{cof-rec-1}
\item If moreover  $h:U \to V$ is a trivial cofibration in $\M$ then $H$ is a  weak equivalence in 
both $\msxinj$ and $\msxproj$. \label{cof-rec-2}
\end{enumerate}
\end{lem}

\begin{proof}
The map $h_{/V}$ is an injective (trivial) cofibration in $\M^{[\1]}$ and since  $\Pr_{t!}$ is a left Quillen functor, we know that  $\Pr_{t!}(h_{/V})$  is a (trivial) cofibration in $\msxinj$. Applying lemma  \ref{pushout-MX} we deduce that $H$ is a weak equivalence in $\msxinj$ but weak equivalences in $\msxinj$ and $\msxproj$ are the same.
\end{proof}

\subsubsection{The local `Co-Segalification'\index{Segal!Co-Segalification!local} process}
 Let $t$ be a fixed object in $\S_X(A,B)$ and   $\Fa$ be an object of  $\msx$.\ \\

As $\M$ is a model category we can factorize the map $\Fa(u_t)$ as: $ \Fa(u_t) = j \circ h$ where $h: \Fa(A,B) \hookrightarrow U$ is a cofibration and $j: U \twoheadrightarrow \Fa t$ is a trivial fibration. \ \\
The pair $(\Id_{\Fa(A,B)}, j)$ defines a morphism $S(j,h) \in \Hom_{\M^{[\1]}}[h, \Fa(u_t)]$ in a tautological way:
\[
\xy
(0,20)*+{\Fa(A,B)}="A";
(20,20)*+{\Fa(A,B)}="B";
(0,0)*+{U}="C";
(20,0)*+{\Fa t}="D";
{\ar@{=}^{ \Id }"A";"B"};
{\ar@{->}_{h}"A";"C"};
{\ar@{->}_{j}"C";"D"};
{\ar@{->}^{\Fa(u_t)}"B";"D"};
\endxy
\]
When necessary we will write $h=h(\Fa, t)$ and $j= j(\Fa,t)$ to mention that we working with the factorization of $\Fa (u_t)$. \ \\

By adjunction we have a unique map $T(h,j,\Fa,t) \in \Hom_{\msx}[\Pr_{t!}(h),\Fa]$ `lifiting' $S(j,h)$ .\ \\
%%%%%%%%%%%%%%%%%%%%%%%%%%%%%%%%%%%%%%%%%%%%%%%%
%%%%%%%%%%%%%%%%%%COMMENT STARTS%%%%%%%%%%%%%%%%%%%%%%%%%%%%%%
%%%%%%%%%%%%%%%%%%%%%%%%%%%%%%%%%%%%%%%%%%%%%%%%
\begin{comment}

The diagram below is a pushout square in $\M^{[\1]}$ and exhibits $j$ as the pushout-object: %%% rectifier un peu
\[
\xy
%%%%%%%%%%%face arriere%%%%%
(-10,20)*+{\Fa(A,B)}="A";
(20,20)*+{\Fa(A,B)}="B";
(-10,0)*+{U}="C";
(20,0)*+{\Fa t}="D";
{\ar@{->}^{ \Id }"A";"B"};
{\ar@{->}_{h}"A"+(0,-3);"C"};
{\ar@{->}^{\Fa(u_t)}"B";"D"};
{\ar@{.>}^{\ \ \ \ \ j}"C";"D"};
%%%%%%%%%%%%%%%%%%%%face avant%%%%
(-10,20)+(-15,-10)*+{U}="E";
(20,20)+(-15,-10)*+{U}="F";
(-10,0)+(-15,-10)*+{U}="G";
(20,0)+(-15,-10)*+{\Fa t}="H";
{\ar@{->}^{\ \ \ \ \ \Id}"E";"F"};
{\ar@{->}_{\Id}"E"+(0,-3);"G"};
{\ar@{->}_{j}"F";"H"};
{\ar@{->}^{j}"G";"H"};
%%%%%%%%%%% fleches sortantes arrieres vers avant %%%%%%%
{\ar@{->}_{h}"A";"E"};
{\ar@{->}_{h}"B";"F"};
{\ar@{->}^{ \Id }"C";"G"};
{\ar@{->}^{ \Id }"D";"H"};
\endxy
\]  
\end{comment}
%%%%%%%%%%%%%%%%%%%%%%%%%%%%%%%%%%%%%%%%%%%%%%%%
%%%%%%%%%%%%%%%%%%COMMENT ENDS %%%%%%%%%%%%%%%%%%%%%%%%%%%%%%
%%%%%%%%%%%%%%%%%%%%%%%%%%%%%%%%%%%%%%%%%%%%%%%%
Define the \emph{Gluing Construction} $\Sim_t^1(\Fa)$ to be the object of $\msx$ given by the pushout  diagram
\[
\xy
(-20,20)*+{\Pr_{t!}(h)}="A";
(20,20)*+{\Fa}="B";
(-20,0)*+{\Pr_{t!}\Id_U}="C";
(20,0)*+{\Sim_t^1(\Fa)}="D";
{\ar@{->}^{T(h,j,\Fa,t) }"A";"B"};
{\ar@{_{(}->}_{\Pr_{t!}(h_{/U})}"A"+(0,-3);"C"};
{\ar@{-->}^{H_1}"B";"D"};
{\ar@{-->}^{\alpha}"C";"D"};
\endxy
\]   

\begin{prop}
With the above notations the following holds.
\begin{enumerate}
\item For every such factorization $(h,j)\in (\cof, \we \cap \fib)$ of $\Fa(u_t)$  the map $H_1: \Fa \to \Sim_t^1(\Fa)$ is an injective cofibration in $\msx$. 
\item If $\Fa(u_t)$ is a weak equivalence in $\M$, then $H_1$ is an injective trivial cofibration. In particular $H_1$ is a weak equivalence in both $\msxinj$ and $\msxproj$ and the map $[\Sim_t^1\Fa]u_t$ is a weak equivalence in $\M$. 
\item If the factorization axioms in $\M$ is functorial then the operation $\Fa \mapsto \Sim_t^1\Fa$ is a functor. 
\end{enumerate}
\end{prop}

\begin{proof}[Sketch of proof]
As $h$ is a cofibration in $\M$, the assertion $(1)$ follows immediately from the lemma \ref{cof-rectification} \ref{cof-rec-1}.\ \\

If $\Fa(u_t)= j \circ h$ is a weak equivalence and as $j$ is a weak equivalence by hypothesis,  then by the $3$-out-of-$2$ property we deduce that $h$ is also a weak equivalence; therefore $h$ is a trivial cofibration and the half of assertion $(2)$ follows also from the lemma \ref{cof-rectification} \ref{cof-rec-2}. \ \\

By definition of map in $\msx$, we know that the pair $(H_{1,AB}, H_{1,t})$ defines a map in $\M^{[\1]}$ from $\Fa (u_t)$ to $[\Sim_t^1\Fa](u_t)$. In particular we have an equality:
$$[\Sim_t^1\Fa]u_t \circ H_{1,AB}= H_{1,t} \circ \Fa(u_t).$$ 
 
Since $H_1$ is a weak equivalence, then both $H_{1,AB}$, $H_{1,t}$  are weak equivalences in $\M$; it follows that $H_{1,t} \circ \Fa(u_t)$ is a weak equivalence if $\Fa(u_t)$ is so. Now by the $3$-out-of-$2$ property we deduce that $[\Sim_t^1\Fa](u_t)$ is also a weak equivalence. This complete the proof of $(2)$. 
 \ \\

The assertion $(3)$ is clear and is left to the reader.
\end{proof}

\begin{rmk}\label{intermed-trivial-fib}
By adjoint transpose we have the following commutative square in $\M^{[\1]}$:
\[
\xy
(-20,20)*+{h}="A";
(20,20)*+{\Fa(u_t)}="B";
(-20,0)*+{\Id_U}="C";
(20,0)*+{\Sim_t^1(\Fa)(u_t)}="D";
{\ar@{->}^{S(j,h)}"A";"B"};
{\ar@{_{(}->}_{h_{/U}}"A"+(0,-3);"C"};
{\ar@{-->}^{\Pr_{t}(H_1)}"B";"D"};
{\ar |{ \  \ol{\alpha} \  }"C";"D"};
\endxy
\]
To simplify the notations in diagrams we will write $\Fa^1$ for $\Sim_t^1(\Fa)$. The above diagram is displayed as a commutative cube  in $\M$ 
 \[
\xy
%%%%%%%%%%%face arriere%%%%%
(-10,20)*+{\Fa(A,B)}="A";
(20,20)+(10,0)*+{\Fa(A,B)}="B";
(-10,0)*+{U}="C";
(20,0)+(10,0)*+{\Fa t}="D";
{\ar@{->}^{ \Id }"A";"B"};
{\ar@{.>}^{}"A"+(0,-3);"C"};
{\ar@{->}^{\Fa(u_t)}"B";"D"};
{\ar@{.>}^{\ \ \ \ \ j}"C";"D"};
%%%%%%%%%%%%%%%%%%%%face avant%%%%
(-10,20)+(-25,-15)*+{U}="E";
(20,20)+(-25,-15)+(10,0)*+{\Fa^1(A,B)}="F";
(-10,0)+(-25,-15)*+{U}="G";
(20,0)+(-25,-15)+(10,0)*+{\Fa^1 t}="H";
{\ar@{->}^{\ol{\alpha}_{AB}}"E";"F"};
{\ar@{->}_{\Id}"E"+(0,-3);"G"};
{\ar@{->}^{\Fa^1 u_t}"F";"H"};
{\ar@{->}^{\ol{\alpha}_t}"G";"H"};
%%%%%%%%%%% fleches sortantes arrieres vers avant %%%%%%%
{\ar@{->}_{h}"A";"E"};
{\ar@{->}_{H_{1,AB}}"B";"F"};
{\ar@{->}_{ \Id }"C";"G"};
{\ar@{->}^{H_{1,t} }"D";"H"};
\endxy
\]  
 
 From the upper and bottom faces of that cube we deduce that $H_{1,AB}=  \ol{\alpha}_{AB} \circ h$ and $\ol{\alpha}_{t}=  H_{1,t} \circ j$; from the front face we have that $ \ol{\alpha}_t= \Fa^1 u_t \circ \ol{\alpha}_{AB} $. If we put these together we see that  in the diagram below everything is commutative (triangles and squares):
% $\Pr_t(H_1)$ 
\[
\xy
(-25,20)*+{\Fa(A,B)}="E";
(20,20)*+{\Fa^1(A,B)}="F";
(-25,-5)*+{\Fa t}="G";
(20,-5)*+{\Fa^1 t}="H";
(-25,20)+(10,-9)*+{U}="D";
{\ar@{->}^{H_{1,AB}}"E";"F"};
{\ar@{->}_{\Fa u_t}"E";"G"};
{\ar@{->}^{\Fa^1 u_t}"F";"H"};
{\ar@{->}^{H_{1,t}}"G";"H"};
%%%%%%%%%%%%%%%%%%%
{\ar@{^{(}->}_{h}"E"+(3,-4);"D"};
{\ar@{->>}^{j}_{\wr}"D";"G"};
{\ar@{->}^{\ol{\alpha}_{AB}}"D";"F"};
\endxy
\]

\end{rmk}

\begin{comment}

  define $S(h,\Fa,t )=\Hom_{\M^{[\1]}}(h, \Fa(u_t))$  the set of morphisms in  $\M^{[\1]}$ from $h$ to $\Fa(u_t)$. By adjunction $S(h,\Fa,t )$ is in one-one correspondance with the set $T(h,\Fa,t)=\Hom_{\msx}[\Pr_{t!}(h),\Fa]$.  %by the previous lemma we can get a universal object $\Ga$  together with a map $H: \Fa \to \Ga$ such that $\Ga$ has the property 
\ \\
\ \\

By construction for each $h$ and each $\sigma \in T(h,\Fa,t)$ there is a horn filling  \begin{tabular}{c}
\xy
(0,20)*+{\Pr_{t!}(h)}="A";
(20,20)*+{\Sim_t^1(\Fa)}="B";
(0,0)*+{\Pr_{t!}\Id_V}="C";
{\ar@{->}^{ H_1 \circ \sigma }"A";"B"};
{\ar@{->}_{\Pr_{t!}(h_{/V})}"A";"C"};
{\ar@{.>}_{\sigma'}"C";"B"};
\endxy
\end{tabular}
\ \\
\ \\
\end{comment}

\begin{warn}
For the rest of the discussion we assume that the factorization axioms in $\M$ is functorial. 
\end{warn}

For $k>1$ we define inductively objects $\Sim_t^k(\Fa)$ of $\msx$ by setting: $ \Sim_t^k(\Fa):= \Sim_t^1[\Sim_t^{k-1}(\Fa)]$ with $\Sim_t^0(\Fa)= \Fa$.  One uses a (functorial) factorization $(h_k, j_k)\in (\cof, \we \cap \fib)$ of the map  $\Sim_t^{k-1}(\Fa)u_t$ and apply the previous construction.  \ \\

We have a canonical map $H_k: \Sim_t^{k-1}(\Fa) \to \Sim_t^{k}(\Fa)$ which is a cofibration in $\msxinj$. \ \\
We have a $\kappa$-sequence in $\msxinj$:
$$ \Fa =\Sim_t^0(\Fa)  \xhookrightarrow{H_1}  \Sim_t^1(\Fa) \hookrightarrow \cdots \hookrightarrow \Sim_t^{k-1}(\Fa) \xhookrightarrow{H_{k}} \Sim_t^k(\Fa)  \hookrightarrow \cdots$$
\ \\
Define $\Sim_t^{\infty}(\Fa)= \colim_k \Sim_t^k(\Fa)$ and denote by  $\eta_t: \Fa \to \Sim_t^{\infty}(\Fa)$ the canonical map.

\begin{prop}\label{coseg-t-1}
For every $\Fa \in \msx$ then:
\begin{enumerate}
\item The map $\Sim_t^{\infty}(\Fa)(u_t) $ has the RLP with respect to all cofibrations in $\M$ i.e  it's a trivial fibration in $\M$
\item The map $\eta_t$ is a cofibration in $\msxinj$.  
\item If $\Fa(u_t)$ is a weak equivalencein $\M$, then $\eta_t$ is trivial cofibration in $\msxinj$, in particular a weak equivalence in $\msx$. \label{coseg-1-3} 
\end{enumerate}

\end{prop}

\begin{proof}

For notational convenience we will write in this proof $\Fa^k= \Sim_t^k(\Fa)$ and  $\fainf=\Sim_t^{\infty}(\Fa)$.\ \\

The assertion $(2)$ and $(3)$ are  straightforward: if $\Fa (u_t)$ is a weak equivalence then applying inductively the lemma \ref{cof-rectification}, we get that all $H_k$ are  either cofibrations in case $(2)$ or trivial cofibration in case $(3)$. In both cases $\eta_t$ is a transfinite composition of such morphisms so it's also either a cofibration or a trivial cofibration in $\msxinj$. \ \\

To prove the assertion $(1)$ we use the small object argument in the locally presentable category $\msx$. Choosing $\kappa$ big enough we can assume that $\Pr_{t!}(g)$ is small for every cofibration $g$ of $\M$.\ \\ 

Let $g$ be a cofibration in $\M$ and consider a lifting problem defined by $g$ and $\fainf(u_t)$:
 \[
\xy
(-10,20)*+{P}="A";
(20,20)*+{\fainf(A,B)}="B";
(-10,0)*+{Q}="C";
(20,0)*+{\fainf t}="D";
{\ar@{->}^{f}"A";"B"};
{\ar@{_{(}->}_{g}"A"+(0,-3);"C"};
{\ar@{->}^{\fainf u_t}"B";"D"};
{\ar@{->}^{l}"C";"D"};
\endxy
\] 
\ \\
Such a lifting problem is equivalent to a morphism $\theta \in \Hom_{\M^{[\1]}}[g,\fainf(u_t)]$. By adjunction $\theta$ corresponds to a unique morphism $\tld{\theta} \in \Hom_{\msx}[\Pr_{t!}g,\fainf]$. 
Since $\Pr_{t!}g$ is $\kappa$-small, $\tld{\theta}$ factorizes through one of the $\Fa^{k}$, say $\Fa^{k_0}$:  there is  a map 
 $\tld{\theta}_0: \Pr_{t!}(h) \to \Fa^{k_0}$ such that $ \tld{\theta}=\iota_{k_0} \circ \tld{\theta}_0 $, where $\iota_{k_0}: \Fa^{k_0} \to \fainf$ is the canonical map.\ \\
 
By construction  $\iota_{k_0} = \iota_{k_0 +1} \circ H_{k_0}$ and from the  adjunction we have an equality:
 $$ \theta=  \Pr_t (\iota_{k_0}) \circ  \theta_0=  \Pr_t (\iota_{k_0 +1}) \circ  \Pr_t (H_{k_0}) \circ  \theta_0$$
 where $\theta_0=(f_0, l_0)$ is the adjoint transpose of $\tld{\theta}_0$. 
\ \\ 
 It follows that  our original lifting problem can be factorized as:
 
\begin{tabular}{cc}
\begin{tabular}{c}
\xy
(-10,20)*+{P}="A";
(20,20)*+{\fainf(A,B)}="B";
(-10,0)*+{Q}="C";
(20,0)*+{\fainf t}="D";
{\ar@{->}^{f}"A";"B"};
{\ar@{_{(}->}_{g}"A"+(0,-3);"C"};
{\ar@{->}_{\fainf u_t}"B";"D"};
{\ar@{->}^{l}"C";"D"};
\endxy
\end{tabular}
=
\begin{tabular}{c}
\xy
(-10,20)*+{P}="A";
(10,20)*+{\Fa^{k_0}(A,B)}="B";
(-10,0)*+{Q}="C";
(10,0)*+{\Fa^{k_0} t}="D";
(20,20)+(10,-6)*+{U_{k_0}}="U";
{\ar@{->}^{f_0}"A";"B"};
{\ar@{_{(}->}_{g}"A"+(0,-3);"C"};
{\ar@{->}_{\Fa^{k_0} u_t}"B";"D"};
{\ar@{->}_{l_0}"C";"D"};
%%%%% extending the square
(20,20)+(30,0)*+{\Fa^{k_0 +1}(A,B)}="F";
(20,0)+(30,0)*+{\Fa^{k_0 +1} t}="H";
{\ar@{->}^{\Fa^{k_0 +1} u_t}"F";"H"};
{\ar@{->}^{H_{k_0,AB} }"B";"F"};
{\ar@{->}^{H_{k_0,t}}"D";"H"};
%%%%%% extending square %%%
(20,20)+(30,0)+(30,0)*+{\fainf(A,B)}="J";
(20,0)+(30,0)+(30,0)*+{\fainf t}="L";
{\ar@{->}^{\fainf u_t}"J";"L"};
{\ar@{->}^{\iota}"F";"J"};
{\ar@{->}^{\iota}"H";"L"};
%%%%%%%%%% connecting U
{\ar@{->}_{h_{k_0}}"B";"U"};
{\ar@{->>}^{j_{k_0}}"U";"D"};
%%%%%%%
{\ar@{->}_{\ol{\alpha}_{k_0}}"U";"F"};
%%%%% lifting solution %%%
{\ar@{-->}^{}"C";"U"};
\endxy
\end{tabular}

\end{tabular}

The pair $(h_{k_0},j_{k_0})$ is `the' factorization cofibration-trivial fibration used to construct $H_{k_0}$ and  $\ol{\alpha}_{k_0}$ is the obvious map (see Remark \ref{intermed-trivial-fib}). As $j_{k_0}$ is a trivial fibration, the induced lifting problem by $g$ and $j_{k_0}$ has a solution:  there is a map $\beta_0: Q \to U$ satisfying the obvious equalities. We leave the reader to check that the composite 
$$\iota \circ \ol{\alpha}_{k_0} \circ \beta_0: Q \to \fainf(A,B)$$
is a solution to the original lifting problem. \ \\

It follows that $\fainf(u_t)$ has the RLP with respect to all cofibrations of $\M$, thus it's a trivial fibration as desired. 

\begin{comment}

\[
\xy
(0,20)*+{\Pr_{t!}(h)}="A";
(50,20)*+{\Sim_t^{k_0+2}(\Fa)}="B";
(0,0)*+{\Pr_{t!}\Id_V}="C";
{\ar@{->}^{H_{k_0 +2} \circ H_{k_0 +1} \circ \sigma_0}"A";"B"};
{\ar@{->}_{\Pr_{t!}(h_{/V})}"A";"C"};
{\ar@{.>}_{\beta_0}"C";"B"};
\endxy
\]  

Clearly the composite $\iota_{k_0+2} \circ \beta_0$ gives a horn filling in the orginal diagram since $\iota_{k_0} =  \iota_{k_0 +2} \circ  H_{k_0 +2} \circ H_{k_0 +1}$:

\[
\xy
(0,20)*+{\Pr_{t!}(h)}="A";
(50,20)*+{\Sim_t^{k_0+2}(\Fa)}="B";
(0,0)*+{\Pr_{t!}\Id_V}="C";
(75,20)*+{\Sim_t^{\infty}(\Fa)}="D";
{\ar@{->}^{H_{k_0 +2} \circ H_{k_0 +1} \circ \sigma_0}"A";"B"};
{\ar@{->}_{\Pr_{t!}(h_{/V})}"A";"C"};
{\ar@{.>}^{\beta_0}"C";"B"};
{\ar@{.>}_{\iota_{k_0+2} \circ \beta_0}"C";"D"};
{\ar@{->}^{\iota_{k_0+2}}"B";"D"};

\endxy
\]  

\begin{equation*}
\begin{split}
\iota_{k_0+2} \circ \beta_0 \circ \Pr_{t!}(h_{/V}) 
&=\iota_{k_0+2} \circ (H_{k_0 +2} \circ H_{k_0 +1} \circ \sigma_0)  \\
 &= (\iota_{k_0+2} \circ H_{k_0 +2} \circ H_{k_0 +1} )\circ \sigma_0 \\
 &= \iota_{k_0}\circ \sigma_0 \\
&=\sigma \\
\end{split}
\end{equation*}
\end{comment} %%%% 
\end{proof}
\paragraph{$\Sim_t^{\infty}$ is homotopically minimal}
Let $\homsx$ be the homotopy category associated to both $\msxinj$ and $\msxproj$. Given a map $\sigma : \Fa \to \Ga$ in $\msx$, we will denote by  $[\sigma]$ the classe of $\sigma$ in $\homsx$.\ \\

Denote by $\Rc \subset \msx$ the subcategory consisting of co-Segal categories; these are object $\Fa$ such that for  every $t$, $\Fa (u_t)$ is a weak equivalence. For a fixed $t$ denote by $\Rc_t \subset \msx$ the subcategory of object $\Fa$ such that $\Fa(u_t)$ is a weak equivalence.  We have $\Rc \subset \Rc_t \subset \msx$.\ \\

Thanks to Proposition \ref{coseg-t-1}, for any $\Fa \in \msx$ we have $\Sim_t^{\infty} \Fa \in \Rc_t$. In what follows we show that among  all objects of $\Rc_t$,  $\Sim_t^{\infty} \Fa$ is the  `homotopic-nearest object' to $\Fa$. \ \\

\begin{df}
Let $\Fa$ be an object of $\msx$ and $\Ga$ be an object of $\Rc_t$.  A map $\sigma_0 :\Fa \to \Ga$ is \emph{homotopically minimal} with respect  to $\Rc_t$ if
for any $\Qa \in \Rc_t$ and any morphism $\sigma: \Fa \to \Qa$ there exist a morphism  $\gamma: [\Ga] \to [\Qa]$ in $\homsx$ such that $[\sigma]= \gamma \circ[\sigma_0]$.\ \\

Diagrammatically this is displayed in $\homsx$ as 
\[
\xy
(0,20)*+{[\Fa]}="A";
(30,20)*+{[\Qa]}="B";
(0,0)*+{[\Ga]}="C";
{\ar@{->}^{[\sigma]}"A";"B"};
{\ar@{->}_{[\sigma_0]}"A";"C"};
{\ar@{.>}_{\gamma}^{}"C";"B"};
\endxy
\]  
\end{df} 
 
\begin{prop}
 For every $\Fa \in \msx$ the map $\eta_t: \Fa \to \Sim_t^{\infty} \Fa$ is homotopically minimal with respect to $\Rc_t$. 
 \end{prop}
 
 \begin{proof}
 For a map $\sigma: \Fa \to \Qa$ with $\Qa \in \Rc_t$, by functoriality we have an induced map 
 $$\Sim_t^{\infty}(\sigma): \Sim_t^{\infty} \Fa \to \Sim_t^{\infty} \Qa$$ and the following commutes:
 \[
\xy
(0,20)*+{[\Fa]}="A";
(30,20)*+{[\Qa]}="B";
(0,0)*+{[\Sim_t^{\infty} \Fa]}="C";
(30,0)*+{[\Sim_t^{\infty} \Qa]}="D";
{\ar@{->}^{\sigma}"A";"B"};
{\ar@{->}_{(\eta_t)_{\Fa}}"A";"C"};
{\ar@{.>}_{\Sim_t^{\infty} \sigma}"C";"D"};
{\ar@{->}_{(\eta_t)_{\Qa}}"B";"D"};
\endxy
\]  
 
Note that the map  $\Sim_t^{\infty}\sigma$ is induced by universal property of the pushout (inductively), so it's a universal morphism. Since $\Qa \in \Rc_t$ we have from the Proposition \ref{coseg-t-1} (3) that $(\eta_t)_{\Qa}$ is a weak equivalence in $\msx$, thus $[(\eta_t)_{\Qa}]$ is an isomorphism in $\homsx$. Take $\gamma= [(\eta_t)_{\Qa}]^{-1} \circ [\Sim_t^{\infty} \sigma]$.
 \end{proof}

\subsubsection{The global Co-Segalification\index{Segal!Co-Segalification!global} process}

In what follows we use the previous functors $ \Sim_t^{\infty}$ to construct the desired functor $\Sim$ such that for any $t$ and any $\Fa$, $\Sim(\Fa)u_t$ is a weak equivalence, that is   $\Sim(\Fa)$ is an object of $\Rc$. \ \\
\ \\
Denote by $\morsx$ the set of all $1$-morphisms $t$ of degree $>1$ in $\S_{\ol{X}}$. \ \\
Define the general gluing construction to be the object $\Sim^1 \Fa$  obtained from the generalized pushout diagram formed by all the morphisms $\eta_t: \Fa \to \Sim_t^{\infty} \Fa$ as $t$ runs through the set of all $1$-morphisms of degree $>1$:
\[
\xy
(-10,20)*+{\Fa}="A";
(25,20)*+{\Sim_t^{\infty} \Fa}="B";
(-10,0)*+{\Sim_{t'}^{\infty} \Fa}="C";
(20,0)*+{\Sim_{t''}^{\infty} \Fa}="D";
{\ar@{^{(}->}^{}"A"+(3,0);"B"};
{\ar@{_{(}->}_{}"A"+(0,-3);"C"};
{\ar@{^{(}->}^{}"A"+(2,-2);"D"};
\endxy
\]

$$ \Sim^1 \Fa:= \colim_{t, \degb(t)>1 } \{ \eta_t: \Fa \to \Sim_t^{\infty} \Fa \} .$$

Let $\eta^1: \Fa \to \Sim^1 \Fa$ and $\iota^1_t : \Sim_t^{\infty} \Fa \to \Sim^1 \Fa$ be the canonical maps. It follows that for all $t$ we have $\eta^1= \iota^1_t  \circ \eta_t$. \ \\
\begin{comment}
For each $t$, we end up with an endofunctor $\Sim_t^{\infty} : \msx \to \msx$  together with a natural transformation $\eta_t: \Id_{\msx} \to \Sim_t^{\infty}$ which is an object-wise cofibration in $\msxinj$. \ \\
 With all the functors $\Sim_t^{\infty}$  where $t$ runs the (small) set of morphisms in $\S_X$.\ \\
 \end{comment}
Like every $\Sim_t^{\infty}$, $\Sim^1$ is functorial in $\Fa$.

\begin{rmk}
We leave the reader to check that we have the following properties. 
\begin{enumerate}
\item For every $\Fa$ the map $\eta^1$ is a cofibration in $\msxinj$ (see lemma \ref{cone-cofib}).
\item If $\Fa \in \Rc$ then  $\Sim^1 \Fa \in \Rc$. Equivalently $\Sim^1$ induces an endofunctor  on $\Rc$. Morever  $\eta^1$ is a trivial cofibration in $\msxinj$; in particular a weak equivalence in $\msx$. 
\end{enumerate}
\end{rmk}
Define inductively  have a sequence of functors $\Sim^{k}$  by $\Sim^k (\Fa):=  \Sim^1[\Sim^{k-1}(\Fa)]$ :
$$ \Fa =\Sim^0(\Fa)  \xhookrightarrow{\eta_1}  \Sim^1(\Fa) \hookrightarrow \cdots \hookrightarrow \Sim^{k-1}(\Fa) \xhookrightarrow{\eta_{k}} \Sim_t^k(\Fa)  \hookrightarrow \cdots$$

Set  $\Sim(\Fa):= \colim_k \Sim^k(\Fa)$; denote by $\eta_{\Fa}: \Fa \to \Sim(\Fa)$ the canonical map. 
\begin{prop}\label{global-cosegalification}
For every $\Fa \in \msx$, the following holds.
\begin{enumerate}
\item For all $t$ $\Sim(\Fa)(u_t)$  is a trivial fibration in $\M$, in particular a weak equivalence, thus  $\Sim(\Fa) \in \Rc$ i.e satisfies the Co-Segal conditions.
\item The canonical map $\eta: \Fa \to \Sim(\Fa)$ is a cofibration in $\msxinj$.
\item If $\Fa \in \Rc$ then $\eta: \Fa \to \Sim(\Fa)$  is a trivial cofibration in $\msxinj$, in particular a weak equivalence in $\msx$.
\end{enumerate} 
\end{prop}

\begin{proof}[Sketch of proof]
The assertion $(2)$ and $(3)$ are clear and are left to the reader. \ \\
To prove $(1)$ one proceeds exactly in same way as in the proof of Proposition \ref{coseg-t-1}. We will adopt the notations $\Fa^k=\Sim^{k} (\Fa)$ for simplicity.\ \\
\ \\
For any cofibration $g$ of $\M$, using suitably the small object argument and the adjunction $\Pr_{t!} \dashv \Ev_{u_t}$, any lifting problem defined by $g$ and $(\Sim \Fa) u_t$ can be factorized, for some $k_0$, as:

\begin{tabular}{c}
\begin{tabular}{c}
\xy
(-10,20)*+{P}="A";
(10,20)*+{(\Sim \Fa)(A,B)}="B";
(-10,0)+(0,-10)*+{Q}="C";
(10,0)+(0,-10)*+{(\Sim \Fa) t}="D";
{\ar@{->}^-{f}"A";"B"};
{\ar@{_{(}->}_-{g}"A"+(0,-3);"C"};
{\ar@{->}_-{(\Sim \Fa) u_t}"B";"D"};
{\ar@{->}^-{l}"C";"D"};
\endxy
\end{tabular}
=
\begin{tabular}{c}
\xy
(-10,20)+(-3,0)*+{P}="A";
(10,20)+(-3,0)*+{\Fa^{k_0}(A,B)}="B";
(-10,0)+(0,-10)+(-3,0)*+{Q}="C";
(10,0)+(0,-10)+(-3,0)*+{\Fa^{k_0} t}="D";
(30,20)+(0,-7)*+{(\Sim_t^{\infty} \Fa^{k_0})(A,B)}="U";
(30,0)+(0,7)+(0,-10)*+{(\Sim_t^{\infty} \Fa^{k_0})t}="V";
{\ar@{->}^-{f_0}"A";"B"};
{\ar@{_{(}->}_{g}"A"+(0,-3);"C"};
{\ar@{->}_-{\Fa^{k_0} u_t}"B";"D"};
{\ar@{->}_-{l_0}"C";"D"};
{\ar@{->}_{}"B";"U"};
{\ar@{->}_{}"D";"V"};
%%%%% extending the square
(20,20)+(30,0)+(5,0)*+{\Fa^{k_0 +1}(A,B)}="F";
(20,0)+(30,0)+(0,-10)+(5,0)*+{\Fa^{k_0 +1} t}="H";
{\ar@{->}^{\Fa^{k_0 +1} u_t}"F";"H"};
{\ar@{->}^{}"B";"F"};
{\ar@{->}^{}"D";"H"};
%%%%%% extending square %%%
(20,20)+(30,0)+(30,0)+(5,0)*+{(\Sim \Fa)(A,B)}="J";
(20,0)+(30,0)+(30,0)+(0,-10)+(5,0)*+{(\Sim \Fa) t}="L";
{\ar@{->}^-{(\Sim \Fa) u_t}"J";"L"};
{\ar@{->}^-{\tx{can}}"F";"J"};
{\ar@{->}^-{\tx{can}}"H";"L"};
%%%%%%%%%% connecting U
{\ar@{->>}^-{(\Sim_t^{\infty}\Fa^{k_0})u_t}_{\wr}"U";"V"};
{\ar@{->}^{\iota_t}"U";"F"};
%%%%%%%
{\ar@{->}_{\iota_t}"V";"H"};
%%%%% lifting solution %%%
{\ar@{-->}^{}"C";"U"};
\endxy
\end{tabular}

\end{tabular}
\ \\
In the above diagram, everything  is commutative (squares and triangles), and since $(\Sim_t^{\infty}\Fa^{k_0})u_t$ has the RLP with respect to all cofibration (Proposition \ref{coseg-t-1}(1)) there is a solution  $\beta: Q \to (\Sim_t^{\infty}\Fa^{k_0})(A,B)$ to the lifting problem induced by $g$ and $(\Sim_t^{\infty}\Fa^{k_0})u_t$. Clearly the composite 
 $$\tx{can} \circ \iota^{k_0}_t \circ \beta: Q \to \Sim(\Fa)(A,B) $$
 is a solution to the original lifting problem. Here of course `can' is the canonical map going to the colimit. \ \\
 
 Consequently  $(\Sim \Fa) u_t$ has the RLP with respect to any cofibration $g$ in $\M$, thus it's a trivial fibration as desired. 
\end{proof}

\begin{note}
 Since weak equivalences in $\msxinj$ and $\msxproj$ are the same, if we choose a functorial factorization in $\msxproj$ of the map $\eta_{\Fa}$ as:
$$\eta_{\Fa}: \Fa  \xhookrightarrow{\tld{\eta}_{\Fa}} \Za \xtwoheadrightarrow[\sim]{q}  \Sim(\Fa)$$
  
 where $\tld{\eta}_{\Fa}$ is cofibration and $q$ is a trivial fibration, then we can set $\Sim(\Fa):= \Za$ when working in the projective model structure. This new functor has the same properties as the previous one. 
 \end{note}
\subsection{Localization by weak monadic projections}
%We want to use the functor   $\Sim: \msx \to \msx$ constructed previously to localize the model structure 
\subsubsection{Weak monadic projection}
Let $\Mb$ be a model category and $\Ra \subset \Mb$ be a subcategory stable under weak equivalences. 
We recall very briefly the definition of weak monadic projection as stated in \cite[9.2.2]{Simpson_HTHC}. \ \\

A \emph{weak monadic projection} from $\Mb$ to $\Ra$ is a functor $F: \Mb \to \Mb$  together with a natural transformation 
$\eta_A: A \to F(A)$ such that:

\begin{enumerate}%[label=(WPr\arabic*), align=left, leftmargin=*, noitemsep]
\item $F(A) \in \Ra$ for all $A \in \Mb$;
\item for any $A \in \Ra$, $\eta_A$ is a weak equivalence;
\item for any $A \in \Mb$  the map $F(\eta_A): F(A) \to F(F(A))$ is a weak equivalence;
\item If $f : A \to B$ is a weak equivalence between cofibrant objects then $F(f): F(A) \to F(B)$ is a weak equivalence; and 
\item $F(A)$ is cofibrant for any cofibrant $A \in \Mb $.  
\end{enumerate}

\begin{rmk}\label{mpr-inj-proj}
If $F$ is a monadic projection from $\msxinj$ to $\Ra$ then we can extract a monadic projection $\tld{F}$ from $\msxproj$ to $\Ra$. In fact one uses the (functorial) factorization in $\msxproj$: 
$$\eta_A: A \xhookrightarrow{\tld{\eta}_A} \tld{F}(A)  \xtwoheadrightarrow[\sim]{p_A}F(A)$$
where $\tld{\eta}_A$ is a projective cofibration and $p_A$ a trivial fibration. 
\ \\
\ \\
(WPr1) holds because $p: \tld{F}(A)\to F(A) $ is a weak equivalence and $\Ra$ is stable by weak equivalence;\\
(WPr2) follows by the $3$-out-of-$2$ property of weak equivalences:  $p_A$ is already a weak equivalence, consequently if in addition $\eta_A$ is a weak equivalence then $\tld{\eta}_A$ is also a weak equivalence ;\\
(WPr3) also follows from the $3$-out-of-$2$ property:  from the functoriality of the factorization in $\msxproj$ on has that the following commutes:
 \[
\xy
(0,20)*+{\tld{F}(A)}="A";
(30,20)*+{\tld{F}(\tld{F}(A))}="B";
(0,0)*+{F(A)}="C";
(30,0)*+{F(F(A))}="D";
{\ar@{->}^{\tld{F}(\tld{\eta}_A)}"A";"B"};
{\ar@{->>}_-{p}^-{\wr}"A";"C"};
{\ar@{->}_-{F(\eta_A)}^-{\sim}"C";"D"};
{\ar@{->>}_-{p}^-{\wr}"B";"D"};
\endxy
\]  
 
and all the other maps are weak equivalences. \\
For (WPr4) we use the fact that projective cofibrations are also injective cofibrations. Therefore if $A$ is cofibrant in $\msxproj$, then it's also cofibrant in $\msxinj$. It follows that if $f :A \to B$ is a weak equivalence between cofibrant objects in $\msxproj$, then $F(f):F(A) \to F(B)$ is a weak equivalence in $\msx$. The functoriality of $\tld{F}$ gives a commutative square where all the other maps are weak equivalences:
 \[
\xy
(0,20)*+{\tld{F}(A)}="A";
(30,20)*+{\tld{F}(B)}="B";
(0,0)*+{F(A)}="C";
(30,0)*+{F(B)}="D";
{\ar@{->}^{\tld{F}(f)}"A";"B"};
{\ar@{->>}_-{p_A}^-{\wr}"A";"C"};
{\ar@{->}_-{F(\eta_A)}^-{\sim}"C";"D"};
{\ar@{->>}_-{p_B}^-{\wr}"B";"D"};
\endxy
\] 
 
and $\tld{F}(f)$ is a weak equivalence by $3$-out-of-$2$;\\
(WPr5) holds `on the nose' since $\tld{\eta}_A : A \to \tld{F}(A)$ is a projective cofibration: if $\varnothing \to A$ is a cofibration, by composition $\varnothing \to \tld{F}(A)$ is also a cofibration. 
\end{rmk}

\paragraph*{In our case} We would like to show that the functor $\Sim$ constructed previously is a weak monadic projection from $\msxinj$ or $\msxproj$ to $\Rc$. 
The only nontrivial condition in our case is the condition (WPr3), namely that the map induced by universal property $\Sim(\eta_{\Fa}): \Sim(\Fa) \to \Sim(\Sim \Fa))$ is a weak equivalence. \ \\

But rather than verifying step by step that $\Sim$ is a weak monadic projection,  we will use the more general approach of Simpson \cite[Chap. 9]{Simpson_HTHC}  who used \textbf{Direct localizing systems} to produce weak monadic projections. 
%%%%%%%%% On peut tout faire dans msxinj !! %%%%%%%% 
\subsubsection{Direct localizing system}
The present discussion follows closely \cite[Chap. 9]{Simpson_HTHC}. \ \\
\ \\
Let $(\Mb, I,J)$ be a tractable left proper cofibrantly generated model category which is moreover  locally presentable. Recall that \emph{tractable} means that the domains of maps in $I$ and $J$ are cofibrant. 
Suppose we are given a subclass of objects considered as a full subcategory $\Ra \subset \Mb$, and a subset $\tx{K} \subset \Ar(\Mb)$.  We assume that:
\begin{enumerate}%[label=(A\arabic*), align=left, leftmargin=*, noitemsep]
\item $\tx{K}$ is a small set;
\item $J \subset \tx{K}$;
\item $\tx{K} \subset \cof(I)$ and the domain of arrows in $\tx{K}$ are cofibrant; 
\item If $A \in \Ra$ and if $A \cong B$ in $\Ho(\Mb)$ then $B \in \Ra$; and
\item $\inj(\tx{K}) \subset \Ra$
\ \\
Say that $(\Ra, \tx{K})$ is \emph{direct localizing} if in addition to the above conditions:
\item for all $A \in \Ra$ such that $A$ is fibrant, and any $A \to B$ which is a pushout by an element of $\tx{K}$, there exists $B \to C$ in $cell(\tx{K})$ such that $A \to C$ is a weak equivalence. 
\end{enumerate}

\begin{note}
In our case $(\Mb, I, J)$ will be $(\msxinj, \I_{\msxinj},\Ja_{\msxinj})$ and $\Ra$ will be $\Rc$, the subcategory of Co-Segal categories.
\end{note}
\begin{comment}

We fix some notations which will be useful in the upcoming paragraphs. 

\paragraph{Some important maps}  We work with universes $\U \subsetneq \V \subsetneq \cdots$. We assume that all our categories have $\U$-small set of morphisms for some universe $\U \subsetneq \V $. \ \\
\ \\
Let's  fix a generating trivial cofibration $h \in \I$ and consider $h _{/{\M^{[\1]}}}$ the under category whose object are morphisms $\beta=(e_1,e_2) : h \to g$ in $\M^{[\1]}$. For each such $\beta$, let $E$ be the object obtained from the pushout of $h$ along $e_1$:
\[
\xy
(0,20)*+{U}="A";
(30,20)*+{A}="B";
(0,0)*+{V}="C";
(30,0)*+{B}="D";
(20,10)*+{E}="E";
{\ar@{->}^{e_1}"A";"B"};
{\ar@{->}^{h}"A";"C"};
{\ar@{->}^{g}"B";"D"};
{\ar@{->}^{e_2}"C";"D"};
{\ar@{.>}^{}"C";"E"};
{\ar@{.>}^{f}"B";"E"};
{\ar@{->}^{b}"E";"D"};
\endxy
\]
This gives a factorization of $g: A \xhookrightarrow{f} E \xrightarrow{b} B$. Now use the functorial factorization axiom `cofibration-trivial fibration' in $\M$  to write  $b: E \to B$  as a composite:
$$ E \xhookrightarrow{i} Q \xtwoheadrightarrow[\sim]{j} B.$$
All $E, i, f, b$ are functorial and we assume that a choice is made once and for all. 
\end{comment}
\begin{nota}\ \
\begin{enumerate}
\item We remind the reader that $h_{/V}: h \to \Id_V$ is the map represented by the commutative square:
\[
\xy
(0,20)*+{U}="A";
(20,20)*+{V}="B";
(0,0)*+{V}="C";
(20,0)*+{V}="D";
{\ar@{->}^{h}"A";"B"};
{\ar@{->}^{h}"A";"C"};
{\ar@{->}^{\Id_V}"B";"D"};
{\ar@{->}^{\Id_V}"C";"D"};
\endxy
\]

\item Let  $\kbinj$ be the set 
$$\Ja_{\msxinj} \cup  \coprod_{t \in \sx, \degb(t) >1} \{ \Pr_{t!}(h_{/V}) \}_{h \in \I}.$$ 
\end{enumerate}
\end{nota}

\begin{rmk}\label{Lurie_cofib}
Thanks to a theorem of Lurie \cite[Prop. A.1.5.12]{Lurie_HTT} we can assume that every cofibration of $\M$ is an in $cell(\I)$. It follows that for any cofibration $i: E \to Q$ the map $\Pr_{t!}(i_{/Q})$ is in $cell(\kbinj)$.
\end{rmk}

As we shall see in a moment the maps $h_{/V}$ allow us to transport in a tautological way, a lifting problem defined in $\M$ into a extension (or horn filling) problem in $\M^{[\1]}$. And thanks to the adjunction $\Pr_{t!} \dashv \Ev_{u_t}$, we will be able to test if $\Fa (u_t)$ is a trivial fibration or not in terms of being injective with respect to the maps $\Pr_{t!}(h_{/V})$.\ \\

The main result in this section is the following 

\begin{thm}\label{direct-local}
With the above notations the pair $(\Rc, \kbinj)$ is direct localizing in\\
 $(\msxinj, \I_{\msxinj},\Ja_{\msxinj})$.
\end{thm}
    
\paragraph{Proof of Theorem \ref{direct-local}} To prove the theorem we will verify that all the conditions (A1),..., (A6) hold. \ \\

The conditions (A1) and (A2) are clear. Since we assumed that all objects of $\M$ are cofibrant, it follows that all objects in $\kxinj$ are cofibrant as well; therefore the elements of $\I_{\msxinj} = \Gamma \I_{\kxinj}$ have cofibrant domain by definition of the model structure on $\msxinj$.

By construction $\Pr_{t!}: {\M^{[\1]}}_{\tx{inj}} \to \msxinj$ is a left Quillen functor and since $h_{/V}$ is clearly a cofibration in  ${\M^{[\1]}}_{\tx{inj}}$ when $h \in \I$, we deduce that  $\Pr_{t!}(h_{/V})$ is cofibration in $\msxinj$ (with cofibrant  domain). Putting these together one has (A3).\ \\

The condition (A4) follows from the stability of $\Rc$ under weak equivalence (Proposition \ref{stab-coseg}). We treat (A5) and (A6) in the next paragraphs. 

\paragraph{The condition (A5) holds} To prove this we begin by observing that

\begin{prop}
For a commutative square in $\M$
\[
\xy
(0,20)*+{U}="A";
(20,20)*+{X}="B";
(0,0)*+{V}="C";
(20,0)*+{Y}="D";
{\ar@{->}^{f}"A";"B"};
{\ar@{->}_{h}"A";"C"};
{\ar@{->}^{p}"B";"D"};
{\ar@{->}^{g}"C";"D"};
\endxy
\]  
considered as a morphism $\alpha=(f,g): h \to p$ in $\M^{[\1]}$ the following are equivalent. 

\begin{itemize}
\item There is a lifting in the commutative square above i.e there exists $k: V \to X$  such that: $k \circ h =f$, $p \circ h=g$.
\item We can fill the following `\emph{horn}' of $\M^{[\1]}$:
\[
\xy
(0,20)*+{h}="A";
(20,20)*+{p}="B";
(0,0)*+{\Id_V}="C";
{\ar@{->}^{\alpha}"A";"B"};
{\ar@{->}_{h_{/V}}"A";"C"};
{\ar@{.>}_{}"C";"B"};
\endxy
\]  
that is there exists $\beta=(k,l): \Id_V \to p$ such that $\beta \circ h_{/V}= \alpha$. 
\end{itemize}
\end{prop}

\begin{proof}
Obvious.
\end{proof}

Let $\Fa \in \msx$ be an object in $\inj(\kbinj)$. As $\Fa$ is $\kbinj$-injective, it has the left lifting property with respect to all maps in $\kbinj$, so in particular for any generating cofibration $h \in \I$ and any $t\in \sx$, there is a solution to any lifting problem of the following form:  
\[
\xy
(0,20)*+{\Pr_{t!}(h)}="A";
(20,20)*+{\Fa}="B";
(0,0)*+{\Pr_{t!}\Id_V}="C";
(20,0)*+{\ast}="D";
{\ar@{->}^{a}"A";"B"};
{\ar@{_{(}->}_{\Pr_{t!}(h_{/V})}"A"+(0,-3);"C"};
{\ar@{->}^{!}"B";"D"};
{\ar@{->}^{!}"C";"D"};
{\ar@{.>}^{}"C";"B"};
\endxy
\] 
  
where $\ast$ is the terminal object in $\msx$. But such lifting problem is equivalent to the extension or horn filling problem:
\[
\xy
(0,20)*+{\Pr_{t!}(h)}="A";
(20,20)*+{\Fa}="B";
(0,0)*+{\Pr_{t!}\Id_V}="C";
%(20,0)*+{\ast}="D";
{\ar@{->}^{a}"A";"B"};
{\ar@{_{(}->}_{\Pr_{t!}(h_{/V})}"A"+(0,-3);"C"};
%{\ar@{->}^{!}"B";"D"};
%{\ar@{->}^{!}"C";"D"};
{\ar@{.>}^{}"C";"B"};
\endxy
\]  

It follows by adjunction that $\Fa(u_t)$ has the extension property with respect to all  $h_{/V}$, as $h$  runs through  $\I$. Thanks to the previous proposition, $\Fa(u_t)$ has the RLP with respect to any generating cofibration of $h \in \I$; therefore $\Fa(u_t)$ is a trivial fibration and in particular a weak equivalence. Assembling this for all $t$ we get that $\Fa$ is a Co-Segal category i.e an object of $\Rc$, and  (A5) follows.$\qed$

\paragraph{The condition (A6) holds} The condition is given by the following

\begin{lem}
Let $\Fa$ be a Co-Segal category i.e an ojbect of $\Rc$. For a pushout square in $\msxinj$ 
\[
\xy
(0,20)*+{\Aa }="A";
(30,20)*+{\Fa}="B";
(0,0)*+{\Ba}="C";
(30,0)*+{\Za}="D";
{\ar@{->}^{\beta}"A";"B"};
{\ar@{->}_-{\alpha}"A";"C"};
{\ar@{->}_-{}"C";"D"};
{\ar@{->}^-{q}"B";"D"};
\endxy
\] 

if $\alpha \in \kbinj$ then there exists a map $\varepsilon: \Za \to \Ea $ which is a pushout by an element $\gamma \in cell(\kbinj)$  such that the composite $ \varepsilon \circ q: \Fa \to \Ea$ is a weak equivalence. 
\end{lem}

\begin{proof}
The assertion is clear if $\alpha \in \Ja_{\msxinj}$, just take $\Ea= \Za$ and $\varepsilon= \Id_{\Za}$;  $\varepsilon \in \kbinj$ is the pushout of itself along itself and $q$  is a trivial cofibration so in particular a weak equivalence. \ \\

Assume that $\alpha= \Pr_{t!}(h_{/V}): \Pr_{t!}(h) \to \Pr_{t!}(\Id_V)$; then $\alpha$ is clearly a cofibration in $\msxinj$. 
Our map $\beta: \Pr_{t!}(h) \to \Fa$ corresponds by adjunction to a map $(a_1, a_2) : h \to \Fa (u_t)$ in $\M^{[\1]}$. Denote by $E$ the object we get from the pushout of $h$ along $a_1$. 

\[
\xy
(0,20)*+{U }="A";
(30,20)*+{\Fa(A,B) }="B";
(0,0)*+{V}="C";
(30,0)*+{\Fa t}="D";
(17,10)*+{E}="E";
{\ar@{->}^{a_1}"A";"B"};
{\ar@{_{(}->}_-{h}"A"+(0,-3);"C"};
{\ar@{->}_-{a_2}"C";"D"};
{\ar@{->}^-{\Fa(u_t)}"B";"D"};
{\ar@{.>}_-{f}"B";"E"};
{\ar@{.>}_-{}"C";"E"};
{\ar@{->}_-{r}"E";"D"};
\endxy
\] 

This gives have a factorization of $\Fa(u_t): \Fa(A,B) \xhookrightarrow{f} E \xrightarrow{r} \Fa t$. \ \\

If we analyze our original pushout square at $t$ like in Remark \ref{intermed-trivial-fib} we get a diagram in which every triangle and square is commutative. 
\[
\xy
(0,20)*+{U }="A";
(30,20)*+{\Fa(A,B) }="B";
(0,0)*+{V}="C";
(30,0)*+{\Fa t}="D";
%(17,10)*+{E}="E";
%%%%%%%%%%
(60,20)*+{\Za(A,B) }="J";
(60,0)*+{\Za t }="K";
{\ar@{->}^{a_1}"A";"B"};
{\ar@{_{(}->}_-{h}"A"+(0,-3);"C"};
{\ar@{->}_-{a_2}"C";"D"};
{\ar@{->}^-{}"B";"D"};
%{\ar@{.>}_-{f}"B";"E"};
%{\ar@{.>}_-{}"C";"E"};
%{\ar@{->}_-{r}"E";"D"};
%%%%%%%%%
{\ar@{->}^-{q_{AB}}"B";"J"};
{\ar@{->}_-{q_t}"D";"K"};
{\ar@{->}^{\Za( u_t)}"J";"K"};
%%%%%%%%%%%%%
{\ar@{-->}_-{}"C";"J"};

\endxy
\] 

By universal property of the pushout of $h$ along $a_1$, there exists a unique map $\delta : E \to \Za(A,B)$ making everything commutative. The unicity of the map out of the pushout implies the commutativity of: 

\[
\xy
(0,20)*+{E }="A";
(30,20)*+{\Za(A,B) }="B";
(0,0)*+{\Fa t}="C";
(30,0)*+{\Za t}="D";
%(17,10)*+{E}="E";
{\ar@{->}^{\delta}"A";"B"};
{\ar@{->}_-{r}"A";"C"};
{\ar@{->}_-{q_t}"C";"D"};
{\ar@{->}^-{\Za(u_t)}"B";"D"};
%{\ar@{.>}_-{f}"B";"E"};
%{\ar@{.>}_-{}"C";"E"};
%{\ar@{->}_-{r}"E";"D"};
\endxy
\]

Choose a factorization of $r$:  $  E \xhookrightarrow{i} Q \xtwoheadrightarrow[\sim]{j} \Fa t$; this yields a factorization of $\Fa(u_t)$ by cofibration followed by a trivial fibration:
$$\Fa(u_t) =   \Fa(A,B) \xhookrightarrow{i \circ f}  Q \xtwoheadrightarrow[\sim]{j} \Fa t.$$
This factorization is like the one we used to construct the functor $\Sim_t^{1}$. 
Since $j$ is already a weak equivalence, if $\Fa$ is in $\Rc$ then $\Fa(u_t)$ is a weak equivalence and by $3$-out-of-$2$, $i \circ f$ is a weak equivalence and hence a trivial cofibration.\ \\

Let us set $h'= i \circ f: \Fa(A,B) \to Q$ the previous trivial cofibration and denote as usual $h'_{/Q} =(h', \Id_Q) \in \Hom_{\M^{[\1]}}(h', \Id_Q)$ the obvious map. The morphism $\Pr_{t!}(h'_{/Q})$ is an element of $\kbinj$ and since  $h'$ is trivial cofibration then $\Pr_{t!}(h'_{/Q})$ is also a trivial cofibration in $\msxinj$.\ \\

Introduce as before $S(h', j) \in \Hom_{\M^{[\1]}}(h', \Fa(u_t))$ the commutative square:
\[
\xy
(0,20)*+{\Fa(A,B)}="A";
(20,20)*+{\Fa(A,B)}="B";
(0,0)*+{Q}="C";
(20,0)*+{\Fa t}="D";
{\ar@{=}^{ \Id }"A";"B"};
{\ar@{->}_{h'}"A";"C"};
{\ar@{->}_{j}"C";"D"};
{\ar@{->}^{\Fa(u_t)}"B";"D"};
\endxy
\]
 
and denote by  $T(h',j,\Fa,t) \in \Hom_{\msx}[\Pr_{t!}(h'),\Fa]$ its adjoint transpose.  Denote by $H'_1$ the pushout of $\Pr_{t!}(h'_{/Q})$ along $T(h',j,\Fa,t) $. By stability of cofibrations under pushout we know that $H'_1$ is a trivial cofibration, so in particular a weak equivalence in $\msx$:
\begin{equation}\label{correction}
\xy
(-20,20)*+{\Pr_{t!}(h')}="A";
(20,20)*+{\Fa}="B";
(-20,0)*+{\Pr_{t!}\Id_Q}="C";
(20,0)*+{\Ea}="D";
{\ar@{->}^{T(h',j,\Fa,t) }"A";"B"};
{\ar@{_{(}->}_-{\Pr_{t!}(h'_{/Q})}^-{\wr}"A"+(0,-3);"C"};
{\ar@{_{(}->}^-{H'_1}_-{\wr}"B"+(0,-3);"D"};
{\ar@{->}^{}"C";"D"};
\endxy
\end{equation}
\ \\
\textbf{\ul{Goal}:} The rest of the proof  will be to show that we can factorize $H'_1$ as:
$$ \Fa \xrightarrow{q} \Za \xrightarrow{\varepsilon} \Ea$$
where $\varepsilon$ is a pushout of a map $\gamma \in \kbinj$. This will complete the proof as $H'_1$ is a weak equivalence.

\begin{claim}
The map $\gamma$ is $\Pr_{t!}(i_{/Q}): \Pr_{t!}(i) \to \Pr_{t!}(\Id_{Q})$,  where $i: E \hookrightarrow Q$ is the previous cofibration appearing in the factorization of the morphism $E \xrightarrow{r} \Fa t$.  This map is in $cell(\kbinj)$ (see Remark \ref{Lurie_cofib}).
\end{claim}

Let $R \in \Hom_{\M^{[\1]}}(h,h') $ be the morphism represented by the commutative square:
\[
\xy
(0,20)*+{U}="A";
(20,20)*+{\Fa(A,B)}="B";
(0,0)*+{V}="C";
(20,0)*+{Q}="D";
(10,8)*+{E}="E";
{\ar@{->}^-{a_1 }"A";"B"};
{\ar@{->}_-{h}"A";"C"};
{\ar@{->}_{g}"C";"D"};
{\ar@{->}^-{h'}"B";"D"};
{\ar@{.>}_{}"C";"E"};
{\ar@{.>}_{}"E";"D"};
{\ar@{.>}_{}"B";"E"};
\endxy
\]

The composite  $S(h', j) \circ R: h \to \Fa(u_t)$ is the map $(a_1,a_2)$ whose adjoint transpose is $\beta: \Pr_{t!}(h) \to \Fa$ of the original pushout square.\ \\
\ \\
One can easily check that the morphism $\theta=(f, \Id_Q): h' \to i$ is the pushout  in $\M^{[\1]}$ of $h_{/V}$ along $R$ and that the following commutes:

\begin{tabular}{cc}
\begin{tabular}{c}
\xy
(-20,20)*+{h}="A";
(20,20)*+{h'}="B";
(-20,0)*+{\Id_V}="C";
(20,0)*+{\Id_Q}="D";
(0,10)*+{i}="E";
{\ar@{->}^{R}"A";"B"};
{\ar@{_{(}->}_-{h_{/V}}"A"+(0,-3);"C"};
{\ar@{_{(}->}^-{h'_{/Q}}_-{\wr}"B"+(0,-3);"D"};
{\ar@{->}^{}"C";"D"};
{\ar@{.>}_{}"C";"E"};
{\ar@{^{(}->}_-{i_{/Q}}"E"+(2,-2);"D"};
{\ar@{.>}_{\theta}"B";"E"};
\endxy
\end{tabular}
=
\begin{tabular}{c}
\xy
%%%%%%%%%%%face arriere%%%%%
(-10,20)*+{U}="A";
(20,20)+(10,0)*+{\Fa(A,B)}="B";
(-10,0)*+{V}="C";
(20,0)+(10,0)*+{Q}="D";
{\ar@{->}^-{a_1}"A";"B"};
{\ar@{.>}|-{}"A"+(0,-3);"C"};
{\ar@{->}|-{h'}"B";"D"};
{\ar@{.>}|-{}"C";"D"};
%%%%%%%%%%%%%%%%%%%%face avant%%%%
(-10,20)+(-25,-15)*+{V}="E";
(20,20)+(-25,-15)+(10,0)*+{Q}="F";
(-10,0)+(-25,-15)*+{V}="G";
(20,0)+(-25,-15)+(10,0)*+{Q}="H";
{\ar@{->}^{}"E";"F"};
{\ar@{->}|-{\Id}"E"+(0,-3);"G"};
{\ar@{->}|-{}"F";"H"};
{\ar@{->}|-{ }"G";"H"};
%%%%%%%%%%% fleches sortantes arrieres vers avant %%%%%%%
{\ar@{->}|-{h}"A";"E"};
{\ar@{->}|-{h'}"B";"F"};
{\ar@{->}|-{ \Id }"C";"G"};
{\ar@{->}|-{\Id}"D";"H"};
(-10,20)+(-25,-15)+(25,8)+(6,0)*+{E}="O";
(-10,20)+(-25,-15)+(25,8)+(6,0)+(0,-20)*+{Q}="I";
{\ar@{->}|-{i}"O";"I"};
{\ar@{->}^-{i}"O";"F"};
{\ar@{->}^-{}"I";"H"};
{\ar@{.>}^-{}"E";"O"};
{\ar@{.>}^-{}"G";"I"};
{\ar@{.>}^-{}"B";"O"};
{\ar@{.>}^-{}"D";"I"};
\endxy
\end{tabular}
\end{tabular}
\ \\
\ \\
The functor $\Pr_{t!}$ is a left adjoint, therefore preserves colimits in general, in particular it preserves pushout squares. It follows  that $\Pr_{t!}(\theta)$ is the pushout of $\Pr_{t!}(h_{/V})$ along $\Pr_{t!}(R)$. If we apply  $\Pr_{t!}$ in the above square and join  the pushout square \eqref{correction} we get the following commutative diagram in $\msx$:

\[
\xy
(-20,25)*+{\Pr_{t!}(h)}="A";
(20,25)*+{\Pr_{t!}(h')}="B";
(-20,0)*+{\Pr_{t!}(\Id_V)}="C";
(20,0)*+{\Pr_{t!}(\Id_Q)}="D";
(0,10)+(0,2)*+{\Pr_{t!}(i)}="E";
{\ar@{->}^{\Pr_{t!}(R)}"A";"B"};
{\ar@{_{(}->}_-{\Pr_{t!}(h_{/V})}"A"+(0,-3);"C"};
{\ar|(0.2){\Pr_{t!}(h'_{/Q})}"B"+(0,-3);"D"};
{\ar@{->}^{}"C";"D"};
{\ar@{.>}_{}"C";"E"};
{\ar@{^{(}->}^-{\Pr_{t!}(i_{/Q})}"E"+(4.5,-3);"D"};
{\ar@{.>}_{\Pr_{t!}(\theta)}"B";"E"};
%%%%%%%%%%%% suite
(60,25)*+{\Fa}="J";
(60,0)*+{\Ea}="K";
{\ar@{->}^-{H_1}"J";"K"};
{\ar@{.>}^{T(h',j,\Fa,t) }"B";"J"};
{\ar@{->}_{}"D";"K"};
%%%%% nouveaux pushout
%(40,10)+(0,2)*+{\Ma}="Z";
%(50,5)*+{\Na}="Y";
%{\ar@{.>}_{}"J";"Z"};
%{\ar@{.>}_{}"E";"Z"};
%{\ar@{.>}_{}"Z";"Y"};
%{\ar@{.>}_{}"D";"Y"};
\endxy
\]

The naturality of the adjunction implies that $T(h',j,\Fa,t) \circ\Pr_{t!}(R)= \beta: \Pr_{t!}(h) \to \Fa$. Now introduce the pushout of $\Pr_{t!}(\theta)$ along $T(h',j,\Fa,t)$:
\begin{center}
 $D_1 =$
\begin{tabular}{c}
 \xy
(-20,20)*+{\Pr_{t!}(h)}="A";
(20,20)*+{\Fa}="B";
(-20,0)*+{\Pr_{t!}(i)}="C";
(20,0)*+{\Ma}="D";
{\ar@{->}^{T(h',j,\Fa,t) }"A";"B"};
{\ar@{_{(}->}_-{\Pr_{t!}(\theta)}"A"+(0,-3);"C"};
{\ar@{_{(}->}^-{g}"B"+(0,-3);"D"};
{\ar@{->}^{\xi}"C";"D"};
\endxy
\end{tabular} 
\end{center}
 
In one hand by lemma \ref{collapse_po}, we know that `a pushout of a pushout is a pushout', thus the  concatenation of the two squares below is a pushout of $\Pr_{t!}(h_{/V})$ along $\beta$ ( which is the pushout of the lemma):
\[
\xy
(-10,20)*+{\Pr_{t!}(h)}="A";
(20,20)*+{\Pr_{t!}(h')}="B";
(-10,0)*+{\Pr_{t!}(\Id_V)}="C";
(20,0)*+{\Pr_{t!}(i)}="D";
%(0,10)+(0,2)*+{\Pr_{t!}(i)}="E";
{\ar@{->}^{\Pr_{t!}(R)}"A";"B"};
{\ar@{_{(}->}_-{\Pr_{t!}(h_{/V})}"A"+(0,-3);"C"};
{\ar@{^{(}->}^{\Pr_{t!}(\theta)}"B"+(0,-3);"D"};
{\ar@{.>}^{}"C";"D"};
%{\ar@{.>}_{}"C";"E"};
%{\ar@{^{(}->}^-{\Pr_{t!}(i_{/Q})}"E"+(4.5,-3);"D"};
%{\ar@{.>}_{\Pr_{t!}(\theta)}"B";"E"};
%%%%%%%%%%%% suite
(50,20)*+{\Fa}="J";
(50,0)*+{\Ma}="K";
{\ar@{->}^-{g}"J";"K"};
{\ar@{.>}^{T(h',j,\Fa,t) }"B";"J"};
{\ar@{->}_{\xi}"D";"K"};

\endxy
\]

By unicity of the pushout, we can assume  (up-to a unique isomorphism) that $\Ma =\Za$ and that $g=q$. 
On the other hand if we consider $D_2$ the pushout square of $\gamma= \Pr_{t!}(i_{/Q})$ along the previous map $\xi$:
\begin{center}
 $D_2 =$
\begin{tabular}{c}
 \xy
(-20,20)*+{\Pr_{t!}(i)}="A";
(20,20)*+{\Za}="B";
(-20,0)*+{\Pr_{t!}(\Id_Q)}="C";
(20,0)*+{\Na}="D";
{\ar@{->}^{\xi}"A";"B"};
{\ar@{_{(}->}_-{\Pr_{t!}(i_{/Q})}"A"+(0,-3);"C"};
{\ar@{_{(}->}^-{\varepsilon}"B"+(0,-3);"D"};
{\ar@{->}^{\zeta}"C";"D"};
\endxy
\end{tabular} 
\end{center}
then the vertical concatenation $\frac{D_1}{D_2}$ is a pushout of $\Pr_{t!}(h'_{/Q})$ along $T(h',j,\Fa,t)$. By unicity of the pushout we can assume (up-to a unique isomorphism) that $\Na= \Ea$. Consequently we have $H'_1= \varepsilon \circ q$ and by construction $\varepsilon$ is the pushout of $\Pr_{t!}(i_{/Q}) \in \kbinj$.  This completes the proof of the lemma. 
\end{proof}

\subsubsection{Localization of the injective model structure}

We now go back to the functor $\Sim: \msxinj \to \msxinj$ constructed before.  Recall that $\Sim$ has the following properties:
\begin{enumerate}
\item for every $\Fa \in \msx$, $\Sim(\Fa) \in \Rc$;
\item we have a natural transformation $\eta_{\Fa}: \Fa \to \Sim(\Fa)$ which is a cofibration in $\msxinj$.
\item if $\Fa \in \Rc$ then  $\eta_{\Fa}: \Fa \to \Sim(\Fa)$ is a trivial cofibration therein. 
\end{enumerate}
In order to apply the material developed by Simpson in \cite{Simpson_HTHC}, we need some other properties. \ \\

The first thing we need, that will not be proved for the moment is the
\begin{hypo}\label{left-proper-hypo}
We will assume from now that if $\M$ is left proper then $\msxinj$ is also left proper. 
\end{hypo} 
\begin{rmk}\label{remark-hypo}
We are not sure for the moment that this hypothesis is valid in all cases. But if it's not, there are many reasons to believe that we can have  a structure of a \textbf{catégorie dérivable} in the sense of Cisinski \cite{Cisinski-cat-derivables}
\end{rmk}

\begin{warn}
We modify $\Sim$ by another functor denoted $\siminj$ which is a $\kbinj$-injective replacement functor. $\siminj$ is constructed by the \emph{Gluing construction} (see \cite[Prop. 7.17]{Dwyer_Spalinski}) and the small object argument in the locally presentable category $\msx$. 
\end{warn}

With the above modifications and hypothesis, and thanks to Theorem \ref{direct-local} we have

\begin{prop}
The pair $(\siminj, \eta)$ is a weak monadic projection from $\msxinj$ to $\Rc$. 
\end{prop}

\begin{proof}
This is Lemma 9.3.1 in \cite{Simpson_HTHC}. 
\end{proof}

\paragraph*{New homotopical data}
Let $\Cr$ be cofibrant replacement functor on $\msxinj$.\ \\

Define a map $\sigma : \Fa \to \Ga$ to be
\begin{itemize}%[label=$-$, align=left, leftmargin=*, noitemsep]
\item a \textbf{new weak equivalence} if the map $\siminj(\Cr\sigma): \siminj(\Cr\Fa) \to \siminj(\Cr\Ga)$ is a weak equivalence in $\msx$;
\item a \textbf{new cofibration} if $\sigma$ is a cofibration;
\item a \textbf{new trivial cofibration} if $\sigma$ is a cofibration and a new weak equivalence;
\item a \textbf{new fibration} if it has the RLP with respect to all new trivial cofibrations;  and
\item a \textbf{new trivial fibration} if it's a new fibration and  a new weak equivalence.
\end{itemize} 

With the above definitions we have
\begin{thm}\label{model-msxinjp}
The classes of original cofibrations, new weak equivalences, and new fibrations defined above provide $\msx$ with a structure of closed model category, cofibrantly generated and combinatorial. It is left proper. This structure is the left Bousfield localization\index{Localization!Bousfield} of $\msxinj$ by the original set of maps $\kbinj$.\ \\

The fibrant objects are the $\kbinj$-injective objects, in particular they are Co-Segal categories; and a morphism $\Fa \to \Ga$ to a fibrant object is a fibration if and only if it is in $\inj(\kbinj)$.  We will denote by $\msxinjp$ this new model structure on $\msx$.
\end{thm}

\begin{proof}
Follows from Theorem 9.7.1 in  \cite{Simpson_HTHC}. The fact that fibrant objects are Co-Segal categories follows from the property $(A5)$. 
\end{proof}

\subsection{Localization of the projective model structure}
\subsubsection{The Classical localization}
In the fist place we begin by using the classical localization method to localize the projective model structure $\msxproj$.

We will use the following theorem, due to Smith \cite{Smith_unpub}, as stated by Barwick \cite[Thm 4.7 ]{Barwick_localization}. 
\begin{thm}\label{smith-local}
If $\Mb$ is left proper and $\U$-combinatorial, and $\kb$ is an $\U$-small set of homotopy classes of morphisms of $\Mb$, the left Bousfield localization\index{Localization!Bousfield}  $\lkm$ of $\Mb$ along any set representing $\kb$ exists and satisfies the following conditions.
\begin{enumerate}
\item The model category $\lkm$ is left proper and $\U$-combinatorial.
\item  As a category, $\lkm$ is simply $\Mb$.
\item The cofibrations of $\lkm$ are exactly those of $\Mb$.
\item The fibrant objects of $\lkm$ are the fibrant $\kb$-local objects Z of $\Mb$. 
\item The weak equivalences of $\lkm$ are the $\kb$-local equivalences.
\end{enumerate}
\end{thm}

We introduce some piece of notations. \ \\

Given a cofibration $h$, the map $h_{/V}$  considered previously is an injective cofibration in $\M^{[\1]}$ but not in general a projective cofibration. We will then replace $h_{/V}$ by a slight modification  $\zeta(h)$ which is a projective cofibration.\ \\
If we consider $h_{/V}$ as a commutative diagram, by universal property of the pushout of $h$ along itself, there is a unique map $k: V \cup_U V \to V$ making everything commutative:
\[
\xy
(0,20)*+{U }="A";
(30,20)*+{V}="B";
(0,0)*+{V}="C";
(30,0)*+{V}="D";
%(50,-15)*+{V}="F";
(17,10)*+{V \cup_U V}="E";
{\ar@{^{(}->}^-{h}"A"+(3,0);"B"};
{\ar@{_{(}->}_-{h}"A"+(0,-3);"C"};
{\ar@{->}_-{\Id}"C";"D"};
{\ar@{->}^-{\Id}"B";"D"};
{\ar@{.>}_-{i_1}"B";"E"};
{\ar@{.>}_-{i_0}"C";"E"};
{\ar@{^{(}->}_-{k}"E"+(3,-3);"D"};
%{\ar@{->>}_-{j}^-{\sim}"D";"F"};
%{\ar@{.>}^-{\Id}"B";"F"};
%{\ar@{.>}|{\Id}"C";"F"};
\endxy
\]  

Choose a factorization `cofibration-trivial fibration' of the map $k:V \cup_U V \to V$:
$$r= V \cup_U V \xhookrightarrow{a} Z \xtwoheadrightarrow[\sim]{q} V.$$
Such factorization is a \emph{relative cylinder object} for the cofibration $h:  U \to V$.\ \\
  
For each cofibration $h \in \I$, define $\zeta(h)=(h, ai_0) \in \Hom_{\M^{[\1]}}(h, ai_1)$ to be the induced map represented by the commutative square:
\[
\xy
(0,20)*+{U }="A";
(30,20)*+{V}="B";
(0,0)*+{V}="C";
(30,0)*+{Z}="D";
(17,10)*+{V \cup_U V}="E";
{\ar@{->}^-{h}"A";"B"};
{\ar@{_{(}->}_-{h}"A"+(0,-3);"C"};
{\ar@{_{(}->}_-{ai_0}"C"+(3,0);"D"};
{\ar@{->}^-{ai_1}"B";"D"};
{\ar@{.>}_-{}"B";"E"};
{\ar@{.>}_-{}"C";"E"};
{\ar@{^{(}->}_-{a}"E"+(3,-3);"D"};
\endxy
\]   

By construction we have $ j\circ (ai_0)=\Id_V$ and since $j$ and $\Id_V$ are weak equivalences, it follows by $3$-out-of-$2$ that $ai_0$ is a weak equivalence, hence a trivial cofibration.

\begin{rmk}\ \
\begin{enumerate}
\item It's clear that $\zeta(h)$ is automatically a projective (= Reedy) cofibration in $\M^{[\1]}$; and if $h$ is a trivial cofibration then so is $\zeta(h)$.
\item Since $\Pr_{t!}\dashv \Ev_{u_t}$ is a  Quillen adjunction with the corresponding projective model structures, then $\Pr_{t!} \zeta(h)$ is a (trivial)  cofibration in $\msxproj$ is $h$ is so.
\end{enumerate}
\end{rmk}

The map $\zeta(h)$ plays the role of $h_{/V}$, that is we can detect if  $\Fa(u_t)$ is a weak equivalence in $\M$, in terms of horn  filling property against all the map $\zeta(h)$ (using a homotopy lifting lemma). %This will be needed to prove that (A5) holds.  

\begin{df}
Let $\M$ be a model category. Say  that $g: A \to B$ has the \textbf{right homotopy lifting prorperty} (RHLP) with respect to $h: U \to V$ if for any commutative square:
\[
\xy
(0,20)*+{U }="A";
(30,20)*+{A}="B";
(0,0)*+{V}="C";
(30,0)*+{B}="D";
{\ar@{->}^-{u}"A";"B"};
{\ar@{->}_-{h}"A";"C"};
{\ar@{->}_-{v}"C";"D"};
{\ar@{->}^-{g}"B";"D"};
{\ar@{.>}^-{r}"C";"B"};
\endxy
\]  

there exists a map $r: V \to A$ such that:
\begin{itemize}%[label=$-$, align=left, leftmargin=*, noitemsep]
\item $rh=u$ i.e the upper triangle commutes;
\item $gr$ and $v$ are \textbf{homotopic relative to $U$}, that is we have a relative cylinder object
$$V \cup_U V \xhookrightarrow{a} Z \xtwoheadrightarrow[\sim]{q} V$$
together with a map $f: Z \to  B$ restricting to $gu=vh$ on $U$; and inducing the equalities $fai_0= v$, $fai_1= gr$. 
\end{itemize}
\end{df}

\begin{rmk}\label{invariance-hlift}
It's important to observe that this definition is well defined in the sense that it doesn't depend on choice of the relative cylinder object for $h$. Indeed if $r'= V \cup_U V \xhookrightarrow{a'} Z' \xtwoheadrightarrow[\sim]{q'} V$ is another cylinder, then by the lifting axiom we can find a map $k : Z' \to Z$ such that $a=ka'$ and $q'=qk$. 
Therefore if $f: Z \to B$ is a homotopy lifting with respect to $r$ then automatically $f'= fk$ is a homotopy lifting with respect to $r'$.
\end{rmk}

\begin{prop}\label{horn-fil-eq}
For a cofibration $h:U \to V$ and  a map $g :A \to B$ in a model category $\M$, the following are equivalent.
\begin{enumerate}
\item $g$ has the RHLP with respect to $h$.
\item $g$ is $\{\zeta(h)\}$-injective.
\end{enumerate}
\end{prop} 

\begin{proof}
We simply show how we get (1) from (2). The converse follows by `reversing' the argumentation since we can  assume that the relative cylinder chosen in (1) is the one used to construct $\zeta(h)$ thanks to Remark \ref{invariance-hlift}. \ \\

Assume that $g$ is $\{\zeta(h)\}$-injective.  A lifting problem defined by $g$ and $h$ 

\[
\xy
(0,20)*+{U }="A";
(30,20)*+{A}="B";
(0,0)*+{V}="C";
(30,0)*+{B}="D";
{\ar@{->}^-{u}"A";"B"};
{\ar@{->}_-{h}"A";"C"};
{\ar@{->}_-{v}"C";"D"};
{\ar@{->}^-{g}"B";"D"};
\endxy
\]  

corresponds to a map $\alpha=(u,v) \in \Hom_{\M^{[\1]}}(h, g)$; since  $g$ is $\{\zeta(h)\}$-injective we can fill  the following horn in $\M^{[\1]}$
\[
\xy
(0,20)*+{h}="A";
(20,20)*+{g}="B";
(0,0)*+{ai_1}="C";
{\ar@{->}^{\alpha}"A";"B"};
{\ar@{->}_-{\zeta(h)}"A";"C"};
{\ar@{.>}_{\beta}"C";"B"};
\endxy
\] 
with a map $\beta=(r,f)$.
Displaying this horn in $\M$ we end up with with a commutative diagram where every triangle and square commute:
\[
\xy
%%%%%%%%%%%face arriere%%%%%
(-10,20)*+{U}="A";
(10,20)+(10,0)*+{A}="B";
(-10,0)*+{V}="C";
(10,0)+(10,0)*+{B}="D";
{\ar@{->}^-{u}"A";"B"};
{\ar@{->}_-{h}"A"+(0,-3);"C"};
{\ar@{->}_-{g}"B";"D"};
{\ar@{->}_-{v}"C";"D"};
%%%%%%%%%%%%%%%%%%%%face avant%%%%
(-10,20)+(-25,-15)*+{V}="E";
(-10,0)+(-25,-15)*+{Z}="G";
{\ar@{->}_-{ai_1}"E"+(0,-3);"G"};
%%%%%%%%%%% fleches sortantes arrieres vers avant %%%%%%%
{\ar@{->}_-{h}"A";"E"};
{\ar@{->}_-{ai_0}"C";"G"};
{\ar@{.>}^-{r}"E";"B"};
{\ar@{.>}_-{f}"G";"D"};
\endxy
\]

 One clearly has that $(r,f)$ gives the desired relative homotopy lifting.
\end{proof}

The following is a classical result in model categories  
\begin{lem}\label{homotop-lifting}
In a tractable left proper model category $\M$, if a map $g : A \to B$ between fibrant objects has the RHLP with respect to any generating cofibration $h$, then $g$ is a weak equivalence.
\end{lem}

\begin{proof}
This is Lemma 7.5.1 in \cite{Simpson_HTHC}.
\end{proof}

\begin{nota}\ \
\begin{enumerate}
\item We will denote for short  $\zeta(\I)= \{ \zeta(h)  \}_{U \xrightarrow{h} V \in \I}$.
\item For $t \in \sx$ we will write $\Pr_{t!} \zeta(\I)$ the image of $\zeta(I)$ by $\Pr_{t!}$.
\item Let $\kbproj$ be the set $\Ja_{\msxproj} \cup ( \coprod_{t \in \sx, \degb(t) >1} \Pr_{t!} \zeta(\I))$.
\end{enumerate}
\end{nota}
\ \\
Recall that $\Rc$ is the subcategory of Co-Segal categories. 
Under the hypothesis \ref{left-proper-hypo}, we have by virtue of Theorem \ref{smith-local} that

\begin{thm}\label{classical-local}
There exists a combinatorial model structure on $\msxproj$ such that the fibrant objects are Co-Segal categories ie object of $\Rc$. This model structure is moreover left proper and is the left Bousfield localization\index{Localization!Bousfield}  with respect to $\kbproj$.
\end{thm}

\begin{proof}
Take $\msxproj$ with the Bousfield localization with respect to $\kbproj$ which exists by Theorem \ref{smith-local}.  The maps in $\kbproj$ are weak equivalences in the Bousfield localization and since they are old cofibrations, they become trivial cofibrations.
It follows that if $\Fa$ is fibrant, it is then $\kbproj$-injective.\ \\

Let $\Fa$ be a $\kbproj$-injective object. In one hand as $\jmsxproj \subset \kbproj$, we have that $\Fa$ is fibrant in the old model structure, which means that  $\Fa$ is level-wise fibrant. On the other hand since $\Fa$ is $\{\Pr_{t!} \zeta(\I)\}$-injective for all $t$, it follows  by the adjunction $\Pr_{t!} \dashv \Ev_{u_t}$, that $\Fa(u_t)$ is $\zeta(\I)$-injective. Combining Proposition \ref{horn-fil-eq} and  Lemma \ref{homotop-lifting}, we deduce that $\Fa(u_t)$ is a weak equivalence, thus  $\Fa \in \Rc$.  
\end{proof}

\subsubsection{Direct localizing the projective model structure.}

Let's consider the injective localized model structure $\msxinjp$ constructed with the direct localizing system $(\Rc, \kbinj)$. As we mentioned before it's the left Bousfield localization of the original model structure with respect to $\kbinj$.This is the same Bousfield localization we will have using Theorem \ref{classical-local}.\ \\

Since every map  in $\kbinj$ is a weak equivalence in $\msxinjp$ , we have in particular that for all $h \in I$, the map $\Pr_{t!}(h_{/V}): \Pr_{t!}(h) \to \Pr_{t!}(\Id_V)$ is a weak equivalence in $\msxinjp$. Recall that each map $\zeta(h)$ is constructed out of $h_{/V}$  and we have a factorization of  $h_{/V}$:
$$ h \xhookrightarrow{\zeta(h)} ai_1 \xtwoheadrightarrow{\ell} \Id_V.$$
The factorization is displayed below as:

\[
\xy
(0,20)*+{U }="A";
(20,20)*+{V}="B";
(0,0)*+{V}="C";
(20,0)*+{Z}="D";
(40,20)*+{V}="E";
(40,0)*+{V}="F";
{\ar@{->}^-{h}"A";"B"};
{\ar@{_{(}->}_-{h}"A"+(0,-3);"C"};
{\ar@{_{(}->}_-{ai_0}"C"+(3,0);"D"};
{\ar@{_{(}->}^-{ai_1}_{\wr}"B"+(0,-3);"D"};
{\ar@{->}^-{\Id}"B";"E"};
{\ar@{->>}_-{q}^-{\sim}"D";"F"};
{\ar@{->}^-{\Id}"E";"F"};
\endxy
\] 

As $q$ is a trivial fibration,  $\ell: ai_1 \to \Id$ is a level-wise trivial fibration in $\M^{[\1]}$ in particular a weak equivalence therein. Since we've assumed that all the objects of $\M$ are cofibrant, then both the source and target of $\ell$ are cofibrant in $\arproj$. From Ken Brown lemma $\Pr_{t!}$ preserves weak equivalences between cofibrant objects (as any left Quillen functor); thus  $\Pr_{t!}(\ell)$ is an old weak equivalence, hence a new weak equivalence (= a weak equivalence in $\msxinjp$). \ \\

From the equality $\Pr_{t!}(h_{/V}) = \Pr_{t!}(\ell) \circ \Pr_{t!}\zeta(h)$ we deduce by $3$ for $2$ in the model category $\msxinjp$, that $\Pr_{t!}\zeta(h)$ is a weak equivalence therein; moreover since projective cofibrations are also injective cofibration, then $\Pr_{t!}\zeta(h)$ is an old cofibration, hence a new cofibration. Putting these together one has that every $\Pr_{t!}\zeta(h)$ is a trivial cofibration in $\msxinjp$. \ \\

\paragraph{New projective data} Recall that weak equivalences in $\msxinjp$ are those maps $\sigma$ such that $\siminj(\Cr \sigma)$ is a weak equivalence in the original model structure $\msxinj$. Here  $\siminj$ is a weak monadic projection from $\msxinj$ to $\Rc$ and $\Cr$ is a cofibrant replacement functor in $\msxinj$. As pointed out in \cite[Sec 9.3]{Simpson_HTHC}, the notion of new weak equivalence depends only on $\Rc$ and doesn't depend neither on  $\kbinj$ nor on $\Cr$.\ \\

Using again the fact that projective cofibrations are injective ones, and since old weak equivalences in $\msxinj$ and $\msxproj$ are the same,  we clearly have that a cofibrant replacement functor in $\msxproj$ is also a cofibrant replacement for $\msxinj$. Therefore we can assume that $\Cr$ is a projective cofibrant replacement functor. \ \\

As mentioned in the Remark \ref{mpr-inj-proj} we can extract from the weak monadic projection $\siminj$ a weak monadic projection $\simproj$ from $\msxproj$ to $\Rc$. This is obtained by applying the functorial factorization of the type $(\cof, \fib \cap \we)$ to the natural transformation $\eta: \Id_{\msx} \to \siminj$ in the model category $\msxproj$. In particular there is trivial projective fibration $\simproj \to \siminj$. 

\begin{warn}
In the upcoming paragraphs we will precise `new injective' or `new projective' to avoid confusion when saying `new weak equivalence'.  We will remove this disctinction later since the new weak equivalences will be the same.
\end{warn}

Let $\Cr$ be cofibrant replacement functor on $\msxproj$.\ \\
\ \\
Define a map $\sigma : \Fa \to \Ga$ to be:
\begin{itemize}%[label=$-$, align=left, leftmargin=*, noitemsep]
\item a \textbf{new projective weak equivalence} if the map $\simproj(\Cr\sigma): \simproj(\Cr\Fa) \to \simproj(\Cr\Ga)$ is a weak equivalence in $\msx$ ;
\item a \textbf{new cofibration} if $\sigma$ is a cofibration;
\item a \textbf{new trivial cofibration} if $\sigma$ is a cofibration and a new (projective) weak equivalence;
\item a \textbf{new fibration} if it has the RLP with respect to all new trivial cofibrations;  and
\item a \textbf{new trivial fibration} if it's a new fibration and  a new weak equivalence.
\end{itemize} 

\begin{prop}\label{new-we-same}
The class of new projective weak equivalences and and new injective weak equivalences coincide. 
\end{prop}

\begin{proof}
For any morphism $\sigma: \Fa \to \Ga$ we have by construction the commutativity of:
 \[
\xy
(0,20)*+{\simproj(\Cr \Fa)}="A";
(45,20)*+{\simproj(\Cr \Ga)}="B";
(0,0)*+{\siminj(\Cr \Fa)}="C";
(45,0)*+{\siminj(\Cr \Ga)}="D";
{\ar@{->}^{\simproj(\Cr \sigma)}"A";"B"};
{\ar@{->>}_-{}^-{\wr}"A";"C"};
{\ar@{->}_-{\siminj(\Cr \sigma)}"C";"D"};
{\ar@{->>}_-{}^-{\wr}"B";"D"};
\endxy
\] 
\\
with all the above vertical maps being weak equivalences in $\msx$.  Therefore if one of two maps $\simproj(\Cr \sigma)$, $ \siminj(\Cr \sigma)$ is a weak equivalence in $\msx$ then by $3$ for $2$ of weak equivalence in $\msx$ the other one is also a weak equivalence. 
\end{proof}

\begin{rmk}
\begin{enumerate}
\item It follows from the proposition that the new projective weak equivalences are closed under retract, composition and satisfy the $3$ for $2$ property. We leave the reader to check that the class of new projective (trivia) cofibrations and (trivial) fibrations are also so closed under retract and composition.
\item Since the old projective cofibrations are also old injective cofibrations, by the proposition we deduce that the new projective trivial cofibrations are also a new injective trivial cofibrations.   
\end{enumerate}
\end{rmk}

We remind the reader that the functor $\siminj$ is a $\kbinj$-injective replacement functor.  The following proposition tells us that
\begin{prop}
The functor $\simproj$ is a $\kbproj$-injective replacement functor, that is  there is a lift to any diagram:
\[
\xy
(0,20)*+{\Aa}="A";
(20,20)*+{\simproj(\Fa)}="B";
(0,0)*+{\Ba}="C";
(20,0)*+{\ast}="D";
{\ar@{->}^{}"A";"B"};
{\ar@{->}_{\alpha}"A";"C"};
{\ar@{->}^{!}"B";"D"};
{\ar@{->}^{!}"C";"D"};
{\ar@{.>}^{}"C";"B"};
\endxy
\] 

for any morphism $\alpha: \Aa \to \Ba$ in $\kbproj$.
\end{prop}

\begin{proof}[Sketch of proof]
Such lifting problem is simply an extension problem:
\[
\xy
(0,20)*+{\Aa}="A";
(20,20)*+{\simproj(\Fa)}="B";
(0,0)*+{\Ba}="C";
%(20,0)*+{\ast}="D";
{\ar@{->}^{}"A";"B"};
{\ar@{->}_{}"A";"C"};
%{\ar@{->}^{!}"B";"D"};
%{\ar@{->}^{!}"C";"D"};
{\ar@{.>}^{\exists ?}"C";"B"};
\endxy
\] 
If $\alpha \in \Ja_{\msxproj}\subset \Ja_{\msxinj} \subset \kbinj$, by extending the above diagram using the projective trivial fibration  $p: \simproj(\Fa) \xtwoheadrightarrow{\sim} \siminj(\Fa)$, we find  a lift $\circled{1}$ in the diagram

\[
\xy
(0,20)*+{\Aa}="A";
(20,20)*+{\simproj(\Fa)}="B";
(0,0)*+{\Ba}="C";
%(20,0)*+{\ast}="D";
(50,20)*+{\siminj(\Fa)}="E";
{\ar@{->}^{}"A";"B"};
{\ar@{->}_{\alpha}"A";"C"};
%{\ar@{->}^{!}"B";"D"};
{\ar@{->>}^{p}_{\sim}"B";"E"};
%{\ar@{->}^{!}"C";"D"};
{\ar@{.>}|{\circled{1}}"C";"E"};
{\ar@{.>}|{\circled{2}}"C";"B"};
\endxy
\] 

With the map $\circled{1}$  we have a commutative square which gives a lifting problem defined by $\alpha$ and $p$; by the lifting axiom in the model category $\msxproj$ we can find  a lift $\Ba \xrightarrow{\circled{2}} \simproj(\Fa)$ making everything commutative, in particular we have the desired lift of the original extension problem.\ \\
\ \\
Assume now that $\alpha= \Pr_{t!}\zeta(h)$, then by extending the diagram with the map $p$ and the map $\Pr_{t!}(\ell)$  we end up with the following diagram
\[
\xy
(0,20)*+{\Pr_{t!}h}="A";
(20,20)*+{\simproj(\Fa)}="B";
(0,5)*+{ai_1}="C";
%(20,0)*+{\ast}="D";
(50,20)*+{\siminj(\Fa)}="E";
(0,-10)*+{\Id_V}="F";
{\ar@{->}^{}"A";"B"};
{\ar@{->}^{\Pr_{t!}\zeta(h)}"A";"C"};
%{\ar@{->}^{!}"B";"D"};
{\ar@{->>}^{p}_{\sim}"B";"E"};
{\ar@{->}^{\Pr_{t!} \ell}"C";"F"};
{\ar@{.>}|{\circled{1}}"F";"E"};
{\ar@/_1.5pc /_{\Pr_{t!}h_{/V}}"A";"F"};
{\ar@{.>}_{\circled{2}}"C";"B"};
\endxy
\] 

Since $\Pr_{t!}h_{/V} \in \kbinj$, there is a lift $ \Id_{V} \xrightarrow{\circled{1}} \siminj(\Fa)$; the commutative square we get is a lifting problem defined by the projective cofibration  $\Pr_{t!}\zeta(h)$ and the projective trivial fribration $p$. By the lifting axiom in $\msxproj$ there is a lift $ai_1 \xrightarrow{\circled{2}}  \simproj(\Fa)$ making everything commutative; in particular we have a lift to the original extension problem.  
\end{proof}

\paragraph{Pushout by  new projective trivial cofibrations}
We need a last piece of ingredient in order to apply Smith recognition theorem for combinatorial model category as stated for  example by Barwick \cite[Proposition 2.2 ]{Barwick_localization}. 
\begin{lem}
For a pushout square in $\msxproj$ 
\[
\xy
(0,20)*+{\Aa }="A";
(30,20)*+{\Fa}="B";
(0,0)*+{\Ba}="C";
(30,0)*+{\Ga}="D";
{\ar@{->}^{}"A";"B"};
{\ar@{->}_-{f}"A";"C"};
{\ar@{->}_-{}"C";"D"};
{\ar@{->}^-{g}"B";"D"};
\endxy
\] 
If $f$ is a new projective trivial cofibration then $g$ is a new projective trivial cofibration.  
\end{lem}
\begin{proof}
Since $g$ is already an old projective cofibration it suffices to show that $g$ is a new projective weak equivalence. By the  Proposition \ref{new-we-same}, it's the same thing as being a new injective weak equivalence.\ \\

As pointed out above $f$ is also a new injective trivial cofibration, and since (trivial) cofibrations are closed by pushout in any model category, it follows that $g$ is also new injective trivial cofibration in $\msxinjp$.  In particular $g$ is a new injective weak equivalence, thus a new projective weak equivalence. 
\end{proof}

The new weak equivalences form an accessible subcategory of $\Ar(\msxproj)$ because they are the weak equivalences of the combinatorial model category $\msxinjp$. A map in $\inj(\I_{\msx})$ is a trivial fibration, in particular an old weak equivalence. But old weak equivalence are also new weak equivalence, therefore any map in $\inj(\I_{\msx})$  is a new weak equivalence.\ \\
 
By virtue of Smith theorem we have that:

\begin{thm}\label{msxprojp}
The classes of original cofibrations, new weak equivalences, and new fibrations defined above provide $\msx$ with a structure of closed model category, cofibrantly generated and combinatorial. It is indeed left proper.\\
This model structure is the left Bousfield localization\index{Localization!Bousfield}  of $\msxproj$ with respect to $\kbproj$. Fibrant objects are Co-Segal categories.

We will denote by $\msxprojp$ this new model strucuture on $\msx$. 
\end{thm}  

\begin{proof}
The model structure is guaranteed by Smith's theorem.  A fibrant object $\Fa$ is by definition an $\alpha$-injective injective object for every new trivial cofibration $\alpha$.\\
 
From the previous observations the maps $\Pr_{t!}\zeta(h)$ become new trivial cofibrations, and since the maps in $\Ja_{\msxproj}$ are already new trivial cofibrations, all the elements of $\kbproj$ are new trivial cofibrations. Consequently if  $\Fa$ is fibrant, it's in particular $\kbproj$-injective and one proceeds as in the proof of Theorem \ref{classical-local} to conclude that $\Fa$ is a Co-Segal category which is fibrant in the old model structure (= level-wise fibrant).\ \\

It remains to prove that this model structure is the left Bousfield localization of $\msxproj$ with respect to $\kbproj$. To prove  this we will simply show that the model structure $\msxprojp$ and the one of Theorem \ref{classical-local} are the same; that is we have the same class of  cofibrations and fibrations on the underlying category $\msxproj$.\ \\

For the notations we will denote by  $\msxprojc$ the `classical' localized model structure of Theorem \ref{classical-local}.
In both $\msxprojp$ and $\msxprojc$ the cofibrations are the old cofibrations in $\msxproj$ so we only have to show that we have the same fibrations.

 In the two model structures the fibrations are defined to be the maps having the RLP with respect to all maps which are both cofibration and new weak equivalences. It follows that the fibrations will be the same as soon as we show that we have the same weak equivalences. Thanks to the lemma below we know that the weak equivalences are indeed the same.
\end{proof}

\begin{lem}\label{equiv-of weak-equiv}
Given a map $\sigma: \Fa \to \Ga$ in $\msx$ the following are equivalent

\begin{enumerate}
\item $\sigma$ is a weak equivalence in $\msxprojc$ i.e a $\kbproj$-local equivalence;
\item $\sigma$ is a weak in $\msxprojp$, that is $\simproj(\Cr \sigma)$ is an old weak equivalence;
\item $\sigma$ is a weak in $\msxinjp$, that is $\siminj(\Cr \sigma)$ is an old weak equivalence or equivalenty a $\kbinj$-local equivalence.
\end{enumerate}
\end{lem}

\begin{proof}[Proof of the lemma]
The equivalence between $(2)$ and $(3)$ is clear; we mentioned it there just for a reminder.  We will show that $(3)$ is equivalent to $(1)$. 

The general picture is that $\msxprojc$ is obtained by turning the maps $\Pr_{t!}\zeta(h)$ to weak equivalences while $\msxinjp$ is obtained by turning the maps $\Pr_{t!}h_{/V}$ into weak equivalences. The two type of maps are related by the equality $\Pr_{t!}(h_{/V}) = \Pr_{t!}(\ell) \circ \Pr_{t!}\zeta(h)$ where $\Pr_{t!}(\ell)$ is already a weak equivalence; thus if we turn one of them into a weak equivalence then the other one also become a weak equivalence by $3$ for $2$. This is the general philosophy; we shall now present a proof. \ \\

In the follwoing we use the language of classical Bousfield localization of a model category with respect to a set of maps. This requires the notion of \emph{local object} and \emph{local equivalence}. These definitions can be found for example in \cite{Barwick_localization}, \cite{DHKS}, \cite{DKFC}, \cite{Hirsch-model-loc}, \cite{Hov-model} , \cite{Lurie_HTT}. We will denote by $\Map(-,-)$ a homotopy function complex on $\msx$; the homotopy type of $\Map(-,-)$ depends only on the weak equivalences, so we can use the same for $\msxinj$ and $\msxproj$.\ \\
\ \\
We start with the direction $(3) \Rightarrow (1)$. Let $\sigma: \Fa \to \Ga$ be a $\kbinj$-local equivalences, that is a weak equivalence in $\msxinjp$. Then $\sigma$ will be a weak equivalence in $\msxprojc$, by definition, if for all $\kbproj$-local object  $\Ea$ the induced map of simplicial sets $\sigma^{\star}: \Map(\Ga, \Ea) \to \Map(\Fa,\Ea)$ is weak equivalence.\ \\

Let $\Ea$ be a $\kbproj$-local object. By definition $\Ea$ is fibrant in $\msxproj$ and for any $\alpha: \Aa \to \Ba$ in $\kbproj$ then 
 $\alpha^{\star}: \Map(\Ba, \Ea) \to \Map(\Aa,\Ea)$ is a weak equivalence. As $\Ea$ is not fibrant in $\msxinj$ we have to  introduce a fibrant replacement $\tld{\Ea}$ in  $\msxinj$; we have a weak equivalence $q: \Ea \to  \tld{\Ea}$. But fibrant objects in $\msxinj$ are also fibrant in $\msxproj$, therefore $\tld{\Ea}$ is fibrant in $\msxproj$. 
 \begin{claim}
$\tld{\Ea}$ is also $\kbproj$-local and $\kbinj$-local.
 \end{claim}
The fact that $\tld{\Ea}$ is $\kbproj$-local is a classical thing: $\kbproj$-locality is invariant under weak equvalences of fibrant objects. In fact for any $\alpha: \Aa \to \Ba$ in $\kbproj$,   the following commutes:
\[
\xy
(0,20)*+{\Map(\Ba, \Ea) }="A";
(30,20)*+{\Map(\Aa, \Ea)}="B";
(0,0)*+{\Map(\Ba, \tld{\Ea})}="C";
(30,0)*+{\Map(\Aa, \tld{\Ea})}="D";
{\ar@{->}^{\alpha^{\star}}_-{\sim}"A";"B"};
{\ar@{->}_-{q^{\star}}^-{\wr}"A";"C"};
{\ar@{->}_-{\alpha^{\star}}"C";"D"};
{\ar@{->}^-{q^{\star}}_-{\wr}"B";"D"};
\endxy
\] 

The maps $\alpha^{\star}:\Map(\Ba, \Ea) \to \Map(\Aa, \Ea)$ is a weak equivalences of simplicial sets by hypothesis; the two vertical maps $q^{\star}$ are also weak equivalences (see \cite[Thm 17.7.7]{Hirsch-model-loc}). Therefore  by $3$ for $2$ in the model category of simplicial sets the other map  $\alpha^{\star}:\Map(\Ba,\tld{\Ea}) \to \Map(\Aa, \tld{\Ea})$ is also a weak equivalence which proves that $\tld{\Ea}$ is $\kbproj$-local.\ \\

To prove that $\tld{\Ea}$ is $\kbinj$-local we have to show that for all $\alpha \in \kbinj$ then \\ $\alpha^{\star}:\Map(\Ba,\tld{\Ea}) \to \Map(\Aa, \tld{\Ea})$ is a weak equivalence i.e $\tld{\Ea}$ is $\alpha$-local. If $\alpha \in \Ja_{\msxinj}$ then $\alpha$ is in particular an old weak equivalence so  $\tld{\Ea}$ is $\alpha$-local by  \cite[Thm 17.7.7]{Hirsch-model-loc} (any object is local with respect to any weak equivalence). It remains the case where $\alpha= \Pr_{t!}h_{/V}$. \ \\

In one hand since $\Pr_{t!}(\ell)$ is a weak equivalence, we have by  \cite[Thm 17.7.7]{Hirsch-model-loc} again that $\tld{\Ea}$ is $\Pr_{t!}(\ell)$-local. On the other hand since $\tld{\Ea}$ is $\kbproj$-local, it's in particular $\Pr_{t!}\zeta(h)$-local. It follows that $\tld{\Ea}$ is $\{ \Pr_{t!}(\ell) \circ \Pr_{t!}\zeta(h)\}$-local  (as weak equivalence of simplicial sets are closed by compositions); but 
$\Pr_{t!}(\ell) \circ \Pr_{t!}\zeta(h)= \Pr_{t!}h_{/V}$, thus $\tld{\Ea}$ is $\Pr_{t!}h_{/V}$-local. 
Summing up this one has that $\tld{\Ea}$ is also $\kbinj$-local.\ \\

\begin{claim}
A $\kbinj$-local equivalence is also a $\kbproj$-local equivalence.
\end{claim}

By the above if $\sigma$ is a $\kbinj$-local equivalence then $\sigma^{\star}: \Map(\Ga, \tld{\Ea}) \to \Map(\Fa,\tld{\Ea})$ is a weak equivalence of simplicial sets. If we put this in the commutative diagram below:
\[
\xy
(0,20)*+{\Map(\Ga, \Ea) }="A";
(30,20)*+{\Map(\Fa, \Ea)}="B";
(0,0)*+{\Map(\Ga, \tld{\Ea})}="C";
(30,0)*+{\Map(\Fa, \tld{\Ea})}="D";
{\ar@{->}^{\sigma^{\star}}_-{}"A";"B"};
{\ar@{->}_-{q^{\star}}^-{\wr}"A";"C"};
{\ar@{->}_-{\sigma^{\star}}^-{\sim}"C";"D"};
{\ar@{->}^-{q^{\star}}_-{\wr}"B";"D"};
\endxy
\]
\\
All the three maps are weak equivalences, therefore by $3$ for $2$ the map $\sigma^{\star}: \Map(\Ga, \Ea) \to \Map(\Fa,\Ea)$ is also a weak equivalence. Consequently $\sigma$ is also a $\kbproj$-local equivalence as claimed.
\ \\

For the direction $(1)\Rightarrow (3)$ the proof is the same. Let $\sigma: \Fa \to \Ga$ be a $\kbproj$-local equivalence. By definition for any $\kbproj$-local object $\Ea$, $\sigma^{\star}: \Map(\Ga, \Ea) \to \Map(\Fa,\Ea)$ is a weak equivalence.\ \\
\begin{claim}
If $\Ea$ is a $\kbinj$-local object then it is also a $\kbproj$-local object. 
\end{claim} 
In fact if $\Ea$ is $\kbinj$-local, $\Ea$ is fibrant in $\msxinj$ (hence in $\msxproj$) and the map
 $$\alpha^{\star}:\Map(\Ba,\tld{\Ea}) \to \Map(\Aa, \tld{\Ea})$$ is a weak equivalence for all $\alpha \in \kbinj$.\ \\

Recall that $\kbproj=\Ja_{\msxproj} \cup ( \coprod_{t \in \sx, \degb(t) >1} \Pr_{t!} \zeta(\I))$.\ \\

If $\alpha \in \Ja_{\msxproj}$ then $\alpha$ is in particular an old weak equivalence in $\msxproj$ and $\msxinj$ so $\Ea$ is automatically $\alpha$-local by \cite[Thm 17.7.7]{Hirsch-model-loc}. 

Assume that $\alpha=\Pr_{t!}\zeta(h)$. By applying $\Map(-,\Ea)$ to the equality $\Pr_{t!}(\ell) \circ \Pr_{t!}\zeta(h)= \Pr_{t!}h_{/V}$ we get $\Pr_{t!}\zeta(h)^{\star} \circ  \Pr_{t!}(\ell)^{\star} = \Pr_{t!}h_{/V}^{\star}$; and since $\Ea$ is $\kbinj$-local, it is  in particular $\Pr_{t!}\zeta(h)$-local, thus $\Pr_{t!}\zeta(h)^{\star}$ is a weak equivalence.

 As $\Pr_{t!}(\ell)$ is an old weak equivalence then $\Ea$ is $\Pr_{t!}(\ell)$-local by \cite[Thm 17.7.7]{Hirsch-model-loc} which means that $\Pr_{t!}(\ell)^{\star}$ is a weak equivalence; puting these together we conclude by $3$-out-of-$2$ that  $\Pr_{t!}\zeta(h)^{\star}$ is a weak equivalence and $\Ea$ is $\alpha$-local.\\
 
Now if $\sigma$ is a $\kbproj$-local equivalence, by the above for any $\kbinj$-local object $\Ea$ the map  $$\sigma^{\star}: \Map(\Ga, \Ea) \to \Map(\Fa,\Ea)$$ is also a weak equivalence which means precisely that $\sigma$ is a $\kbinj$-local equivalence.
\end{proof}

% je pense qu'il y a un isomorphisme entre celui-ci et le premier.
%%% put proof here %%%
\subsection{A new fibered model structure on $\mset$}
In the following we want to vary the set $X$ when $\msx$ is equipped with the Bousfield localization with respect to $\kbproj$ constructed previously. \ \\
We will use the following notations.
\begin{nota}\ \\
From now we will specify by $\kbxproj$, $\kbyproj$  the corresponding sets. \\
$\msxprojp=$ the bousfield localization of $\msxproj$ with respect to $\kbxproj$, for a set $X$.\\ 
$\Le_X: \msxproj \to \msxprojp=$ the canonical left Quillen functor.\\ 
\end{nota}

Recall that for any function $f: X \to Y$ there is a Quillen adjunction $$\fex: \msxproj \rightleftarrows \msyproj: \fstar$$
 where $\fex$ is left Quillen and $\fstar$ is right Quillen. 
 \begin{prop}
 For any sets $X, Y$ and any function $f:  X \to Y$ there is an induced Quillen adjunction denoted again $\fex \dashv \fstar$:
 $$\fex: \msxprojp \leftrightarrows \msyprojp: \fstar$$
  and the following commutes:
\[
\xy
(0,20)*+{\msxproj}="A";
(40,20)*+{\msyproj}="B";
(0,0)*+{\msxprojp}="C";
(40,0)*+{\msyprojp}="S";
{\ar@{->}_{\fex}"A"+(10,-1);"B"+(-10,-1)};
{\ar@{<-}^{\fstar}"A"+(10,1);"B"+(-10,1)};
{\ar@{->}^{\Le_X}"A"+(1,-3);"C"+(1,3)};
{\ar@{->}^{\Le_Y}"B"+(1,-3);"S"+(1,3)};
{\ar@{->}_{\fex}"C"+(10,-1);"S"+(-10,-1)};
{\ar@{<-}^{\fstar}"C"+(10,1);"S"+(-10,1)};
\endxy
\] 

 \end{prop}
 
 \begin{proof}[Sketch of proof]
 We  will show that $\fex(\kbxproj) \subset \kbyproj$. The new `$\fex$' will be induced by universal property of the Bousfield localization\index{Localization!Bousfield} . \ \\
Consider $\alpha$ in  $\kbxproj=\Ja_{\msxproj} \cup ( \coprod_{t \in \sx, \degb(t) >1} \Pr_{t!} \zeta(\I))$.

\begin{claim}
If $\alpha \in \Ja_{\msxproj}$ then $\fex(\alpha) \in \Ja_{\msyproj} \subset \kbyproj$.
\end{claim}

To see this first recall that  $\Ja_{\msxproj}= \Gamma \Ja_{\kxproj}$, where $\kx= \prod_{(A,B) \in X^{2}} \Hom[\S_{\ol{X}}(A,B)^{op}, \M]$. Furthermore there is also a Quillen adjunction $\fex: \kxproj \leftrightarrows \kyproj: \fstar$; and by definition  we have: 
 $$\Ja_{\kxproj}= \coprod_{t \in \sx} \{ \Fb^{t}_{h}, h \in \Ja\},$$
where $\Fb^{t}$ is the left adjoint of the evaluation at $t$. It suffices then to show that $\fex(\Ja_{\kxproj}) \subset \Ja_{\kyproj}$. \ \\

Using the adjunction one establishes that for any $m \in \M$ and any $\Ga \in \ky$:
\begin{equation*}
\begin{split}
\Hom(\fex \Fb^{t}_{m}, \Ga)
& \cong \Hom(\Fb^{t}_{m}, \fstar \Ga)\\
& \cong \Hom(m, (\fstar \Ga)(t))\\
& = \Hom(m, \Ga f(t)) \\
&\cong \Hom(\Fb^{f(t)}_{m}, \Ga)\\
\end{split}
\end{equation*}
Consequently $\fex \Fb^{t}_{m} \cong \Fb^{f(t)}_{m}$; similarly $\fex \Fb^{t}_{\alpha} \cong  \Fb^{f(t)}_{\alpha}$ (we can actually assume that we have an equality). One clearly has that $\fex(\Ja_{\kxproj}) \subset \Ja_{\kyproj}$ and the claim holds.

 \begin{claim}
For every $t$ and every $h \in \I$ then $\fex\Pr_{t!} \zeta(h) \cong \Pr_{f(t)!} \zeta(h)$
\end{claim} 
This also holds by the adjunction. For any $h$ and any $\Ga \in \msy$ one has:
\begin{equation*}
\begin{split}
\Hom(\fex \Pr_{t!} \zeta(h), \Ga)
& \cong \Hom(\Pr_{t!} \zeta(h), \fstar \Ga)\\
& \cong \Hom(\zeta(h), (\fstar \Ga)(u_t))\\
& = \Hom(\zeta(h), \Ga \underbrace{f(u_t)}_{= u_{f(t)}}) \\
& = \Hom(\zeta(h), \Ga u_{f(t)}) \\
&\cong \Hom(\Pr_{f(t)!} \zeta(h), \Ga)\\
\end{split}
\end{equation*}
and the claim follows. If we combine the two claims we have the desired result 
\end{proof}

By virtue of Roig-Stanculescu result (Theorem \ref{Stanculescu_bifib}) and under the hypothesis (\ref{left-proper-hypo}) we have 
\begin{thm}\label{fibered-model-mset-plus}
For a symmetric closed monoidal model category $\M$ whose objects are all cofibrant, the category $\mset$ has a Quillen model structure where a map $\sigma=(f,\sigma):\Fa \to \Ga $ is 
\begin{enumerate}
\item  a weak equivalence if $f: X \to Y$ is an isomorphism of sets and  $ \sigma: \Fa \to \fstar \Ga $ is a weak equivalence in $\msxprojp$,
\item a cofibration if the adjoint map $ \tld{\sigma}: \fex \Fa \to \Ga $ is a cofibration in $\msyprojp$,
\item a fibration if $ \sigma: \Fa \to \fstar \Ga $ is a fibration in $\msxprojp$.
\item fibrant objects are Co-Segal categories
\end{enumerate} 
We will denote by $\msetprojp$ the new model structure on $\mset$. There is a canonical left Quillen functor 
$$\Le: \msetproj \to \msetprojp $$
whose component over $X$ is $\Le_X: \msxproj \to \msxprojp$. 
\end{thm}

\begin{proof}
The proof is the same as for Theorem \ref{fibred-model-mset}. 
On $\Set$ we take the minimal model structure: cofibrations and fibrations are all maps, weak equivalences are isomorphisms. One can check easily that all the conditions of Theorem \ref{Stanculescu_bifib} are fulfilled; this gives the model structure described above.\ \\

For $\Fa \in \msxprojp$, $\Fa$ is fibrant if the canonical map $\jmath: \Fa \to \ast$ is a fibration $\msetprojp$, where $\ast$ is the terminal object therein. By definition this is equivalent to have that $\jmath: \Fa \to \jmath^{\star} \ast$ is fibration is $\msxprojp$; and since $\jmath^{\star}$ is a right adjoint then  it preserves the terminal object (as it preserves any limit). \ \\

Summing up this one has that  $\Fa$ is fibrant in $\msetprojp$ if and only if it is fibrant in $\msxprojp$, therefore $\Fa$ is a Co-segal category by Theorem \ref{msxprojp}.
\end{proof}

\subsubsection{Cofibrantly generated}
Since the  cofibrations on $\msetproj$ and $\msetprojp$ are the same, the set $\I_{\msetproj}$ constitutes also a generating set for the cofibrations in $\msetprojp$. Using the fact that $\msetprojp$ is already a model category which is locally presentable, one can easily check that the set  $\I_{\msetproj}$ and the class of weak equivalences of $\msetprojp$ satisfy the conditions Smith's recognition theorem (\cite[Proposition 2.2]{Barwick_localization}). \ \\

By Smith's theorem we have a combinatorial model structure on $\msetproj$ with the same cofibrations and weak equivalences of $\msetprojp$. The set $\I_{\msetproj}$ constitutes a generating set for the cofibrations and there exists a set $\Ja_{\msetprojp}$ which is a set of generating trivial cofibrations.\ \\

But since this new model category has the same cofibrations and weak equivalences (hence the same fibrations) as $\msetprojp$ we deduce that this new model structure is in fact isomorphic to $\msetproj$; thus $\msetprojp$ is combinatorial and in particular cofibrantly generated.

\section*{The monoidal category $(\mset, \otimes_{\S}, \Un )$ }
Given a small category $\Ca$, by construction there is a degree (or length) strict $2$-functor $\degb : \S_{\Ca} \to \S_{\1}$ where $\1$ is the unit category and $\S_{\1} \cong (\Delta_{epi}, +,\0)$.  If $\D$ is another category we can form the genuine fiber product of $2$-categories $\S_{\Ca} \times_{\S_{\1}} \S_{\D}$. 

\begin{prop}
There is an isomorphism of $2$-categories: $ \S_{\Ca \times \D} \cong \S_{\Ca} \times_{\S_{\1}} \S_{\D}$.
\end{prop}

\begin{proof}
Obvious.
\end{proof}
\subsection*{Tensor product of $\S$-diagrams}
Let $\M=(\ul{M}, \otimes , I)$ be a symmetric monoidal category. Let $\Un: (\Delta_{epi}, +)^{op} \to \M$ be the  constant lax functor
of value $I$ and $\Id_{I}$;  the laxity maps are the obvious natural isomorphism $I \otimes I \cong I$. $\Un$ exhibits  $I$ in a tautological way as a (semi) monoid. \ \\

Given $\Fa: (\S_{\Ca})^{\tx{$2$-op}}  \to \M$   and $\Ga:(\S_{\D})^{\tx{$2$-op}}  \to \M$ we define 
$\Fa \otimes_{\S} \Ga: (\S_{\Ca \times \D})^{\tx{$2$-op}} \to \M $ to be the lax functor given as follows. 
\renewcommand{\theenumi}{\arabic{enumi}}
\begin{enumerate}
\item For $1$-morphisms $(s,s') \in (\S_{\Ca \times \D})$ we set $(\Fa \otimes_{\S} \Ga) (s,s'):= \Fa(s) \otimes \Ga(s')$,
\item The laxity maps $\varphi_{\Fa \otimes_{\S} \Ga}: (\Fa \otimes_{\S} \Ga) (s,s') \otimes (\Fa \otimes_{\S} \Ga) (t,t') \to (\Fa \otimes_{\S} \Ga) (s \otimes t,s' \otimes t')$ are obtained as the composite:
$$\Fa(s) \otimes \Ga(s') \otimes \Fa(t) \otimes  \Ga(t')\xrightarrow{ \Id \otimes sym \otimes \Id} \Fa(s) \otimes \Fa(t) \otimes \Ga(s') \otimes  \Ga(t') \xrightarrow{\varphi_{\Fa} \otimes \varphi_{\Ga}} \Fa(s\otimes t) \otimes \Ga(s' \otimes t')$$
where $sym$ is the symmetry isomorphism in $\M$ (we have $sym:\Ga(s') \otimes \Fa(t) \xrightarrow{\cong}  \Fa(t) \otimes \Ga(s')$).
\item One easily see that if $f: \Ca' \to \Ca  $ and $g: \D' \to \D$ then $(f\times g)^{\star} \Fa \otimes_{\S} \Ga \cong \fstar \Fa \os \gstar \Ga$.
\item If $\sigma=(\sigma,f) \in \Hom_{\mset}(\Fa, \Ga)$ and $\gamma=(\gamma,g) \in \Hom(\Fa',\gstar \Ga')$ we define \\
$\sigma \os \gamma=(\sigma \otimes \gamma, f \times g)   \in \Hom_{\mset}[\Fa \os \Ga, \Fa' \os \Ga']$ to be the morphism whose component at $(s,s')$ is $\sigma_s \otimes \sigma_{s'}$. 
\end{enumerate}
\ \\
We leave the reader to check that:
\begin{enumerate}
\item $\os$ is a bifunctor and is associative,
\item we have a canonical symmetry: $\Fa \os \Ga \cong \Ga \os \Fa$, 
\item for any $\Fa$ we have a natural isomorphism $\Fa \os \Un \cong  \Fa$.
\end{enumerate}

%%%%%%%%%%%%%%%%%%%%%%%%%%%%%%%%%%%%%%%%%%%%%%%%%%%%%%%%%%%%%%%%%%%
%%%%%%%%%%%%%%%%%%%%%%%%%%%%%%%%%%%%%%%%%%%%%%%%%%%%%%%%%%%%%%%%%%%
%%%%%%%%%%%%%%%%%%%%%%%%%%%%%%%%%%%%%%%%%%%%%%%%%%%%%%%%%%%%%%%%%%%
%%%%%%%%%%%%%%%%%%%%M IS LOCALLY MODEL CATEGORY%%%%%%%%%%% %%%%%%%%%%%%%%%%%%%%%%%%%%%%%%%%%%%%%%%%%%%%%%%
%%%%%%%%%%%%%%%%%%%%%%%%%%%%%%%%%%%%%%%%%%%%%%%%%%%%%%%%%%%%%%%%%%%
%%%%%%%%%%%%%%%%%%%%%%%%%%%%%%%%%%%%%%%%%%%%%%%%%%%%%%%%%%%%%%%%%%%
%%%%%%%%%%%%%%%%%%%%%%%%%%%%%%%%%%%%%%%%%%%%%%%%%%%%%%%%%%%%%%%%%%%
\section{A model structure for $\M$-$\Cat$ for a $2$-category $\M$ } \label{mode-mcat-deux-cat}
In the following $\M$ is a $2$-category. We will use capital letters $U,V, W$ for the objects of $\M$ and $f,g,h,k$ for $1$-morphisms and $\alpha, \beta, \gamma$ for $2$-morphisms. For $U,V \in \Ob(\M)$ we will write $\M_{UV}$ the category of morphisms from $U$ to $V$; when $U=V$ we will simply write $\M_U$.  If $f,g$ are composable $1$-morphisms, we will denote by $g\otimes f$ the horizontal composition. Similarly for $2$-morphisms $\alpha, \beta$ we will write $\beta \otimes \alpha$ the horizontal composition while $\beta \circ \alpha$ will represents the vertical composition. 
\begin{df}\label{model-2-cat}
Let $\M$ be $2$-category. We will say that $\M$ is \textbf{locally a model category} if the following conditions holds:
\begin{enumerate}
\item Each $\M_{UV}$ is a model category in the usual sense.
\item $\M$ is a biclosed that is: for every $1$-morphism $f$  the following functors have  a right adjoint:
$$f \otimes -  \ \ \ \ \ \ \ \ \ \ \ \  \ - \otimes f $$ %a vérifier 
\item The pushout-product axiom holds: given two cofibrations $\alpha: f \to g  \in \Ar(\M_{UV})$ and  $\beta: k \to h \in \Ar(\M_{VW})$ then the induced map $\beta \square \alpha : h \otimes  f \cup_{k\otimes f} k \otimes g \to h \otimes g$ is a cofibration in $\M_{UW}$ which is moreover a trivial cofibration if either $\alpha$ or $\beta$ is . %%% 
\item For every $U$ and any $1$-morphism $f \in \M_{UV}$, if $Q(\Id_U) \to \Id_U$ is a cofibrant replacement then the induced map  $ Q(\Id_U) \otimes f \to f $ is a weak equivalence in $\M_{UV}$.
\end{enumerate}
\end{df} 
As one can see this is a straightforward generalization of a monoidal model category considered as a $2$-category with one object. The condition $(2)$ allows to distribute colimits with respect to the composition on each factor. \ \\

We recall briefly below the definition of a category enriched over a $2$-category $\M$. \ \\

An $\M$-category $\X$ consists roughly speaking of: 
\begin{enumerate}
\item for each object $U$ of $\M$, a set $\X_U$ of objects \emph{over} $U$;
\item for objects $A,B$ over $U,V$, respectively an arrow $\X(A,B): U \to V \in \M_{UV}$;
\item for each object $A$ over $U$, a $2$-cell $I_A: \Id_U \Longrightarrow \X(A,A) \in \M_U$;
\item for object $A,B,C$ over $U,V,W$, respectively, a $2$-cell $c_{ABC} :\X(B,C)\otimes \X(A,B) \Longrightarrow \X(A,C) \in \M_{UW}$
satisfying the obvious three axioms of left and right identities and associativity.
\end{enumerate} 
\ \\
Equivalently $\X$ can be defined as a lax morphism $\X : \ol{X} \to \M$ or a strict homomorphism $\X: \P_{\ol{X}} \to \M$ (see \cite{SEC1}). Note that for each $U$, we have a category $\X_{|U}$  enriched over the monoidal category $\M_{U}$; the set of objects of $\X_{|U}$  is $\X_U$. \ \\

Given two $\M$-categories $\X$ and $\Y$   an $\M$-functor is given the following data:

\begin{enumerate}
\item a function $\Phi: \Ob(\X) \to \Ob(\Y)$, $\Phi = \coprod_{U} \Phi_U$ with $\Phi_U : \X_U  \to \Y_U$;
\item for  $A,B$  in $\X$ over $U,V$, respectively we have a morphism   $\Phi_{AB}: \X(A,B) \to \Y(\Phi A, \Phi B)$ in $\M_{UV}$
\item for each object $A$ over $U$ we have $I_{\Phi A}= \Phi_{AA} \circ I_A$;
satisfying the obvious compatibility with respect to the composition on both sides.
\end{enumerate} 

$\M$-categories with $\M$-functors form the category $\M$-$\Cat$. There is an obvious category $\M$-$\Graph$ whose objects are $\M$-graphs and morphisms are just the natural ones.  \ \\

We have a forgetful functor just like in the monoidal case $\Ub : \M\tx{-$\Cat$} \to \M\tx{-$\Graph$}$. 
\begin{rmk}
As usual there is a restriction $\Phi_{|U}: \X_{|U} \to \Y_{|U}$ of $\Phi$, which is a $\M_U$-functor. And any $\M$-functor has an underlying morphism between the corresponding $\M$-graphs. 
\end{rmk}
\begin{prop}
Let $\M$ be a biclosed $2$-category which is locally \textbf{locally-presentable} that is: each $\M_{UV}$ is locally presentable in the usual sense. Then $\M$-$\Cat$ is locally presentable. 
\end{prop}
% %
\begin{proof}
All is proved in the same manner as for a monoidal category $\M$. Below we list the different steps:
\begin{enumerate}
\item First  one shows that $\M\tx{-$\Graph$}$ is cocomplete, this is easy, just apply the same method as Wolff \cite{Wo}
\item $\Ub$ is monadic:  construct a left adjoint of $\Ub$ with the same formula given in \cite{Wo}. Then show that $\M$-$\Cat$ has coequalizer of parallel $\Ub$-split pair, again using the same idea in \emph{loc. cit}. As $\Ub$ clearly reflect isomorphisms, it follows by Beck monadicity theorem that $\Ub$ is monadic.
\item Linton's result \cite[Corollary 2]{Linton-cocomplete} applies and one has that $\M$-$\Cat$ is cocomplete as well.
\item Following the same method as Kelly and Lack \cite{Kelly-Lack-loc-pres-vcat} one has that the monad induced by $\Ub$ preserves filtered colimits and $\M$-$\Graph$ is locally presentable. From this we apply \cite[Remark 2.78]{Adamek-Rosicky-loc-pres} to establish that $\M$-$\Cat$ is also locally presentable. 
\end{enumerate}
Note that being biclosed is essential in order to permute (filtered) colimits and $\otimes$. 
\end{proof}

\begin{term}
Let $W$ be a class of $2$-morphisms in $\M$. An $\M$-functor  $\Phi : \X \to  \Y$ is said to be locally in $W$ if for every pair of objects $A,B$ of $\X$ the $2$-morphism $\Phi_{AB}: \X(A,B) \to \Y(\Phi A, \Phi B)$ is in $W$. We will say the same thing for a morphism between  $\M$-graphs.
\end{term}

\subsection*{The category $\mcatx$}
 For each $U \in \M$ let's fix a set $X_U$ of objects over $U$ and consider $X= \coprod_U X_U$. 
Denote by $\mcatx$ the category of $\M$-categories with fixed set of objects $X$ and $\M$-functors which fixe $X$. Similarly there is a category $\mgraphx$ of $\M$-graphs with vertices $X$ and morphisms fixing $X$. \ \\

Just like in the case where $\M$ is  a monoidal category there is a tensor product in $\mgraphx$ defined as follows. If $\X, \Y \in \mgraphx$ one defines $\X \otimes_X \Y$ by:
$$(\X \otimes_{X} \Y) (A,B)= \coprod_{Z \in X} \X(A,Z)  \otimes \Y(Z,B).$$ 
The unity of this product is the $\M$-graph $\In$ given by
 \begin{equation*}
\In(A,B) = \begin{cases}
 \Id_U  &  \tx{if $A$ and $B$ are over the same object $U$} \\
 \varnothing = \tx{initial object in $\M_{UV}$} & \tx{If $A$ over $U$, and $B$ over $V$ with  $U \neq V$}   \\
  \end{cases}
\end{equation*}

As usual it's not hard to see that $\mcatx$ is the category of monoids of $(\mgraphx, \otimes_X, \In)$. We have an obvious isomorphism of categories:
$$\mgraphx \cong \prod_{(U,V) \in \Ob(\M)^2}  \M_{UV}^{(X_U \times X_V)}$$ 
where $ \M_{UV}^{(X_U \times X_V)}= \Hom(X_U \times X_V, \M_{UV})$. 
From this we can endow $\mgraphx$ with the product model structure. In this model structure, fibrations, cofibrations and weak equivalences are simply component wise such morphism. \ \\

\begin{thm}
Let $\M$ be a $2$-category  which is locally a model category and locally cofibrantly generated. 
Assume moreover that all  the objects of $\M$ are cofibrant. 
Then we have: 
\begin{enumerate}
\item the category $\mcatx$ admits a model structure which is cofibrantly generated. 
\item if $\M$ is combinatorial, then so is  $\mcatx$
\end{enumerate}

\end{thm}

\begin{proof}
This is a special case of theorem \ref{model-laxalg} where $\O= \O_{X}$, $\Ca\dt= \ol{X}$ and $\Ma\dt= \M$
\end{proof}

\begin{rmk}
One can remove the hypothesis  `all the objects of $\M$ are cofibrant' by using an analogue of the monoid axiom of \cite{Sch-Sh-Algebra-module}. In fact one can use the  method in \emph{loc. cit} to establish the theorem.  Lurie \cite{Lurie_HTT} also presented a nice description of the model structure for the case where $\M$ is the monoidal category of simplicial sets. It seems obvious that we can adapt his method to calculate the pushout of interest in our case.
\end{rmk}

\appendix
\section{Some classical lemma}\label{section-classical-lema}
We present here the supplying material required in the other sections.\ \\

The first lemma we present is a classical result of category theory concerning the universal property of a pushout diagram. We include this part for completeness.

%Let's recall first the definition of a pushout (see \cite{Mac}). 

\begin{df}
Let $\C$ be a small category and $f:A \to B$, $g:A\to C$ be two morphisms of $\C$ with the same source $A$.\\

A \emph{pushout} of $\langle f,g\rangle$ is a commutative square:
\begin{tabular}{c}
\xy
(0,15)*+{A}="A";
(15,15)*+{B}="B";
(0,0)*+{C}="C";
(15,0)*+{R}="R";
{\ar@{->}^{f}"A";"B"};
{\ar@{->}_{g}"A";"C"};
{\ar@{->}^{u}"B";"R"};
{\ar@{->}^{v}"C";"R"};
\endxy
\end{tabular}

such that for any other commutative square
 \begin{tabular}{c}
\xy
(0,15)*+{A}="A";
(15,15)*+{B}="B";
(0,0)*+{C}="C";
(15,0)*+{S}="S";
{\ar@{->}^{f}"A";"B"};
{\ar@{->}_{g}"A";"C"};
{\ar@{->}^{h}"B";"S"};
{\ar@{->}^{k}"C";"S"};
\endxy
 \end{tabular}
 
\ \\ 
 there exists a unique morphism $t:R \to S$ such that $h=t \circ u$ and $k=t \circ v$. 
\end{df}
\begin{nota}
 To stress the fact a commutative square is a pushout square we will put the symbol `$\lrcorner$' at the center of the diagram: 
\[
\xy
(0,15)*+{A}="A";
(15,15)*+{B}="B";
(0,0)*+{C}="C";
(15,0)*+{R}="R";
(7.5,7.5)*+{\LARGE{\tx{$\lrcorner$}}}="G"; %%% barycentre du caré pour mettre le symbole de pushout
{\ar@{->}^{f}"A";"B"};
{\ar@{->}_{g}"A";"C"};
{\ar@{->}^{u}"B";"R"};
{\ar@{->}^{v}"C";"R"};
\endxy
\]
\end{nota}

\begin{obs}\label{unicity_po}
It follows from the universal property of the pushout that if a commutative square 
\[
\xy
(0,15)*+{A}="A";
(15,15)*+{B}="B";
(0,0)*+{C}="C";
(15,0)*+{S}="S";
{\ar@{->}^{f}"A";"B"};
{\ar@{->}_{g}"A";"C"};
{\ar@{->}^{h}"B";"S"};
{\ar@{->}^{k}"C";"S"};
\endxy
\]
is also a pushout of $\langle f,g \rangle$ then the unique map $t:R \to S$ we get from the definition is an isomorphism.\  \\

In this situation, up to a composition by the morphism $t$, we can assume that $R=S$, $h=u$   and $k=v$.  
\end{obs}

\begin{term}
The map $v:C \to R$ is said to be `the pushout of $f$ along $g$' and by symmetry $v$ is the pushout of $g$ along $f$.
\end{term}

Using again the universal property of the pushout, we get the following lemma which says that `\emph{a pushout of a pushout is a pushout}'.
\begin{lem}\label{collapse_po}
Let $\C$ be a small category. Given two commutative squares in $\C$:

\[
\xy
(0,15)*+{A}="A";
(15,15)*+{B}="B";
(0,0)*+{C}="C";
(15,0)*+{R}="R";
{\ar@{->}^{f}"A";"B"};
{\ar@{->}_{g}"A";"C"};
{\ar@{->}^{u}"B";"R"};
{\ar@{->}^{v}"C";"R"};
\endxy
\xy
(0,15)*+{B}="A";
(15,15)*+{D}="B";
(0,0)*+{R}="C";
(15,0)*+{U}="S";
{\ar@{->}^{p}"A";"B"};
{\ar@{->}_{u}"A";"C"};
{\ar@{->}^{w}"B";"S"};
{\ar@{->}^{q}"C";"S"};
\endxy
\] 
\\

If the two are pushout squares then the `composite square': 
\begin{tabular}{c}
\xy
(0,15)*+{A}="A";
(15,15)*+{D}="B";
(0,0)*+{C}="C";
(15,0)*+{U}="R";
{\ar@{->}^{p\circ f}"A";"B"};
{\ar@{->}_{g}"A";"C"};
{\ar@{->}^{w}"B";"R"};
{\ar@{->}^{q \circ v}"C";"R"};
\endxy
\end{tabular}

is also a pushout square.
\end{lem}
\ \\
\begin{proof} 
Obvious
\end{proof}
\ \\
\begin{rmk}
If $\kappa$ is a regular cardinal and if the set Mor($\C$) of all morphism of $\C$ is $\kappa$-small, then the previous lemma can be applied for any set of consecutive pushout squares indexed by an ordinal $\beta$ with $\beta < \kappa$. 
\end{rmk}

\begin{df}
 Let $K$ be set of cardinality $|K| <\kappa$ and $\uc=\{\o,\o \xrightarrow{Id_{\o}} \o \}$ be the unit category.
We identify $K$ with the discrete category whose set of objects is $K$.\\
The \textbf{cone} associated to the set $K$ is the category $\co(K)$ described as follows.\ 
\ \\
$Ob(\co(K))= K \sqcup \{ \o \}$ and for $x,y \in Ob(\co(K))$ we have
\begin{equation*}
\co(K)(x,y) =
  \begin{cases}
     \{ (\o,y) \} & \text{if $x=\o$ and $y \in K$} \\
     \{ \Id_x \}  & \text{if $x=y$ }\\
     \varnothing & \text{otherwise}
  \end{cases}
\end{equation*}
The composition is the obvious one.
\end{df}
\
\begin{rmk}
Our notation `$\co(K)$' is inspired from the category $\co(n)$ used by Simpson (see \cite{Simpson_HTHC} 13.1). In fact if $K$ is a set of cardinality $n$ then $\co(K)$ is isomorphic to $\co(n)$.\
\ \\ 
Following the terminology in \cite{SEC1}, $\co(K)$ is the \underline{thin} brige from the $\uc$ to $K$ which was denoted therein by  `$\uc < K$'. 
\end{rmk}

With the category $\co(K)$ we can give a general definition.

\begin{df}
Let $\M$ be a category and $K$ be a set of cardinality $|K| <\kappa$.\\
A cone of $\M$ indexed by $K$ is a functor $\tau: \co(K) \to \M$.\\
The ordinal $|K|$ is said to be the \emph{size} of $\tau$.
\end{df}
\

Concretely a cone of $\M$ corresponds to a $K$-indexed family $\{ A \to B_k \}_{k\in K}$ of morphisms of $\M$ having the same domain. We can write 
 $$\tau= \{ A \to B_k \}_{k\in K}$$ with $A=\tau(\o)$, $B_k=\tau(k)$ and $\tau[(\o,k)]= A \to B_k$.

\begin{term}\ \
\begin{itemize}
 \item If $\M$ is a model category then a cone $\tau= \{ A \to B_k \}_{k\in K}$ is said to be a \textbf{cone of cofibrations} if every morphism $A \to B_k$ is a cofibration.\\
\item More generally given a class of maps $\bf{I}$ of a category $\M$, a cone $\tau= \{ A \to B_k \}_{k\in K}$ is said to be a cone of $\bf{I}$ if every map $A \to B_k$ is a member of $\bf{I}$.\\
\item A cone $\tau: \co(K) \to \M$ is said to be small if the index set $K$ is such that $|K| <\kappa$ for some regular cardinal $\kappa$. 
\end{itemize}
\end{term}
\
\begin{df}
Let $\M$ be a category. \\
 A \textbf{generalized pushout diagram} in $\M$ is a colimit of a cone $\tau$ of $\M$. Here the colimit is the colimit of the functor $\tau$.
\end{df}

One can check that for a cone $\tau$ associated to a set $K$ of cardinality $2$, then the colimit of the diagram $\tau$ is given by a classical pushout square.\
\ \\

Using the fact that in a model category, the pushout of a cofibration is again a cofibration we have the following lemma.

\begin{lem}\label{cone-cofib}
Let $\kappa$ be regular cardinal. For any  $\kappa$-small model category $\M$ the following hold.
\begin{enumerate}
 \item Every small cone $\tau$ of $\M$ has a colimit.
 \item If $\tau= \{ A \hookrightarrow B_k \}_{k\in K}$ if a cone of (trivial) cofibrations then all the canonical maps:
 $$B_k \to \colim( \tau)$$
$$A \to \colim (\tau) $$  
are also (trivial) cofibrations.
\end{enumerate}
\end{lem}
\ 
\begin{proof}[Sketch of proof]
The assertion (1) follows from the fact that $\M$ has all small colimits by definition of a model category.\
\ \\

For the assertion (2) it suffices to prove that the canonical map $B_k \to \colim( \tau)$ is a (trivial) cofibration for every $k$. The map 
$A \to \colim (\tau) $ which is the composite of $A \hookrightarrow B_k $ and $B_k \to \colim( \tau)$, will be automatically a (trivial) cofibration since (trivial) cofibrations are stable by composition. For the rest of the proof we will simply treat the trivial cofibration case; the other case is implicitely proved by the same method.\ \\
 
We proceed by induction on the cardinality of $K$.
\begin{itemize}
\item If $|K|=1$, there is nothing to prove.
\item  If $|K|=2$, $K=\{k_1,k_2 \}$, the colimit of $\tau$ is a pushout diagram and the result is well known.
\item Let $K$ be an arbitrary $\kappa$-small set and assume that the assertion is true for any subset $J \subset K$ with $|J|<|K|$.\\

 Let's now choose $k_0 \in K$ and set $J=K-\{k_0\}$ and $\tau':=\tau_{|\co(J)}$.\\
As $|J|<|K|$ the assertion is true for $\tau'$ and we have that the canonical maps:
\begin{equation*}
  \begin{cases}
     A \to \colim(\tau') &  \\
     B_k \to \colim(\tau') & \text{for all $k \in J$}\\
  \end{cases}
\end{equation*}
 are trivial cofibrations .
\begin{comment}
From the universal property of $\tau'$ we know that there exist a unique map $i: \colim(\tau') \to  \colim(\tau)$ such that
\begin{equation*}
  \begin{cases}
     A \to \colim (\tau) = i \circ [A \hookrightarrow  \colim (\tau')] &  \\
     B_k \to \colim (\tau) = i \circ [B_k \hookrightarrow \colim (\tau')] & \text{for all $k \in J.$}\\
  \end{cases}
\end{equation*}
\ \\
\end{comment}
\ \\
Consider in $\M$ the following pushout square:
\[
\xy
(0,15)*+{A}="A";
(20,15)*+{B_{k_{0}}}="B";
(0,0)*+{\colim (\tau')}="C";
(20,0)*+{S}="S";
{\ar@{^{(}->}^{}"A";"B"};
{\ar@{_{(}->}_{}"A";"C"};
{\ar@{->}^{}"B";"S"};
{\ar@{->}^{}"C";"S"};
\endxy
\]

Since trivial cofibrations are closed under pushout, we know that the canonical maps
\begin{equation*}
  \begin{cases}
   \colim (\tau') \to S&  \\
     B_{k_{0}} \hookrightarrow  S & \\
  \end{cases}
\end{equation*}
are trivial cofibrations.  We deduce that the following maps are also trivial cofibrations:
\begin{equation*}
  \begin{cases}
     B_k \hookrightarrow S= [\colim (\tau') \hookrightarrow S] \circ [B_k \hookrightarrow \colim(\tau')] & k \in J \\
    A \hookrightarrow  S. & \\
  \end{cases}
\end{equation*}

Finally one can easily verify that the object $S$ equipped with the morphisms: 
\begin{equation*}
  \begin{cases}
     A \hookrightarrow  S &  \\
     B_{k_{0}} \hookrightarrow  S & \\
     B_k \hookrightarrow S & \text{for all $k \in J.$}\\
  \end{cases}
\end{equation*}
is a colimit of the functor $\tau$, that is, it satisfies the universal property of `the' colimit of $\tau$. So we can actually take $\colim (\tau)=S$, and the assertion follows. 
\end{itemize}
\end{proof}

\begin{rmk}
If $\bf{I}$ is the set of cofibrations of $\M$ then what we've just showed can be rephrased in term of  \emph{relative $\bf{I}$-cell complex}. We refer the reader to \cite[Ch. 2.1.2]{Hov-model}or \cite[Ch. 8.7]{Simpson_HTHC} and references therein for the definition of relative cell complex.

In this terminology we've just showed that each map $B_k \to \colim(\tau)$ is a \emph{relative $\bf{I}$-cell complex}. Now it's well known that a relative $\bf{I}$-cell complex is an element of some set $\bf{I}\tx{-cof}$ (see \cite[Lemma 2.1.10]{Hov-model}). In general for an arbitrary class of maps $\bf{I}$  we have an inclusion  $\bf{I} \subset \bf{I}\tx{-cof}$,  but in a model category with $\bf{I}$ is the set of cofibrations of $\M$ we have an equality  $\bf{I}\tx{-cof}=\bf{I}$ (see \cite[Ch. 2.1.2]{Hov-model}).
\end{rmk}

\section{Adjunction Lemma}\label{section-adjunction-lemma}
\subsection{Lemma 1}\label{lemme_adjunction_1}
In the following we fix $\M$ a cocomplete  a symmetric closed monoidal category.\ \\
We remind the reader that being symmetric closed implies that the tensor product commutes with colimits on both sides. In particular  for every (small) diagram $D: J \to \M$, with $J$  a discrete category i.e a set, then we have:
$$ [\coprod_{j \in J}D(j) ] \otimes P \cong \coprod_{j \in J}[D(j) \otimes P]$$ 
$$ P \otimes  [\coprod_{j \in J}D(j)]  \cong \coprod_{j \in J}[P \otimes D(j)].$$
\ \\
Let $X$ nonempty $\kappa$-small set and $\Ub: \M_{\S}(X) \to \prod_{(A,B) \in X^{2}} \Hom[\S_{\ol{X}}(A,B)^{op}, \M]$ be the functor  defined as follows.\\

\begin{equation*}
  \begin{cases}
 \Ub(F)= \{F_{AB} \}_{(A,B) \in X^{2} } & \tx{for $F=\{F_{AB},\varphi_{s,t} \}_{(A,B) \in X^{2}} \in \M_{\S}(X)$} \\
 \Ub(\sigma)=  \{\sigma_{AB}: F_{AB} \to G_{AB} \}_{(A,B) \in X^2} & \tx{for $ F \xrightarrow{\sigma} G$}\\
  \end{cases}
\end{equation*}
\\
So concretely the functor $\Ub$ forgets the laxity maps `$\varphi_{s,t}$'.
\ \\

Our goal is to prove the following lemma.
\begin{lem}\label{adjoint-ub}
The functor $\Ub$ has a left adjoint, that is there exists a functor  
$$\Gamma :  \prod_{(A,B) \in X^{2}} \Hom[\S_{\ol{X}}(A,B)^{op}, \M] \to \M_{\S}(X)$$ 
\\
such that for all $F \in \M_{\S}(X)$ and all $\X \in \prod_{(A,B) \in X^{2}} \Hom[\S_{\ol{X}}(A,B)^{op}, \M]$,   we have an isomorphism of set: 

$$ \Hom[\Gamma[\X], F] \cong \Hom[\X,\Ub(F)]$$
\\
which is natural in $F$ and $\X$.
\end{lem}

We will  adopt the following conventions.
\begin{conv}\ \
\begin{itemize}
\item If $(U_1,\cdots,U_n)$ is a $n$-tuple of objects of $\M$ we will write $U_1 \otimes \cdots \otimes U_n$ for the tensor product  of $U_1,\cdots,U_n$ with all pairs of parentheses starting in front. 
\item For  a nonempty set $J$  and $J_1, J_2$ two nonempty subsets of $J$ such that $J_1 \bigsqcup J_2=J$ then for every family $(U_j)_{j \in J}$ of objects of  $\M$ we will freely identify the two objects $\coprod_{j \in J} U_j$ and $(\coprod_{j \in J_1}U_j) \coprod (\coprod_{j \in J_2} U_j) $ and we will call it ``the'' coproduct of the $U_j$. 

In particular for each $k \in  J_1$, the  three canonical maps
\begin{equation*}
  \begin{cases}
 i_k: U_k \to  \coprod_{j \in J} U_j  \\
  i_{J_1} : \coprod_{j \in J_1} U_j \to  \coprod_{j \in J} U_j \\
 i_{k,J_1}: U_k \to  \coprod_{j \in J_1} U_j  \\
  \end{cases}
\end{equation*}
are linked by the equality: $ i_k = i_{J_1} \circ  i_{k,J_1}$. 
\end{itemize}
\end{conv}
\ \\
Before giving the proof we make some few observations.

\begin{obs}
Let $(A,B)$ be a pair of elements of $X$ and $t \in \S_{\ol{X}}(A,B)$. Denote by  $\d$ the degree (or length) of $t$.\ \\ 

Consider the set $\tx{Dec}(t)$ of all decompositions or `presentations' of $t$ given by:
$$\tx{Dec}(t)= \coprod_{0 \leq l \leq \d-\1 } \{ (t_0, \cdots ,t_l), \tx{with $t_0 \otimes \cdots \otimes t_l = t$} \}$$

where for $l=0$ we have $t_0=t_l=t$. \\

Given $t' \in  \S_{\ol{X}}(A,B)$ of length $\d'$ and a morphism $u: t \to  t'$ (hence $\d'\leq \d$) then for any $(t'_0, \cdots ,t'_l) \in \tx{Dec}(t')$, there exists a unique  $(t_0, \cdots ,t_l) \in \tx{Dec}(t)$ together with a unique $(l+1)$-tuple of morphisms $(u_0, \cdots ,u_l)$ with $u_i : t_i \to t'_i$ such that: 
$$u= u_0 \otimes \cdots \otimes u_l.$$ 

This follows from the fact in  $\S_{\ol{X}}$ the composition is a concatenation of chains `side by side' which is a generalization of the ordinal addition in $(\Delta_{epi}, +,0)$. In fact by construction each $\S_{\ol{X}}(A,B)$ is a category of elements of a functor from $\Delta_{epi}$  to the category of sets. In particular the morphisms in $\S_{\ol{X}}(A,B)$ are parametrized by the morphism of $\Delta_{epi}$ and we clearly have this property of decomposition of morphisms in $(\Delta_{epi}, +)$. \ \\

It follows that any map $u: t \to  t'$ of $\S_{\ol{X}}(A,B)$ determines a unique function $\tx{Dec}(u): \tx{Dec}(t') \to \tx{Dec}(t)$. Moreover it's not hard to see that if we have two composable maps $u:t \to t'$, $u': t' \to t''$ then $\tx{Dec}(u' \circ u)= \tx{Dec}(u') \circ \tx{Dec}(u)$.\ \\
\end{obs}
\begin{rmk}
One can observe that for $t \in \S_{\ol{X}}(A,B)$ and $s \in \S_{\ol{X}}(B,C)$ we have a canonical map of sets
 \begin{equation*}
\tx{Dec}(s) \times \tx{Dec}(t) \to \tx{Dec}(s \otimes t)
 \end{equation*}
which is injective, so that we can view $\tx{Dec}(s) \times \tx{Dec}(t)$ as a subset of $\tx{Dec}(s \otimes t)$.\ \\
And more generally for each $(t_0, \cdots ,t_l) \in \tx{Dec}(t)$ with $l>0$ we can identify $\tx{Dec}(t_0) \times \cdots \times  \tx{Dec}(t_l)$ with a subset of $\tx{Dec}(t)$.
\end{rmk} 
\
\subsubsection{Proof of lemma \ref{adjoint-ub} }
Let $\X=(\X_{AB})$ be an object of   $\prod_{(A,B) \in X^{2}}\Hom[\S_{\ol{X}}(A,B)^{op}, \M]$.\\
To prove the lemma we will proceed as follows.
\begin{itemize}
\item[$\ast$] First we give the construction of the components $\Gamma[\X]_{AB}$.
\item[$\ast$] Then we define the laxity maps $\xi_{s,t}$.
\item[$\ast$] Finally we check that we have the universal property i.e that the functor $F \mapsto \Hom[\X, \Ub(F)]$ is co-represented by $\Gamma[\X]$.
\end{itemize}
\paragraph{The components $\Gamma[\X]_{AB}$}\ \\
\begin{enumerate}[label=\arabic*., align=left, leftmargin=*, noitemsep]
\item We define $\Gamma[\X](t)$  by induction on  degree of $t$ by:\\
\begin{itemize}
\item $\Gamma[\X](t)= \X(t)$  if $deg(t)=\1$, i.e $t=(A,B)$ for some $(A,B) \in X^2$ \\
\item And if $deg(t)>\1$ then we set $$\Gamma[\X](t) = \X(t) \coprod_{l > 0, (t_0, \cdots ,t_l) \in \tx{Dec}(t)} \Gamma[\X](t_0) \otimes \cdots \otimes \Gamma[\X](t_l).$$  
\end{itemize}
\ \\
This formula is  well defined since for every $(t_i)_{0 \leq i \leq l} \in \tx{Dec}(t)$  with $l>0$ we have $deg(t_i)< deg(t)$ and therefore each $\Gamma[\X](t_i)$ is already defined by the induction hypothesis.\\

\item Note that by construction we have the following canonical maps:
\begin{equation}\label{eta-t}
\begin{cases}
\Gamma[\X](t_0) \otimes \cdots \otimes \Gamma[\X](t_l) \xrightarrow{\xi_{(t_0,\cdots,t_l)}} \Gamma[\X](t) & \tx{with $l>0$}\\
 \X(t) \xrightarrow{\eta_{t}} \Gamma[\X](t) & 
\end{cases}
\end{equation}
\\
\item Given a map $u: t \to t'$ of $\S_{\ol{X}}(A,B)$, we also define the map $\Gamma[\X](u): \Gamma[\X](t') \to \Gamma[\X](t)$ by induction. \ \\ 
 \begin{itemize}
\item If $t$ is of degree $\1$  we take $\Gamma[\X](u)=\X(u)$.\\
\item If the degree of $t$ is $ > \1$, for  each  $(t_0, \cdots ,t'_l) \in \tx{Dec}(t')$  we have a unique $(t_0, \cdots ,t_l) \in \tx{Dec}(t)$ together with maps
 $u_i : t_i \to t'_i$ such that $u= u_0 \otimes \cdots \otimes u_l$.\ \\
\end{itemize}
By the induction hypothesis all the maps $ \Gamma[\X](u_i) : \Gamma[\X](t'_i) \to \Gamma[\X](t_i) $ are defined, and we can consider the maps: 
 \begin{equation}\label{prod_x(u)}
 \begin{cases}
 \Gamma[\X](t'_0) \otimes \cdots \otimes \Gamma[\X](t'_l) \xrightarrow{\otimes \Gamma[\X](u_i) } \Gamma[\X](t_0) \otimes \cdots \otimes \Gamma[\X](t_l) & \tx{with $l>0$}\\
\X(u): \X(t') \to \X(t) & \\
\end{cases}
\end{equation}
The composite of the maps in \eqref{prod_x(u)} followed by the maps in \eqref{eta-t} gives the following
\begin{equation}\label{x(u)}
\begin{cases}
\Gamma[\X](t'_0) \otimes \cdots \otimes \Gamma[\X](t'_l)  \xrightarrow{\xi_{(t_0,\cdots ,t_l)} \circ (\otimes \Gamma[\X](u_i)) } \Gamma[\X](t)  &  \tx{with $l>0$} \\
\X(t') \xrightarrow{\eta_{t} \circ \X(u)} \Gamma[\X](t) & 
\end{cases}
\end{equation}
Finally using the universal property of the coproduct, we know that the maps in \eqref{x(u)} determines  a unique map:
$$\Gamma[\X]_{AB}(u): \Gamma[\X]_{AB}(t') \to \Gamma[\X]_{AB}(t).$$\\
%%%%%%%
\item It's not hard to check that these data determine a functor $\Gamma[\X]_{AB}: \S_{\ol{X}}(A,B)^{op} \to \M$. \\
\item We leave the reader to check that all the maps in \eqref{eta-t} are natural in all variable $t_i$ including $t$. In particular the maps $\eta_t$ determine a natural transformation 
$$\eta_{AB}: \X_{AB} \to \Gamma[\X]_{AB}$$ 
These natural transformations $\{ \eta_{AB} \}_{(A,B) \in X^2}$ will constitute the unity of the adjunction.
\end{enumerate}
\ \\
\begin{rmk}\label{equiv-deux-def}
We can define alternatively $\Gamma[\X]$ without using induction by the formula:
$$\Gamma[\X](t) =  \coprod_{(t_0, \cdots ,t_l) \in \tx{Dec}(t)} \X(t_0) \otimes \cdots \otimes \X(t_l).$$
where we include also the case $l=0$ to have $\X(t)$ in the coproduct. \ \\

It's not hard to check that the $\Gamma[\X]$  we get by this formula and the previous one are naturally isomorphic. But for simplicity we will work with the definition by induction. 
\end{rmk}
%Les morphismes de laxité%%%%%%%%%%%%
\paragraph{The laxity maps} 
The advantage of the definition by induction is that we have `on the nose' the  laxity maps which correspond to the canonical maps :
 $$\Gamma[\X](s) \otimes \Gamma[\X](t)  \xrightarrow[(s,t) \in \tx{Dec}(s \otimes t)]{\xi_{s,t}} \Gamma[\X](s \otimes t)$$
for all pair of composable morphisms $(s,t)$. And one can check that these laxity maps satisfy the coherence axioms of a lax morphism, so that $\Gamma[\X]$ is indeed an  $\S_{\ol{X}}$-diagram. \ \\ 

Given two objects $\X$, $\X'$ with a morphism $\delta: \X  \to \X'$, one defines $\Gamma(\delta)=\{ \Gamma(\delta)_t\}$ with:
 $$\Gamma(\delta)_t= \delta_t \coprod_{(t_0, \cdots ,t_l) \in \tx{Dec}(t),l>0} \Gamma(\delta)_{t_0} \otimes \cdots \otimes \Gamma(\delta)_{t_l}.$$ 
 
We leave the reader check that $\Gamma$ is a functor.
\begin{rmk}\label{gamma-preserve-we}
If $\M$ is a symmetric monoidal model category such that all the objects are cofibrant, then by the pushout-product axiom one has that the class of (trivial) cofibrations is closed under tensor product. It's also well known that  (trivial) cofibrations are also closed under coproduct. If we combine these two fact one clearly see that if $\delta: \X \to \X'$ is a level-wise (trivial) cofibration, then $\Gamma\delta$ is also a level-wise (trivial) cofibrations.\ \\

The level-wise (trivial) cofibrations are precisely the (trivial) cofibration on $\kxinj$ (=  $\kx$ equipped with the  \emph{injective model structure}).
We have also by  Ken Brown lemma  (see \cite[Lemma 1.1.12]{Hov-model}) that both $\otimes$ and $\coprod$  preserve the weak equivalences between cofibrant objects; thus if $\delta$ is a level-wise weak equivalence then so is  $\Gamma \delta$. 
\end{rmk}

\paragraph{The universal property}
It remains to show that we have indeed an isomorphism of set: 
$$ \Hom[\Gamma[\X], F] \cong \Hom[\X,\Ub(F)]$$
\ \\
Recall that the functor $\Ub$ forgets only the laxity maps, so for any  $F, G \in \M_{\S}(X)$ we have:
\begin{itemize}
\item[-] $(\Ub F)_{AB}= F_{AB}$,
\item[-] For any  $\sigma \in \Hom_{\M_{\S}(X)}(F,G)$  then  $\Ub[\sigma] = \sigma.$
\end{itemize} 
So when there is no confusion, we will write $F$ instead of $\Ub(F)$ and $\sigma$ instead of $\Ub(\sigma)$. %And to simplify the notations we will sometime omit the subscript `AB' in $F_{AB}, \Gamma[\X]_{AB}, \sigma_{AB}, \eta_{AB}$, etc.
\ \\

Consider $\eta: \X \to \Gamma[\X]$ the canononical map appearing in the construction of $\Gamma[\X]$. $\eta$ is actually a map from $\X$ to $\Ub(\Gamma[\X])$. \ \\
\ \\
Let $\theta :  \Hom[\Gamma[\X], F ] \to \Hom[\X,\Ub(F)] $ be the function defined by:
$$\theta(\sigma)=  \Ub(\sigma) \circ \eta$$ 
with $\theta(\sigma)_t= \sigma_t \circ \eta_t$. By abuse of notations we will write $\theta(\sigma)= \sigma \circ \eta$.  \ \\

\begin{claim}
 $\theta$ is one-to-one and onto, hence an isomorphism of sets
\end{claim}

\paragraph{Injectivity of $\theta$}\ \\
Suppose we have $\sigma, \sigma' \in \Hom[\Gamma[\X], F]$ such that $\theta(\sigma)= \theta(\sigma')$ . We proceed by induction on the degree of $t$ to show that for all $t$, $\sigma_t=\sigma'_t$.\\

$\bullet$ For $t$ of degree $\1$ we have $\Gamma[\X](t)= \X(t)$ and  $\eta_t= \Id_{\Gamma[\X](t)}$ therefore $\theta(\sigma)_t= \sigma_t$ and $\theta(\sigma')_t= \sigma'_t$. The assumption $\theta(\sigma)_t= \theta(\sigma')_t$ gives $\sigma_t=\sigma'_t$.\ \\

$\bullet$  Let $t$ be of degree $\d > \1$ and assume that  $\sigma_w=\sigma'_w$ for all $w$ of degree $\leq \d-\1$. We will denote for short by $\tx{Dec}(t)^*$ the set $\tx{Dec}(t) - \{t \}$ which is exacltly the set of all $(w_0, \cdots ,w_p)$ in $\tx{Dec}(t)$  with $p>0$. \ \\ 

For each $(w_0, \cdots ,w_p) \in \tx{Dec}(t)^*$ we have two canonical maps: 
\begin{equation}
 \begin{cases}
\Gamma[\X](w_0) \otimes \cdots \otimes \Gamma[\X](w_p) \xrightarrow{\xi_{(w_0,\cdots,w_p)}} \Gamma[\X](w_0 \otimes \cdots  \otimes w_p)=\Gamma[\X](t) & \\
F(w_0) \otimes \cdots \otimes F(w_p) \xrightarrow{\varphi_{(w_0,\cdots,w_p)}}F(w_0 \otimes \cdots  \otimes w_p)=F(t) & \\
 \end{cases}
\end{equation}
\\
The map $\xi_{(w_0,\cdots,w_p)}$ is the one in \eqref{eta-t} and the map $\varphi_{(w_0,\cdots,w_p)}$ is uniquely determined by  laxity maps  of $F$ and their coherences,  together with the bifunctoriality of the product $\otimes$ and it associativity. We remind the reader that the choice and order of composition of these maps (laxity, associativity, identities)  we use to build $\varphi_{(w_0,\cdots,w_p)}$  doesn't matter (Mac Lane coherence theorem \cite[Ch. 7]{Mac}).\ \\

Now by definition of a morphism of $\S_{\ol{X}}$-diagrams, for any $(w_0,w_1)$ with $w_0 \otimes w_1= t$ the following diagram commutes:

\[
\xy
(-10,20)*+{ \Gamma[\X](w_0) \otimes  \Gamma[\X](w_1)}="A";
(30,20)*+{ \Gamma[\X](t)}="B";
(-10,0)*+{F(w_0) \otimes F(w_1)}="C";
(30,0)*+{F(t)}="D";
{\ar@{->}^{\ \ \ \ \ \xi_{w_0,w_1}}"A";"B"};
{\ar@{->}_{\sigma_{w_0} \otimes \sigma_{w_1}}"A";"C"};
{\ar@{->}^{\sigma_{t}}"B";"D"};
{\ar@{->}^{\varphi_{w_0,w_1}}"C";"D"};
\endxy
\]
\\
And using again a `Mac Lane coherence style' argument we have a general commutative diagram for each $(w_0, \cdots ,w_p) \in \tx{Dec}(t)$:

\begin{equation}\label{eq-lax-sigma-t}
\xy
(-10,20)*+{  \Gamma[\X](w_0) \otimes \cdots \otimes \Gamma[\X](w_p)}="A";
(40,20)*+{ \Gamma[\X](t)}="B";
(-10,0)*+{F(w_0) \otimes \cdots \otimes F(w_p)}="C";
(40,0)*+{F(t)}="D";
{\ar@{->}^{\ \ \ \ \ \ \ \xi_{(w_0,\cdots,w_p)}}"A";"B"};
{\ar@{->}_{\bigotimes\sigma_{w_i}}"A";"C"};
{\ar@{->}^{\sigma_{t}}"B";"D"};
{\ar@{->}^{\varphi_{(w_0,\cdots,w_p)}}"C";"D"};
\endxy
\end{equation}
\\
If we replace $\sigma$ by $\sigma'$ everywhere we get a commutative diagram of the same type.\ \\
\ \\ %Now if we add to each of the two diagram 
Let's denote by $\tx{Diag}_t(\sigma)$ and $\tx{Diag}_t(\sigma')$ the set of maps:
$$\tx{Diag}_t(\sigma)= \left\lbrace \varphi_{(w_0,\cdots,w_p)} \circ (\otimes\sigma_{w_i}), (w_0, \cdots ,w_p) \in \tx{Dec}(t)^* \right\rbrace \sqcup \left\lbrace \theta(\sigma)_t \right\rbrace $$
$$\tx{Diag}_t(\sigma')=\left\lbrace \varphi_{(w_0,\cdots,w_p)} \circ (\otimes\sigma'_{w_i}), (w_0, \cdots ,w_p) \in \tx{Dec}(t)^* \right\rbrace  \sqcup \left\lbrace \theta(\sigma')_t \right\rbrace .$$
\\
Using the induction hypothesis $\sigma_{w_i} =\sigma'_{w_i}$  and the fact that $\theta(\sigma)=\theta(\sigma')$ we have:
\begin{equation*}
\begin{cases}
\varphi_{(w_0,\cdots,w_p)} \circ (\otimes\sigma_{w_i}) = \varphi_{(w_0,\cdots,w_p)} \circ (\otimes\sigma'_{w_i}) & \tx{for all $(w_0, \cdots ,w_p) \in \tx{Dec}(t)^* $}\\
\theta(\sigma)_t = \theta(\sigma')_t & \\
\end{cases}
\end{equation*}
so that $\tx{Diag}_t(\sigma) = \tx{Diag}_t(\sigma')$.\ \\
\ \\
The universal property of the coproduct says that there exists a \textbf{unique} map $\delta_t:  \Gamma[\X](t) \to F(t)$ such that:
\begin{equation*}
\begin{cases}
\varphi_{(w_0,\cdots,w_p)} \circ (\otimes\sigma_{w_i}) =  \delta_t \circ \xi_{(w_0,\cdots,w_p)}& \\
\varphi_{(w_0,\cdots,w_p)} \circ (\otimes\sigma'_{w_i}) = \delta_t \circ \xi_{(w_0,\cdots,w_p)} &\\
\theta(\sigma)_t = \delta_t \circ \eta_t & \\
\theta(\sigma')_t= \delta_t \circ \eta_t & \\
\end{cases}
\end{equation*}
But we know from the commutative diagrams \eqref{eq-lax-sigma-t} that both  $\sigma_t$ and $\sigma'_t$ satisfy these relations so by unicity we have $\sigma_t=\delta_t=\sigma'_t$ .\\ 

By induction we have the equality $\sigma_t= \sigma'_t$ for all $t$ which means that $\sigma=\sigma'$ and $\theta$ is injective. \ \\
\begin{rmk}\label{rmk-co-unit}
The set of maps:
\begin{equation*}
\begin{cases}
\varphi_{(w_0,\cdots,w_p)}: F(w_0) \otimes \cdots \otimes F(w_p) \to F(t) & \tx{for all  $(w_0,\cdots,w_p) \in \tx{Dec}(t)^*$} \\
\Id_{F(t)}: F(t) \to F(t) &\\
\end{cases}
\end{equation*}
determines by the universal property of the coproduct a unique map :
$$\varepsilon'_{t}: \coprod_{\tx{Dec}(t) } F(w_0) \otimes \cdots \otimes F(w_p)  \to F(t).$$ 

Note that the source of that map is a coproduct taken on $\tx{Dec}(t)$ i.e including $t$. We have the obvious factorizations of $\varphi_{(w_0,\cdots,w_p)}$ and $\Id_{F(t)}$ through 
$\varepsilon'_{t}$. From  the Remark \ref{equiv-deux-def} we have
 $$\Gamma[\Ub F](t) \cong  \coprod_{\tx{Dec}(t) } F(w_0) \otimes \cdots \otimes F(w_p)$$ and $\varepsilon'_t$ gives a map $\varepsilon_t: \Gamma[\Ub F](t) \to F(t)$.  That map $\varepsilon_t$ will constitute the co-unit of the adjunction. 
\end{rmk}

%%% mettre une remarque ici pour indiquer la co-unité %%
\paragraph{Surjectivity of $\theta$}\ \\
Let $\pi: \X \to \Ub(F)$ be an element of $\Hom[\X,\Ub(F)]$. In the following we construct by induction a morphism of $\S_{\ol{X}}$-diagrams $\sigma: \Gamma[\X] \to F$ such that $\theta(\sigma)= \pi$.\ \\
Let $t$ be a morphism in $\S_{\ol{X}}$ of degree $\d$.\ \\
\ \\
$\bullet$ For $\d=\1$ since $\Gamma[\X](t)= \X(t)$ and $\eta_t= \Id_{\Gamma[\X](t)}$, we set $\sigma_t= \pi_t$. We  have: 
$$\theta(\sigma)_t=  \sigma_t \circ \eta_t = \pi_t \circ \Id_{\Gamma[\X](t)}= \pi_t.$$ 
\\
$\bullet$ For $\d > \1$,  let's assume that we've construct $\sigma_{w_i}: \Gamma[\X](w_i) \to F(w_i)$ for all $w_i$ of degree $\leq \d-\1$ such that:
$$ \theta(\sigma)_{w_i}=  \sigma_{w_i} \circ \eta_{w_i} = \pi_{w_i}.$$ 
\\
Using the  the universal property of the copruduct with respect to the following set of maps: 

 \begin{equation*}
\begin{cases}
\varphi_{(w_0,\cdots,w_p)} \circ (\otimes\sigma_{w_i})  :\Gamma[\X](w_0) \otimes \cdots \otimes \Gamma[\X](w_p) \to F(t) & \tx{for all  $(w_0,\cdots,w_p) \in \tx{Dec}(t)^*$} \\
\pi_t : \X(t) \to F(t) &\\
\end{cases}
\end{equation*}
we have a unique map $\sigma_t: \Gamma[\X](t) \to F(t) $ such that the following factorizations hold:
 \begin{equation*}
\begin{cases}
\varphi_{(w_0,\cdots,w_p)} \circ (\otimes\sigma_{w_i})= \sigma_t \circ \xi_{(w_0,\cdots,w_p)}  & \tx{for all  $(w_0,\cdots,w_p) \in \tx{Dec}(t)^*$} \\
\pi_t= \sigma_t \circ \eta_t. &\\
\end{cases}
\end{equation*}
\\
These factorizations implies that all the diagrams of type \eqref{eq-lax-sigma-t} are commutative. So by induction  for all $t$ we have these commutative diagrams which gives the required axioms for transformations of lax-morphisms (see Definition \ref{mor-s-diag}).\ \\

Moreover we clear have by construction that $\theta(\sigma)_t= \sigma_t \circ \eta_t= \pi_t$ for all $t$, that is $\theta(\sigma)= \pi$ and $\theta$ is surjective. \ \\

We leave the reader to check that all the constructions considered in the previous paragraphs are natural in both $\X$ and $F$. $\qed$

\subsection{Evaluations on morphism have left adjoint}\label{lemme_adjunction_2}
In the following $\M$ represents a cocomplete category.  Given a small category $\Ca$ and a morphism $\alpha$ of $\Ca$ we will denote by $\Ev_{\alpha}: \Hom(\Ca, \M) \to \M^{[\1]}$ the functor that takes $\Fa \in \Hom(\Ca, \M)$ to $\Fa(\alpha)$.\ \\
 
Let $[\1]=(0 \to 1)$ be the interval category. A morphism $\alpha$ of $\Ca$ can be identified with a functor denoted again $\alpha : [\1] \to \Ca$, that takes $0$ to the source of $\alpha$ and $1$ to the target. \ \\
Then we can identify $\Ev_{\alpha}$ with the pullback functor $\alpha^{\star}: \Hom(\Ca, \Ma) \xrightarrow{\alpha \circ } \Hom([\1],\M)=\M^{2}$. With this identification one easily establish that 

\begin{lem}
For a small category $\Ca$ and any morphism $\alpha \in \Ca$ the functor  $\Ev_{\alpha}$ has a left adjoint $\Fb^{\alpha}: \M^{[\1]} \to \Hom(\Ca, \M)$. 
\end{lem}

\begin{proof}
This is a special case of a more general situation where we have a functor $\alpha:  \Ba \to \Ca$;  the left adjoint of $\alpha^{\star}$ is $\alpha_{!}$ which corresponds to the left Kan extension functor along $\alpha$. The reader can find a detailed proof for example in  \cite[A.2.8.7]{Lurie_HTT}, \cite[Ch. 7.6.1]{Simpson_HTHC}, \cite[Ch. 11.6]{Hirsch-model-loc}. 
\end{proof}

\begin{cor}\label{adjoint-preserve-proj-cof}
$\Fb^{\alpha}$ is a left Quillen functor between the corresponding projective model structures. 
\end{cor}

\begin{proof}
A left adjoint is a left Quillen functor if and only if its right adjoint is a right Quillen functor (see lemma \cite[1.3.4]{Hov-model}); consequently $\Fb^{\alpha}$ is a left Quillen functor if and only if $\Ev_{\alpha}$ is a right Quillen functor. 

But in the respective projective model structure on $\Hom(\Ca, \M)$ and $\M^{[\1]} =\Hom([\1], \M)$, the (trivial) fibrations are the level-wise ones, so $\Ev_{\alpha}$ clearly preserves them. Thus $\Ev_{\alpha}$ is a right Quillen functor and the result follows. 
\end{proof}

\begin{rmk}
From Ken Brown lemma (\cite[Lemma 1.1.12]{Hov-model}), $\Fb^{\alpha}$ preserves weak equivalences between cofibrant objects. Recall that a cofibrant object in $\arproj$ is  a cofibration in $\M$ with a cofibrant domain. In particular if all the objects of $\M$ are cofibrant, then the cofibrant objects in $\arproj$ are simply all the cofibrations of $\M$.  Therefore if all the objects of $\M$ are cofibrant then $\Fb^{\alpha}$ preserves weak equivalence between any two cofibrations. 
\end{rmk}

Our main interest in the above lemma is when $\Ca=  \S_{\ol{X}}(A,B)^{op}$ and $\alpha= u_t$, the unique morphism in $\S_{\ol{X}}(A,B)^{op}$ from $(A,B)$ to $t$ (see Observation \ref{Co-Segal-final-cond}).\ \\

Each category $\S_{\ol{X}}(A,B)^{op}$ is an example of \emph{direct category}, that is a a category $\Ca$ equiped with a \emph{linear extension functor} $\degb: \Ca \to \lambda$, where $\lambda$ is an ordinal. One requires furthermore that $\degb$ takes nonidentiy maps to nonidentiy maps; this way any nonidentity map raises the degree. Note that $\S_{\ol{X}}(A,B)^{op}$ has an initial object $e$ which corresponds to $(A,B)$. \ \\

For such categories $\Ca$ one has the following
\begin{lem}\label{adjoint_ev_inj}
Let $\Ca$ be a direct category. Then for any model category $\M$, the adjunction 
$$\Fb^{\alpha}: \M^{[\1]} \rightleftarrows \Hom(\Ca,\M): \Ev_{\alpha}$$
 is also a Quillen adjunction with the injective model structures on each side.
\end{lem}

\begin{cor}\label{lem-pres-level-wise-cof}
For any $t \in \S_{\ol{X}}(A,B)$  the functor 
 $$\Fb^{u_t}: \M^{[\1]}_{\tx{inj}} \to \Hom[\S_{\ol{X}}(A,B)^{op} ,\M]_{\tx{inj}}$$ is a left Quillen functor.
\end{cor}

\begin{proof}[Proof of Lemma \ref{adjoint_ev_inj}]
We assume that $\alpha$ is an nonidentity map since the assertion is obvious when $\alpha$ is an identity. 
Let $c_0= \alpha(0)$ and $c_1= \alpha(1)$ be respectively the source and target of $\alpha$, and  $h : U \to V$  be an object in  $\M^{[\1]}= \Hom([\1],\M)$. Note that since $\alpha$ is an nonidentity map we have $\degb(c_0)< \degb(c_1)$, thus $\Hom(c_1, c_0)= \varnothing$.\ \\

By definition of the left Kan extension along $\alpha$ one defines $\Fb^{\alpha}_{h}(c)$ as:
$$\Fb^{\alpha}_{h}(c)= \colim_{(\alpha(i)\xrightarrow{k} c) \in \alpha_{/c}} \alpha(i)$$
 
where $\alpha_{/c}$ represents the over category (see \cite{Mac}, \cite[7.6.1]{Simpson_HTHC}). Let $D_{\alpha}(c_0,c) \subset \Hom(c_0,c)$ be the subset of morphisms that factorizes through $\alpha$. One can check that the previous colimit  is the following coproduct:
$$\Fb^{\alpha}_{h}(c)= (\coprod_{f \in D_{\alpha}(c_0,c)} V) \coprod  (\coprod_{f \notin D_{\alpha}(c_0,c) }  U) \coprod (\coprod_{f \in \Hom(c_1,c)} V) .$$

In the above expression, when the set indexing the coproduct is empty, then the coproduct is the initial object of $\M$.\ \\

Given $g: c \to c'$ the structure map $\Fb^{\alpha}_{h}(g): \Fb^{\alpha}_{h}(c) \to  \Fb^{\alpha}_{h}(c')$ is given as follows:
\begin{enumerate}
\item  On $\coprod_{f \in D_{\alpha}(c_0,c)} V$, one sends the $V$-summand corresponding to $f: c_0 \to c$ to the $V$-summand corresponding to $gf: c_0 \to c'$ by the identity $\Id_V$. Note that this is well defined since $gf$ factorizes also through $\alpha$.
\item  On $\coprod_{f \notin D_{\alpha}(c_0,c) }  U$, one sends the $U$-summand corresponding to $f: c_0 \to c$ to: 
  \begin{itemize}[label=$-$, align=left, leftmargin=*, noitemsep]
  \item the $V$-summand corresponding to $gf$ by the morphism $h: U \to V$, if $gf \in D_{\alpha}(c_0,c')$. 
  \item the $U$-summand corresponding to $gf$, if $gf \notin D_{\alpha}(c_0,c')$ by the morphism $\Id_U$.
  \end{itemize}
\item On $\coprod_{f \in \Hom(c_1,c)} V$ we send the $V$-summand indexed by $f$ to the $V$-summand in $\Fb^{\alpha}_{h}(c')$ corresponding to $gf$ by the morphism $\Id_V$. 
\item If one of the coproduct vanish to the initial object $\varnothing$ of $\M$ then one use simply the unique map out of it.
\end{enumerate}

It follows that  given an injective (trivial) cofibration $\theta=(a,b): h \to h'$:
\[
\xy
(0,20)*+{U}="A";
(20,20)*+{U'}="B";
(0,0)*+{V}="C";
(20,0)*+{V'}="D";
{\ar@{^{(}->}^{a}"A";"B"};
{\ar@{->}_{h}"A";"C"};
{\ar@{^{(}->}_{b}"C";"D"};
{\ar@{->}^{h'}"B";"D"};
\endxy
\]

the $c$-component of $\Fb^{\alpha}_{\theta}$ is the coproduct
$$\Fb^{\alpha}_{\theta,c}= (\coprod_{f \in D_{\alpha}(c_0,c)} b) \coprod  (\coprod_{f \notin D_{\alpha}(c_0,c) }  a) \coprod (\coprod_{f \in \Hom(c_1,c)} b).$$
Since (trivial) cofibrations are closed under coproduct we deduce that $\Fb^{\alpha}_{\theta,c}$ is a (trivial) cofibration if $\theta$ is so, thus $\Fb^{\alpha}$ is a left Quillen functor as desired.
\end{proof}

\section{$\M_{\S}(X)$ is cocomplete if $\M$  is so } \label{proof-MX-cocomplete}
In this section we want to prove the following

\begin{thm}
Given a co-complete symmetric monoidal category $\M$, for any nonempty set $X$ the category  $\M_{\S}(X)$ is co-complete.
\end{thm}
The proof of this theorem follows exactly the same ideas as the proof  of the co-completeness of  $\M$-$\Cat$ given by  Wolff \cite{Wo}. \ \\

We will proceed as follows. 
\begin{itemize}
\item We will show first that $\M_{\S}(X)$ is monadic over $\K_X= \prod_{(A,B) \in X^{2}} \Hom[\S_{\ol{X}}(A,B)^{op}, \M]$ using the Beck monadicity theorem \cite{Beck_monadicity}. 
\item As the adjunction is monadic we know that $\M_{\S}(X)$ has a coequalizers of $\Ub$-split pair of morphisms and preserve them.
\item Since  $\K_X$ is cocomplete  by a result of Linton \cite[Corollary 2]{Linton-cocomplete} the cateogry of algebra of the monad, which is equivalent to $\M_{\S}(X)$, is cocomplete.
\end{itemize} 

\subsection{$\M_{\S}(X)$ has coequalizers of reflexive pairs}

The question of existence of coequalizers in $\M_{\S}(X)$ is similar to that of coequalizers in $\M$-$\Cat$ which was treated by Wolff \cite{Wo}. For our needs we only treat  the question of coequalizer of reflexive pairs. \ \\ 
Given a parallel pair of morphisms in $\M_{\S}(X)$  $(\sigma_1,\sigma_2): D \rightrightarrows F$ one can view it as defining a `relation' `$\Ra= \Im(\sigma_1 \times \sigma_2) \subset F \times F$' on $F$ . We will call such relation `precongruence'. In this situation the question is to find out when a quotient object `$E=F / \Ra$' (`the coequalizer') exists in $\M_{\S}(X)$.  We will proceed in the same maner as Wolff; below we outline the different steps before going to details. 

\begin{enumerate}
\item We will start by giving a criterion which says under which conditions the coequalizer computed in  $\K_X$ lifts to a coequalizer in $\M_{\S}(X)$. 
\item When a parallel pair of morphisms is a reflexive pair  (=the analogue of  the relation $\Ra$ to be reflexive) we will show that the conditions of the criterion are fullfiled and the result will follow.
\end{enumerate}

\subsubsection{Lifting of coequalizer} \label{lifting-coeq}
\begin{df}
Let $F$ be an object of $\M_{\S}(X)$. 
\begin{enumerate}
\item A \textbf{precongruence} in $F$ is a pair of parallel morphisms in $\K_X$
\[
\xy
(0,0)*++{A}="X";
(20,0)*++{\Ub F}="Y";
{\ar@<-0.5ex>@{->}_{\sigma_1}"X";"Y"};
{\ar@<0.5ex>@{->}^{\sigma_2}"X";"Y"};
\endxy
\]
for some object $A \in \K_X$.\ \\

\item Let $E$ be a coequalizer in $\K_X$ of $(\sigma_1, \sigma_2)$,  with $L: \Ub F \to E$ the canonical map. We say that the precongruence is a \textbf{congruence} if:
\begin{itemize}
\item $E= \Ub(\tld{E})$ for some $\tld{E} \in \M_{\S}(X)$  and 
\item $L = \Ub (\tld{L})$ for a (unique) morphism $\tld{L}: F \to \tld{E}$ in  $\M_{\S}(X)$.
\end{itemize} 
\end{enumerate}
When there is no confusion we simply write $E$ for $\tld{E}$ and $L$ for $\tld{L}$. 
\end{df}
%Intuitively, a precongruence is a relation `$\Ra= \Im(\sigma_1 \times \sigma_2) \subset F \times F$' with `$E=F / \Ra$' the quotient in $\K_X$ of $F$ by $\Ra$. The precongruence is a congruence if the quotient $E$  is an object of  $\M_{\S}(X)$.
\begin{lem}\label{precong-lemma}
Let $F$ be an object of $\M_{\S}(X)$ and consider a precongruence:
\[
\xy
(0,0)*++{A}="X";
(20,0)*++{\Ub F}="Y";
(40,0)*++{E}="Z";
{\ar@<-0.5ex>@{->}_{\sigma_1}"X";"Y"};
{\ar@<0.5ex>@{->}^{\sigma_2}"X";"Y"};
{\ar@{->}^{L}"Y";"Z"};
\endxy
\]
\\
Denote by $\varphi_{s,t}: F(s) \otimes F(t) \to F(s \otimes t)$ be the laxity maps of $F$. \ \\
\\
Then the precongruence is a congruence if and only if for any pair $(s,t$) of composable $1$-morphisms in $\S_{\ol{X}}$ the following equalities hold:
$$L_{s \otimes t} \circ [\varphi_{s,t}\circ (\Id_{F(s)} \otimes \sigma_1(t))]= L_{s \otimes t} \circ [\varphi_{s,t}\circ (\Id_{F(s)} \otimes \sigma_2(t))]$$
$$L_{s \otimes t} \circ [\varphi_{s,t}\circ (\sigma_1(s) \otimes \Id_{F(t)})]= L_{s \otimes t} \circ [\varphi_{s,t}\circ (\sigma_2(s) \otimes \Id_{F(t)})].$$

In this case the structure of $\S_{\ol{X}}$-diagram on $E$ is unique. 
\end{lem}
\
\begin{rmk}
If $F$ was (the nerve of) an $\M$-category, taking $s=(A,B)$, and $t=(B,C)$, we have $s \otimes t =(B,C)$ and the laxity maps correpond to the composition: $c_{ABC}: F(A,B) \otimes F(B,C) \to F(A,C)$. One can check that the previous conditions are the same as in \cite[Lemma 2.7]{Wo}. 
\end{rmk}

\begin{proof}[Sketch of proof]
The fact that having a congruence implies the equalities is easy  and follows from the fact that $L$ is a morphism in $\M_{\S}(X)$ and that $L$ is a coequalizer of $\sigma_1$ and $\sigma_2$. We will then only  prove that the equalities force a congruence. \ \\

To prove the statement we need to provide the laxity maps for $E$: $\phi_{s,t}: E(s) \otimes E(t) \to E(s \otimes t)$. \ \\
By definition of $E$, for any $1$-morphism $s \in \S_{\ol{X}}$, $E(s)$ is a coequalizer of $(\sigma_1(s), \sigma_2(s))$, which is a particular case of colimit in $\M$. Since $\M$ is a symmetric monoidal closed, colimits commute on each factor with the tensor product $\otimes$ of $\M$. It follows that given a pair $(s,t)$ of composable morphisms then $E(s) \otimes E(t)$ is the colimit of the `diagram':
$$\epsilon(s,t)= \{ A(s) \otimes A(t) \xrightarrow{\sigma_i(s) \otimes \sigma_j(t)} F(s) \otimes F(t) \}_{i,j \in \{1,2 \} }.$$

\ \\
Now  we claim that all the composite 
$L_{s \otimes t} \circ \varphi_{s,t} \circ  [\sigma_i(s) \otimes \sigma_j(t)]: A(s) \otimes A(t) \to E(s\otimes t)$ are equal.\\
 This equivalent to say that the diagram 
$$ \epsilon'(s,t)= L_{s \otimes t} \circ \varphi_{s,t} \circ \epsilon(s,t):=  \{L_{s \otimes t} \circ \varphi_{s,t} \circ  [\sigma_i(s) \otimes \sigma_j(t)] \}_{i,j \in \{1,2 \} }$$
is a compatible cocone. Before telling why this is true let's see how we get the laxity maps for $E$. For that it suffices to observe that since $\epsilon'(s,t):=L_{s \otimes t} \circ \varphi_{s,t} \circ \epsilon(s,t)$ is a compatible cocone, by the universal property of the colimit of $\epsilon(s,t)$ there exists a unique map  $\psi_{s,t}: E(s) \otimes E(t) \to E(s \otimes t)$ making the obvious diagrams commutative. In particular for any $s,t$ the following is commutative:
\[
\xy
(0,20)*+{F(s) \otimes F(t)}="A";
(30,20)*+{F(s \otimes t)}="B";
(0,0)*+{E(s) \otimes E(t)}="C";
(30,0)*+{E(s \otimes t)}="D";
{\ar@{->}^{\varphi_{s,t}}"A";"B"};
{\ar@{->}_{L_s \otimes L_t}"A";"C"};
{\ar@{->}^{L_{s \otimes t}}"B";"D"};
{\ar@{-->}^{\psi_{s,t}}"C";"D"};
\endxy
\]

The fact that the morphism $\psi_{s,t}$ fit coherently is left to the reader as it's straightforward: it suffices to introduce a cocone $\epsilon(s,t,u)$ whose colimit `is' $E(s) \otimes E(t) \otimes E(u)$ and use the universal property of the colimit. This shows that $(E,\psi_{s,t})$ is an object of $\M_{\S}(X)$ and that $L$ extends to a morphism in $\M_{\S}(X)$.\ \\

Now with some easy but tedious computations one gets successively the desired equalities:\ \\

\begin{equation*}
\begin{split}
L_{s \otimes t} \circ \varphi_{s,t} \circ  [\sigma_1(s) \otimes \sigma_2(t)]
&=L_{s \otimes t} \circ \varphi_{s,t} \circ  [(\Id_{F(s)} \otimes \sigma_2(t))  \circ (\sigma_1(s) \otimes \Id_{A(t)})]  \\
&=\underbrace{L_{s \otimes t} \circ  [\varphi_{s,t} \circ  (\Id_{F(s)} \otimes \sigma_2(t))]}_{=L_{s \otimes t} \circ [\varphi_{s,t} \circ  (\Id_{F(s)} \otimes \sigma_1(t))]} \circ (\sigma_1(s) \otimes \Id_{A(t)}) \\
%&=\{L_{s \otimes t} \circ [\varphi_{s,t} \circ  (\Id_{F(s)} \otimes \sigma_2(t))] \} \circ (\sigma_1(s) \otimes \Id_{A(t)}) \\
%(\ast) &=\{L_{s \otimes t} \circ [\varphi_{s,t} \circ  (\Id_{F(s)} \otimes \sigma_1(t))] \} \circ (\sigma_1(s) \otimes \Id_{A(t)}) \\
(1)&=L_{s \otimes t} \circ [\varphi_{s,t} \circ  (\Id_{F(s)} \otimes \sigma_1(t))] \circ (\sigma_1(s) \otimes \Id_{A(t)}) \\
&=\ul{ L_{s \otimes t} \circ \varphi_{s,t} \circ  [\sigma_1(s) \otimes \sigma_1(t)]} \\
\tx{From $(1)$ $\rightsquigarrow$}&= L_{s \otimes t} \circ \varphi_{s,t} \circ  [(\Id_{F(s)} \otimes \sigma_1(t)) \circ (\sigma_1(s) \otimes \Id_{A(t)})] \\
&= L_{s \otimes t} \circ \varphi_{s,t} \circ  [(\sigma_1(s) \otimes \Id_{F(t)}) \circ (\Id_{A(s)} \otimes \sigma_1(t) )] \\
&= \underbrace{L_{s \otimes t} \circ [\varphi_{s,t} \circ  (\sigma_1(s) \otimes \Id_{F(t)})]}_{=L_{s \otimes t} \circ [\varphi_{s,t} \circ  (\sigma_2(s) \otimes \Id_{F(t)})]} \circ (\Id_{A(s)} \otimes \sigma_1(t) ) \\
%&= \{L_{s \otimes t} \circ [\varphi_{s,t} \circ  (\sigma_1(s) \otimes \Id_{F(t)})] \} \circ (\Id_{A(s)} \otimes \sigma_1(t) ) \\
%(\ast \ast) &= \{L_{s \otimes t} \circ [\varphi_{s,t} \circ  (\sigma_2(s) \otimes \Id_{F(t)})] \} \circ (\Id_{A(s)} \otimes \sigma_1(t) ) \\
(2)&= L_{s \otimes t} \circ [\varphi_{s,t} \circ  (\sigma_2(s) \otimes \Id_{F(t)})] \circ (\Id_{A(s)} \otimes \sigma_1(t) ) \\
&=\ul{ L_{s \otimes t} \circ \varphi_{s,t} \circ  [\sigma_2(s) \otimes \sigma_1(t)]} \\
\tx{From $(2)$ $\rightsquigarrow$}&= L_{s \otimes t} \circ [\varphi_{s,t} \circ  (\sigma_2(s) \otimes \Id_{F(t)})] \circ (\Id_{A(s)} \otimes \sigma_1(t) ) \\
&= L_{s \otimes t} \circ \varphi_{s,t} \circ  [(\Id_{F(s)} \otimes \sigma_1(t)) \circ (\sigma_2(s)  \otimes \Id_{A(t)} )] \\
&= \underbrace{L_{s \otimes t} \circ [\varphi_{s,t} \circ  (\Id_{F(s)} \otimes \sigma_1(t))}_{=L_{s \otimes t} \circ [\varphi_{s,t} \circ  (\Id_{F(s)} \otimes \sigma_2(t))} \circ (\sigma_2(s)  \otimes \Id_{A(t)}) \\
&=  L_{s \otimes t} \circ [\varphi_{s,t} \circ  (\Id_{F(s)} \otimes \sigma_2(t))  \circ (\sigma_2(s)  \otimes \Id_{A(t)}) \\
&=\ul{ L_{s \otimes t} \circ \varphi_{s,t} \circ  [\sigma_2(s) \otimes \sigma_2(t)] } \\
\end{split}
\end{equation*}
\end{proof}

Following Linton \cite{Linton-cocomplete} we introduce the
\begin{df}
Let $\sigma_1, \sigma_2: A \rightrightarrows \Ub F$ be a parallel pair of morphisms in $\K_X$ i.e a precongruence in $F$. We will say that a morphism $L: F \to E$ in $\M_{\S}(X)$ is a \textbf{coequalizer relative to $\Ub$} if:
\begin{enumerate}
\item If $\Ub L \circ \sigma_1= \Ub L \circ \sigma_2$ and 
\item if for any morphism  $L': F \to Z$  in $\M_{\S}(X)$ which satisfies  $\Ub L' \circ \sigma_1= \Ub  L' \circ \sigma_2$ then there exists a unique morphism $H: E \to Z$ in   $\M_{\S}(X)$ such that $L'= H \circ L$.
\end{enumerate}
\end{df}

\begin{lem}\label{rel-coeq}
If a precongruence 
\[
\xy
(0,0)*++{A}="X";
(20,0)*++{\Ub F}="Y";
(40,0)*++{E}="Z";
{\ar@<-0.5ex>@{->}_{\sigma_1}"X";"Y"};
{\ar@<0.5ex>@{->}^{\sigma_2}"X";"Y"};
{\ar@{->}^{L}"Y";"Z"};
\endxy
\]
 is a congruence then the morphism $L: F \to E$ is a coequalizer rel. to  $\Ub$.
\end{lem}
\begin{proof}
Obvious: follows from the construction of $(E, \psi_{s,t})$.

\end{proof}

\begin{lem}[Linton]\label{Linton-relative-coeq}
Let $U: D \to C$ be  a \ul{faithful} functor and $\tld{f},\tld{g}: A \to B$  be a pair of parallel of morphisms  in $D$. Denote by  $f= U\tld{f}$ and $g= U\tld{g}$. Then for $p: A \to E$ in $D$,  the following are equivalent:
\begin{itemize}
\item $p$ is a coequalizer  of $(\tld{f},\tld{g})$,
\item $p$ is a coequalizer rel. to $U$ of $(f,g)$. 
\end{itemize}
\end{lem}

\begin{proof}
See \cite[Lemma 1]{Linton-cocomplete}.
\end{proof}

For a given $F \in \M_{\S}(X)$, among the precongruences defined in $F$  we have the ones coming from parallel pair of morphisms in $\M_{\S}(X)$ namely:
\[
\xy
(0,0)*++{\Ub D}="X";
(20,0)*++{\Ub F}="Y";
{\ar@<-0.5ex>@{->}_{\Ub \tld{\sigma}_1}"X";"Y"};
{\ar@<0.5ex>@{->}^{\Ub \tld{\sigma}_2}"X";"Y"};
\endxy
\]

for some $D \in \M_{\S}(X)$.\ \\

The following lemma tells about these congruences.

\begin{lem}\label{coeq-split-pair}
Let $F$ be an object of $\M_{\S}(X)$. Consider a precongruence in $F$

\[
\xy
(0,0)*++{\Ub D}="X";
(20,0)*++{\Ub F}="Y";
(40,0)*++{E}="Z";
{\ar@<-0.5ex>@{->}_{\Ub \tld{\sigma}_1}"X";"Y"};
{\ar@<0.5ex>@{->}^{\Ub\tld{\sigma}_2}"X";"Y"};
{\ar@{->}^{L}"Y";"Z"};
\endxy.
\]

If  there exists a  \textbf{split} i.e a morphism $p: \Ub F \to \Ub D$ in $\K_X$  such that $ \Ub \tld{\sigma}_1 \circ p= \Ub \tld{\sigma}_2 \circ p = \Id_{\Ub F}$
then:
\begin{itemize}
\item the precongruence is a congruence and hence
\item  $L: F \to E$ is a coequalizer in $\M_{\S}(X)$ of the pair $(\tld{\sigma}_1,\tld{\sigma}_2): D \rightrightarrows F$ (and is obviously preserved by $\Ub$).
\end{itemize} 
\end{lem}

\begin{proof}
We will simply need to show that the equalities of Lemma \ref{precong-lemma} holds. We will reduce the proof to  the first equalities since the second ones are treated in the same manner by simply permuting $\Id_F \otimes \sigma_i$ to $\sigma_i \otimes \Id_F$. \ \\

Setting $\sigma_1 =  \Ub \tld{\sigma}_1$ and $\sigma_2= \Ub \tld{\sigma}_2$, these equalities become:
$$L_{s \otimes t} \circ [\varphi_{s,t}\circ (\Id_{F(s)} \otimes \sigma_1(t))]= L_{s \otimes t} \circ [\varphi_{s,t}\circ (\Id_{F(s)} \otimes \sigma_2(t))]$$

And to simplify the notations we will remove the letters `$s \otimes t, s, t$' but will mention `$\varphi_{_F}$' or `$\varphi_{_D}$' for the laxity maps of $F$ and $D$ repsectively. The previous equality will be then written, when there is no confusion, as follows: 
$$L \circ [\varphi_{_F} \circ (\Id_F \otimes \sigma_1)]= L \circ [\varphi_{_F} \circ (\Id_F \otimes \sigma_2)].$$
\ \\

Since $\sigma_1 \circ p= \Id_F = \sigma_2 \circ p$, we have that 
\begin{equation*}
\begin{split}
\Id_F \otimes \sigma_1
&= (\sigma_1 \circ p) \otimes (\sigma_1 \circ \Id_D)  \\
&=(\sigma_1 \otimes \sigma_1) \circ (p \otimes \Id_D). \\
\end{split}
\end{equation*}

Similarly for $\sigma_2$: $\Id_F \otimes \sigma_2=(\sigma_2 \otimes \sigma_2) \circ (p \otimes \Id_D)$. \\

With these observations we can compute
\begin{equation*}
\begin{split}
L \circ [\varphi_{_F} \circ (\Id_F \otimes \sigma_1)]
&=L \circ \varphi_{_F} \circ (\sigma_1 \otimes \sigma_1) \circ (p \otimes \Id_D) \\
 &= L \circ \sigma_1 \circ \varphi_{_D} \circ (p \otimes \Id_D) \ \ \ \ \ \ \  \ \ \ \ \ \ \ \ \ \ \ \ \ \ \ \  \ \ \ \ \ \ \  \ \ \ \  \ \ \ \ \ (a) \\
 &= L \circ \sigma_2 \circ \varphi_{_D} \circ (p \otimes \Id_D) \ \ \ \ \ \ \  \ \ \ \ \ \ \ \ \ \ \ \ \ \ \ \  \ \  \ \ \ \ \ \ \ \ \ \ \ \ \ \ (\ast) \\
&=L \circ \varphi_{_F} \circ (\sigma_2 \otimes \sigma_2) \circ (p \otimes \Id_D) \ \ \ \ \ \ \  \ \ \ \ \ \ \ \ \ \ \ \ \ \ \ \  \ \ \ \ \ \ \  (b) \\
&=L \circ [\varphi_{_F} \circ (\Id_F \otimes \sigma_2)]\\
\end{split}
\end{equation*}
This gives the desired equalities. We justify the different steps below.\ \\
%%%
\begin{itemize}
	
\item In $(a)$ and $(b)$ we've used the fact that $\sigma_1$ and $\sigma_2$ are \textbf{morphisms in $\M_{\S}(X)$}, which implies in particluar that 
$$\varphi_{_F} \circ (\sigma_i \otimes \sigma_i) = \sigma_i \circ \varphi_{_D}.$$
This last equality is equivalent to say that the following diagram commutes:
\[
\xy
(0,20)*+{D(s) \otimes D(t)}="A";
(30,20)*+{D(s \otimes t)}="B";
(0,0)*+{F(s) \otimes F(t)}="C";
(30,0)*+{F(s \otimes t)}="D";
{\ar@{->}^{\varphi_{_D}}"A";"B"};
{\ar@{->}_{\sigma_i \otimes \sigma_i}"A";"C"};
{\ar@{->}^{\sigma_i}"B";"D"};
{\ar@{-->}^{\varphi_{_F}}"C";"D"};
\endxy
\]

\item In $(\ast)$ we've used the fact that $L$ is a coequalizer  of $\sigma_1$ and $\sigma_2$ therefore: $L \circ \sigma_1=  L \circ \sigma_2$.
\end{itemize}
\ \\
Now from the lemma \ref{rel-coeq}, we know that $L$ is a coequalizer of $(\Ub \tld{\sigma}_1, \Ub \tld{\sigma}_2)$ rel. to $\Ub$, since $\Ub$ is clearly faithful Lemma \ref{Linton-relative-coeq} applies, hence $L$ is a coequalizer in $\M_{\S}(X)$ of $(\tld{\sigma}_1,\tld{\sigma}_2)$ (and is obviously preserved by $\Ub$).
\end{proof}

\begin{cor}\label{corol-coeq-split}
$\M_{\S}(X)$ has coequalizers of reflexive pairs and $\Ub$ preserves them.
\end{cor}
\begin{proof}
Given a pair of parallel morphisms $(\tld{\sigma}_1,\tld{\sigma}_2): D \rightrightarrows F$ with a split  $\tld{p}$ in $\M_{\S}(X)$ , setting $p= \Ub \tld{p}$, by definition of morphism in  $\M_{\S}(X)$,  $p$ is a split of $(\Ub \tld{\sigma}_1, \Ub \tld{\sigma}_2)$ and we conclude by the Lemma \ref{coeq-split-pair}.
\end{proof}

\subsection{Monadicity and Cocompleteness}

\begin{thm}\label{MX-monadic-KX}
If $\M$ is cocomplete then $\M_{\S}(X)$ is monadic over $\K_X$.
\end{thm}

\begin{proof}
We use Beck's monadicity theorem (see \cite[Chap.6, Sec.7, Thm.1]{Mac})  for $\Ub: \M_{\S}(X) \to \K_X$  since:
\begin{itemize}
\item $\Ub$ has a left adjoitn $\Gamma$ ( Lemma \ref{adjoint-ub}), 
\item $\Ub$ clearcly reflect isomorphisms since by definition a morphism $\sigma$ is an isomorphism in $\M_{\S}(X)$ if $\Ub(\sigma)$ is so, 
\item $\M_{\S}(X)$  has has coequalizers of $\Ub$-split parallel pair (= reflexive pair) and $\Ub$ preserves them (Corollary \ref{corol-coeq-split}).
\end{itemize}
\end{proof}

\begin{thm}
For a cocomplete symmetric closed monoidal category $\M$,  the category $\M_{\S}(X)$ is also cocomplete. 
\end{thm}

\begin{proof}
By the previous theorem $\M_{\S}(X)$ is equivalent to the category of algebra of the monad $\T= \Ub \Gamma$ on $\K_X$. Since $\K_X$ is cocomplete and $\M_{\S}(X)$ (hence $\T$-alg) has coequalizer of reflexive pair, a result of Linton \cite[Corollary 2]{Linton-cocomplete} implies that $\T$-alg (hence $\M_{\S}(X)$) is cocomplete.
\end{proof}

%%%%%%%% comment following %%%%%%
\begin{comment}
Recall that the category $\K_X= \prod_{(A,B) \in X^{2}} \Hom[\S_{\ol{X}}(A,B)^{op}, \M]$ has coequalizers which are taken level-wise. With the adjunction $\Ub:\M_{\S}(X) \rightleftarrows \K_X : \Gamma$, we would like to know when the (naive) coequalizer computed in $\K_X$ lift to a coequalizer in . We  treat this question following the same ideas as Wolff \cite{Wo} who prove for categories and graphs use the language of congruences like in \cite{Wo}. 
Below we give a criterion which says when the naive coequalizer 
Given a pair of parallel morphisms of $\S_{\ol{X}}$-diagrams 
\end{comment} 

%%%%%%%%%%%%%%%%%%%%%%%%%%%%%%%New Section  %%%%%%%% %%%%%%%%%
%%%%%%%%%%%%%%%%%%%%%%%%%%%%%%%%%%%%%%%%%%%%%%%%%%
%%%%%%%%%%%%%%%%%%%%%%%%%%%%%%%%%%%%%%%%%%%%%%%%%%%%%%%%%%%%%%%
%%%%%%%%%%%%%%%%%%%%%%%%%%%%%%%%%%%%%%%%%%%%%%%%%%%%%%%%%%%%%%%
%%%%%%%%%%%%%%%%%%%%%%%%%%%%%%%%%%%%%%%%%%%%%%%%%%%%%%%%%%%%%%%
%%%%%%%%%%%%%%%OPERAD -COCOMPLETENESS-LOCAL-PRESENTABILITY%%%%%%%%%%% %%%%%%%%%%%%%%%%%%%%%%%%%%%%%% %%%%%%%%%%%%%%%%%%%%%%%%%%%%%%%%%%%%%%%%%%%%%%%%
%%%%%%%%%%%%%%%%%%%%%%%%%%%%%%%%%%%%%%%%%%%%%%%%%%%%%%%%%%%%%%%
%%%%%%%%%%%%%%%%%%%%%%%%%%%%%%%%%%%%%%%%%%%%%%%%%%%%%%%%%%%%%%%
%%%%%%%%%%%%%%%%%%%%%%%%%%%%%%%%%%%%%%%%%%%%%%%%%%%%%%%%%%%%%%%
%%%%%%%%%%%%%%%%%%%%%%%%%%%%%%%%%%%%%%%%%%%%%%%%%%%%%%%%%%%%%%%
%%%%%%%%%%%%%%%%%%%%%%%%%%%%%%%%%%%%%%%%%%%%%%%%%%%%%%%%%%%%%%%
%%%%%%%%%%%%%%%%%%%%%%%%%%%%%%%%%%%%%%%%%%%%%%%%%%%%%%%%%%%%%%%
%%%%%%%%%%%%%%%%%%%%%%%%%%%%%%%%%%%%%%%%%%%%%%%%%%%%%%%%%%%%%%%

%%%%%%%%%%%% A mettre en appendice
\section{$\Laxalg(\Ca\dt, \Ma\dt)$ is locally presentable} \label{Proof-local-present-O-alg}
Our goal in this section is to prove the following 
\begin{thm}
Let $\Ma\dt$ be a locally presentable $\O$-algebra. For any $\O$-algebra $\Ca\dt$ the category of lax $\O$-morphisms $\Laxalg(\Ca\dt, \Ma\dt)$ is locally presentable.  
\end{thm}

The proof of this theorem is technical and a bit long so we will divide it into small pieces, but for the moment we give hereafter an outline of what we will do. 

\begin{enumerate}
\item We will construct a left adjoint $\Fb$, of the forgetful functor $\Ub:  \Laxalg(\Ca\dt, \Ma\dt) \to \prod_{i \in C} \Hom(\Ca_i, \Ma_i)$. 
\item We will show that $\Ub$ is monadic that is $\Laxalg(\Ca\dt, \Ma\dt)$ is equivalent to the category of algebra of the induced monad $\Ub \Fb$. We will transfer the local-presentability by monadic adjunction following the same idea as Kelly and Lack \cite{Kelly-Lack-loc-pres-vcat} who proved that $\M$-$\Cat$ is locally presentable if $\M$ is so. All we need will be to check  that the monad is finitary i.e preserve filtered colimits and the result will follow by a classical argument.
\end{enumerate}

\begin{comment}
\begin{prop}
Let $\O$ be a multisorted operad and $\Ma\dt$ a cocomplete $\O$-algebra. Assume that $\Ma\dt$ is $\kappa$-small for some sufficiently large cardinal $\kappa$. Then for any $\O$-algebra $\Ca\dt$ the category $\Laxalg(\Ca\dt, \Ma\dt)$ is cocomplete. 
\end{prop}
\end{comment}

\begin{prop}\label{prop-laxalg-cocomplete}
Let $\O$ be a multisorted operad and $\Ma\dt$ a cocomplete $\O$-algebra. Then for any $\O$-algebra $\Ca\dt$ the category $\Laxalg(\Ca\dt, \Ma\dt)$ is cocomplete. 
\end{prop}
We will denote by $\K_{\Ca\dt}=\prod_{i \in C} \Hom(\Ca_i, \Ma_i)$ and $\K_{\Ca_i}= \Hom(\Ca_i, \Ma_i)$ the $i$th-factor.\ \\
Consider  $\Ub: \Laxalg(\Ca\dt, \Ma\dt) \to  \K_{\Ca\dt}$ the functor which forgets the laxity maps. \ \\
For the proof of the proposition we will establish first the following lemma. 

\begin{lem}\label{lem-interm-cocomp}
Let $\O$ be a multisorted operad and $\Ma\dt$ a cocomplete $\O$-algebra. Then   the functor 
$$\Ub : \Laxalg(\Ca\dt, \Ma\dt) \to  \K_{\Ca\dt} $$  
\\
has a left adjoint.
\end{lem}

As the proof is long we dedicate the next section to it

\subsection{The functor $\Ub$ has a left adjoint}\label{ub-has-left-adjoint}

In the following we give a `free algebra construction process' which associates to any family of functors $\Fa\dt= (\Fa_i)_{i \in C}$, a lax $\O$-morphism  $\Fb \Fa\dt$. One can consider it as the analogue of  the `monadification' of a classical (one-sorted) operad (see for example \cite[Part I, Section 3]{Kriz_May_OAMM}.
For the operad $\O_X$ it will cover the process which associates an $\M$-graph to the corresponding  free $\M$-category.\ \\ 
\begin{nota}\ \
If $(x, c_1, ...,c_n) \in \O(i_1,...,i_n; j) \times \Ca_{i_1} \times \cdots \times \Ca_{i_n}$, we will write $\otimes_x(c_1,...c_n)= \rho_{_{i_{\dt}|j}}(x,c_1,..,c_n)$.\ \\
$\rho_{_{i_{\dt}|j}}^{-1}(c)=$ the subcategory of $\O(i_1,...,i_n; j) \times \Ca_{i_1} \times \cdots \times \Ca_{i_n}$ whose objects are $(x, c_1, ...,c_n)$ such that $\otimes_x(c_1,...c_n)=c$ and morphisms $(f, u_1, ..., u_n)$ such that $\rho_{_{i_{\dt}|j}}(f, u_1, ..., u_n)= \Id_c$. \ \\
$\Fb_{-}^{c}: \Ma_j \to \Hom(\Ca_j, \Ma_j) =$ the left adjoint of the evaluation functor at $c$: $\Ev_c: \Hom(\Ca_j, \Ma_j) \to \Ma_j$. \\
\end{nota}

Informally from a family of (abstract) objects $(\L_i)_{i \in C}$, one defines the associated free $\O$-algebra $(\Fb \L_i)_{i \in C}$ as follows
$$\Fb \L_j= \coprod_{n \in \N}( \coprod_{(i_1, ..., i_n)} \O(i_1,...,i_n; j) \boxtimes [\L_{i_1} \otimes \cdots \otimes \L_{i_n}]).$$

Here `$\boxtimes$' is the action of $\O$ on the category containing the object $\L_i$ and $\otimes$ is the internal product of that category.
The algebra structure is simply given by the multiplication of $\O$ and shuffles; the reader can find a description of the algebra structure, for the one sorted case, in \emph{loc. cit}.\ \\

In our case we want to use the description presented previously, according to which we view a lax $\O$-morphism as an $\O$-algebra of some category $\Ar(\Cat)_+$ (see \ref{coherences-lax}).  The action of $\O$ on $\Ar(\Cat)_+$ was denoted by `$\odot$'.\ \\

So if we start with a family of functors $(\Fa_i: \Ca_i \to \Ma_i)_{i \in C}$ we would like to  define the associated  free algebra by:
$$\tx{ ` $\Fb \Fa_j= \coprod_{n \in \N}( \coprod_{(i_1, ..., i_n)} \O(i_1,...,i_n; j) \odot  [\Fa_{i_1} \otimes \cdots \otimes \Fa_{i_n}])$ '} $$

But as one can see in this coproduct the functors are not defined over the same category, which we want to be $\Ca_j$, so the previous expression actually doesn't make sense in general. But still it guides us to the correct object which is some left Kan extension of something similar. \ \\
\ \\
For each $(i_1, ..., i_n) \in C^n$ introduce $\Lan_j (\O,\Fa_{i_{\dt}}) $ the left Kan extension of the functor 
$$\rho_{_{i_{\dt}|j}} \circ [ \O(i_1,...,i_n; j) \odot \prod \Fa_i] =  \rho_{_{i_{\dt}|j}} \circ [\Id_{\O(i_1,...,i_n; j)} \times \prod \Fa_i]$$ along the functor $\theta_{_{i_{\dt}|j}}$. This  left Kan extension exists since $\Ma_j$ is cocomplete and we have the following diagram 
\[
\xy
(-20,20)*+{ \O(i_1,...,i_n; j) \times \Ca_{i_1} \times \cdots \times \Ca_{i_n}}="X";
(30,20)*+{\Ca_j}="Y";
(-20,0)*+{\O(i_1,...,i_n; j) \times \Ma_{i_1} \times \cdots \times \Ma_{i_n}}="A";
(30,0)*+{\Ma_j}="B";
{\ar@{->}^{\ \ \ \ \ \ \ \ \ \ \ \ \ \ \ \ \ \ \  \theta_{_{i_{\dt}|j}}}"X";"Y"};
{\ar@{->}_{ \ \\ \ \ \ \ \ \ \ \ \ \ \ \ \ \ \ \ \ \ \rho_{_{i_{\dt}|j}}}"A";"B"};
{\ar@{->}^{\Lan_j (\O,\Fa_{i_{\dt}})}"Y";"B"};
{\ar@{->}_{\Id_{\O(i_1,...,i_n; j)} \times \Fa_{i_1} \times \cdots \times \Fa_{i_n}}"X";"A"};
%% insertion de phi%%%
{\ar@{=>}^{\varepsilon_{_{i_{\dt}|j}}}(21,2);(28,9)};
\endxy
 \]

 %%%%%%%%%%%%%%%%%%%%%% tout ceci est inutile %%%%%%%%% il y a plus simple dans le cas d'une algebre libre 
Here $\varepsilon_{_{i_{\dt}|j}}$ is the universal natural transformation arising in the construction of the left Kan extension (see \cite[Ch. X]{Mac}). This diagram represents a morphism  in $\Ar(\Cat)_+$ 
$$\Phi(i_1,...,i_n; j): \O(i_1,...,i_n; j) \odot [\Fa_{i_1} \times \cdots \times \Fa_{i_n}] \to \Lan_j (\O,\Fa_{i_{\dt}})$$
 which, according to the notations in section (\ref{coherences-lax}), is $(\theta_{_{i_{\dt}|j}} ; \rho_{_{i_{\dt}|j}} ; \varepsilon_{_{i_{\dt}|j}})$. \ \\

\subsubsection{Definition of $\Fb$ is $\Ca\dt$ is a free $\O$-algebra} 
Let $\Ca\dt$ be a free $\O$-algebra.  $\Fa\dt \in \K_{\Ca\dt}$. Then we define $\Fb \Fa\dt$ by 
$$\Fb (\Fa)_j= \coprod_{n \in \N} \coprod_{(i_1, ..., i_n)} \Lan_j (\O,\Fa_{i_{\dt}}).$$

Note that in this coproduct there is a hidden term for $n=0$ which is just $\Fa_j$ itself, since  $\Fa_j$ is the left Kan extension of itself along the identity $\Id_{\Ca_j}$. The inclusion in the coproduct yields a natural transformation:
$$\eta_j: \Fa_j \to \Fb (\Fa)_j.$$
We have to specify the morphisms: $ \O(i_1,...,i_n; j) \odot [\Fb(\Fa)_{i_1} \times \cdots \times \Fb(\Fa)_{i_n}] \to \Fb(\Fa)_j $. Before doing this we need to outline some important fact about free $\O$-algebras:
\begin{rmk}
Since $\Ca\dt$ is a \ul{free} algebra, $\Ca\dt$ is a defined by a collections of categories $(L_i)_{i\in C}$ and one has 
$$ \Ca_j= \coprod_{n \in \N}( \coprod_{(i_1, ..., i_n)} \O(i_1,...,i_n; j) \times [\L_{i_1} \times \cdots \times \L_{i_n}]).$$
It follows that each multiplication $\theta_{_{i\dt|j}}: \O(i_1,...,i_n; j) \times \Ca_{i_1} \times \cdots \times \Ca_{i_n} \to   \Ca_j$ is an inclusion to a coproduct, one view it as a `grafting trees' operation; it's image defines a connected component of $\Ca_j$.
Therefore any $c_j \in \Im(\theta_{_{i\dt|j}})$ has a unique presentation $(x,c_1,...c_n) \in \O(i_1,...,i_n; j) \times \Ca_{i_1} \times \cdots \times \Ca_{i_n}$.\ \\
Consequently  the Kan extension $\Lan_j (\O,\Fa_{i_{\dt}})$ consists to take the image by $(\Fa_i)$ of the presentation i.e : $\Lan_j(\O, \Fa_{i_{\dt}}) c= \rho_{_{i\dt|j}}(x,\Fa_{i_1}c_1,..., \Fa_{i_n}c_n).$
\end{rmk}
\ \\
With this description  we define the morphism: $ \O(i_1,...,i_n; j) \odot [\Fb(\Fa)_{i_1} \times \cdots \times \Fb(\Fa)_{i_n}] \to \Fb(\Fa)_j $ as follows.
\begin{itemize}[label=$-$, align=left, leftmargin=*, noitemsep]
\item First if we expand  $\O(i_1,...,i_n; j) \odot [\Fb(\Fa)_{i_1} \times \cdots \times \Fb(\Fa)_{i_n}]$ we have: 
\begin{equation*}
\begin{split}
\O(i_{\dt}; j) \odot [\Fb(\Fa)_{i_1} \times \cdots \times \Fb(\Fa)_{i_n}] 
&= \coprod_{n \in \N} \coprod_{(i_1, ..., i_n)} \coprod_{(h_{1,1}, ..., h_{n,k_n})}  \Id_{\O(i_{\dt}; j)} \times \prod_i \Lan_i(\O(h_{_{i,\dt}}|i),\Fa_{h_{i,l_{\dt}}}) \\
\end{split}
\end{equation*}
\item Then introduce $\Lan_j[\O(i\dt|j),  \Lan_i(\O(h_{_{i,\dt}}|i), \Fa_{h_{i,l_{\dt}}}) ]$, the left Kan extension of
 the summand \\  $\Id_{\O(i_{\dt}; j)} \times \prod_i \Lan_i(\O(h_{_{i,\dt}}|i),\Fa_{h_{i,l_{\dt}}})$ along $\theta_{_{i_{\dt}|j}}$. This left Kan extension comes equipped with a natural  transformation: 
 $$\delta:\Id_{\O(i_1,...,i_n; j)} \times \prod_i \Lan_i(\O(h_{_{i,\dt}}|i),\Fa_{h_{i,l_{\dt}}}) \to \Lan_j[\O(i\dt|j),  \Lan_i(\O(h_{_{i,\dt}}|i), \Fa_{h_{i,l_{\dt}}}) ].$$   
\end{itemize} 
%%%%%%%%
\begin{claim}
We have an equality 
$\Lan_j[\O(i\dt|j),  \Lan_i(\O(h_{_{i,\dt}}|i), \Fa_{h_{i,l_{\dt}}}) ] = \Lan_j[\O(h_{\dt,l_{\dt}}|j),  \Fa_{h_{\dt,l_{\dt}}}].$
\end{claim}
Before telling why the claim holds, one defines the desired map by sending the summand\\ $\Id_{\O(i_{\dt}; j)} \times \prod_i \Lan_i(\O(h_{_{i,\dt}}|i),\Fa_{h_{i,l_{\dt}}})$ of  $ \O(i_1,...,i_n; j) \odot [\Fb(\Fa)_{i_1} \times \cdots \times \Fb(\Fa)_{i_n}]$ to the summand \\
$\Lan_j[\O(h_{\dt,l_{\dt}}|j),  \Fa_{h_{\dt,l_{\dt}}}]$ of $\Fb(\Fa)_j$ by the composite:
$$\Id_{\O(i_{\dt}; j)} \times \prod_i \Lan_i(\O(h_{_{i,\dt}}|i),\Fa_{h_{i,l_{\dt}}}) \xrightarrow{\delta} \Lan_j[\O(h_{\dt,l_{\dt}}|j),  \Fa_{h_{\dt,l_{\dt}}}] \hookrightarrow \Fb(\Fa)_j.$$

By the universal property of the coproduct  we have a unique map:
$$ \Phi_{_{i_{\dt}|j}}: \O(i_1,...,i_n; j) \odot [\Fb(\Fa)_{i_1} \times \cdots \times \Fb(\Fa)_{i_n}] \to \Fb(\Fa)_j .$$ 
\ \\
To see that the claim holds one proceeds as follows. Consider  $[x,(x_i, d_{i,1},..., d_{i,k_i})_{1\leq i \leq n}] $  an object of $ \O(i\dt|j) \times [\O(h_{_{1,\dt}}|i_1) \times \Ca_{_{1,\dt}}] \times \cdots \times [\O(h_{_{n,\dt}}|i_n) \times \Ca_{_{n,\dt}}]$. Such presentation defines a unique object $c \in \Ca_j$, and each $(x_i, d_{i,1},..., d_{i,k_i})$ defines a unique object $c_i \in  \Ca_i$.\ \\

In the free algebra $\Ca$, one declares that the following objects are equal to $c \in \Ca_j$
\begin{itemize}[label=$-$, align=left, leftmargin=*, noitemsep]
\item $(\gamma(x,x_i);d_{1,1},...,d_{n,k_n}) = c$,
\item  $(x,c_1,..., c_n)=c$,  
\end{itemize} 
Here $\gamma$ is the substitution in $\O$ and $(x_i; d_{i,1},..., d_{i,k_i})= c_i$.  Then one computes in one hand 
$$\Lan_j[\O(h_{\dt,l_{\dt}}|j),  \Fa_{h_{\dt,l_{\dt}}}](c)= \otimes_{\gamma(x,x_i)}(\Fa d_{1,1},...,\Fa d_{n,k_n}).$$
On the other hand one has:
$$\Lan_j[\O(i\dt|j),  \Lan_i(\O(h_{_{i,\dt}}|i), \Fa_{h_{i,l_{\dt}}}) ](c)= \otimes_x [\otimes_{x_i}(\Fa (d_{1 \dt}),...,\otimes_{x_n}(\Fa d_{n \dt})] .$$ 
\\
Now as $\Ma\dt$ is an $\O$-algebra one has the equality:
$$\otimes_{\gamma(x,x_i)}(\Fa d_{1,1},...,\Fa d_{n,k_n})= \otimes_x [\otimes_{x_i}(\Fa (d_{1 \dt}),...,\otimes_{x_n}(\Fa d_{n \dt})]$$  
which means that the two functors are equal as claimed.\ \\ 

By virtue of the previous discussion we have the
\begin{prop}\label{laxity-f-infini}\ \
\begin{enumerate}
\item The family $\Fb (\Fa_i)_{i \in C}$ forms a lax $\O$-morphism of algebra. Equivalently $\Fb (\Fa_i)_{i \in C}$ is an $\O$-algebra of $\Ar(\Cat)_+$.
\item For any $\Ga\dt=(\Ga_i)_{i \in C} \in \Laxalg(\Ca\dt,\Ma\dt)$ we have a functorial isomorphism of sets:
$$ \Hom[\Fb (\Fa_i)_{i \in C} , \Ga\dt] \cong \prod_i \Hom[\Fa_i,\Ga_i]$$
\end{enumerate}
\end{prop}

\begin{proof}[Proof of Proposition \ref{laxity-f-infini}]

We will simply give a proof of the assertion $(1)$. The statement $(2)$ is tedious but straightforward to check. \ \\
%%%%%%%%%%%%
The only thing we need to check is the fact that the natural transformations $\Phi_{_{i_{\dt}|j}}$ fit coherently. 
We are asked to say if for any $(h_{1,1},...,h_{1,l_1}; i_1);  \cdots ;  (h_{n,1},...,h_{n,l_n}; i_n)$ the following is commutative:

\begin{equation}\label{cohe-lax-free}
\xy
(-40,20)*+{\O(i\dt|j) \times \O(h_{_{1,\dt}}|i_1) \times \cdots \times \O(h_{_{n,\dt}}|i_n) \times \prod\Fb( \Fa_{h_{\dt,l_{\dt}}})}="X";
(40,20)*+{\O(h_{\dt,l_{\dt}}|j) \times \prod\Fb( \Fa_{h_{\dt,l_{\dt}}})}="Y";
(-40,0)*+{\O(i\dt|j) \times \prod \Fb (\Fa_i) }="A";
(40,0)*+{ \Fb (\Fa_j)}="B";
{\ar@{->}^-{\gamma \times \Id}"X";"Y"};
{\ar@{->}_-{ \Phi_{_{i_{\dt}|j}}}"A";"B"};
{\ar@{->}^-{ \Phi_{h_{\dt,l_{\dt}}|j}}"Y";"B"};
{\ar@{->}_-{(\Id_{\O(i\dt|j)} \times \prod \Phi_{_{h_{i,l_{\dt}}|i}}) \circ \tx{shuffle} }"X";"A"};
\endxy
\end{equation}

But the commutativity of that diagram boils down to check that the following maps are identities:
$$\Lan_j[\O(h_{\dt,l_{\dt}}|j),  \Fa_{h_{\dt,l_{\dt}}}] \xrightarrow{\tx{canonical}} \Lan_j[\O(i\dt|j),  \Lan_i(\O(h_{_{i,\dt}}|i), \Fa_{h_{i,l_{\dt}}}) ] $$

But this follows from the previous discussion. Consequently the maps  $\Phi_{_{i_{\dt}|j}}$ fit coherently and $(\Fb \Fa_i)_{i \in C}$ equipped with $ \Phi_{_{i_{\dt}|j}}$ is a lax $\O$-morphism of algebra.\ \\ 
\end{proof}

\subsubsection{Definition of $\Fb$ for an arbitrary $\O$-algebra $\Ca\dt$} 
Let $\Ca\dt$ be an arbitrary $\O$-algebra and   $\Fa\dt \in \K_{\Ca\dt}$. For each $j \in C$ consider,
$$\Fa^1_j= \coprod_{n \in \N} \coprod_{(i_1, ..., i_n)} \Lan_j (\O,\Fa_{i_{\dt}}).$$
This is the same type of expression as before;  the inclusion in the coproduct yields a natural transformation:
$$e_j: \Fa_j \to \Fa^1_j.$$
\\
For each $c$ and each $(x,c_1,..., c_n) \in \rho^{-1}c$, we have a map:
$$\varepsilon: \otimes_x(\Fa c_1,..., \Fa c_n) \to \Lan_j (\O,\Fa_{i_{\dt}})(c)  \hookrightarrow \Fa^1 c .$$ 
By the adjunction $\Fb^c \dashv \Ev_c$, the map $\varepsilon$ corresponds to a unique map in $\K_j= \Hom(\Ca_j, \Ma_j)$: 
$$\Fb_{\otimes_x(\Fa c_1,..., \Fa c_n)}^{c} \to \Fa^1. $$
Let $\Ra(e;x,c_1,...,c_n)$ be the object defined by the pushout diagram in $\K_j$:
\[
\xy
(-10,20)*+{\Fb_{\otimes_x(\Fa c_1,..., \Fa c_n)}^{c}}="A";
(30,20)*+{\Fa^1}="B";
(-10,0)*+{\Fb_{\otimes_x(\Fa^1 c_1,..., \Fa^1 c_n)}^{c}}="C";
(30,0)*+{\Ra(e;x,c_1,...,c_n)}="D";
{\ar@{->}^{}"A";"B"};
{\ar@{->}_-{\Fb_{\otimes_x(e_1,...e_n)}^{c}}"A"+(0,-3);"C"};
{\ar@{-->}^-{p(e;x,c_1,...,c_n)}"B";"D"};
{\ar@{-->}^{}"C";"D"};
\endxy
\]

%%%%%%%%%%%%%%%%%%%%%%%%%%%%%%%%%%%%%
\paragraph{An intermediate coherence}
\ \\
Let  $[x,(x_i, d_{i,1},..., d_{i,k_i})_{1\leq i \leq n}] $ be an object of $ \O(i\dt|j) \times [\O(h_{_{1,\dt}}|i_1) \times \Ca_{_{1,\dt}}] \times \cdots \times [\O(h_{_{n,\dt}}|i_n) \times \Ca_{_{n,\dt}}]$ 
such that:
\begin{itemize}[label=$-$, align=left, leftmargin=*, noitemsep]
\item $\otimes_{x_i} (d_{i,1},..., d_{i,k_i})= c_i$,
\item $\otimes_{\gamma(x,x_i)}(d_{1,1},...,d_{n,k_n}) = c$, and
\item  $\otimes_x(c_1,..., c_n)=c$; here $\gamma$ is the substitution in $\O$. 
\end{itemize}

From the map 
$$\eta \varepsilon: \otimes_{\gamma(x,x_i)}(\Fa d_{1,1},..., \Fa d_{n,k_n}) \to \Fa^{1,1} (\otimes_{\gamma(x,x_i)}(d_{1,1},...,d_{n,k_n})) = \Fa^{1,1} c.$$
Using the adjunction $\Fb^c \dashv \Ev_c$, we define the object $Q(x,x_i;d_{1,1},..., d_{n,k_n})$ given by the pushout diagram in $\K_j$:
\[
\xy
(-10,20)*+{\Fb_{\otimes_{\gamma(x,x_i)}(\Fa d_{1,1},..., \Fa d_{n,k_n})}^{c}}="A";
(40,20)*+{\Fa^1}="B";
(-10,0)*+{\Fb_{\otimes_x(\Ra c_1,..., \Ra  c_n)}^{c}}="C";
(40,0)*+{Q(x,x_i;d_{1,1},..., d_{n,k_n})}="D";
{\ar@{->}^{}"A";"B"};
{\ar@{->}_-{\Fb_{\otimes_x [(p \varepsilon)_1, ..., (p \varepsilon)_n]}^{c}}"A"+(0,-4);"C"};
{\ar@{.>}^-{g(x,x_i;d_{1,1},..., d_{n,k_n})}"B";"D"};
{\ar@{.>}^{}"C";"D"};
\endxy
\]
Introduce $\Za(x,x_i;d_{1,1},..., d_{n,k_n})$ to be the object obtained from the pushout: \ \\
\[
\xy
(-10,20)*+{\Fa^1}="A";
(40,20)*+{\Ra(e,\gamma(x,x_i);d_{1,1},..., d_{n,k_n})}="B";
(-10,0)*+{Q(x,x_i;d_{1,1},..., d_{n,k_n})}="C";
(40,0)*+{\Za(x,x_i;d_{1,1},..., d_{n,k_n})}="D";
{\ar@{->}^-{p}"A";"B"};
{\ar@{->}_-{g(x,x_i;d_{1,1},..., d_{n,k_n})}"A"+(0,-4);"C"};
{\ar@{.>}^-{g'(x,x_i;d_{1,1},..., d_{n,k_n})}"B";"D"};
{\ar@{.>}^{}"C";"D"};
\endxy
\]
\\
Denote by $\Za_{h\dt, i\dt, j}(c): \rho^{-1}c \to \Fa^{1}_{/ \K_j}$, the  functor that takes $[x,(x_i, d_{i,1},..., d_{i,k_i})]$ to natural transformation:
 $$\Fa^1 \to  \Za(x,x_i;d_{1,1},..., d_{n,k_n}).$$ 
Let $\Fa^{1,c}_{h\dt, i\dt, j}$  be the colimit of $\Za_{h\dt, i\dt, j}(c)$ and denote by $\eta_{h\dt, i\dt, j}: \Fa^1 \to \Fa^{1,c}_{h\dt, i\dt, j}$ the canonical map. 

\begin{df}\label{construction-t}
Define $\Fa^{1,1}_{h\dt, i\dt, j}$ to be the object obtained by the generalized pushout diagram in $\K_j$ as $c$ runs through $\Ca_j$:
$$ \Fa^{1,1}_{h\dt, i\dt, j}= \colim_{c \in \Ca_j} \{ \Fa^1 \xrightarrow{\eta_{1,c}} \Fa^{1,c}_{h\dt, i\dt, j} \}.$$
We have canonical maps $\eta: \Fa^1 \to \Fa^{1,1}_{h\dt, i\dt, j}$ and $\delta_{+}(x,(x_i, d_{i,1},..., d_{i,k_i})): \Za(x,x_i;d_{1,1},..., d_{n,k_n}) \to  \Fa^{1,1}_{h\dt, i\dt, j}$.\ \\
Define $\Ta(e,\Fa, \Fa^1)_j$ to be the object defined also by the generalized pushout:
 $$ \Ta(e,\Fa, \Fa^1)_j = \colim \coprod_{n \in \N} \coprod_{(i_1, ..., i_n)} \coprod_{(h_{1,1}, ..., h_{n,k_n})} \{ \eta_{h\dt, i\dt, j}: \Fa^1 \to \Fa^{1,1}_{h\dt, i\dt, j}  \}$$ 
where $e: \Fa \to \Fa^1$ is the original morphism from the left Kan extension which gives  the first laxity maps $\varepsilon$.
\end{df}

We will write $\Fa^2= \Ta(e,\Fa, \Fa^1)$ and $\eta^1: \Fa^1 \to \Fa^2$ the canonical map. By construction we end up with new laxity maps $\varepsilon_1: \otimes_x(\Fa^1 c_1,..., \Fa^1 c_n) \to \Fa^2(c)$ which are not coherent, but  we can iterate the process to build an object  $\Fa^3= \Ta(\Fa^1, \Fa^2)$ which `bring the coherences of $\varepsilon_1$'.  But the new laxity maps are not coherent so we have to repeat the process an infinite number of time.\ \\

Proceeding by induction we define for $k \in \lambda$, an object $\Fa^k$ wiht maps $\eta^k: \Fa^k \to \Fa^{k+1}$ by:
\begin{enumerate}
\item $\Fa^0:= \Fa$ and $\eta^0= e$. 
\item $\Fa^{k+1}:= \Ta(\eta^{k-1}, \Fa^{k-1}, \Fa^k)$, we have a map $\eta^k: \Fa^k \to \Fa^{k+1}$ from the construction `$\Ta$'.
\end{enumerate}
We therefore have a $\lambda$-sequence in $\K_j$: 
$$\Fa= \Fa^0 \to \Fa^1 \to \cdots \to \Fa^k \to \Fa^{k+1} \to \cdots$$
and we can take the colimit $\Fa^{\infty}=  \colim_{k \in \lambda} \{ \Fa^k \xrightarrow{\eta^k} \Fa^{k+1} \}$. 
\begin{df}
For a family $\Fa\dt=(\Fa_j)_{j \in C}$ we define $\Fb(\Fa\dt)$ by setting:
 $$\Fb(\Fa\dt)_j:= \Fa^{\infty}.$$
We have a canonical map $\eta: \Fa_j \to \Fb(\Fa\dt)_j$.  
\end{df}

We have also sequences of objects $\Ra^k, Q^k$ and $\Za^k$ which are created step by step and there are also maps induced by universal property from $\Ra^k \to \Ra^{k+1}$ and similarly for $Q^k$ and $\Za^k$. \ \\  

\begin{prop}\label{free-laxalg-functor}
For a cocomplete $\O$-algebra $\Ma\dt$ we have that:
\begin{enumerate}
\item the family $\Fb(\Fa\dt)$ is a lax morphism from $\Ca\dt$ to $\Ma\dt$ and 
\item  the functor $\Fb$ is  left adjoint of $\Ub$. 
\end{enumerate}

\end{prop} 

\begin{proof}[Sketch of proof] The assertion $(2)$ is straightforward so we leave it to the reader. For the assertion $(1)$ we need to specify the laxity maps and check that they satisfy the coherence conditions. \ \\
 
By construction the following diagram involving $\Ra^k$ commutes

\[
\xy
(-40,20)*+{\Fb_{\otimes_x(\Fa^k c_1,..., \Fa c_n)}^{c}}="A";
(10,20)*+{\Fa^{k+1}}="B";
(-40,0)*+{\Fb_{\otimes_x(\Fa^{k+1} c_1,..., \Fa^{k+1} c_n)}^{c}}="C";
(10,0)*+{\Ra^{k+1}(\eta^k;x,c_1,...,c_n)}="D";
(50,0)*+{\Fa^{k+2}}="E";
(-40,-20)*+{\Fb_{\otimes_x(\Fa^{k+2} c_1,..., \Fa^{k+2} c_n)}^{c}}="X";
(50,-20)*+{\Ra^{k+2}(\eta^{k+1};x,c_1,...,c_n)}="Y";
{\ar@{->}^{}"A";"B"};
{\ar@{->}_{\Fb_{\otimes_x(\eta^k_1,...\eta^k_n)}^{c}}"A"+(0,-3);"C"};
{\ar@{-->}_{p(\eta^k;x,c_1,...,c_n)}"B";"D"};
{\ar@{-->}^{}"C";"D"};
{\ar@{..>}^{}"D";"E"};
{\ar@{->}_{\Fb_{\otimes_x(\eta^{k+1}_1,...\eta^{k+1}_n)}^{c}}"C"+(0,-3);"X"};
{\ar@{-->}^{}"X";"Y"};
{\ar@{-->}^{}"E";"Y"};
{\ar@{.>}^{}"D";"Y"};
{\ar@{.>}^{\eta^{k+1}}"B";"E"};
\endxy
\]

From this diagram it's easy to see that the sequences $\Ra^k$ and $\Fa^k$ `converge' to the same object, that is they have the same colimit object. A simple analyse shows that also $Q^k$ and $\Za^k$ have as colimit object $\Fa^{\infty}$. \ \\

Since $\lambda$ is a regular cardinal for any $\lambda$-small cardinal $\mu$  the diagonal functor $d: \lambda \to \prod_{\mu} \lambda$ from $\lambda$ to the product of $\mu$ copies of $\lambda$ is cofinal: a consequence is that diagrams indexed by $\lambda$ and $\prod_{\mu} \lambda$  have the same colimits. It follows, in particular, that for any $(i_1,...,i_n) \in C^{n}$ the following colimits are the same
\begin{equation*}
\begin{cases}
\colim_{(k_1, ...,  k_n) \in \lambda^n} \{ \otimes_x(\Fa^{k_1}_{i_1}  c_1,..., \Fa^{k_n}_{i_n}  c_n) \}\\
\\
  \colim_{k \in \lambda}  \{ \otimes_x(\Fa^k_{i_1}  c_1,..., \Fa^k_{i_n}  c_n) \} \\
  \end{cases}
\end{equation*}

%These colimits are taken in functor-category $\Hom[\O(i_1,...,i_n; j) \times \Ca_{i_1} \times \cdots \times \Ca_{i_n}, \Ma_j]$ and the morphisms connecting the vertices of the cocone are the natural transformations 
%$ \otimes_{-} (\prod \eta^{k}_i)$. \ \\%
One of the assumptions on the algebra  $\Ma\dt=(\Ma_i)_{i \in C}$ is  the possibilty to commute colimits computed in $\Ma\dt$ and the  tensor products `$\otimes_x$'.\ \\
\ \\
If we put these together the first colimit is easily computed as:
$$\colim_{(k_1, ...,  k_n) \in \lambda^n} \{ \otimes_x(\Fa^{k_1}_{i_1}  c_1,..., \Fa^{k_n}_{i_n}  c_n) \} =  \otimes_x( \Fb (\Fa_{i_1} )c_1,...,   \Fb (\Fa_{i_n} )c_n) .$$ 
\\
And we deduce that: 
$$\colim_{k  \in \lambda} \{ \otimes_x(\Fa^{k}_{i_1}  c_1,..., \Fa^{k}_{i_n}  c_n) \} =  \otimes_x( \Fb (\Fa_{i_1} )c_1,...,   \Fb (\Fa_{i_n} )c_n) .$$ 
\\
All these are natural in $(x, c_1,.., c_n)$.  One gets the laxity maps by the universal property of the colimit with respect to the compatible cocone which ends at $\Fb(\Fa_j)c$
\[
\xy
(-60,20)*+{\otimes_x(\Fa^k c_1,..., \Fa c_n)}="A";
(-10,20)*+{\Fa^{k+1} c}="B";
(-60,0)*+{\otimes_x(\Fa^{k+1} c_1,..., \Fa^{k+1} c_n)}="C";
(-10,0)*+{\Ra^{k+1}(\eta^k;x,c_1,...,c_n)c}="D";
(40,0)*+{\Fa^{k+2} c}="E";
(40,-20)*+{\Ra^{k+2}(\eta^{k+1};x,c_1,...,c_n) c}="Y";
(80,-20)*+{\Fb(\Fa_j)c}="Z";
{\ar@{->}^{}"A";"B"};
{\ar@{->}_{\otimes_x(\eta^k_1,...\eta^k_n)}"A"+(0,-3);"C"};
{\ar@{-->}_{p(\eta^k;x,c_1,...,c_n)c}"B";"D"};
{\ar@{-->}^{}"C";"D"};
{\ar@{..>}^{}"D";"E"};
{\ar@{-->}^{}"E";"Y"};
{\ar@{.>}^{}"D";"Y"};
{\ar@{.>}^{\eta^{k+1}}"B";"E"};
{\ar@{.>}^{\tx{canonical}}"E";"Z"};
{\ar@{.>}^{\ \ \ \ \  \ \  \ \ \tx{canonical}}"Y";"Z"};
{\ar@{.>}^{}"C";"C"+(0,-20)};
{\ar@/_1.5pc/@{.>}"C"+(1,-20); "Z"};
\endxy
\]

So we get a unique map $\varphi^{\infty}(x,c_1, ..., c_n): \otimes_x( \Fb (\Fa_{i_1} )c_1,...,   \Fb (\Fa_{i_n} )c_n) \to \Fb(\Fa_j)c$  which makes the bovious diagram commutative. As usual the maps $\varphi^{\infty}(x,c_1, ..., c_n)$ are natural in $(x,c_1, ..., c_n)$. \ \\

The fact that these maps $\varphi^{\infty}(x,c_1, ..., c_n)$ satisfy the coherence conditions is easy bu tedious to check. One use the diagram involving $Q^k$ and $\Za^k$ and take the colimit everywhere; the universal property of the colimit will force (by unicity)  the equality between the two maps out of $ \otimes_{\gamma(x,x_i)}( \Fb (\Fa_{h_{1,1}} )d_{1,1},...,   \Fb (\Fa_{h_{n,k_n}} )d_{n,k_n})$ and going to $\Fb(\Fa_j)c$. %which makes everything commutative. 
 For the record these maps are: 
\begin{itemize}
\item[-] $\varphi^{\infty}(\gamma(x, x_i),d_{1,1}, ..., d_{n,k_n})$
\item[-] $ \varphi^{\infty}(x,c_1, ..., c_n) \circ [\otimes_x (\varphi^{\infty}(x_1,d_{1,1}, ..., d_{1,k_1}), \cdots, \varphi^{\infty}(x_n,d_{n,1}, ..., d_{n,k_n}))].$ 
\end{itemize}
\end{proof}
\begin{rmk}\label{fb-preserves-filtered-colim}
As $\Ub$ has a left adjoint $\Fb$ we have an induced monad $\T= \Ub \Fb$. It's not hard to see that $\T$ automatically preserves the colimits appearing in the definition of $\Fb$ namely the $\lambda$-directed ones. And since directed colimits are the same as  filtered ones we deduce that $\T$ preserves filtered colimits as well. 
\end{rmk} 

\subsection{$\Laxalg(\Ca\dt, \Ma\dt)$ is monadic over $\Hom(\Ca\dt, \Ma\dt)$ }

Let $\sigma_1, \sigma_2: \Fa\dt \to \Ga\dt$ be a pair of parallel morphisms between two lax-morphisms in $\Laxalg(\Ca\dt,\Ma\dt)$. Denote by $L: \Ga\dt \to \Ea\dt$ the coequalizer of $\sigma_1, \sigma_2$ in $\K_{\Ca\dt}=\prod_{i \in C} \Hom(\Ca_i, \Ma_i)$:

\[
\xy
(0,0)*++{\Ub \Fa\dt}="X";
(20,0)*++{\Ub \Ga\dt}="Y";
(40,0)*++{\Ea\dt}="Z";
{\ar@<-0.5ex>@{->}_{\sigma_1}"X";"Y"};
{\ar@<0.5ex>@{->}^{\sigma_2}"X";"Y"};
{\ar@{->}^{L}"Y";"Z"};
\endxy
\]
Note that we've freely identified $\sigma_i$ and it's image $\Ub \sigma_i$.  The following lemma is the general version of lemma \ref{precong-lemma} except that we do not use the language of precongruences. 

\begin{lem}\label{coeq-laxalg}
Consider $\Fa\dt, \Ga\dt , \Ea\dt$ with $\sigma_1, \sigma_2$ and $L$ as before. Assume that 
 for every $(x,c_1, ..., c_n)  \in \O(i_1,...,i_n; j) \times \Ca_{i_1} \times \cdots \times \Ca_{i_n}$ with $c= \otimes_x(c_1,...,c_n)$ and any $l \in \{1,...,n\}$ the following equality holds:
$$ L_c \circ \varphi_{_{\Ga} }(x,c_1, ...,  c_n) \circ [\otimes_x(\Id_{\Ga c_1}, ...,\sigma_1(c_{l}), ..., \Id_{\Ga c_n}) ]=  L_c \circ \varphi_{_{\Ga} }(x,c_1, ...,  c_n) \circ [\otimes_x(\Id_{\Ga c_1}, ...,\sigma_2 (c_{l}), ..., \Id_{\Ga c_n})].$$
 Then we have:
\begin{enumerate}
\item $\Ea\dt$ becomes a lax morphism and 
\item $L$ is the coequalizer of $\sigma_1$ , $\sigma_2$ in $\Laxalg(\Ca\dt,\Ma\dt)$.
\end{enumerate}
\end{lem}
When there is no confusion we will simply write $\varphi_{_{\Ga} }$ instead of $ \varphi_{_{\Ga} }(x,c_1, ...,  c_n) $. 
\begin{proof}[Sketch of proof]
The assertion $(2)$ will follow from the proof of $(1)$. To prove $(1)$ we will simply give the laxity maps; the coherence conditions are straightforward.

As mentioned before, the assumptions on $\Ma\dt$ allow to distribute (factor wise) colimits over each tensor product $\otimes_x$. It follows that $\otimes_x(\Ea c_1, ... \Ea c_n)$ equipped with the maps $\otimes_x(L_{c_1},...,L_{c_n})$ is the colimit of the diagram 
$$ \epsilon(\sigma_1,\sigma_2;c_1,...,c_n)= \coprod_{ (\tau_1, ..., \tau_n) \in \{1,2\}^n } \{\otimes_x( \sigma_{\tau_1} c_1, ..., \sigma_{\tau_n} c_n): \otimes_x( \Fa c_1, ...,\Fa c_n) \to \otimes_x( \Ga c_1, ...,\Ga c_n)  \}.$$
(For each $l$, $\sigma_{\tau_l}$ is either $\sigma_1$ or $\sigma_2$). \ \\

For each $(\tau_1, ..., \tau_n) \in \{1,2\}^n$ let  $\Theta(\tau_1, ..., \tau_n)= L_c \circ \varphi_{_{\Ga} }(x,c_1, ...,  c_n) \circ \otimes_x( \sigma_{\tau_1} c_1, ..., \sigma_{\tau_n} c_n)$  be the map illustrated in the diagram below:

\[
\xy
(-10,15)*+{\otimes_x( \Fa c_1, ...,\Fa c_n)}="X";
(30,15)*+{\Fa c}="Y";
(-10,0)*+{\otimes_x( \Ga c_1, ...,\Ga c_n)}="A";
(30,0)*+{\Ga c}="B";
%{\ar@{->}^{\varphi(x,c_1,...,c_n)}"X";"Y"};
{\ar@{->}_{\ \ \ \ \ \ \ \varphi_{_{\Ga} }}"A";"B"};
%{\ar@{->}^{\sigma_{j,\otimes_x(c_{\dt})}}"Y";"B"};
{\ar@{->}_{\otimes_x( \sigma_{\tau_1} c_1, ..., \sigma_{\tau_n} c_n)}"X";"A"};
(30,-15)*+{\Ea c}="E";
{\ar@{->}_{L_c}"B";"E"};
\endxy
 \]
 
Now we claim that $\Theta(\tau_1, ..., \tau_n)= \Theta(\tau'_1, ..., \tau'_n)$ for all $(\tau_1, ..., \tau_n)$ , $(\tau'_1, ..., \tau'_n)$ in $\{1,2\}^n$. The claim will holds as soon as we show that   for every $l \in  \{1,...,n\}$ we have $\Theta(\tau_1, ...,\tau_l, ... \tau_n)= \Theta(\tau_1, ...,\tau'_l,...,  \tau_n)$ where $\tau_l$ and $\tau'_l$ are `conjugate' that is: if $\tau_l=1$ then $\tau'_l=2$ and vice versa. \ \\
Let's assume that $\tau_l=1$ (hence $\tau'_l=2$) that is $\sigma_{\tau_l}=\sigma_1$.

We establish successively the following equalities:

\begin{equation*}
\begin{split}
\Theta(\tau_1, ..., \tau_n)
&= L_c \circ \varphi_{_{\Ga} } \circ \otimes_x( \sigma_{\tau_1} c_1, ..., \sigma_1 c_l,..., \sigma_{\tau_n} c_n)  \\
&=L_c \circ \varphi_{_{\Ga} } \circ [\otimes_x(\Id_{\Ga c_1}, ...,\sigma_1c_{l} , ..., \Id_{\Ga c_n})  \circ \otimes_x( \sigma_{\tau_1} c_1, ..., \Id_{\Fa c_l},..., \sigma_{\tau_n} c_n) ]\\
&=\{L_c \circ \varphi_{_{\Ga} } \circ [\otimes_x(\Id_{\Ga c_1}, ...,\sigma_1c_{l} , ..., \Id_{\Ga c_n}) ] \} \circ \otimes_x( \sigma_{\tau_1} c_1, ..., \Id_{\Fa c_l},..., \sigma_{\tau_n} c_n) \\
(\ast) &=\{L_c \circ \varphi_{_{\Ga} } \circ [\otimes_x(\Id_{\Ga c_1}, ...,\sigma_2(c_{l}), ..., \Id_{\Ga c_n}) ]\} \circ \otimes_x( \sigma_{\tau_1} c_1, ..., \Id_{\Fa c_l},..., \sigma_{\tau_n} c_n) \\
&=L_c \circ \varphi_{_{\Ga} } \circ [\otimes_x(\Id_{\Ga c_1}, ...,\sigma_2 c_{l}, ..., \Id_{\Ga c_n})  \circ \otimes_x( \sigma_{\tau_1} c_1, ..., \Id_{\Fa c_l},..., \sigma_{\tau_n} c_n)] \\
&=L_c \circ \varphi_{_{\Ga} } \circ \otimes_x( \sigma_{\tau_1} c_1, ..., \sigma_2 c_l,..., \sigma_{\tau_n} c_n) \\
&=\Theta(\tau_1, ..., \tau'_l,..., \tau_n)
\end{split}
\end{equation*}

(In $(\ast)$ we use the hypothesis to switch $\sigma_1$ and $\sigma_2$ in the expression contained in `\{ \}'.)
\end{proof}

\begin{rmk}
If $\Fa\dt$ is simply an object of $\K_{\Ca\dt}$ but not necessarily a lax morphism but $\Ga\dt$ is a lax morphism, we will have a precongruence in $\Ga\dt$ and the first assertion of the lemma will hold. The proof will exactly be the same. 
\end{rmk}

The next lemma tells about the existence of coequalizer of a parallel $\Ub$-split pair. This is the generalization of lemma \ref{coeq-split-pair}.

\begin{lem}\label{coeq-usplit-laxalg}
Consider $\Fa\dt, \Ga\dt , \Ea\dt$ with $\sigma_1, \sigma_2$ and $L$ as before. Assume that there is a \textbf{$\Ub$-split} i.e a morphism $p: \Ub \Ga\dt \to \Ub \Fa\dt$ in $\K_{\Ca\dt}$  such that $ \sigma_1 \circ p= \sigma_2 \circ p = \Id_{\Ub \Ga\dt}.$\ \\
Then  we have:
\begin{enumerate}
\item $\Ea\dt$ becomes a lax $\O$-morphism,
\item $L: \Ga\dt \to \Ea\dt$ is a coequalizer in $\Laxalg(\Ca\dt,\Ma\dt)$ of the pair $(\sigma_1,\sigma_2)$  and $\Ub$ obviously preserves it (as a coequalizer). 
\end{enumerate} 
\end{lem}

\begin{proof}
All is proved in the same manner as for lemma \ref{coeq-split-pair}. We simply have to show that the equalities of lemma \ref{coeq-laxalg} holds. 

Since for each $ \tau \in \{1,2 \}$), $\sigma_\tau \circ p= \Id_{\Ga}$, for every  $(c_1,..., c_n)$ and every $l \in \{1,...,n \}$  we have that:
$$ \otimes_x(\Id_{\Ga c_1}, ...,\sigma_\tau c_{l}, ..., \Id_{\Ga c_n})=  \otimes_x( \sigma_{\tau} c_1, ..., \sigma_{\tau} c_l,..., \sigma_{\tau} c_n) \circ  \otimes_x( p c_1, ..., \Id_{\Fa c_l},..., p c_n). $$

Moreover as $\sigma_{\tau}$ is a morphism of lax-morphis the following commutes:
\[
\xy
(-10,15)*+{\otimes_x( \Fa c_1, ...,\Fa c_n)}="X";
(30,15)*+{\Fa c}="Y";
(-10,0)*+{\otimes_x( \Ga c_1, ...,\Ga c_n)}="A";
(30,0)*+{\Ga c}="B";
{\ar@{->}^{\ \ \ \ \ \ \ \varphi_{_{\Fa}}}"X";"Y"};
{\ar@{->}_{\ \ \ \ \ \ \ \varphi_{_{\Ga} }}"A";"B"};
{\ar@{->}^{\sigma_{\tau}c}"Y";"B"};
{\ar@{->}_{\otimes_x( \sigma_{\tau} c_1, ..., \sigma_{\tau} c_n)}"X";"A"};
%(30,-15)*+{\Ea c}="E";
%{\ar@{->}_{L_c}"B";"E"};
\endxy
 \] 
 which means that we have an equality:
 $\sigma_{\tau}c \circ \varphi_{_{\Fa}}= \varphi_{_{\Ga} } \circ \otimes_x( \sigma_{\tau} c_1, ..., \sigma_{\tau} c_n).$
 If we combine all the previous discussion we establish successively the following.

\begin{equation*}
\begin{split}
L_c \circ \varphi_{_{\Ga} } \circ [\otimes_x(\Id_{\Ga c_1}, ...,\sigma_1 c_{l}, ..., \Id_{\Ga c_n}) ]
&= L_c \circ \varphi_{_{\Ga} } \circ [\otimes_x( \sigma_1 c_1, ..., \sigma_{1} c_l,..., \sigma_1 c_n) \circ  \otimes_x( p c_1, ..., \Id_{\Fa c_l},..., p c_n)]  \\
&=[L_c \circ \underbrace{ \varphi_{_{\Ga} } \circ \otimes_x( \sigma_1 c_1, ..., \sigma_{1} c_l,..., \sigma_1 c_n)}_{=\sigma_{1}c \circ \varphi_{_{\Fa}}}] \circ  \otimes_x( p c_1, ..., \Id_{\Fa c_l},..., p c_n)  \\ \\
&=[\underbrace{L_c  \circ \sigma_{1}c}_{=L_c \circ \sigma_{2}c} \circ \varphi_{_{\Fa}}] \circ   \otimes_x( p c_1, ..., \Id_{\Fa c_l},..., p c_n) \\ \\
&=[L_c \circ \underbrace{ \sigma_{2}c \circ \varphi_{_{\Fa}}}_{=\varphi_{_{\Ga} } \circ \otimes_x( \sigma_2 c_1, ..., \sigma_{2} c_l,..., \sigma_2 c_n)}] \circ   \otimes_x( p c_1, ..., \Id_{\Fa c_l},..., p c_n)  \\ \\
&=[L_c \circ \varphi_{_{\Ga} } \circ \otimes_x( \sigma_2 c_1, ..., \sigma_{2} c_l,..., \sigma_2 c_n)] \circ  \otimes_x( p c_1, ..., \Id_{\Fa c_l},..., p c_n)  \\ 
&=L_c \circ \varphi_{_{\Ga} } \circ [\otimes_x( \sigma_2 c_1, ..., \sigma_{2} c_l,..., \sigma_2 c_n) \circ  \otimes_x( p c_1, ..., \Id_{\Fa c_l},..., p c_n)]  \\
&=L_c \circ \varphi_{_{\Ga} } \circ [\otimes_x(\Id_{\Ga c_1}, ...,\sigma_2 c_{l}, ..., \Id_{\Ga c_n}) ].
\end{split}
\end{equation*}

\end{proof}

\begin{cor}\label{cor_monadicity}
Let $\Ma\dt$ be a cocomplete $\O$-algebra. Then we have
\begin{enumerate}
\item The functor $\Ub: \Laxalg(\Ca\dt,\Ma\dt) \to \K_{\Ca\dt}$ is monadic.
\item  $\Laxalg(\Ca\dt,\Ma\dt)$ is cocomplete.
\item If moreover $\Ma\dt$ is locally presentable then so is $\Laxalg(\Ca\dt,\Ma\dt)$.
\end{enumerate}
\end{cor}

\begin{proof}[Sketch of proof]
The assertion $(1)$ follows from Beck monadicity theorem since:
\begin{itemize}
\item $\Ub$ has a left adjoint $\Fb$ (Proposition \ref{prop-laxalg-cocomplete}),
\item $\Ub$ clearly reflect isomorphisms,
\item $\Laxalg(\Ca\dt,\Ma\dt)$ has coequalizers of parallel $\Ub$-split pairs and $\Ub$ preserves them (Lemma \ref{coeq-usplit-laxalg}).
\end{itemize}
It follows that $\Laxalg(\Ca\dt,\Ma\dt)$ is equivalent to the category $\T$-alg for the monad $\T= \Ub \Fb$.  The assertion $(2)$ follows from Lintons's theorem \cite{Linton-cocomplete} since $\T$ is defined on $\K_{\Ca\dt}$ which is cocomplete and  $\T$-alg has coequalizer of reflexive pair. \ \\
From the Remark \ref{fb-preserves-filtered-colim} we know that $\T$ preserves filtered colimits, and since $\K_{\Ca\dt}$ is locally presentable we know from \cite{Adamek-Rosicky-loc-pres} that $\T$-alg (hence $\Laxalg(\Ca\dt,\Ma\dt)$) is automatically locally presentable and the assertion $(3)$ follows. 
\end{proof}

\section{Pushout in $\Laxalg(\Ca\dt, \Ma\dt)$}  \label{pushout-laxalg}

In this section we want to show that for a trivial cofibration $\alpha \in \kc$ then the pushout of $\Fb \alpha$ is a weak equivalence in $\Laxalg(\Ca\dt, \Ma\dt)$, when $\Ma\dt$ is a special Quillen $\O$-algebra (Definition \ref{quillen-alg}). On $\kc$ we will consider the \emph{injective and projective} model structures; these are product model structures of the ones on each $\K_j= \Hom(\Ca_j, \Ma_j)$.  \ \\

Given a  diagram in $\Laxalg(\Ca\dt, \Ma\dt)$ 
\[
\xy
(0,20)*+{\Fb \Aa}="A";
(20,20)*+{\Fa}="B";
(0,0)*+{\Fb \Ba}="C";
%(20,0)*+{V}="D";
{\ar@{->}^{\sigma}"A";"B"};
{\ar@{->}^-{\Fb\alpha}"A";"C"};
%{\ar@{->}^{\Id_V}"B";"D"};
%{\ar@{->}^{\Id_V}"C";"D"};
\endxy
\]

with $\alpha$ is a trivial cofibrationin $\kc$; if we want to calculate the pushout, then the first thing to do is to consider the pushout in $\kc$ then build the laxity map etc. But the left adjoint $\Fb$ we've constructed previously, when considered as an endofunctor on $\kc$,  may not preserve weak equivalences for arbitrary $\O$-algebra $\Ca\dt$ and $\Ma\dt$. In particular it may not preserve trivial cofibrations. So the pushout of $\Fb \alpha$ will hardly be a weak equivalence. 
\ \\

The obstruction of $\Fb$ to be a left Quillen functor can be seen in the following phenomena:
\begin{enumerate}
\item first the left Kan extension we've considered to define $\Fa^1$ may not  in general preserves level-wise (trivial) cofibrations:
\[
\xy
(-20,20)*+{ \O(i_1,...,i_n; j) \times \Ca_{i_1} \times \cdots \times \Ca_{i_n}}="X";
(30,20)*+{\Ca_j}="Y";
(-20,0)*+{\O(i_1,...,i_n; j) \times \Ma_{i_1} \times \cdots \times \Ma_{i_n}}="A";
(30,0)*+{\Ma_j}="B";
{\ar@{->}^-{\theta_{_{i_{\dt}|j}}}"X";"Y"};
{\ar@{->}_-{\rho_{_{i_{\dt}|j}}}"A";"B"};
{\ar@{->}^{\Lan_j (\O,\Fa_{i_{\dt}})}"Y";"B"};
{\ar@{->}_{\Id_{\O(i_1,...,i_n; j)} \times \Fa_{i_1} \times \cdots \times \Fa_{i_n}}"X";"A"};
%% insertion de phi%%%
{\ar@{=>}^{\varepsilon_{_{i_{\dt}|j}}}(21,2);(28,9)};
\endxy
 \]
 \item second the  $ \Fa^{1,c}_{h\dt, i\dt, j}$ appearing in the construction of $\Fb$  may not be left Quillen functor . In fact  $\Fa^{1,c}_{h\dt, i\dt, j}$ is a colimit of a functor:
$$\Za_{h\dt, i\dt, j}(c): \rho^{-1}c \to \Fa^{1}/ \K_j$$
where the source $\rho^{-1}c$ can \emph{a priori} be any category; so the colimit may not preserve (trivial) cofibrations. 
\end{enumerate}

These two facts lead us to some restrictions on our statements, for the moment. \ \\
We will reduce our statement to the $\O$-algebra $\Ca$ such that $\rho^{-1}c$ is a discrete category i.e a set. This way the colimit of $\Za_{h\dt, i\dt, j}(c)$ is a generalized pushout diagram in $\K_j$; and pushouts interact nicely with (trivial) cofibrations. \ \\  
\ \\
So rather than trying to figure out under which conditions $\Fb$ preserves the level-trivial cofibration as endofunctor on $\kc$, we will work by assuming that it is.\ \\ 

This reduction may appear to be too restrictive, but hopefully the cases we encounter in `the nature' will be in this situation.  Usually this will be the case for all the `simple objects' we use to built complicated ones eg: the operad $\O_X$, $\Delta$, $\ol{X}$, $\S_{\ol{X}}$, every $1$-categorie $\D$, free $\O$-algebra, etc.  
\ \\
\ \\
Recall that we introduced previously the
\begin{df}\label{iro-hco}
Let $(\Ca\dt,\rho)$ and $(\Ma\dt, \theta)$ be two $\O$-algebra.

\begin{enumerate}
\item Say that $\Ca\dt$ is  \textbf{$\O$-well-presented}, or \textbf{$\O$-identity-reflecting} if:\\
for every $n+1$-tuple $(i_1,...,i_n; j)$ the following functor reflects identities

$$\rho: \O(i_1,...,i_n; j) \times \Ca_{i_1} \times \cdots \times \Ca_{i_n} \to \Ca_j.$$
This means that the image of $(u, f_1, ...,f_n) \in \O(i_1,...,i_n; j) \times \Ca_{i_1} \times \cdots \times \Ca_{i_n} $ is an identity morphism in $\Ca_j$ (if and) only if all $u, f_1,...,f_n$ are simultaneously identities.
\item Say that $(\Ca\dt, \Ma\dt)$ is an \textbf{$\O$-homotopy-compatible pair} if $\Fb: \kc \to \kc$ preserves level-wise trivial cofibrations, where $\kc$ is endowed with the injective model structure. 
\end{enumerate}
\end{df}

\begin{rmk}\ \
\begin{enumerate}
\item A consequence of the definition is  that  if $\Ca\dt$ is an $\O$-identity-reflecting  algebra (henceforth $\iro$-algebra), then the fiber $\rho^{-1}c = \rho^{-1}\{\Id_c\}$ is a set. 
\item Any free $\O$-algebra $\Ca\dt$ is an $\iro$-algebra; and for any special Quillen $\O$-algebra $\Ma\dt$ having all its objects cofibrant, the pair $(\Ca\dt, \Ma\dt)$ is $\O$-homotopy compatible (henceforth $\hco$ pair).
\end{enumerate}
\end{rmk}

With the previous material we can announce the main result:
\begin{lem}\label{lem-pushout-laxalg}
Let $\Ma\dt$ be a special Quillen $\O$-algebra such that all objects of $\Ma\dt$ are cofibrant. Let $\Ca\dt$ be $\iro$-algebra such that the $(\Ca\dt, \Ma\dt)$ is an $\hco$ pair. Then for any pushout square  in $\Laxalg(\Ca\dt, \Ma\dt)$ 
\[
\xy
(0,20)*+{\Fb \Aa}="A";
(20,20)*+{\Fa}="B";
(0,0)*+{\Fb \Ba}="C";
(20,0)*+{\Ga}="D";
{\ar@{->}^{\sigma}"A";"B"};
{\ar@{->}^-{\Fb\alpha}"A";"C"};
{\ar@{->}^-{H_{\alpha}}"B";"D"};
{\ar@{->}^{}"C";"D"};
\endxy
\]
$H_{\alpha}: \Fa \to \Ga$ is a level-wise trivial cofibration if $\alpha$ is so. 
\end{lem}

\subsection*{Proof of the lemma}
By the adjunction $\Fb \dashv \Ub$ the map $\sigma : \Fb \Aa \to \Fa$ in the pushout is induced by a unique map $\Aa \to \Ub \Fa$ in $\kc$. Similarly the map $\Fb \Ba \to \Ga$ is also induced by a map $\Ba \to \Ub \Ga$. We will construct $\Ga$ out of $\Fa$  and focus our analyse on the construction of the map $H_{\alpha}: \Fa \to \Ga$; the map $\Ba \to \Ub \Ga$ will follow automatically.
The first thing to do is to consider the pushout square in $\kc$:
\[
\xy
(0,20)*+{\Ub \Fb \Aa}="A";
(20,20)*+{\Ub \Fa}="B";
(0,0)*+{\Ub \Fb \Ba}="C";
(20,0)*+{\Ea}="D";
{\ar@{->}^{\Ub \sigma}"A";"B"};
{\ar@{->}^-{\Ub\Fb\alpha}"A";"C"};
{\ar@{->}^-{p}"B";"D"};
{\ar@{->}^{}"C";"D"};
\endxy
\]

Since we assumed that $\Fb \alpha$ is a level-wise trivial cofibration, the map $p: \Fa \to \Ea$ is automatically a level-wise trivial cofibration as well.

\paragraph{Intermediate laxity maps}

Let $(x,c_1, ..., c_n)$ be an object in $\O(i_1,...,i_n; j) \times \Ca_{i_1} \times \cdots \times \Ca_{i_n}$ with $c= \otimes_x(c_1,...,c_n) \in \Ca_j$. \ \\
Using the adjunction $\Ev_c: \Hom(\Ca_j,\Ma_j) \leftrightarrows \Ma_j : \Fb^c$ for the laxity map 
$$\otimes_x(\Fa c_1,..., \Fa c_n) \to \Fa(\otimes_x(c_1,...,c_n))= \Fa(c)$$
define $\Ra(p;x,c_1,...,c_n)$ to be the object we get from the pushout diagram in $\K_j$:
\begin{center}
$L_1=$
\begin{tabular}{c}
\xy
(-10,20)*+{\Fb_{\otimes_x(\Fa c_1,..., \Fa c_n)}^{c}}="A";
(30,20)*+{\Fa}="B";
(-10,0)*+{\Fb_{\otimes_x(\Ea c_1,..., \Ea c_n)}^{c}}="C";
(30,0)*+{\Ra(p;x,c_1,...,c_n)}="D";
{\ar@{->}^{}"A";"B"};
{\ar@{->}_-{\Fb_{\otimes_x(p_1,...p_n)}^{c}}"A"+(0,-3);"C"};
{\ar@{.>}^-{h(x,c_1,...,c_n)}"B";"D"};
{\ar@{.>}^-{q}"C";"D"};
\endxy
\end{tabular}
\end{center}

As each $p_{k}$ is a trivial cofibration (with cofibrant domain)  and since $\M$ is a special Quillen $\O$-algebra, we have that 
$\otimes_x(p_1,...p_n)$ is trivial cofibration in $\Ma_j$. Applying the left Quillen functor $\Fb^c$, we have  that $\Fb_{\otimes_x(p_1,...p_n)}^{c}$ is  a projective (hence an injective) trivial cofibration. It follows that $h(x,c_1,...,c_n)$ is also a projective trivial cofibration (as a puhout of such morphism) and therefore a level-wise trivial cofibration.\\
When the context is clear we will simply write $\Ra(x,c\dt)$ or $\Ra_{c\dt}$, and $p(x,c\dt)$, etc.

\paragraph{Intermediate coherences} With the `temporary' laxity maps we need to have a `temporary coherence' as well. We start with the objects on the fiber $\rho^{-1}c= \otimes^{-1} \{ c \}$.\ \\

Let  $[x,(x_i, d_{i,1},..., d_{i,k_i})_{1\leq i \leq n}] $ be an object of $ \O(i\dt|j) \times [\O(h_{_{1,\dt}}|i_1) \times \Ca_{_{1,\dt}}] \times \cdots \times [\O(h_{_{n,\dt}}|i_n) \times \Ca_{_{n,\dt}}]$ 
such that:
\begin{itemize}[label=$-$, align=left, leftmargin=*, noitemsep]
\item $\otimes_{x_i} (d_{i,1},..., d_{i,k_i})= c_i$,
\item $\otimes_{\gamma(x,x_i)}(d_{1,1},...,d_{n,k_n}) = c$, and
\item  $\otimes_x(c_1,..., c_n)=c$.
\end{itemize}

The coherence condition on the lax morphism $\Fa$ is equivalent to say that the upper face of the semi-cube below is commutative.

\[
\xy
%Le digram de dessus%%%
(-60,20)*+{\otimes_{\gamma(x,x_i)}(\Fa d_{1,1},..., \Fa d_{n,k_n}) }="A";
(20,30)+(-20,0)+(10,0)*+{\Fa c}="B";
(-20,-10)+(0,20)*+{ \otimes_x(\Fa c_1,..., \Fa c_n)}="C";
(60,0)+(0,20)+(-20,0)+(10,0)*+{\Fa c}="D";
{\ar@{->}^-{\varphi}"A";"B"};
{\ar@{->}_{\otimes_x(\varphi_1,...,\varphi_n) }"A";"C"};
{\ar@{=}^-{\Id}"B";"D"};
{\ar@{->}^{\varphi}"C";"D"};
% le diagram d'en dessous%%%
(-60,-10)*+{\otimes_{\gamma(x,x_i)}(\Ea d_{1,1},..., \Ea d_{n,k_n})}="X";
(20,0)+(-20,0)+(10,0)*+{\Ra(\gamma(x,x_i),d_{1,1},...,d_{n,k_n})}="Y";
(-20,-40)+(0,20)*+{\otimes_x(\Ra_{d_{1\dt}}c_1,..., \Ra_{ d_{n\dt}}c_n)}="Z";
%(60,-30)+(0,20)+(-20,0)*+{m'}="W";
{\ar@{.>}^{}"X";"Y"};
{\ar@{->}_{\otimes_{\gamma(x,x_i)} p}"A";"X"};
%{\ar@{..>}^{}"Y";"W"};
{\ar@{->}^>>>>>>>{\otimes_x(h_{c_i})}"C";"Z"};
{\ar@{.>}_{}"B";"Y"};
%{\ar@{..>}^{\beta}"D";"W"};
{\ar@{->}|{\circled{1}}"X";"Z"};
%{\ar@{.>}_{}"Z";"W"};
\endxy
\]
Here $\circled{1}$ represents the map: %$\otimes_x[q(d_{1\dt}),..., q(d_{n\dt})]$: 
$$\otimes_{\gamma(x,x_i)}(\Ea d_{1,1},..., \Ea d_{n,k_n}) = \otimes_x[\otimes_{x_1}(\Ea d_{1\dt}),...,\otimes_{x_n}(\Ea d_{n\dt})] \xrightarrow{\otimes_x[q_{d_{1\dt}},..., q_{d_{n\dt}}]} \otimes_x(\Ra_{d_{1\dt}}c_1,..., \Ra_{ d_{n\dt}}c_n)$$
with $q_{d_{i\dt}}: \otimes_{x_i}(\Ea d_{i,1}, ..., \Ea d_{i,k_i}) \to \Ra(x_i, d_{i,1},...,d_{i,k_i}).$\ \\

Extend the upper face by the commutative square $(L_1)$ above;  then extend the face on the right by taking the pushout of the trivial cofibration $\otimes_x(h_{c_i})$ along the trivial cofibration $\otimes_x(p_{c_i})$. We get a new semi-cube $C(x,x_i, d_i)$ where the face in the back is unchanged. \\
Since the face in the back is a pushout square and the vertical map in the front is a trivial cofibration,  we are in the situation of the Reedy style lemma \ref{semi-cub-cof}.\ \\

Introduce $O(x,x_i,d_i)$ to be the colimit of the semi-cube $C(x,x_i, d_i)$ . By virtue of lemma \ref{semi-cub-cof}, the canonical map $\beta: \Fa c  \to O(x,x_i,d_i)$ is a trivial cofibration.  Applying the left Quillen functor $\Fb^c$ we get a projective trivial cofibration $\Fb^c_{\beta}: \Fb^{c}_{\Fa c} \to \Fb^{c}_{O(x,x_i,d_i)}$.\ \\

The co-unit of the adjunction $\Fb^c \dashv \Ev_c$ corresponds to a map $e: \Fb^{c}_{\Fa c} \to \Fa$. Define $Q(x,x_i, d_i)$ to be the functor we get by the pushout of $\Fb^c_{\beta}$ along $e$ in $\K_j$:
\[
\xy
(-10,20)*+{\Fb^{c}_{\Fa c}}="A";
(20,20)*+{\Fa_j}="B";
(-10,0)*+{\Fb^{c}_{O(x,x_i,d_i)}}="C";
(20,0)*+{Q(x,x_i, d_i)}="D";
{\ar@{->}^{e}"A";"B"};
{\ar@{_{(}->}_-{\Fb^{c}_{\beta}}^-{\wr}"A"+(0,-5);"C"};
{\ar@{.>}^-{p_{(x,x_i,d_i)}}"B";"D"};
{\ar@{.>}^{}"C";"D"};
\endxy
\]

Define $\Fa^{1,c}_j$ to be $ \colim_{(x,x_i, d_i) \in \rho^{-1}c} \ \{ \Fa_j \xrightarrow{p_{(x,x_i,d_i)}} Q(x,x_i, d_i) \}$, where:
$$\rho^{-1}c= \coprod_{n \in \N^{\ast}} \coprod_{(i_1, ..., i_n)} \coprod_{(h_{1,1}, ..., h_{n,k_n})}  \rho_{h_{\dt},i_{\dt},j}^{-1}  \{\Id_c \} .$$

Since we assumed that $\Ca$ is an $\iro$-algebra then $\rho^{-1}c$ is a \ul{set}, therefore the colimit is a generalized pushout diagram in $\K_j$. This is  what we called a cone of trivial cofibrations in the model category $\K_{j\tx{-inj}}$.  By Lemma \ref{cone-cofib} we deduce that all the canonical maps going to the colimit are trivial cofibrations in $\K_{j\tx{-inj}}$; in particular the map 
 $\iota_c: \Fa_j \to \Fa^{1,c}_j$ is an injective trivial cofibration. \ \\
 
\paragraph{The construction `$\Po$' } Recall that all the previous construction are obtained from the map $p: \Fa \to \Ea$ which is an object of the under category $\Fa_{/ \kc}$. It's not hard to see that these constructions are functorial in $p$. 

\begin{df}
For each $j$, define $\Po(j,p,\Fa, \Ea)$ to be the colimit of the cone of trivial cofibrations  in $\K_j$:
$$\Po(j,p,\Fa, \Ea)= (\coprod_{c \in \Ob(\Ca_j)}   \iota_c : \Fa_j \to \Fa^{1,c}_j) \cup \{ p_j: \Fa \to \Ea_j \} .$$ 
Denote by $\eta^1_j: \Fa_j \to \Po(j,p,\Fa, \Ea)$ and $\delta^1_j: \Ea_j \to \Po(j,p,\Fa, \Ea)$ the canonical trivial cofibrations. 
\end{df} 
 
By the above remark one clear see that $\Po$ is an endofunctor of $\Fa_{/ \kc}$, that takes $p$ to $\eta^1$. Moreover for any $j$ the following commutes:

\[
\xy
(-20,0)*+{\Fa_j}="X";
(20,0)*+{\Po(j, p, \Fa,\Ea)}="Y";
(0,-10)*+{\Ea_j}="E";
{\ar@{->}^-{\eta^1_j}"X";"Y"};
{\ar@{->}_-{p_j}"X";"E"};
{\ar@{->}_-{\delta^1_j}"E";"Y"};
\endxy
\]   
As $\Po$ is an endofunctor, we can repeat the process and apply the previous construction  to \\ $\eta^1= \{ \Fa_j \xrightarrow{\eta^1_j} \Po(j,p,\Fa, \Ea) \}$ and repeat again and so forth.\ \\ 

Let $\kappa$ be a regular cardinal. For each $j$ we define a $\kappa$-sequence $(\Fa^{k}_j)_{k \in \kappa }$ in $\K_j$ as follows.

\begin{enumerate}
\item $\Fa_j = \Fa^{0}_j$,
\item $\Fa_j^{1} = \Ea_j$,
\item $\Fa^{k}_j= \Po(j,\eta^{k-1},\Fa, \Fa^{k-1})$ for $k \geq 2$,
\item  there are canonical maps $\delta^k : \Fa^{k-1}_j \to \Fa^{k}_j $ and $\eta^k: \Fa_j \to \Fa_j^k$ such that 
$\eta^k= \delta^k \circ   \eta^{k-1}$; with  $\eta_j^0= p_j$ 
\end{enumerate}
We end up with a $\kappa$-directed diagram in $\K_j$:
$$  \Fa_j = \Fa^{0}_j \xrightarrow{p_j} \Fa^{1}_j \xrightarrow{\delta^{1}_j} \cdots \xrightarrow{\delta^{k-1}_j} \Fa^{k}_j \xrightarrow{\delta^{k}_j} \Fa^{k+1}_j \xrightarrow{\delta^{k+1}_j} \cdots$$

Define $\Fa^{\infty}_j$ to be the colimit in $\K_j$ of that $\kappa$-sequence and denote by  $\eta^{\infty}_j: \Fa_j \to \Fa^{\infty}_j$ the canonical map.

\begin{rmk}
Since both $\delta^k$ and $\eta^k$ are   trivial cofibrations, it follows that  $\eta^{\infty}_j$  is also a trivial cofibration.
Furthermore we have a factorization of $\eta^{\infty}_j$  as:
$\Fa_j \xhookrightarrow{p_j} \Ea_j \xhookrightarrow{\delta^{\infty}_j} \Fa^{\infty}_j.$ By construction we have also other $\kappa$-sequences $(\Ra^k)_{k \in \kappa}$,  $(O^k)_{k \in \kappa}$ and $(Q^k)_{k \in \kappa}$; $\Ra^k$ bring the laxity maps and $Q^k$ bring the coherences. These objects interact in the semi-cubes  $C^k(x,x_i,d_i)$.\\
For each $j$ and each $c \in \Ca_j$,  all the three sequences $\{\Ra^k(c)\}_{k \in \kappa}$,  $\{O^k(c)\}_{k \in \kappa}$ and $\{Q^k(c)\}_{k \in \kappa}$ have the same colimit object which is $\Fa^{\infty}_j(c)$.
\end{rmk}

We complete the proof with the following
\begin{prop}\ \
\begin{enumerate}
\item For every  laxity map $\otimes_x(\Fa c_1,..., \Fa c_n) \to \Fa(c)$ we have a map
$\otimes_x(\Fa^{\infty} c_1,..., \Fa^{\infty} c_n) \to  \Fa^{\infty}(c)$
 and the following commutes:
\[
\xy
(0,15)*+{\otimes_x(\Fa c_1,..., \Fa c_n)}="A";
(45,15)*+{\Fa(c)}="B";
(0,0)*+{\otimes_x(\Fa^{\infty} c_1,..., \Fa^{\infty} c_n)}="C";
(45,0)*+{\Fa^{\infty}(c)}="S";
{\ar@{->}^-{\varphi}"A";"B"};
{\ar@{->}_-{\otimes_x(\eta^{\infty} c_1,..., \eta^{\infty} c_n)}"A";"C"};
{\ar@{->}^-{\eta^{\infty}}"B";"S"};
{\ar@{->}^-{\varphi^{\infty}}"C";"S"};
\endxy
\] 
\item The maps $\varphi^{\infty}$ fit coherently and $(\Fa^{\infty}_j)_j$ equipped with $\varphi^{\infty}$ is a lax $\O$-morphism i.e an object of $\Laxalg(\Ca\dt,\Ma\dt)$.
\item The map $\eta^{\infty}= (\eta^{\infty}_j): \Fa \to \Fa^{\infty}$ is the pushout in $\Laxalg(\Ca\dt,\Ma\dt)$ of $\Fb \alpha$ along $\sigma$. 
\item  $\Ub(\eta^{\infty})$ is also a level-wise trivial cofibration, so in particular a weak equivalence.
\end{enumerate}
\end{prop}

\begin{proof}[Sketch of proof]
The proof of $(1)$ is exactly the same for the Proposition \ref{free-laxalg-functor}. One gets the laxity maps by the universal property of the colimit of $\{\otimes_x(\Fa^k c_1,..., \Fa^k c_n)\}_{k \in \kappa}$,  with respect to the following compatible cocone which ends at $\Fa^{\infty}(c)$ (and starts from $\otimes_x(\Fa c_1,..., \Fa c_n) \to \Fa(c)$):
\[
\xy
(-60,20)*+{\otimes_x(\Fa^k c_1,..., \Fa^k c_n)}="A";
(-10,20)*+{\Fa^{k+1} c}="B";
(-60,0)*+{\otimes_x(\Fa^{k+1} c_1,..., \Fa^{k+1} c_n)}="C";
(-10,0)*+{\Ra^{k+1}(x,c_1,...,c_n)c}="D";
(40,0)*+{\Fa^{k+2} c}="E";
(40,-20)*+{\Ra^{k+2}(x,c_1,...,c_n) c}="Y";
(80,-20)*+{\Fa^{\infty}(c)}="Z";
{\ar@{->}^{}"A";"B"};
{\ar@{->}_-{\otimes_x(\delta^k_1,...\delta^k_n)}"A"+(0,-3);"C"};
{\ar@{-->}_-{p(x,c_1,...,c_n)_c}"B";"D"};
{\ar@{-->}^{}"C";"D"};
{\ar@{..>}^{}"D";"E"};
{\ar@{-->}^{}"E";"Y"};
{\ar@{.>}^{}"D";"Y"};
{\ar@{.>}^-{\delta^{k+1}}"B";"E"};
{\ar@{.>}^{\tx{canonical}}"E";"Z"};
{\ar@{.>}^-{ \tx{canonical}}"Y";"Z"};
{\ar@{.>}^{}"C";"C"+(0,-20)};
{\ar@/_1.5pc/@{.>}"C"+(1,-20); "Z"};
\endxy
\]
\ \\

One computed the colimit of $\{\otimes_x(\Fa^k c_1,..., \Fa^k c_n)\}_{k \in \kappa}$ by the same method explained in the proof of Proposition \ref{free-laxalg-functor}.  The map $\varphi^{\infty}(x,c_1, ..., c_n): \otimes_x(\Fa^{\infty} c_1,..., \Fa^{\infty} c_n) \to  \Fa^{\infty}(c)$ is the unique map  which makes everything commutative.\ \\

The coherence condition follows by construction; one takes the colimit everywhere in the universal cube defined by the semi-cubes  $C^k(x,x_i,d_i)$. The coherence is given by `the cube at the infinity'. 
The assertion $(3)$ is easily checked and follows by construction:  $\Fa^{\infty}$ with the obvious maps satisfies the universal property of the pushout. It's important to notice that this is valide because both $\Fb \Aa$ and $\Fb \Ba$ are free objects, therefore the map $\Fb \Ba \to \Fa^{\infty}$ is induced by the composite $\Ba \to \Fb \Ba \to  \Ea \to \Fa^{\infty}$.  

The assertion $(4)$ is obvious. 
\end{proof}

\bibliographystyle{plain}
\bibliography{Bibliography_These}

\begin{thebibliography}{10}

\bibitem{Adamek-Rosicky-loc-pres}
Ji{\v{r}}{\'{\i}} Ad{\'a}mek and Ji{\v{r}}{\'{\i}} Rosick{\'y}.
\newblock {\em Locally presentable and accessible categories}, volume 189 of
  {\em London Mathematical Society Lecture Note Series}.
\newblock Cambridge University Press, Cambridge, 1994.

\bibitem{Angel_Reedy}
Vigleik Angeltveit.
\newblock Enriched {R}eedy categories.
\newblock {\em Proc. Amer. Math. Soc.}, 136(7):2323--2332, 2008.

\bibitem{SEC1}
H.~V. {Bacard}.
\newblock {Segal Enriched Categories I}.
\newblock http://arxiv.org/abs/1009.3673.

\bibitem{Barwick_localization}
Clark Barwick.
\newblock On left and right model categories and left and right {B}ousfield
  localizations.
\newblock {\em Homology, Homotopy Appl.}, 12(2):245--320, 2010.

\bibitem{Ben2}
Jean B{\'e}nabou.
\newblock Introduction to bicategories.
\newblock In {\em Reports of the {M}idwest {C}ategory {S}eminar}, pages 1--77.
  Springer, Berlin, 1967.

\bibitem{Ber-Moer_Reedy_cat}
C.~{Berger} and I.~{Moerdijk}.
\newblock {On an extension of the notion of Reedy category}.
\newblock http://arxiv.org/abs/0809.3341.

\bibitem{Berger_Moer_htpy_cat}
C.~{Berger} and I.~{Moerdijk}.
\newblock {On the homotopy theory of enriched categories}.
\newblock http://arxiv.org/abs/1201.2134.

\bibitem{Ber_Moer_multisorted}
Clemens Berger and Ieke Moerdijk.
\newblock Resolution of coloured operads and rectification of homotopy
  algebras.
\newblock In {\em Categories in algebra, geometry and mathematical physics},
  volume 431 of {\em Contemp. Math.}, pages 31--58. Amer. Math. Soc.,
  Providence, RI, 2007.

\bibitem{Bonnin}
F.~{Bonnin}.
\newblock {Les groupements}.
\newblock http://arxiv.org/abs/math/0404233.

\bibitem{Cisinski-prefaisceau-model}
Denis-Charles Cisinski.
\newblock Les pr\'efaisceaux comme mod\`eles des types d'homotopie.
\newblock {\em Ast\'erisque}, (308):xxiv+390, 2006.

\bibitem{Cisinski-cat-derivables}
Denis-Charles Cisinski.
\newblock {Cat{\'e}gories d{\'e}rivables}.
\newblock {\em Bulletin de la soci{\'e}t{\'e} math{\'e}matique de France},
  138(3):317--393, 2010.

\bibitem{DKFC}
W.~G. Dwyer and D.~M. Kan.
\newblock Function complexes in homotopical algebra.
\newblock {\em Topology}, 19(4):427--440, 1980.

\bibitem{Dwyer_Spalinski}
W.~G. Dwyer and J.~Spali{\'n}ski.
\newblock Homotopy theories and model categories.
\newblock In {\em Handbook of algebraic topology}, pages 73--126.
  North-Holland, Amsterdam, 1995.

\bibitem{DHKS}
William~G. Dwyer, Philip~S. Hirschhorn, Daniel~M. Kan, and Jeffrey~H. Smith.
\newblock {\em Homotopy limit functors on model categories and homotopical
  categories}, volume 113 of {\em Mathematical Surveys and Monographs}.
\newblock American Mathematical Society, Providence, RI, 2004.

\bibitem{Fukaya_1}
Kenji Fukaya, Yong-Geun Oh, Hiroshi Ohta, and Kaoru Ono.
\newblock {\em Lagrangian intersection {F}loer theory: anomaly and obstruction.
  {P}art {I}}, volume~46 of {\em AMS/IP Studies in Advanced Mathematics}.
\newblock American Mathematical Society, Providence, RI, 2009.

\bibitem{Fukaya_2}
Kenji Fukaya, Yong-Geun Oh, Hiroshi Ohta, and Kaoru Ono.
\newblock {\em Lagrangian intersection {F}loer theory: anomaly and obstruction.
  {P}art {II}}, volume~46 of {\em AMS/IP Studies in Advanced Mathematics}.
\newblock American Mathematical Society, Providence, RI, 2009.

\bibitem{Hirsch-model-loc}
Philip~S. Hirschhorn.
\newblock {\em Model categories and their localizations}, volume~99 of {\em
  Mathematical Surveys and Monographs}.
\newblock American Mathematical Society, Providence, RI, 2003.

\bibitem{Hov-model}
Mark Hovey.
\newblock {\em Model categories}, volume~63 of {\em Mathematical Surveys and
  Monographs}.
\newblock American Mathematical Society, Providence, RI, 1999.

\bibitem{Jacobs_cat_type}
Bart Jacobs.
\newblock {\em Categorical logic and type theory}, volume 141 of {\em Studies
  in Logic and the Foundations of Mathematics}.
\newblock North-Holland Publishing Co., Amsterdam, 1999.

\bibitem{Ke}
G.~M. Kelly.
\newblock Basic concepts of enriched category theory.
\newblock {\em Repr. Theory Appl. Categ.}, (10):vi+137, 2005.
\newblock Reprint of the 1982 original [Cambridge Univ. Press, Cambridge;
  MR0651714].

\bibitem{Kelly-Lack-loc-pres-vcat}
G.~M. Kelly and Stephen Lack.
\newblock {$\mathscr{V}$}-{C}at is locally presentable or locally bounded if
  {$\mathscr{V}$} is so.
\newblock {\em Theory Appl. Categ.}, 8:555--575, 2001.

\bibitem{Kon_Soi_NCG1}
M.~{Kontsevich} and Y.~{Soibelman}.
\newblock {Notes on A-infinity algebras, A-infinity categories and
  non-commutative geometry. I}.
\newblock June 2006.
\newblock http://arxiv.org/abs/math/0606241.

\bibitem{Kriz_May_OAMM}
Igor K{\v{r}}{\'{\i}}{\v{z}} and J.~P. May.
\newblock Operads, algebras, modules and motives.
\newblock {\em Ast\'erisque}, (233):iv+145pp, 1995.

\bibitem{Lack_icons}
Stephen Lack.
\newblock Icons.
\newblock {\em Appl. Categ. Structures}, 18(3):289--307, 2010.

\bibitem{Lei1}
T.~{Leinster}.
\newblock {Basic Bicategories}.
\newblock http://arxiv.org/abs/math/9810017.

\bibitem{Lei3}
T.~{Leinster}.
\newblock {Up-to-Homotopy Monoids}.
\newblock http://arxiv.org/abs/math/9912084.

\bibitem{Leinster_HOHC}
Tom Leinster.
\newblock {\em Higher operads, higher categories}, volume 298 of {\em London
  Mathematical Society Lecture Note Series}.
\newblock Cambridge University Press, Cambridge, 2004.

\bibitem{Linton-cocomplete}
F.~E.~J. Linton.
\newblock Coequalizers in categories of algebras.
\newblock In {\em Sem. on {T}riples and {C}ategorical {H}omology {T}heory
  ({ETH}, {Z}\"urich, 1966/67)}, pages 75--90. Springer, Berlin, 1969.

\bibitem{Lurie_HTT}
Jacob Lurie.
\newblock {\em Higher topos theory}, volume 170 of {\em Annals of Mathematics
  Studies}.
\newblock Princeton University Press, Princeton, NJ, 2009.

\bibitem{Lyub_A_inf}
Volodymyr Lyubashenko.
\newblock Category of {$A_\infty$}-categories.
\newblock {\em Homology Homotopy Appl.}, 5(1):1--48, 2003.

\bibitem{Mac}
Saunders Mac~Lane.
\newblock {\em Categories for the working mathematician}, volume~5 of {\em
  Graduate Texts in Mathematics}.
\newblock Springer-Verlag, New York, second edition, 1998.

\bibitem{Quillen_HA}
Daniel~G. Quillen.
\newblock {\em Homotopical algebra}.
\newblock Lecture Notes in Mathematics, No. 43. Springer-Verlag, Berlin, 1967.

\bibitem{Roig_bifib}
Agust{\'{\i}} Roig.
\newblock Model category structures in bifibred categories.
\newblock {\em J. Pure Appl. Algebra}, 95(2):203--223, 1994.

\bibitem{Sch-Sh-Algebra-module}
Stefan Schwede and Brooke~E. Shipley.
\newblock Algebras and modules in monoidal model categories.
\newblock {\em Proc. London Math. Soc. (3)}, 80(2):491--511, 2000.

\bibitem{Simpson_HTHC}
Carlos Simpson.
\newblock {\em Homotopy theory of higher categories}, volume~19 of {\em New
  Mathematical Monographs}.
\newblock Cambridge University Press, Cambridge, 2012.

\bibitem{Smith_unpub}
J.~Smith.
\newblock Combinatorial model categories.
\newblock {\em Unpublished}.

\bibitem{Stanculescu_bifib}
Alexandru Stanculescu.
\newblock Bifibrations and weak factorisation systems.
\newblock {\em Applied Categorical Structures}.

\bibitem{Str}
Ross Street.
\newblock Enriched categories and cohomology.
\newblock {\em Repr. Theory Appl. Categ.}, (14):1--18, 2005.
\newblock Reprinted from Quaestiones Math. 6 (1983), no. 1-3, 265--283
  [MR0700252], with new commentary by the author.

\bibitem{Wo}
Harvey Wolff.
\newblock {$V$}-cat and {$V$}-graph.
\newblock {\em J. Pure Appl. Algebra}, 4:123--135, 1974.

\end{thebibliography}

\end{document}